\documentclass[1 leqno,11pt]{amsart}
\usepackage{amssymb, amsmath,latexsym,amsfonts,amsbsy, amsthm,mathtools,graphicx,,color}
\usepackage{float}
\usepackage{hyperref}

\numberwithin{equation}{section}

\allowdisplaybreaks

\let\al=\alpha

\let\e=\varepsilon

\let\la=\lambda

\let\f=\frac

\let\om=\omega

\let\La=\Lambda

\let\Om=\Omega

\let\na=\nabla
\let\th=\theta
\let\pa=\partial

\def\R{\mathbb R}

\def\no{\noindent}

\def\dv{\mbox{div}}

\newcommand{\beq}{\begin{equation}}
\newcommand{\eeq}{\end{equation}}
\newcommand{\ben}{\begin{eqnarray}}
\newcommand{\een}{\end{eqnarray}}
\newcommand{\beno}{\begin{eqnarray*}}
\newcommand{\eeno}{\end{eqnarray*}}

\newtheorem{theorem}{Theorem}[section]

\newtheorem{lemma}[theorem]{Lemma}
\newtheorem{proposition}[theorem]{Proposition}

\newtheorem{remark}[theorem]{Remark}

\begin{document}

\title{Tollmien-Schlichting waves in the subsonic regime} 

\author[N. Masmoudi]{Nader Masmoudi}
\address{NYUAD Research Institute, New York University Abu Dhabi, Saadiyat Island, Abu Dhabi, P.O. Box 129188, United Arab Emirates, Courant Institute of Mathematical Sciences, New York University, 251 Mercer Street New York, NY 10012 USA}
\email{masmoudi@cims.nyu.edu}

\author[Y. Wang]{Yuxi Wang}
\address{School of Mathematics, Sichuan University, Chengdu 610064, China}
\email{wangyuxi@scu.edu.cn}

\author[D. Wu]{Di Wu}
\address{School of Mathematics, South China University of Technology, Guangzhou, 510640,  P. R. China}
\email{wudi@scut.edu.cn}

\author[Z. Zhang]{Zhifei Zhang}
\address{School of Mathematical Sciences, Peking University, 100871, Beijing, P. R. China}
\email{zfzhang@math.pku.edu.cn}

\maketitle

\begin{abstract}
The Tollmien-Schlichting (T-S) waves play a key role in the early stages of boundary layer transition. In a breakthrough work \cite{GGN-DMJ}, Grenier, Guo and Nguyen gave the first rigorous construction of the T-S waves of temporal mode for the incompressible fluid. Yang and Zhang \cite{YZ} recently made an important contribution by constructing the compressible T-S waves of temporal mode for certain boundary layer profiles with Mach number $m<\f1{\sqrt 3}$. In this paper,  we construct the T-S waves of both temporal mode and spatial mode to the linearized compressible Navier-Stokes system around the boundary layer flow in the whole subsonic regime $m<1$, including the Blasius profile. Our approach is based on a novel iteration scheme between the quasi-incompressible and quasi-compressible systems, with a key ingredient being the solution of an Orr-Sommerfeld type equation using a new Airy-Airy-Rayleigh iteration  instead of Rayleigh-Airy iteration introduced by Grenier, Guo and Nguyen. We believe the method developed in this work can be applied in solving other related problems for subsonic flows.

 \end{abstract}

\section{Introduction}

In this paper, we are concerned with the stability problem for an isentropic compressible flow over a flat plate at high Reynolds number regime $Re\gg1 $. The motion of the fluid is described by the non-dimensional compressible Navier-Stokes system
\begin{align}\label{eq:NS}
\left\{
\begin{aligned}
&\pa_t\rho^\nu+\na\cdot(\rho^\nu U^\nu)=0,\\
&\rho^\nu\big(\pa_tU^\nu+U^\nu\cdot \na U^\nu\big)-\mu\nu\Delta U^\nu-\lambda\nu\na \na\cdot U^\nu+\na P(\rho^\nu)=0,\\
&U^\nu|_{y=0}=0.
\end{aligned}
\right.
\end{align}
Here $\rho^\nu$ is the density, $U^\nu$ is the velocity, and  $P(\rho^\nu)$ is the pressure.  Here $\mu, \lambda$ are the viscosity coefficients and $\nu>0$ is propositional to  $Re^{-1}$.

At high Reynolds number, the behavior of the flow away from the boundary can be described by the Euler equations and the behavior near the boundary can be  described by the Prandtl equation. For the outer flow $\rho^e=1, U^e=(1,0)$, the behavior of the solution could be described by the Blasius solution, which is a self-similar solution of the steady Prandtl equation and takes the form
\begin{align}\nonumber
 (u_B,v_B)=\Big(f'(\zeta),\frac{1}{2\sqrt{x+x_0}}\{\zeta f'(\zeta)-f(\zeta)\}\Big),
\end{align}
where $\zeta=\frac{y}{\sqrt{\nu(x+x_0)}}$ with $x_0>0$ as a free parameter, and the profile $f$ satisfies
\beno
\frac{1}{2}ff''+f'''=0,\quad  f(0)=f'(0)=0,\quad \lim_{\zeta\to +\infty}f'(\zeta)=1.
\eeno
The Blasius solution has been experimentally confirmed with remarkable accuracy as a basic flow over a flat plate \cite{Sch} . 
In a remarkable work \cite{IM}, Iyer and Masmoudi justify the Prandtl expansion with the Blasius solution as a leading order approximation of the steady incompressible Navier-Stokes equations. In a recent work \cite{GWZ}, Guo, Wang and Zhang prove that the Blasius solution is also a stable solution of the unsteady Prandtl equation. This gives another indication why the Blasius solution is reasonable as a basic flow. 
See \cite{GI, Iyer, IM-Ars, WZ} and references therein for more related works. 

Considering a local region in the streamwise direction, the flow $U_s^\nu=\Big(U_s\big(\f y {\sqrt{\nu}}\big),0\Big)$ with a Blasius profile $U_s(Y)$ could be seen as an approximation of the Blasius solution. To predict the boundary layer transition, a key problem is to study the stability of the flow under the perturbations. To this end,  we consider  the linearized compressible Navier-Stokes system(CNS) around the boundary layer flow $(1,U_s^\nu)$ on $\Omega=[0,L]\times\mathbb R_+$:
\begin{align}\label{eq:LCNS}
	\left\{
	\begin{aligned}
		&\partial_t\varrho+U_s\partial_x\rho+\nabla\cdot U=0,\\
		&\partial_tU+U_s\partial_x U+m^{-2}\nabla p+\partial_y U_s(\frac{y}{\sqrt{\nu}})(U_2,0)-\nu\Delta U-\lambda\nu\nabla(\nabla\cdot U)+\rho\partial_Y^2U_s=0,\\
		&U|_{y=0}=\lim_{y\to+\infty} U(y)=0.
	\end{aligned}
	\right.
\end{align}
Here $m=\f 1 {\sqrt{P'(1)}}$ is the Mach number and $\mu$ is taken to be 1 for convenience. 

Our goal is to construct  the solution to the linearized CNS \eqref{eq:LCNS} with the following form 
\begin{align*}
	(\varrho,U_1, U_2)(t,x,y)=(\rho,u,v)(Y)e^{\mathrm i\alpha(X-c\tau)}+\mathrm c.\mathrm c.,
\end{align*}
where
\begin{align*}
	\tau=\frac{t}{\sqrt{\nu}},\quad X=\frac{x}{\sqrt{\nu}},\quad Y=\frac{y}{\sqrt{\nu}}.
\end{align*}
Then $(\rho(Y),u(Y),v(Y))$ satisfies the following steady system
\begin{align}\label{eq:LCNS-Y}
	\left\{
	\begin{aligned}
		&\mathrm i\alpha(U_s-c)\rho+\mathrm{div}_\alpha(u,v)=0,\\
		&\sqrt{\nu}(\partial_Y^2-\alpha^2)u+\mathrm i\alpha\lambda\sqrt{\nu}\mathrm{div}_\alpha(u,v)-\mathrm i\alpha(U_s-c)u\\
		&\qquad\quad-(\mathrm i\alpha m^{-2}+\sqrt{\nu}\partial_Y^2 U_s)\rho-v\partial_Y U_s=0,\\
		&\sqrt{\nu}(\partial_Y^2-\alpha^2)v+\lambda\sqrt{\nu}\partial_Y\mathrm{div}_\alpha(u,v)-\mathrm i\alpha(U_s-c)v-m^{-2}\partial_Y\rho=0,
	\end{aligned}
	\right.
\end{align}
where $\mathrm{div}_\alpha(u,v)=\mathrm i\alpha u+\partial_Yv$ and $(u,v)$ satisfies the non-slip boundary condition
\begin{align}\label{BC: u(0)=v(0)=0}
	u(0)=v(0)=\lim_{Y\to\infty}u(Y)=\lim_{Y\to\infty}v(Y)=0.
\end{align}
 For convenience, we define the linear operator in \eqref{eq:LCNS-Y} by $\mathcal{L}$.\smallskip
 
  In this paper, we will seek two kinds of important unstable modes, which are so called Tollmien-Schlichting (T-S) waves.
 
 \begin{itemize}
 
 \item[(1)] {\bf Temporal mode}: $\al\in \R$ and $c\in \mathbb{C}$. When $\al>0$ and $\mathrm {Im} c=c_i>0$, this kind of solution will grow in time;
 
 \item[(2)] {\bf Spatial mode}: $\al\in \mathbb{C}$ and $\alpha c\in \R$. When $\mathrm{Im} \al=\al_i<0$, this kind of solution will grow along the streamwise direction $x$. 
 So, it is more relevant in physics and the growth rate $-\al_i$ is a key quantity of $e^N$ criterion in the boundary layer transition.
 \end{itemize}

 For the incompressible fluid,  the construction of T-S waves started with a series of  classical works by Heisenberg, Tollmien, Schlichting, C. C. Lin etc. \cite{Lin, DR}.  Based on the asymptotic analysis,  they found that there exist lower and upper marginal stability branches $\al_{low}(\nu),\al_{up}(\nu)$ depending on the profile $U_s(y)$ so that  when $\al\in (\al_{low},\al_{up})$, the Orr-Sommerfeld equation exist unstable solutions with $\mathrm{Im}(c)>0$. In a breakthrough work \cite{GGN-DMJ}, Grenier, Guo and Nyugen constructed the T-S waves of temporal mode for $\al$ lying in an interval [a,b] with $\al_{low}\ll a<b\ll \al_{up}$ for generic flows $U_s(y)$. Moreover, they develop a robust Rayleigh-Airy iteration method to solve the Orr-Sommerfeld equation. Let us refer to \cite{GMM-duke, GMM-arxiv, GM, GZ, CWZ1, CWZ2, GN-arxiv, GN} and references therein for more works on nonlinear stability(instability) of boundary layer flows.  Let us refer to \cite{Rom, GGN-AM, AH, CWZ-CPAM, CWZ-CMP}  for the linear stability results about the channel flow or pipe flow.

 In a recent important work \cite{YZ},  Yang and Zhang give a first rigorous construction of  the compressible T-S waves of temporal mode for $\al$ near the lower branch $\al_{low}$ when the Mach number $m<\f 1 {\sqrt{3}}$ for a class of boundary layer profiles satisfying strong structure assumptions, which excludes the Blasius profile. Their proof does not use the Rayleigh-Airy iteration method developed in \cite{GGN-DMJ}. Instead, they developed a quai-compressible and Stokes iteration. \smallskip

Throughout this paper, we make the following structure assumptions on the profile $U_s(Y)$:
\begin{align}\label{eq:S-A}
\begin{split}
	&U_s(0)=0,\quad\lim_{Y\to\infty}U_s(Y)=1,\\
	&0<U_s<1, \quad \text{ for } 0<Y<\infty,\quad U_s'(0)>0,\\
	&\sup_{Y\geq0}|e^{\eta_0 Y}\partial_Y^k(U_s(Y)-1)|<+\infty,\quad k=0,1,\dots,4.
\end{split}
\end{align}

Our main result is stated as follows.

 \begin{theorem}\label{thm:main}
	Let $m<1$. Suppose that $U_s(Y)$ satisfies the structure assumption \eqref{eq:S-A}. Then there exist $0<\nu(m)\ll1$ and $1\ll A_0<B_0$ such that for any $\nu\leq \nu(m)$ the following statement holds true. For any $\alpha\in\mathbb G=\big\{\alpha\in \mathbb C:\alpha_r=A\nu^\f18 \text{ with }A\in(A_0,B_0), |\alpha_i|\leq \gamma\alpha_r \text{ with }\gamma\ll1 \big\}$, we can find $c(\alpha)$ such that there exists a solution $(\varrho,U_1, U_2)$ to  the system \eqref{eq:LCNS} with the form of 
	\begin{align*}
	(\varrho,U_1, U_2)(t,x,y)=(\rho,u,v)(\frac{y}{\sqrt{\nu}})e^{-\mathrm i c\alpha\nu^{-\f12}t}e^{\mathrm i\alpha\nu^{-\f12}x}+\mathrm c.\mathrm c.,
    \end{align*}
    where $(\rho(Y),u(Y),v(Y))\in W^{1,\infty}\times W^{2,\infty}\times W^{2,\infty}$ is a solution to the system \eqref{eq:LCNS-Y} and  \eqref{BC: u(0)=v(0)=0}. Moreover, such $c,\alpha$ satisfy the following properties
	\begin{align*}
	c_r=U'_s(0)^{-1}(1-m^2)^{-\f12}\al_r+o(\al),\quad	0<c_i-U_s'(0)^{-1}(1-m^2)^{-\f12}\alpha_i\sim A^{-1}\nu^\f18.
	\end{align*}
In particular, for any $\alpha_r=A\nu^\f18$ with $A\in(A_0,B_0)$, we have the following results:
	\begin{itemize}
		\item there exists  $\alpha^0=\alpha_r+\mathrm i\alpha^0_i\in\mathbb G$  such that 
		      \begin{align}\label{eq:S-mode}		      	
		      \alpha^0c(\alpha^0)\in\mathbb R,\quad \alpha_i^0<0\text{ and }|\alpha_i^0|\sim\nu^\f18.
		      \end{align}

		\item there exists  $c(\alpha_r)$ such that    
		      \begin{align}\label{eq:T-mode}
		      	c_i(\alpha_r)>0 \text{ and }|c_i(\alpha_r)|\sim\nu^\f18.
		      \end{align}
		     
		      \end{itemize}
\end{theorem}

\medskip

Let us give some remarks on our results.

\begin{enumerate}

\item  This is the first time that  a mixed type unstable mode has been constructed, which will grow in both  space and time: 	
\begin{align}\nonumber
	\left|(\varrho,U_1, U_2)(t,x,y)\right|\sim e^{C\nu^{-\f14}t}e^{CA\nu^{-\f38}U'_s(0)^{-1}(1-m^2)^{-\f12}\al_i t}e^{-\nu^{-\f12}\al_i x}.
	\end{align}	
The special case \eqref{eq:S-mode} is related to the T-S waves of spatial mode, which will grow in the space:
\begin{align}\nonumber
|(\varrho,U_1,U_2)(t,x,y)|\sim e^{C\nu^{-\f38}x}.
\end{align}
The special case \eqref{eq:T-mode} is related to  the T-S waves of temporal mode, which will grow in the time:
 \begin{align}\nonumber
|(\varrho,U_1,U_2)(t,x,y)|\sim e^{c\nu^{-\f14}t}
\end{align}

	 \item  The Blasius  profile satisfies the structure assumptions \eqref{eq:S-A}.
	 
	 \item All the estimates in our paper are uniformly in $m\in(0,1)$. When $m\to 0$, we can construct the T-S waves of both temporal mode and spatial mode for the incompressible flow, where temporal mode was analyzed in \cite{GGN-DMJ} by Grenier-Guo-Nguyen. The vanishing Mach number limit for temporal mode was also obtained in the work of Yang and Zhang \cite{YZ}.

	 \item  Our method could be used to construct the T-S waves for $\al\in [a,b]$ with $\al_{low}\ll a<b\ll \al_{up}$ as in \cite{GGN-DMJ}.
	  
\item For the supersonic flow $m>1$, there exist inviscid unstable modes, which plays more dominant role in the boundary layer transition \cite{LL}. This will be conducted in future works.
\end{enumerate}

\section{Sketch of the proof}

The key ingredient of the proof is to construct a non-trivial solution to the system \eqref{eq:LCNS-Y}.
To this end,  we will introduce a new iteration scheme between the quasi-incompressible system \eqref{eq:LCNS-OS} 
and the quasi-compressible system\eqref{eq:LCNS-Stokes} and start this iteration scheme with the slow or fast solution to the quasi-incompressible  system \eqref{eq:LCNS-OS}. See the beginning of section 8 for the detailed description of iteration scheme. \smallskip

First of all, motivated by \cite{YZ}, we introduce the following quasi-incompressible system
  \begin{align}\label{eq:LCNS-OS}
	\left\{
	\begin{aligned}
		&\mathrm i\alpha(U_s-c)\rho_{os}+\mathrm i\alpha u_{os}+\partial_Y v_{os}=0,\\
		&\sqrt{\nu}(\partial_Y^2-\alpha^2)[u_{os}+(U_s-c)\rho_{os}]-\mathrm i\alpha(U_s-c)u_{os}-v_{os}\partial_Y U_s-\mathrm i\alpha m^{-2}\rho_{os}=f_u,\\
		&\sqrt{\nu}(\partial_Y^2-\alpha^2)v_{os}-\mathrm i\alpha(U_s-c)v_{os}-m^{-2}\partial_Y\rho_{os}=f_v.
	\end{aligned}
	\right.
\end{align}
Compared with \eqref{eq:LCNS-Y}, the system \eqref{eq:LCNS-OS} mainly drop all the divergence terms in the equations of velocity in \eqref{eq:LCNS-Y} and add some terms such that we can derive an Orr-Sommerfeld type equation. Indeed, let us introduce the generalized stream function $\phi(Y)$ defined by
\begin{align}\nonumber
	\partial_Y\phi(Y)=u_{os}+(U_s-c)\rho_{os}, \quad-\mathrm i\alpha\phi(Y)=v_{os}.
\end{align}
Then the system \eqref{eq:LCNS-OS} can be reduced to an Orr-Sommerfeld type system
\begin{align}\label{eq:GOS}
	\left\{
	\begin{aligned}
		&\varepsilon\Lambda(\partial_Y^2-\alpha^2)\phi-(U_s-c)\Lambda\phi+\partial_Y(A^{-1}\partial_Y U_s)\phi=\Om(f_u, f_v),\\
		&\rho=m^2A^{-1}(Y)\left(\varepsilon(\partial_Y^2-\alpha^2)\partial_Y\phi-(U_s-c)\partial_Y\phi+\partial_Y U_s\phi-\frac{f_u}{\mathrm i\alpha}\right),
	\end{aligned}
	\right.
\end{align}
where 
\begin{align}
\varepsilon=\f{\sqrt{\nu}}{\mathrm i\al},\quad \Lambda=\partial_Y(A^{-1}\partial_Y)-\alpha^2,\quad A(Y)=1-m^2(U_s-c)^2,
\end{align}
and 
\begin{align}
\Om(f_u, f_v)=&\f{1}{\mathrm i\al}\pa_Y(A^{-1}f_u)-f_v.\nonumber
\end{align}
When $m=0$, the system \eqref{eq:GOS} is reduced to the classical Orr-Sommerfeld equation. See section 6 for the details.

To control the compressible part, we introduce a quasi-compressible system as follows
\begin{align}\label{eq:LCNS-Stokes}
	\left\{
	\begin{aligned}
		&\mathrm i\alpha(U_s-c)\rho_{st}+\mathrm{div}_\alpha(u_{st},v_{st})=0,\\
		&\sqrt{\nu}(\partial_Y^2-\alpha^2)u_{st}+\mathrm i\alpha\lambda\sqrt{\nu}\mathrm{div}_\alpha(u_{st},v_{st})\\
		&\qquad\quad-\mathrm i\alpha(U_s-c)u_{st}-(\mathrm i\alpha m^{-2}+\sqrt{\nu}\partial_Y^2 U_s)\rho_{st}=q_1,\\
		&\sqrt{\nu}(\partial_Y^2-\alpha^2)v_{st}+\lambda\sqrt{\nu}\partial_Y\mathrm{div}_\alpha(u_{st},v_{st})-\mathrm i\alpha(U_s-c)v_{st}\\
		&\qquad\quad-\int_Y^\infty \mathrm i\al \pa_Y U_s v_{st} dY'-m^{-2}\partial_Y\rho_{st}=q_2.
	\end{aligned}
	\right.
\end{align}
Compared with the system \eqref{eq:LCNS-Y}, the main difference in \eqref{eq:LCNS-Stokes} is that we remove non-local term $-\pa_Y U_s v$ in the second equation and add a new term $-\int_Y^\infty i\al \pa_Y U_s v_{st} dY'$ in the third equation. With theses modifications, we can derive a good equation for the density by taking divergence of the equations of velocity. Indeed,  the density $\rho_{st}$ satisfies 
 \begin{align*}
-\mathrm i\al \sqrt\nu (1+\la)(\pa_Y^2-\al^2)((U_s-c)\rho_{st})&-\al^2(U_s-c)^2 \rho_{st}-m^{-2}(\pa_Y^2-\al^2)\rho_{st}\\
\nonumber
&-\mathrm i\al \sqrt\nu \pa_Y^2 U_s\rho_{st}=\mathrm i\al q_1+\pa_Yq_2.
\end{align*}
Furthermore, the density equation can be written as 
\begin{align*}
	\begin{split}
			(\partial_Y^2-\beta^2)\rho_{st}=&\alpha^2(A(Y)-A_{\infty})\rho_{st}-\mathrm i m^2\al \sqrt\nu (1+\la)(\pa_Y^2-\al^2)((U_s-c)\rho_{st})\\
	&-\mathrm i\alpha m^2\sqrt{\nu}\partial_Y^2U_s\rho_{st}-\mathrm i\alpha m^2q_1-m^2\partial_Y q_2,
	\end{split}
\end{align*}
where 
\begin{align*}
\beta=\al A_{\infty}^\f12,\quad A_\infty=\lim_{Y\to +\infty} A(Y)=1-m^2(1-c)^2.
\end{align*}
When $m<1$, this could be viewed as an elliptic equation. Once we obtain the density $\rho_{st}$, thus $\mathrm{div}_\alpha(u_{st},v_{st})$, then $u_{st}$ can be recovered by solving the following Airy type equation
 \begin{align*}
	\varepsilon(\partial_Y^2-\alpha^2)u_{st}-(U_s-c)u_{st}=(\mathrm i\alpha^{-1})q_1+\big(m^{-2}+\varepsilon\partial_Y^2U_s+\mathrm i\lambda\sqrt{\nu}\alpha(U_s-c)\big)\rho_{st}.
\end{align*}
The vertical velocity $v_{st}$ can be recovered by solving the following equation
   \begin{align*}
&(U_s-c)v_{st}+\int_Y^\infty \pa_Y (U_s-1) v_{st} dY'\\
\nonumber
&=\e(\partial_Y^2-\alpha^2)v_{st}+\lambda\e\partial_Y\mathrm{div}_\alpha(u_{st},v_{st})-m^{-2}(\mathrm i\al)^{-1}\partial_Y\rho_{st}-(\mathrm i\al)^{-1}q_2,
\end{align*}
which can be written as 
\begin{align*}
(1-c+\e\al^2)v_{st}=&\int_Y^\infty  (U_s-1) \pa_Yv_{st} dY'+\e\pa_Y^2 v_{st}+\lambda\e\partial_Y\mathrm{div}_\alpha(u_{st},v_{st})\\
&-m^{-2}(\mathrm i\al)^{-1}\partial_Y\rho_{st}-(\mathrm i\al)^{-1}q_2.
\end{align*}
See section 7 for the details. 
\smallskip

Most of the paper will be devoted to solving the Orr-Sommerfeld type system:
\begin{align}\label{eq:OS-F}
\varepsilon\Lambda(\partial_Y^2-\alpha^2)\phi-(U_s-c)\Lambda\phi+\partial_Y(A^{-1}\partial_Y U_s)\phi=F.
\end{align}
In the incompressible case, Grenier, Guo and Nguyen \cite{GGN-DMJ} introduce a robust Rayleigh-Airy iteration scheme to solve the Orr-Sommerfeld equation. However, this iteration scheme does not work in the compressible case due to the fact that the commutator $[\Lambda,\partial_Y^2]$ will lead to one loss of the derivative. Note that $[\Lambda,\partial_Y^2]=0$ when $m=0$. To overcome this difficulty, we introduce a new Airy-Airy-Rayleigh
iteration scheme. See the beginning of section 5 for the details. To carry out this iteration scheme, we solve the Rayleigh type equation conducted in section 3:
\begin{align*}
	(U_s-c)\Lambda\varphi-\partial_Y(A^{-1}\partial_YU_s)\varphi=F.
\end{align*}
The proof is based on an iteration argument in a well-chosen functional space. 
In section 4, we solve the Airy type equation
\begin{align*}
\varepsilon(\partial_Y^2-\alpha^2)\big(\partial_Y(A(Y)^{-1}\partial_Y)-\alpha^2\big)\psi-(U_s-c)\big(\partial_Y(A(Y)^{-1}\partial_Y)-\alpha^2\big)\psi =F.
\end{align*}
To solve the Airy type equation, we introduce a modified Langer transform, which is the key to avoid using the analyticity of $U_s(Y)$. One of key ingredients is the estimate of the Green function related to the approximate Airy equation, see Lemma \ref{lem:A1A2}. 
Our approach is different from Yang and Zhang \cite{YZ}. They made strong assumptions about the structure of boundary layer profiles, excluding the Blasius flow, to obtain energy estimates for the Airy-type equation \eqref{eq:OS-F} with Mach number $m<1/\sqrt{3}$. In contrast, our method provides pointwise estimates for $\phi$ in the entire subsonic regime, and includes the Blasius flow by relaxing assumptions on the boundary layer profiles.

To construct the T-S waves, we construct the slow mode $(\rho^s,u^s,v^s)$ and fast mode $(\rho^f,u^f,v^f)$ to the system \eqref{eq:LCNS-Y} with the boundary condition
\begin{align*}
	\lim_{Y\to\infty}u^s(Y)=\lim_{Y\to\infty}v^s(Y)=\lim_{Y\to\infty}u^f(Y)=\lim_{Y\to\infty}v^f(Y)=0.
\end{align*}
See section 8 for the construction, which is based on the iteration  between the system \eqref{eq:LCNS-OS} 
and the system \eqref{eq:LCNS-Stokes}.

To match the non-slip boundary condition,  we will find the solution $(\rho, u, v)$ of the system \eqref{eq:LCNS-Y}  with the form
\begin{align*}
	(\rho,u,v)=C_s(\rho^s,u^s,v^s)+C_f(\rho^f,u^f,v^f),
\end{align*}
where
\begin{align*}
	C_su^s(0)+C_fu^f(0)=0\text{ and }C_sv^s(0)+C_fv^f(0)=0.
\end{align*}
The existence of $(C_s,C_f)$ is guaranteed by the following dispersion relation:
\begin{align*}
	\frac{u^s(0)}{v^s(0)}=\frac{u^f(0)}{v^f(0)}.
\end{align*}
By our construction, we have an explicit asymptotic expansion for the main parts of the slow mode and fast mode. 
In a suitable regime of $(\al, c)$, we can find a solution of the dispersion relation via an implicit function theorem. 
See section 9 for the details. \smallskip

Let us conclude this section by the following fact: for $m\in(0,1)$, $A(Y)$ has upper and below bound,  i.e.,  there exists $C>1$ such that
\begin{align}\label{est: A}
0<C^{-1}\leq A(Y)\leq C\quad \text{ for any } Y\geq 0.
\end{align}

\section{The Rayleigh type equation}
As we mentioned in section 2, to solve the Orr-Sommerfeld type equation \eqref{eq:OS-F}, we need to perform an Airy-Airy-Rayleigh iteration. In this section, we first solve the following Rayleigh type equation
\begin{align}\label{eq:Ray-non}
	(U_s-c)\Lambda\varphi-\partial_Y(A^{-1}\partial_YU_s)\varphi=F,
\end{align}
where $\Lambda=\partial_Y(A^{-1}\partial_Y)-\alpha^2$ with $A(Y)=1-m^2(U_s-c)^2$.
When $m=0$, this is reduced to the classical Rayleigh equation. For convenience, we denote
\begin{align}\nonumber
Ray[\cdot]=(U_s-c)\La-\partial_Y(A^{-1}\partial_YU_s).
\end{align}

We construct the solution to the Rayleigh equation in the following functional spaces:
\begin{align}\label{def: Y_theta}
	\|f\|_{\mathcal{Y}_\theta}=&|\e|^\f23\|(U_s-c)\pa_Y^4 f\|_{L^\infty_{\theta}}+|\e|^\f13\|(U_s-c)\pa_Y^3 f\|_{L^\infty_{\theta}}+\|(U_s-c)\pa_Y^2 f\|_{L^\infty_{\theta}}\\\nonumber
&+\|\pa_Y f\|_{L^\infty_{\theta}}+\| f\|_{L^\infty_{\theta}},
\end{align}
and
\begin{align}\label{def: Y_bar-theta}
	\|f\|_{\tilde{\mathcal Y}_\theta}=\|(U_s-c)\partial_Y^2 f\|_{L^\infty_\theta}+\|\pa_Y f\|_{L^\infty_{\theta}}+\| f\|_{L^\infty_{\theta}},
\end{align}
where
\begin{align*}
	\|f\|_{L^\infty_\theta}=\|e^{\th Y} f\|_{L^\infty}.
\end{align*}

We denote
\begin{align}\label{def:H1set}
	\mathbb H_1=\big\{(\alpha,c)\in\mathbb C^2:(|\alpha|+|c|)|\log c_i|\ll1,\alpha_r\geq C_0|\alpha_i|,c_i>0\big\}.
\end{align}

In this section, we always assume that $(\al, c)\in \mathbb H_1$.

\subsection{Non-homogeneous Rayleigh equation}

In this subsection, we construct the solution to \eqref{eq:Ray-non} in the case when $e^{\theta Y}F\in L^\infty(\mathbb R_+)$, where $\th>0$ is a universe constant. Such results will be applied to construct a non-trivial solution to the homogeneous Rayleigh type equation \eqref{eq:homo-ray} and perform the Airy-Airy-Rayleigh iteration in section 5.
\smallskip

According to the definition of $\Lambda$, the equation \eqref{eq:Ray-non} can be written as
\begin{align*}
	\partial_Y\left(A^{-1}(U_s-c)^2\partial_Y\left(\frac{\varphi_{non}}{U_s-c}\right) \right)=F+\alpha^2(U_s-c)\varphi_{non}.
\end{align*}
We use the method of iteration to construct a solution to \eqref{eq:Ray-non}. To this end, we define 
\begin{align}\label{def: varphi^(0)_non}
	\varphi^{(0)}_{non}(Y)=(U_s-c)\int_Y^{+\infty}\frac{A(Y')\int_{Y'}^{+\infty}F(Z)dZ}{(U_s-c)^2}dY',
\end{align}
which satisfies 
\begin{align}\label{eq:Ray-non-app-d}
	\partial_Y\left(A^{-1}(U_s-c)^2\partial_Y\left(\frac{\varphi^{(0)}_{non}}{U_s-c}\right) \right)=F.
\end{align}
For any $j\geq 0$, we define 
\begin{align}\label{def: varphi^(j+1)_non}
	\varphi_{non}^{(j+1)}(Y)=\alpha^2(U_s-c)\int_Y^{+\infty}\frac{A(Y)\int_{Y'}^{+\infty}(U_s(Z)-c)\varphi_{non}^{(j)}(Z)dZ}{(U_s-c)^2}dY',
\end{align}
which satisfies 
\begin{align*}
	Ray[\varphi_{non}^{(j+1)}]=\alpha^2(U_s-c)(\varphi_{non}^{(j)}-\varphi_{non}^{(j+1)}).
\end{align*}
Formally, $\varphi_{non}=\sum_{j=0}^\infty\varphi_{non}^{(j)}$ is a solution to \eqref{eq:Ray-non}.

\begin{lemma}\label{lem: integral}
Suppose $|c|\ll 1$ and $c_i\neq0$. Let $Y_c$ satisfy $U_s(Y_c)-c_r=0$ and $Y, Y_0$ be constants such that $0\leq Y<Y_c<Y_0$. Then it holds that
\begin{align*}
\Big|\int_{Y}^{Y_0}\f{G(Y')}{(U_s-c)^2}dY'-\f{G(Y)}{U_s'(Y_c)(U_s(Y)-c)}\Big|
\leq C(\|G\|_{L^\infty}+\|\pa_YG\|_{L^\infty})|\log |c_i||.
\end{align*}

\end{lemma}
\begin{proof}
Thanks to $1=\f{U_s'(Y)}{U_s'(Y_c)}+\f{U_s'(Y_c)-U_s'(Y)}{U_s'(Y_c)}$, we get by integration by parts that
\begin{align*}
\int_{Y}^{Y_0}\f{G(Y')}{(U_s-c)^2}dY'=&-\f{1}{U_s'(Y_c)}\int_{Y}^{Y_0} G(Y') d\Big(\f{1}{U_s-c}\Big)+\f{1}{U_s'(Y_c)}\int_{Y}^{Y_0}\f{(U_s'(Y_c)-U_s')G(Y')}{(U_s-c)^2}dY'\\
=&\f{1}{U_s'(Y_c)}\f{G(Y)}{U_s(Y)-c}-\f{1}{U_s'(Y_c)}\f{G(Y_0)}{U_s(Y_0)-c}+\f{1}{U_s'(Y_c)}\int_{Y}^{Y_0}\f{\pa_{Y'}G(Y')}{U_s-c} dY'\\
&+\f{1}{U_s'(Y_c)}\int_{Y}^{Y_0}\f{(U_s'(Y_c)-U_s')G(Y')}{(U_s-c)^2}dY'\\
=&\f{1}{U_s'(Y_c)}\f{G(Y)}{U_s(Y)-c}+I_1+I_2+I_3.
\end{align*}

Due to $|U_s(Y_0)-c|\geq C^{-1}>0$ for $Y_0>Y_c$, we infer that 
\begin{align*}
|I_1|\leq C|G(Y_0)|\leq C\|G\|_{L^\infty}.
\end{align*}
 Thanks to
\begin{align*}
\int_{Y}^{Y_0}\f{1}{|U_s-c|} dY'\leq C\int_{Y}^{Y_0}\f{1}{|Y-Y_c|+|c_i|}dY'\leq C|\log |c_i||,
\end{align*}
we have
\begin{align*}
|I_2|\leq C\|\pa_YG\|_{L^\infty}\int_{Y}^{Y_0}\f{1}{|U_s-c|} dY'\leq C\|\pa_YG\|_{L^\infty}|\log |c_i||,\\
|I_3|\leq C\|G\|_{L^\infty}\int_{Y}^{Y_0}\Big|\f{U_s'(Y_c)-U_s'}{(U_s-c)^2}\Big| dY'\leq C\|G\|_{L^\infty}|\log |c_i||,
\end{align*}
by using the fact $\Big|\f{U_s'(Y_c)-U_s'}{(U_s-c)^2}\Big|\leq \f{C}{|U_s-c|} $. Then the lemma follows.
\end{proof}

\begin{proposition}\label{pro: phi_non}
If  $e^{\theta Y}F\in L^\infty$, then there exists a solution $\varphi_{non}\in \tilde{\mathcal Y}_\theta$ to \eqref{eq:Ray-non} with 	
\begin{align*}
&\|\varphi_{non}\|_{\tilde{\mathcal Y}_\theta}\leq C|\log c_i|\|e^{\theta Y}F\|_{L^\infty}.
\end{align*}
In addition, if $e^{\theta Y} \partial_Y^jF\in L^\infty$ for $j=1,2$ and $c_i\geq c_0|\e|^\f13>0$, then $\varphi_{non}\in W^{4,\infty}$ satisfies
\begin{align*}
	&|\e|^\f13|e^{\theta Y}(U_s-c)\partial_Y^3\varphi_{non}(Y)|\leq C|\log c_i|\|e^{\theta Y}F\|_{L^\infty}+C|\e|^\f13\|e^{\theta Y}\pa_YF\|_{L^\infty},\\
&|\e|^\f23|e^{\theta Y}(U_s-c)\partial_Y^4\varphi_{non}(Y)|\leq C|\log c_i|\|e^{\theta Y}F\|_{L^\infty}+C|\e|^\f13\|e^{\theta Y}\pa_YF\|_{L^\infty}+C|\e|^\f23\|e^{\theta Y}\pa_Y^2F\|_{L^\infty}.
\end{align*}
As a  consequence, we obtain 
\begin{align*}
	\|\varphi_{non}\|_{\mathcal Y_\theta}\leq C|\log c_i|\|e^{\theta Y}F\|_{L^\infty}+C|\e|^\f13\|e^{\theta Y}\pa_YF\|_{L^\infty}+C|\e|^\f23\|e^{\theta Y}\pa_Y^2F\|_{L^\infty}.
\end{align*}

\end{proposition}
\begin{proof}
	By our assumption $0<|c|\ll 1$, there exists $Y_0>Y_c$ such that for any $Y\geq Y_0$,
	\begin{align*}
		|U_s(Y)-c|\geq C^{-1}>0.
	\end{align*}
		
Using \eqref{est: A} and \eqref{def: varphi^(0)_non}, we infer that for any $Y\geq Y_0$,
\begin{align*}
|e^{\theta Y}\varphi_{non}^{(0)}(Y)|\leq&C\|e^{\th Y} F\|_{L^\infty}\Big|\int_Y^{+\infty}\int_{Y'}^{+\infty} e^{\th(Y-Y'')}dY'' dY'\Big|\leq C\|e^{\th Y} F\|_{L^\infty}.
\end{align*}
For $0\leq Y\leq Y_0$, we decomposite $\varphi_{non}^{(0)}$ into two parts:
\begin{align*}
\varphi_{non}^{(0)}(Y)=&(U_s-c)\int_{Y_0}^{+\infty}\frac{A(Y')\int_{Y'}^{+\infty}F(Z)dZ}{(U_s-c)^2}dY'+(U_s-c)\int_Y^{Y_0}\frac{A(Y')\int_{Y'}^{+\infty}F(Z)dZ}{(U_s-c)^2}dY'.
\end{align*}
As above, for the first term, we have 
\begin{align}\label{est: varphi^0_non-1}
\Big|e^{\th Y}(U_s-c)\int_{Y_0}^{+\infty}\frac{A(Y')\int_{Y'}^{+\infty}F(Z)dZ}{(U_s-c)^2}dY'\Big|\leq C\|e^{\theta Y}F\|_{L^\infty}.
\end{align}
For the second term, we get by Lemma \ref{lem: integral} with $G(Y)=1$ that
\begin{align*}
\Big|e^{\th Y}(U_s-c)\int_Y^{Y_0}\frac{A(Y')\int_{Y'}^{+\infty}F(Z)dZ}{(U_s-c)^2}dY'\Big|\leq &C\|e^{\th Y} F\|_{L^\infty}\Big|(U_s-c)\int_{Y}^{Y_0} \f{1}{(U_s-c)^2} dY'\Big|\\
\leq& C|\log c_i|\|e^{\th Y} F\|_{L^\infty},
\end{align*}
which along with \eqref{est: varphi^0_non-1} shows that for $Y\geq0$,
	\begin{align}
		|e^{\theta Y}\varphi_{non}^{(0)}(Y)|\leq C|\log c_i|\|e^{\theta Y}F\|_{L^\infty}.
	\end{align}
	
Next we give the estimate of $\pa_Y\varphi_{non}^{(0)}$.	By the definition of $\varphi_{non}^{(0)}$ in \eqref{def: varphi^(0)_non}, we have 
	\begin{align*}
		\partial_Y\varphi_{non}^{(0)}(Y)=U'_s(Y)\int_Y^{+\infty}\frac{A(Y')\int_{Y'}^{+\infty}F(Z)dZ}{(U_s-c)^2}dY'-\frac{A(Y)\int_Y^{+\infty}F(Z)dZ}{(U_s(Y)-c)}.
	\end{align*}
Then for any $Y\geq Y_0$, we have 
	\begin{align*}
		|e^{\theta Y}\partial_Y\varphi_{non}^{(0)}(Y)|\leq C\|e^{\theta Y}F\|_{L^\infty}.
	\end{align*}
For $0\leq Y\leq Y_0$, notice that 
\begin{align}\nonumber
	\begin{split}
		\partial_Y\varphi_{non}^{(0)}(Y)=&U'_s(Y)\int_Y^{Y_0}\frac{A(Y')\int_{Y'}^{+\infty}F(Z)dZ}{(U_s-c)^2}dY'-\frac{A(Y)\int_Y^{+\infty}F(Z)dZ}{(U_s(Y)-c)}\\
		&+U'_s(Y)\int_{Y_0}^{+\infty}\frac{A(Y')\int_{Y'}^{+\infty}F(Z)dZ}{(U_s-c)^2}dY'=I_1+I_2+I_3.
	\end{split}
\end{align}
For $I_3$, we have
\begin{align*}
|I_3|\leq&C\|e^{\theta Y}F\|_{L^\infty}.
\end{align*}
Let $G(Y')=A(Y')\int_{Y'}^{+\infty}F(Z)dZ$. We write
\begin{align*}
I_1=&U'_s(Y)\int_Y^{Y_0}\frac{G(Y')}{(U_s-c)^2}dY'\\
=&U'_s(Y)\Big(\int_Y^{Y_0}\frac{G(Y')}{(U_s-c)^2}dY'-\f{G(Y)}{U_s'(Y_c)(U_s(Y)-c)}\Big)+\f{U'_s(Y)G(Y)}{U_s'(Y_c)(U_s(Y)-c)},\\
I_2=&-\f{G(Y)}{U_s(Y)-c},
\end{align*}
which yield that
\begin{align*}
I_1+I_2=&U'_s(Y)\Bigg(\int_Y^{Y_0}\frac{G(Y')}{(U_s-c)^2}dY'-\f{G(Y)}{U_s'(Y_c)(U_s(Y)-c)}\Bigg)-\Bigg(1-\f{U'_s(Y)}{U_s'(Y_c)}\Bigg)\f{G(Y)}{U_s(Y)-c}.
\end{align*}
By \eqref{est: A}, we have
\begin{align*}
\|G\|_{L^\infty}\leq C\|e^{\theta Y}F\|_{L^\infty},\quad \|\pa_YG\|_{L^\infty}\leq C\|e^{\theta Y}F\|_{L^\infty}.
\end{align*}
Then we apply Lemma \ref{lem: integral} and the fact $\Big|(1-\f{U'_s(Y)}{U_s'(Y_c)})\f{G(Y)}{U_s(Y)-c}\Big|\leq C\|G\|_{L^\infty}$ to deduce that
\begin{align*}
|I_1+I_2|\leq C(\|G\|_{L^\infty}+\|\pa_YG\|_{L^\infty})|\log c_i|\leq C|\log c_i|\|e^{\theta Y}F\|_{L^\infty},
\end{align*}
which along with the estimate of $I_3$ shows that for $0\leq Y\leq Y_0$,
\begin{align*}
|\partial_Y\varphi_{non}^{(0)}(Y)|\leq C|\log c_i|\|e^{\theta Y}F\|_{L^\infty}.
\end{align*}
This proves that for any $Y\geq 0$,
    \begin{align*}
		|e^{\theta Y}\partial_Y\varphi_{non}^{(0)}(Y)|\leq C|\log c_i|\|e^{\theta Y}F\|_{L^\infty}.
	\end{align*}
Then by using the equation \eqref{eq:Ray-non-app-d}, we infer that for any $Y\geq 0$,
    \begin{align*}
    	|e^{\theta Y}(U_s-c)\partial_Y^2\varphi_{non}^{(0)}(Y)|\leq C|\log c_i|\|e^{\theta Y}F\|_{L^\infty}.
    \end{align*}
    
Summing up, we conclude 
\begin{align}
\|\varphi_{non}^{(0)}\|_{\tilde{\mathcal{Y}}_\theta}\leq C|\log c_i|\|e^{\theta Y}F\|_{L^\infty}.\nonumber
\end{align}

Using the formula \eqref{def: varphi^(j+1)_non} and the argument as above, we can show that for $j\geq 0$,
\begin{align*}
\|\varphi_{non}^{(j+1)}\|_{\tilde{\mathcal{Y}}_\theta}&\leq C|\al|^2|\log c_i|\|e^{\theta Y}\varphi_{non}^{(j)}\|_{L^\infty}\\
&\leq C|\al|^2|\log c_i| \|\varphi_{non}^{(j)}\|_{\tilde{\mathcal{Y}}_\theta}\leq C(|\al|^2|\log c_i|)^{j+1}\|\varphi_{non}^{(0)}\|_{\tilde{\mathcal{Y}}_\theta}.
\end{align*}
Since $(\alpha,c)\in\mathbb H_1$ so that $|\al|^2|\log c_i|\leq \f12$, we have
\begin{align}\nonumber
\|\varphi_{non}\|_{\tilde{\mathcal{Y}}_\theta}\leq &\sum_{j=0}^\infty\|\varphi_{non}^{(j)}\|_{\tilde{\mathcal{Y}}_\theta}\leq C\sum_{j=0}^\infty(|\al|^2|\log c_i|)^{j}\|\varphi_{non}^{(0)}\|_{\tilde{\mathcal{Y}}_\theta}\\
\nonumber
\leq& C\|\varphi_{non}^{(0)}\|_{\tilde{\mathcal{Y}}_\theta}\leq C|\log c_i|\|e^{\theta Y}F\|_{L^\infty},
\end{align}
which gives the first statement. 

In order to get the estimate for  $\pa_Y^k\varphi_{non}$, $k=3,4$, we take $\pa_Y$ and $\pa_Y^2$ on \eqref{eq:Ray-non} to have
\begin{align*}
A^{-1}(U_s-c)\pa_Y^3\varphi_{non}=&\pa_YF-\pa_Y((U_s-c)A^{-1})\pa_Y^2\varphi_{non}\\
&-\pa_Y\Big((U_s-c)\pa_Y(A^{-1})\pa_Y \varphi_{non}-\al^2(U_s-c)\varphi_{non}-\pa_Y(A^{-1}\pa_Y U_s)\varphi_{non}\Big),\\
A^{-1}(U_s-c)\pa_Y^4\varphi_{non}=&\pa_Y^2 F-2\pa_Y(A^{-1}(U_s-c))\pa_Y^3\varphi_{non}-\pa_Y^2((U_s-c)A^{-1})\pa_Y^2\varphi_{non}\\
&-\pa_Y^2\Big((U_s-c)\pa_Y(A^{-1})\pa_Y \varphi_{non}-\al^2(U_s-c)\varphi_{non}-\pa_Y(A^{-1}\pa_Y U_s)\varphi_{non}\Big).
\end{align*}
Then we can conclude our result by using the estimate of  $\|\varphi_{non}\|_{\tilde{\mathcal{Y}}_\theta}$, \eqref{est: A} and 
\begin{align}\label{est: |U_s-c|>e^1/3} 
|U_s-c|\geq c_i\geq c_0|\e|^\f13>0.
\end{align}

This completes the proof of the proposition.
\end{proof}

\subsection{Homogeneous Rayleigh equation}
In this subsection, we construct a decay solution to the homogeneous Rayleigh equation:
\begin{align}\label{eq:homo-ray}
	(U_s-c)\Lambda\varphi_{Ray}-\partial_Y(A^{-1}\partial_YU_s)\varphi_{Ray}=0.
\end{align}
Such solution $\varphi_{Ray}$ is the main part of the slow mode of the homogeneous Orr-Sommerfeld equation \eqref{eq:OS-F}.
\smallskip

 We define 
\begin{align}
	\varphi_{Ray}^{(0)}(Y)=2\beta e^{\beta Y}(U_s-c)\int_Y^{+\infty}\frac{A(Z)}{(U_s-c)^2e^{2\beta Z}}dZ,
\end{align}
where
\begin{align}\label{def: beta}
\beta=\al A_{\infty}^\f12=\beta_r+i\beta_i,\quad A_\infty=\lim_{Y\to +\infty} A(Y)=1-m^2(1-c)^2.
\end{align}
For $m\in(0,1),$ we have 
\begin{align}
\mathrm {Re} A_\infty>0,\quad \mathrm{Im} A_\infty>0,\nonumber
\end{align} 
and 
\begin{align}
\mathrm{Re} A_{\infty}^\f12\sim (1-m^2)^{\f12}>0,\quad \mathrm{Im} A_{\infty}^\f12\sim m^2 c_i>0,\nonumber
\end{align}
which along with the fact $\alpha=\alpha_r+\mathrm i\alpha_i$ with $\alpha_r\gg|\alpha_i|$ implies 
\begin{align}\label{est: beta}
0<C^{-1}|\al| \leq|\beta_r| \leq |\beta|\leq C|\al|.
\end{align}

By a direct calculation, we find that
\begin{align}\label{eq:ray-homo-app}
	\begin{split}
		(U_s-c)\Lambda\varphi_{Ray}^{(0)}-\partial_Y(A^{-1}\partial_YU_s)\varphi_{Ray}^{(0)}=&2\beta U_s' A^{-1}\varphi_{Ray}^{(0)}-\beta\frac{A'}{A^2}(U_s-c)\varphi_{Ray}^{(0)}\\
&+(A^{-1}\beta^2-\al^2)(U_s-c)\varphi_{Ray}^{(0)}.
	\end{split}
\end{align}
We next define $\varphi_R$ to be a solution to the following Rayleigh equation, which is constructed according to Proposition \ref{pro: phi_non}:
\begin{align}\nonumber
	Ray[\varphi_R]=&-2\beta U_s' A^{-1}\varphi_{Ray}^{(0)}+\beta\frac{A'}{A^2}(U_s-c)\varphi_{Ray}^{(0)}-(A^{-1}\beta^2-\al^2)(U_s-c)\varphi_{Ray}^{(0)}.
\end{align}
Then $\varphi_{Ray}=\varphi_{Ray}^{(0)}+\varphi_{R}$ is a solution to \eqref{eq:homo-ray}.\smallskip

We first give the asymptotic expansions of the main part $\varphi_{Ray}^{(0)}(Y)$.

\begin{lemma}\label{lem: asymp-phi_0}
There exists $Y_0>0$ such that 
	\begin{itemize}
		\item For $Y\geq Y_0$,
   \begin{align*}
   	&\Big|\varphi_{Ray}^{(0)}(Y)-\frac{A_\infty}{(1-c)^2}(U_s-c)e^{-\beta Y}\Big|\leq C|\al| e^{-\eta_0 Y},\\
	&\Big|\pa_Y\varphi_{Ray}^{(0)}(Y)-\frac{A_\infty}{(1-c)^2}\partial_Y((U_s-c)e^{-\beta Y})\Big|\leq C|\al| e^{-\eta_0 Y}.
   \end{align*}

     \item For $0\leq Y\leq Y_0$,
      \begin{align*}
   	&\Big|\varphi_{Ray}^{(0)}(Y)-\frac{A_\infty}{(1-c)^2}(U_s-c)e^{-\beta Y}-\frac{2\beta e^{-\beta Y}}{U'_s(Y_c)}\Big|\leq C|\alpha||U_s(Y)-c||\log c_i|,\\
	&\Big|\pa_Y \varphi_{Ray}^{(0)}-\frac{A_\infty}{(1-c)^2}U_s'e^{-\beta Y}\Big|\leq C|\al| |\log c_i|.
   \end{align*}
     In particular, we have 
     \begin{align*}
   	&\varphi_{Ray}^{(0)}(0)=-(1-m^2)c+\frac{2\beta}{U'(Y_c)}+\mathcal O\Big((|\alpha|+|c|)|c||\log c_i|\Big),\\
	&\pa_Y\varphi_{Ray}^{(0)}(0)=U_s'(0)(1-m^2)+\mathcal O\Big((|\al|+|c|)|\log c_i|\Big).
   \end{align*}
   
   \item If  $c_i\geq c_0|\e|^\f13>0$, then we have
   \begin{align}\label{est: varphi^(0)_Ray,R}
	\big\|\varphi_{Ray}^{(0)}-\frac{A_\infty}{(1-c)^2}(U_s-c)e^{-\beta Y}\big\|_{\mathcal Y_{\eta_0}}\leq C|\alpha||\log c_i|.
\end{align}
	\end{itemize}
\end{lemma}

\begin{proof}
Recall that there exists $Y_0>Y_c$ such that for any $Y\geq Y_0$,
	\begin{align*}
		|U_s(Y)-c|\geq C^{-1}>0.
	\end{align*}
By the definitions of $\varphi_{Ray}^{(0)}$, $A(Y)$ and $A_\infty$, we notice that for any $Y\geq Y_0$,
\begin{align*}
	&\frac{A(Y)}{(U_s-c)^2}-\frac{A_\infty}{(1-c)^2}\\
	&=\frac{A(Y)}{(U_s-c)^2}-\frac{A_\infty}{(U_s-c)^2}+\frac{A_\infty}{(U_s-c)^2}-\frac{A_\infty}{(1-c)^2}\\
	&=\frac{m^2((1-c)^2-(U_s-c)^2)}{(U_s-c)^2}+A_\infty\frac{((1-c)^2-(U_s-c)^2)}{(U_s-c)^2(1-c)^2},
\end{align*}
   which implies that for any $Y\geq Y_0$,
   \begin{align}\label{est: difference 1}
   \Big|\f{A(Y)}{(U_s-c)^2}-\frac{A_\infty}{(1-c)^2}\Big|\leq Ce^{-\eta_0 Y}.
   \end{align}
Then we infer from \eqref{est: beta} that
\begin{align*}
&\left|\varphi_{Ray}^{(0)}(Y)-2\beta (U_s-c)e^{\beta Y}\frac{A_\infty}{(1-c)^2}\int_Y^{+\infty}e^{-2\beta Z}dZ\right|\\
&=\left|2\beta (U_s-c)e^{\beta Y}\int_Y^{+\infty}\Big(\f{A(Z)}{(U_s-c)^2}-\frac{A_\infty}{(1-c)^2}\Big)e^{-2\beta Z}dZ\right|\\
&\leq C|\beta|e^{-\eta_0Y}\leq C|\al| e^{-\eta_0Y}.
\end{align*}
 A direct calculation gives 
   \begin{align*}
   	2\beta (U_s-c)e^{\beta Y}\frac{A_\infty}{(1-c)^2}\int_Y^{+\infty} e^{-2\beta Z}dZ=(U_s-c)e^{-\beta Y}\frac{A_\infty}{(1-c)^2}.
   \end{align*}
   Thus, we infer that for any $Y\geq Y_0$,
   \begin{align}\nonumber
  \Big|\varphi_{Ray}^{(0)}(Y)-(U_s-c)e^{-\beta Y}\frac{A_\infty}{(1-c)^2}\Big|\leq C|\al| e^{-\eta_0 Y}.
   \end{align}

   Now we turn to the estimates for the case of $0\leq Y\leq Y_0$. For this, we first decompose $\varphi_{Ray}^{(0)}$ as follows
   \begin{align*}
   	\varphi_{Ray}^{(0)}(Y)=I_1(Y)+I_2(Y)+I_3(Y),
   \end{align*}
   where 
   \begin{align*}
   	&I_1(Y)=2\beta (U_s-c)e^{\beta Y}\frac{A_\infty}{(1-c)^2}\int_{Y_0}^{+\infty}e^{-2\beta Z}dZ,\\
   	&I_2(Y)=2\beta(U_s-c)e^{\beta Y}\int_{Y_0}^{+\infty}\left(\frac{A(Z)}{(U_s-c)^2}-\frac{A_\infty}{(1-c)^2}\right)e^{-2\beta Z}dZ,\\
   	&I_3(Y)=2\beta(U_s-c)e^{\beta Y}\int_{Y}^{Y_0}\frac{A(Z)}{(U_s-c)^2e^{2\beta Z}}dZ.
   \end{align*}
   By a similar argument as above, we can deduce that for any $Y\leq Y_0$,
   \begin{align*}
   &I_1(Y)= (U_s-c)\frac{A_\infty}{(1-c)^2}e^{\beta Y-2\beta Y_0}=(U_s-c)\frac{A_\infty}{(1-c)^2}e^{-\beta Y}+\mathcal O(|\al||U_s-c|),\\
& |I_2(Y)|\leq C|\beta||U_s-c|e^{-\eta_0 Y}\leq C|\al||U_s-c|.
   \end{align*}
   We write $I_3(Y)$ as follows
   \begin{align*}
   	I_3(Y)=2\beta(U_s-c)e^{\beta Y}\int_{Y}^{Y_0}\frac{1}{(U_s-c)^2e^{2\beta Z}}dZ-2\beta m^2(U_s-c)e^{\beta Y}\int_Y^{Y_0}e^{-2\beta Z}dZ.
   \end{align*}
   For the second term of $I_3(Y)$,  we notice that for any $Y\leq Y_0$,
   \begin{align*}
   	\left|2\beta m^2(U_s-c)e^{\beta Y}\int_Y^{Y_0}e^{-2\beta Z}dZ\right|\leq C|\al(U_s-c)|.
   \end{align*}
For the first term in $I_3(Y)$, we get by Lemma \ref{lem: integral} with $G(Y)=e^{-2\beta Y}$ and $\|G\|_{L^\infty}\leq C$, $\|G'\|_{L^\infty}\leq C|\al|$ that 
\begin{align*}
2\beta(U_s-c)e^{\beta Y}\int_{Y}^{Y_0}\frac{1}{(U_s-c)^2e^{2\beta Z}}dZ=&2\beta(U_s-c)e^{\beta Y}\Big(\f{e^{-2\beta Y}}{U_s'(Y_c)(U_s-c)}+\mathcal{O}(|\log c_i|)\Big)\\
=&\f{2\beta e^{-\beta Y}}{U_s'(Y_c)}+\mathcal{O}(|\al( U_s-c)||\log c_i|).
\end{align*}
This shows that  for any $0\leq Y\leq Y_0$,
   \begin{align*}
   	\Big|\varphi_{Ray}^{(0)}(Y)-\frac{A_\infty}{(1-c)^2}(U_s-c)e^{-\beta Y}-\frac{2\beta e^{-\beta Y}}{U'_s(Y_c)}\Big|\leq C|\al||U_s(Y)-c||\log c_i|.
   \end{align*}
   In particular, by the fact $A_\infty=1-m^2+\mathcal{O}(|c|)$, we have 
   \begin{align*}
   	\varphi_{Ray}^{(0)}(0)=-(1-m^2)c+\frac{2\beta}{U'(Y_c)}+\mathcal O\Big((|\alpha|+|c|)|c||\log c_i|\Big).
   \end{align*}

Next we give the estimate of $\pa_Y\varphi_{Ray}^{(0)}.$ 
For $Y\geq Y_0$,  we have
\begin{align*}
	&\partial_Y\varphi_{Ray}^{(0)}(Y)-\partial_Y\left((U_s-c)e^{-\beta Y}\frac{A_\infty}{(1-c)^2}\right)\\
	&=\partial_Y\Big(2\beta e^{\beta Y}(U_s-c)\int_Y^{+\infty}\Big[\frac{A(Z)}{(U_s-c)^2}-\frac{A_\infty}{(1-c)^2} \Big]e^{-2\beta Z} \Big)dZ,
\end{align*}
which along with \eqref{est: difference 1} implies that 
\begin{align*}
	\left|\partial_Y\varphi_{Ray}^{(0)}(Y)-\partial_Y\Big((U_s-c)e^{-\beta Y}\frac{A_\infty}{(1-c)^2}\Big)\right|\leq C|\alpha|e^{-\eta_0 Y}.
\end{align*}
For $Y\leq Y_0$, we  write
\begin{align*}
\pa_Y\varphi_{Ray}^{(0)}=&\beta\varphi_{Ray}^{(0)}-2\f{\beta e^{-\beta Y}}{U_s-c}A(Y)+2\beta e^{\beta Y} U_s'\int_Y^{+\infty} \f{A(Z)}{(U_s-c)^2e^{2\beta Z}}dZ\\
=&\beta\varphi_{Ray}^{(0)}-2\f{\beta e^{-\beta Y}}{U_s-c}A(Y)+II.
\end{align*}
Using the fact that
\begin{align*}
\Big|\varphi_{Ray}^{(0)}(Y)-\frac{A_\infty}{(1-c)^2}(U_s-c)e^{-\beta Y}-\frac{2\beta e^{-\beta Y}}{U'_s(Y_c)}\Big|\leq C|\alpha||U_s(Y)-c||\log c_i|,
\end{align*}
we have
\begin{align}\label{est: al phi_0}
|\beta \varphi_{Ray}^{(0)}|\leq&C|\al| |U_s-c|+C|\al|^2+C|\al| |U_s-c||\log c_i|\\
\nonumber
\leq& C|\al| |U_s-c||\log c_i|+C|\al|^2.
\end{align}
We decompose $II$ into the following parts:
\begin{align*}
II=&2\beta e^{\beta Y} U_s'\int_Y^{+\infty} \frac{A_\infty}{(1-c)^2}e^{-2\beta Z}dZ\\
&+2\beta e^{\beta Y} U_s'\int_{Y_0}^{+\infty} \Bigg(\f{A(Z)}{(U_s-c)^2}-\frac{A_\infty}{(1-c)^2}\Bigg)e^{-2\beta Z}dZ\\
&-2\beta e^{\beta Y} U_s'\int_Y^{Y_0}\frac{A_\infty}{(1-c)^2}e^{-2\beta Z}dZ-2\beta e^{\beta Y} U_s'\int_Y^{Y_0} m^2e^{-2\beta Z} dZ\\
&+2\beta e^{\beta Y} U_s'\int_Y^{Y_0} \f{1}{(U_s-c)^2}e^{-2\beta Z}dZ\\
=&U_s'e^{-\beta Y}\frac{A_\infty}{(1-c)^2}+II_2+II_3+II_4+II_5.
\end{align*}
By \eqref{est: difference 1}, we have
\begin{align*}
|II_2|\leq C|\al|.
\end{align*}
It is easy to see that
\begin{align*}
|II_3|+|II_4|\leq C|\al|.
\end{align*}
So, we focus on $II_5.$ By Lemma \ref{lem: integral} with $G(Y)=e^{-2\beta Y}$, we have
\begin{align*}
II_5=&2\beta e^{\beta Y} U_s'\Big(\f{e^{-2\beta Y}}{U_s'(Y_c)(U_s-c)}+\mathcal{O}(|\log c_i|)\Big)=\f{2\beta e^{-\beta Y}U_s'(Y)}{U_s'(Y_c)(U_s-c)}+\mathcal{O}(|\al||\log c_i|).
\end{align*}
Then we have
\begin{align*}
II_5-2\f{\beta e^{-\beta Y}}{U_s-c}A(Y)=&\f{2\beta e^{-\beta Y}U_s'(Y)}{U_s'(Y_c)(U_s-c)}+\mathcal{O}(|\al||\log c_i|)-2\f{\beta e^{-\beta Y}}{U_s-c}A(Y)\\
=&\f{2\beta e^{-\beta Y}}{U_s-c}\Big(\f{U_s'(Y)}{U_s'(Y_c)}-1\Big)+\mathcal{O}(|\al||\log c_i|)\\
=&\mathcal{O}(|\al|)+\mathcal{O}(|\al||\log c_i|)=\mathcal{O}(|\al||\log c_i|).
\end{align*}
Thus, we arrive at
\begin{align*}
\Big|II-U_s'e^{-\beta Y}\frac{A_\infty}{(1-c)^2}-\f{2\beta e^{-\beta Y}}{U_s-c}A(Y)\Big|\leq C|\al| |\log c_i|,
\end{align*}
which along with \eqref{est: al phi_0} shows that for $Y\leq Y_0$ 
\begin{align*}
\Big|\pa_Y \varphi_{Ray}^{(0)}-U_s'e^{-\beta Y}\frac{A_\infty}{(1-c)^2}\Big|\leq C|\al| |\log c_i|.
\end{align*}
In particular, we have 
\begin{align*}
\pa_Y\varphi_{Ray}^{(0)}(0)=U_s'(0)(1-m^2)+\mathcal O\Big((|\al|+|c|) |\log c_i|\Big).
\end{align*}

It remains to prove  \eqref{est: varphi^(0)_Ray,R}. We denote 
\begin{align*}
\varphi_{Ray, R}^{(0)}=\varphi_{Ray}^{(0)}-\varphi_{Ray,0}^{(0)},\quad \varphi_{Ray,0}^{(0)}=\frac{A_\infty}{(1-c)^2}(U_s-c)e^{-\beta Y}.
\end{align*}
A direct calculation gives 
\begin{align*}
Ray[\varphi_{Ray, 0}^{(0)}]=-2\beta U_s'A^{-2}\varphi_{Ray, 0}^{(0)}+(A^{-1}\beta^2-\al^2)(U_s-c)\varphi_{Ray,0}^{(0)},
\end{align*}
which along with \eqref{eq:ray-homo-app} gives
\begin{align*}
Ray[\varphi_{Ray, R}^{(0)}]=&2\beta U_s' A^{-1}\varphi_{Ray}^{(0)}-\beta\frac{A'}{A^2}(U_s-c)\varphi_{Ray}^{(0)}+2\beta U_s'A^{-2}\varphi_{Ray, 0}^{(0)}\\
&+(A^{-1}\beta^2-\al^2)(U_s-c)(\varphi_{Ray}^{(0)}-\varphi_{Ray,0}^{(0)}):=F_{\varphi_{Ray, R}^{(0)}}.
\end{align*}
According to the above results, we know that
\begin{align*}
|\varphi_{Ray}^{(0)}|,~|\pa_Y\varphi_{Ray}^{(0)}| \leq C, \quad |\varphi_{Ray,0}^{(0)}|,~|\pa_Y\varphi_{Ray,0}^{(0)}|, ~|\pa_Y^2\varphi_{Ray,0}^{(0)}| \leq C,
\end{align*}
which along with the facts that $|U_s'|,~|A'|\leq Ce^{-\eta_0 Y}$ and $|A^{-1}\beta^2-\al^2|\leq C|\al|^2e^{-\eta_0 Y}$, show that
\begin{align*}
\|e^{\eta_0 Y}F_{\varphi_{Ray, R}^{(0)}}\|_{L^\infty}\leq C|\al|,\quad \|e^{\eta_0 Y}\pa_YF_{\varphi_{Ray, R}^{(0)}}\|_{L^\infty}\leq C|\al|.
\end{align*}
Then we get by Proposition \ref{pro: phi_non} that 
\begin{align*}
\|\varphi_{Ray, R}^{(0)}\|_{\tilde{\mathcal Y}_{\eta_0}}\leq C|\al||\log c_i|.
\end{align*}
Especially, we have
\begin{align*}
\|e^{\eta_0 Y}(U_s-c)\pa_Y^2\varphi_{Ray, R}^{(0)}\|_{L^\infty}\leq C|\al||\log c_i|,
\end{align*}
which implies 
\begin{align*}
\|e^{\eta_0 Y}(U_s-c)\pa_Y^2F_{\varphi_{Ray, R}^{(0)}}\|_{L^\infty}\leq C|\al||\log c_i|.
\end{align*}
If $c_i\geq c_0|\e|^\f13>0$, we apply Proposition \ref{pro: phi_non} again  to conclude that
\begin{align*}
|\e|^\f13|e^{\eta_0 Y}(U_s-c)\partial_Y^3\varphi_{Ray, R}^{(0)}(Y)|\leq& C|\al||\log c_i|,\\
|\e|^\f23|e^{\eta_0 Y}(U_s-c)\partial_Y^4\varphi_{Ray, R}^{(0)}(Y)|\leq& C|\al||\log c_i|,
\end{align*}
which give \eqref{est: varphi^(0)_Ray,R}. 
\end{proof}

Recall that 
	\begin{align}\label{eq: varphi_R}
Ray[\varphi_{Ray}^{(0)}]=&2\beta U_s' A^{-1}\varphi_{Ray}^{(0)}-\beta\frac{A'}{A^2}(U_s-c)\varphi_{Ray}^{(0)}+(A^{-1}\beta^2-\al^2)(U_s-c)\varphi_{Ray}^{(0)}.
\end{align}
Since $\varphi_{R}$ is constructed in Proposition \ref{pro: phi_non}, we can decompose it into $\varphi_{R}(Y)=\varphi_1(Y)+\varphi_2(Y)+\varphi_3(Y)+\varphi_{4}(Y)$, where
\begin{align*}
	\varphi_1(Y)=&-\alpha^2(U_s-c)\int_Y^{+\infty}\frac{A(Y')}{(U_s-c)^2}\int_{Y'}^\infty(A^{-1}A_\infty-1)(U_s-c)\varphi_{Ray}^{(0)}dY''dY',\\
	\varphi_2(Y)=&-2\beta(U_s-c)\int_Y^{+\infty}\frac{A(Y')}{(U_s-c)^2}\int_{Y'}^\infty U_s'A^{-1}\varphi_{Ray}^{(0)}dY''dY',\\
\varphi_3(Y)=&\beta(U_s-c)\int_Y^{+\infty}\frac{A(Y')}{(U_s-c)^2}\int_{Y'}^\infty \frac{A'}{A^2}(U_s-c)\varphi_{Ray}^{(0)}dY''dY',
\end{align*}
and $\varphi_4(Y)$ is the solution to 
\begin{align*}
	Ray[\varphi_4]=\alpha^2(U_s-c)(\varphi_1+\varphi_2+\varphi_3).
\end{align*}

\begin{lemma}\label{lem: varphi_R}
It holds that
\begin{itemize}
\item For $Y\geq Y_0$,
\begin{align*}
&|\varphi_R(Y)|\leq C|\al| e^{-\eta_0 Y},\quad |\pa_Y\varphi_R(Y)|\leq C|\al| e^{-\eta_0 Y}.
\end{align*}

\item For $0\leq Y\leq Y_0$,
\begin{align*}
&\varphi_R(Y)=-\f{\beta e^{-\beta Y}}{U_s'(Y_c)}+\mathcal O\Big(|\al|(|\al|+|U_s-c|)|\log c_i|\Big),\\
&\pa_Y\varphi_R(Y)=\mathcal{O}(|\al| |\log c_i|).
\end{align*}
In particular, we have
\begin{align*}
&\varphi_R(0)=-\f{\beta}{U_s'(Y_c)}+\mathcal{O}\Big((|\al|^2+|c|^2)|\log c_i|\Big),\\
&\pa_Y\varphi_R(0)=\mathcal{O}(|\al| |\log c_i|).
\end{align*}

\item If $c_i\geq c_0|\e|^\f13>0$, then we have
\begin{align}\label{est: varphi_R-Y-1}
\|\varphi_R\|_{\mathcal Y_{\eta_0}}\leq C|\alpha||\log c_i|.
\end{align}
\end{itemize}

\end{lemma}

\begin{proof}
By Proposition 3.2 and Lemma 3.3, we have 
\begin{align}\label{eq:ray-varphi4}
	\|\varphi_4\|_{\tilde{\mathcal Y}_{\eta_0}}\leq C|\alpha|^2|\log c_i|\|\varphi_1+\varphi_2+\varphi_3\|_{L^\infty_{\eta_0}}.
\end{align}
Hence, to obtain the estimates of $\varphi_R$ and $\partial_Y\varphi_R$, it enough to show the control of $\varphi_i$ and $\partial_Y\varphi_i$ for $i=1,2,3$.

We first estimate $\varphi_1(Y)$. Using the fact that
\begin{align}
\f{A_{\infty}}{A}-1=\f{A_\infty-A}{A}=m^2(U_s-1)(U_s+1-2c),
\end{align}
we get by \eqref{eq:S-A} and \eqref{est: A}  that
\begin{align*}
\Big|\f{A_{\infty}}{A}-1\Big|\leq Ce^{-\eta_0 Y}.
\end{align*}
From the estimate of $\varphi_{non}^{(0)}$  in the proof of Proposition \ref{pro: phi_non} and Lemma \ref{lem: asymp-phi_0}, we deduce  that for any $Y\geq Y_0$,
\begin{align}\label{est: phi_1-1}
|\varphi_1(Y)|\leq&C|\al|^2e^{-\eta_0 Y}\|e^{\eta_0 Y}(A^{-1}A_\infty-1)(U_s-c)\varphi_{Ray}^{(0)}\|_{L^\infty}
\leq C|\al|^2e^{-\eta_0 Y},
\end{align}
and for $0\leq Y\leq Y_0$,
\begin{align}\label{est: phi_1-2}
|\varphi_1(Y)|\leq&C|\al|^2|\log c_i |\|e^{\eta_0 Y}(A^{-1}A_\infty-1)(U_s-c)\varphi_{Ray}^{(0)}\|_{L^\infty}\leq C|\al|^2 |\log c_i|.
\end{align}

 We next estimate $\varphi_2(Y)$ and $\varphi_3(Y)$. For $Y\geq Y_0$, similar to the estimate of $\varphi_1(Y)$ along with the fact 
\begin{align*}
|A'|\leq CU_s'\leq C e^{-\eta_0 Y},
\end{align*}
we infer that
\begin{align*}
|\varphi_2(Y)|+|\varphi_3(Y)|\leq C|\al|  e^{-\eta_0 Y}.
\end{align*}
This along with \eqref{est: phi_1-1} implies that for $Y\geq Y_0$,
\begin{align}\label{eq:varphi-4-est}
|\varphi_1(Y)+\varphi_2(Y)+\varphi_3(Y)|\leq C|\al|e^{-\eta_0 Y}.
\end{align}

For $0\leq Y\leq Y_0$, we divide $\varphi_2(Y)$ and $\varphi_3(Y)$ as follows
\begin{align*}
\varphi_2(Y)=&-2\beta \f{A_\infty}{(1-c)^2}(U_s-c)\int_Y^{+\infty}\frac{A(Y')}{(U_s-c)^2}\int_{Y'}^\infty U_s'A^{-1}e^{-\beta Y''}(U_s-c)dY''dY'\\
&-2\beta(U_s-c)\int_Y^{+\infty}\frac{A(Y')}{(U_s-c)^2}\int_{Y'}^\infty U_s'A^{-1}W_1(Y'')dY''dY'\\
=&\varphi_{2,1}(Y)+\varphi_{2,2}(Y),
\end{align*}
and
\begin{align*}
\varphi_3(Y)=&\beta \f{A_\infty}{(1-c)^2}(U_s-c)\int_Y^{+\infty}\frac{A(Y')}{(U_s-c)^2}\int_{Y'}^\infty\frac{A'}{A^2}(U_s-c)^2e^{-\beta Y''}dY''dY'\\
&+\beta(U_s-c)\int_Y^{+\infty}\frac{A(Y')}{(U_s-c)^2}\int_{Y'}^\infty\frac{A'}{A^2}(U_s-c)W_1(Y'')dY''dY'\\
=&\varphi_{3,1}(Y)+\varphi_{3,2}(Y),
\end{align*}
where 
\begin{align}\nonumber
W_1(Y)=\varphi_{Ray}^{(0)}(Y)-\f{A_\infty}{(1-c)^2}(U_s-c)e^{-\beta Y}.
\end{align}
By Lemma \ref{lem: asymp-phi_0}, we have
\begin{align}
|W_1|\leq& C|\al| e^{-\eta_0 Y},\quad Y\geq Y_0,\label{est: W_1-1}\\
|W_1|\leq& C|\al|  |\log c_i|,\quad 0\leq Y\leq Y_0.\label{est: W_1-2}
\end{align}
As in the estimate of $\varphi_1$, by using \eqref{est: W_1-1}-\eqref{est: W_1-2},  \eqref{est: A} and the fact 
\begin{align*}
|A'|\leq CU_s'\leq C e^{-\eta_0 Y},
\end{align*}
we infer that
\begin{align}
|\varphi_{2,2}(Y)|\leq&C|\al|^2|\log c_i|, \label{est: phi_22}\\
|\varphi_{3,2}(Y)|\leq&C|\al|^2|\log c_i|. \label{est: phi_32}
\end{align}

We get by integration by parts that
\begin{align*}
2\int_{Y'}^\infty U_s'A^{-1}e^{-\beta Y''}(U_s-c)dY''=&\int_{Y'}^\infty A^{-1}e^{-\beta Y''}\pa_{Y''}(U_s-c)^2dY''\\
=&\beta\int_{Y'}^\infty A^{-1}e^{-\beta Y''}(U_s-c)^2dY''+\int_{Y'}^\infty \f{A'}{A^2}e^{-\beta Y''}(U_s-c)^2dY''\\
&-A^{-1}e^{-\beta Y'}(U_s(Y')-c)^2.
\end{align*}
As a result, we obtain
\begin{align*}
\varphi_{2,1}(Y)+\varphi_{3,1}(Y)=&-\beta^2 \f{A_\infty}{(1-c)^2}(U_s-c)\int_Y^{+\infty}\frac{A(Y')}{(U_s-c)^2}\int_{Y'}^\infty A^{-1}e^{-\beta Y''}(U_s-c)^2dY''dY'\\
&+\beta \f{A_\infty}{(1-c)^2}(U_s-c)\int_Y^{+\infty}e^{-\beta Y'} dY'=\varphi_{2,1}^1(Y)+\varphi_{2,1}^2(Y).
\end{align*}
It is easy to see that
\begin{align}\label{est: phi_21^2}
\varphi_{2,1}^2(Y)=&(U_s-c)e^{-\beta Y}\f{A_\infty}{(1-c)^2}.
\end{align}

For $\varphi_{2,1}^1(Y)$, we decompose it as follows
\begin{align*}
\varphi_{2,1}^1(Y)=&-\beta^2\f{A_\infty}{(1-c)^2}(U_s-c)\int_{Y_0}^{+\infty}\frac{A(Y')}{(U_s-c)^2}\int_{Y'}^\infty A^{-1}(U_s-c)^2e^{-\beta Y''}dY''dY'\\
&-\beta^2 \f{A_\infty}{(1-c)^2}(U_s-c)\int_Y^{Y_0}\frac{A(Y')}{(U_s-c)^2}\int_{Y'}^\infty A^{-1}(U_s-c)^2 e^{-\beta Y''}dY''dY'\\
=&\varphi_{2,1}^{11}(Y)+\varphi_{2,1}^{12}(Y).
\end{align*}
Using the fact \eqref{est: A} and a similar process for \eqref{est: difference 1} that for $Y''\geq Y_0$,
\begin{align*}
\Big|A^{-1}(U_s-c)^2-\f{(1-c)^2}{A_\infty}\Big|\leq Ce^{-\eta_0 Y''},
\end{align*}
we have
\begin{align*}
\Big|\int_{Y'}^\infty \Big(A^{-1}(U_s-c)^2-\f{(1-c)^2}{A_\infty}\Big) e^{-\beta Y''}dY''\Big|\leq&C\int_{Y'}^\infty e^{-\eta_0 Y''}dY''
\leq Ce^{-\eta_0 Y'},
\end{align*}
and
\begin{align*}
&\Big|\beta^2\f{A_\infty}{(1-c)^2}(U_s-c)\int_{Y_0}^{+\infty}\frac{A(Y)}{(U_s-c)^2}\int_{Y'}^\infty \Big(A^{-1}(U_s-c)^2-\f{(1-c)^2}{A_\infty}\Big) e^{-\beta Y''}dY''dY'\Big|\\
&\leq C|\alpha|^2|U_s-c| \int_{Y_0}^{+\infty}\frac{A(Y')}{(U_s-c)^2}e^{-\eta_0 Y'}dY'\leq C|\al|^2.
\end{align*}
On the other hand, we have
\begin{align*}
\int_{Y}^\infty\f{(1-c)^2}{A_\infty} e^{-\beta Y'}dY'=\beta^{-1}\f{(1-c)^2}{A_\infty}e^{-\beta Y}.
\end{align*}
Thus,  we conclude that for $Y\leq Y_0$,
\begin{align}\label{est: phi_21-11}
\varphi_{2,1}^{11}(Y)=&-\beta(U_s-c)\int_{Y_0}^{+\infty}\frac{A(Y')e^{-\beta Y'}}{(U_s-c)^2}dY'+\mathcal{O}(|\al|^2)\\
\nonumber
=&-\beta\f{A_\infty}{(1-c)^2}(U_s-c)\int_{Y_0}^{+\infty}e^{-\beta Y'}dY'\\
&+\beta(U_s-c)\int_{Y_0}^{+\infty}\Big(\frac{A(Y')}{(U_s-c)^2}-\f{A_\infty}{(1-c)^2}\Big)e^{-\beta Y'}dY'
+\mathcal{O}(|\al|^2)\nonumber\\
\nonumber
=&-(U_s-c)e^{-\beta Y}\f{A_\infty}{(1-c)^2}+\mathcal{O}\Big(|\al|(|\al|+|U_s-c|)\Big),
\end{align}
where we used  \eqref{est: difference 1} in the last step.

Noticing that 
\begin{align*}
\beta\int_{Y'}^\infty A^{-1}(U_s-c)^2 e^{-\beta Y''}dY''=\f{e^{-\beta Y'}(1-c)^2}{A_\infty}+\mathcal O(|\al|)=\f{(1-c)^2}{A_\infty}+\mathcal O(|\al|),
\end{align*}
we infer that
\begin{align*}
\varphi_{2,1}^{12}(Y)=&-\beta \f{A_\infty}{(1-c)^2}(U_s-c)\int_Y^{Y_0}\frac{A(Y)}{(U_s-c)^2}\Big(\f{(1-c)^2}{A_\infty}+\mathcal O(|\al|)\Big)dY\\
=&-\beta(U_s-c)\int_Y^{Y_0}\frac{A(Y')}{(U_s-c)^2} dY'+\mathcal O(|\al|^2|\log c_i|)\\
=&-\beta(U_s-c)\int_Y^{Y_0}\frac{1}{(U_s-c)^2} dY'+\mathcal O\Big(|\al|(|\al|+|U_s-c|)|\log c_i|\Big),
\end{align*}
where we used Lemma \ref{lem: integral}  with $G(Y)=A(Y)$ so that
\begin{align}\nonumber
\Big|(U_s-c)\int_Y^{Y_0}\frac{A(Y)}{(U_s-c)^2}dY\Big|\leq C|\log c_i|
\end{align}
in the second step and $A(Y)=1-m^2(U_s-c)^2$ in the last step.
By Lemma \ref{lem: integral} again, we have
\begin{align*}
\beta(U_s-c)\int_Y^{Y_0}\frac{1}{(U_s-c)^2} dY=\f{\beta e^{-\beta Y}}{U_s'(Y_c)}+\mathcal O(|\al||U_s-c||\log c_i|).
\end{align*}
Thus, we obtain that for  $0\leq Y\leq Y_0$,
\begin{align}\label{est: phi_21-12}
\varphi_{2,1}^{12}(Y)=-\f{\beta e^{-\beta Y}}{U_s'(Y_c)}+\mathcal O\Big(|\al|(|\al|+|U_s-c|)|\log c_i|\Big).
\end{align}

Summing up \eqref{est: phi_21-11} and \eqref{est: phi_21-12}, we conclude that for $0\leq Y\leq Y_0$,
\begin{align}\label{est: phi_21^1}
\varphi_{2,1}^1(Y)=-(U_s-c)e^{-\beta Y}\f{A_\infty}{(1-c)^2}-\f{\beta e^{-\beta Y}}{U_s'(Y_c)}+\mathcal O\Big(|\al|(|\al|+|U_s-c|)|\log c_i|\Big).
\end{align}
Then we infer from \eqref{est: phi_21^2} and \eqref{est: phi_21^1}  that
\begin{align}\nonumber
\varphi_{2,1}(Y)=&-\f{\beta e^{-\beta Y}}{U_s'(Y_c)}+\mathcal O\Big(|\al|(|\al|+|U_s-c|)|\log c_i|\Big),
\end{align}
which along with \eqref{est: phi_22}-\eqref{est: phi_32} gives
\begin{align*}
\varphi_2(Y)+\varphi_3(Y)=&-\f{\beta e^{-\beta Y}}{U_s'(Y_c)}+\mathcal O\Big(|\al|(|\al|+|U_s-c|)|\log c_i|\Big).
\end{align*}
This along with \eqref{eq:varphi-4-est} shows that for $Y\leq Y_0$,
\begin{align*}
\varphi_1(Y)+\varphi_2(Y)+\varphi_3(Y)=&-\f{\beta e^{-\beta Y}}{U_s'(Y_c)}+\mathcal O\Big(|\al|(|\al|+|U_s-c|)|\log c_i|\Big),
\end{align*}
which along with \eqref{eq:ray-varphi4} and \eqref{eq:varphi-4-est} implies
\begin{align}\label{eq:est-phi4-norm}
	\|\varphi_4\|_{L^\infty_{\eta_0}}+\|\partial_Y\varphi_4\|_{L^\infty_{\eta_0}}\leq C|\alpha|^2|\log c_i|\|\varphi_1+\varphi_2+\varphi_3\|_{L^\infty_{\eta_0}}\leq C|\alpha|^3|\log c_i|.
\end{align}
Therefore, by the definition of $\varphi_R$ and the above estimates, we infer that for $Y\leq Y_0$
\begin{align*}
	\varphi_R(Y)=-\f{\beta e^{-\beta Y}}{U_s'(Y_c)}+\mathcal O\Big(|\al|(|\al|+|U_s-c|)|\log c_i|\Big).
\end{align*}

In the following, we compute $\pa_Y\varphi_R(Y)$. We first give the estimate of $\pa_Y\varphi_1(Y)$. A direct calculation gives
\begin{align*}
\pa_Y\varphi_1(Y)=&-\alpha^2U_s'\int_Y^{+\infty}\frac{A(Y')}{(U_s-c)^2}\int_{Y'}^\infty(A^{-1}A_\infty-1)(U_s-c)\varphi_{Ray}^{(0)}dZdY'\\
&+\alpha^2\f{A(Y)}{U_s-c}\int_{Y}^\infty(A^{-1}A_\infty-1)(U_s-c)\varphi_{Ray}^{(0)}dY'.
\end{align*}
 As in the estimate of $\varphi_1$ and using the fact $|A^{-1}A_\infty-1|\leq C|U_s-1|\leq C e^{-\eta_0 Y}$, it is easy to see that for $Y\geq Y_0$,\begin{align*}
|\pa_Y\varphi_1(Y)|\leq& C|\al|^2 e^{-\eta_0 Y}.
\end{align*}
For $0\leq Y\leq Y_0$, let $G(Y)=A(Y)\int_{Y}^\infty (A^{-1}A_\infty-1)(U_s-c)\varphi_{Ray}^{(0)}dY'$, which holds that 
\begin{align*}
|G(Y)|\leq C,\quad |G'(Y)|\leq C.
\end{align*}
Then we  get by  Lemma \ref{lem: integral}  that
\begin{align*}
\pa_Y\varphi_1(Y)=&-\alpha^2U_s'\int_Y^{+\infty}\frac{G(Y')}{(U_s-c)^2}dY'+\al^2\f{G(Y)}{U_s-c}\\
=&-\al^2U_s'\Big(\f{G(Y)}{U_s'(Y_c)(U_s-c)}+\mathcal{O}(|\log c_i|)\Big)+\al^2\f{G(Y)}{U_s-c}\\
=&\al^2\f{G(Y)}{U_s-c}\Big(1-\f{U_s'(Y)}{U_s'(Y_c)}\Big)+\mathcal{O}(|\al|^2|\log c_i|)\\
=&\mathcal{O}(|\al|^2|\log c_i|),
\end{align*}
where we used $\Big|(U_s-c)^{-1}\Big(1-\f{U_s'(0)}{U_s'(Y_c)}\Big)\Big|\leq C$ in the last step.


Next we give the estimates of $\pa_Y\varphi_2(Y), \pa_Y\varphi_3(Y).$ 
Thanks to $|A'|\leq C|U_s'|\leq  Ce^{-\eta_0 Y}$ and \eqref{est: A}, we use the same argument in $\pa_Y \varphi_1(Y)$ to infer that for any $Y\geq Y_0$,
\begin{align*}
|\pa_Y\varphi_2(Y)|\leq& C|\al|  e^{-\eta_0 Y},\\
|\pa_Y\varphi_3(Y)|\leq& C|\al | e^{-\eta_0 Y},
\end{align*}
and for $ Y\leq Y_0$,
\begin{align*}
|\pa_Y\varphi_2(Y)|\leq& C|\al| |\log c_i| ,\\
|\pa_Y\varphi_3(Y)|\leq& C|\al| |\log c_i|.
\end{align*}

Summing up the above estimates and applying \eqref{eq:est-phi4-norm} , we conclude that for any $Y\geq Y_0$,
\begin{align}\nonumber
|\pa_Y \varphi_R(Y)|\leq C|\al| e^{-\eta_0 Y},
\end{align}
and for $ Y\leq Y_0$,
\begin{align}
|\pa_Y\varphi_R(Y)|\leq& C|\al| |\log c_i|.\nonumber
\end{align}

It remains to prove \eqref{est: varphi_R-Y-1}.
Recall the equation of $\varphi_R$ in \eqref{eq: varphi_R}:
\begin{align*}
Ray[\varphi_R]=F_{\varphi_R},
\end{align*}
where 
\begin{align*}
F_{\varphi_R}=&-2\beta U_s' A^{-1}\varphi_{Ray}^{(0)}+\beta\frac{A'}{A^2}(U_s-c)\varphi_{Ray}^{(0)}-(A^{-1}\beta^2-\al^2)(U_s-c)\varphi_{Ray}^{(0)}.
\end{align*}
By Lemma \ref{lem: asymp-phi_0}, we have
\begin{align*}
\|\varphi_{Ray}^{(0)}\|_{\mathcal Y_{\eta_0}}\leq C,
\end{align*}
which implies 
\begin{align*}
\|e^{\eta_0 Y}F_{\varphi_R}\|_{L^\infty}+\|e^{\eta_0 Y}\pa_YF_{\varphi_R}\|_{L^\infty}+\|e^{\eta_0 Y}(U_s-c)\pa_Y^2F_{\varphi_R}\|_{L^\infty}\leq&C|\al|.
\end{align*}
Then it follows from Proposition \ref{pro: phi_non} that
\begin{align*}
\|\varphi_R\|_{\mathcal Y_{\eta_0}}\leq C|\alpha||\log c_i|.
\end{align*}
\end{proof}

With Lemma \ref{lem: asymp-phi_0} and Lemma \ref{lem: varphi_R} in hand, we directly conclude that 
\begin{proposition}\label{pro: varphi_Ray(0)}
Under the same assumptions as in Lemma \ref{lem: varphi_R},  there exists a solution $\varphi_{Ray}\in \tilde{\mathcal Y}_{\beta_r}$ to \eqref{eq:homo-ray}, which holds that
\begin{align*}
&\varphi_{Ray}(0)=-(1-m^2)c+\f{\beta}{U_s'(Y_c)}+\mathcal{O}((|\al|^2+|c|^2) |\log c_i|),\\
&\pa_Y\varphi_{Ray}(0)=(1-m^2)U_s'(0)+\mathcal{O}(|\al|| \log c_i|).
\end{align*}
In particular, we have
\begin{align*}
\f{\varphi_{Ray}(0)}{\pa_Y\varphi_{Ray}(0)}=-\f{c}{U_s'(0)}+\f{\beta}{(1-m^2)(U_s'(0))^2}+\mathcal{O}((|\al|^2+|c|^2) |\log c_i|).
\end{align*}
In addition, if $c_i\geq C_0|\varepsilon|^\f13$, we have
   \begin{align*}
	\|\varphi_{Ray}-\frac{A_\infty}{(1-c)^2}(U_s-c)e^{-\beta Y}\|_{\mathcal Y_{\eta_0}}\leq C|\alpha||\log c_i|.
\end{align*}

\end{proposition}

\section{The Airy type equation}
This section is devoted to constructing the solution to the Airy type equation:
\begin{align}\label{eq:Airy-eq}
\left\{
\begin{aligned}
&\varepsilon(\partial_Y^2-\alpha^2)w-(U_s-c)w =F,\quad Y>0,\\
&(\partial_Y(A(Y)^{-1}\partial_Y)-\alpha^2)\psi=w,\quad Y>0,\\
&\lim_{Y\to\infty}\psi(Y)=0.
\end{aligned}
\right.
\end{align}
Here $\e=\f{\sqrt \nu}{\mathrm i\al}$, $c=c_r+\mathrm i c_i\in \mathbb{C}$, $\al=\al_r+\mathrm i\al_i=|\al|e^{-3\th_0 \mathrm i}\in \mathbb{C}$ with $|\al_i|\ll |\al_r|$ and $\th_0\in(-\f{\pi}{100}, \f{\pi}{100})$. 

The related results in this section will be applied to perform the Airy-Airy-Rayleigh iteration and the construction of fast mode of the homogeneous Orr-Sommerfeld equation \eqref{eq:OS-F} in section 5. Moreover, they are also used to construct the tangential velocity $u_{st}$ of the quasi-compressible system \eqref{eq: Stokes} in section 7.

\subsection{Langer transform and Airy function}
Since we construct a solution to \eqref{eq:Airy-eq} for $Y\geq 0$, we introduce the following modified Langer transform instead of linear approximation $Y-Y_c$. In details, we define
\begin{align}\label{def: eta_out^La}
	\eta_{out}^{La}(Y)=\left\{
	\begin{array}{ll}
		\left(\f32\int_{Y_c}^Y \left(\frac{U_s(Z)-c_r}{U_s'(Y_c)}\right)^\f12dZ \right)^\f23,\quad Y\geq Y_c\\
		-\left(\f32\int^{Y_c}_Y \left(\frac{c_r-U_s(Z)}{U_s'(Y_c)}\right)^\f12dZ \right)^\f23,\quad Y\leq Y_c,
	\end{array}
	\right.
\end{align}
and $\eta_{in}^{La}(Y)=Y-Y_c$. It is easy to verify that
\begin{align}\label{eq: eta_out^La}
U_s'(Y_c)\eta_{out}^{La}(\pa_Y\eta_{out}^{La})^2=U_s-c_r.
\end{align}

Now we introduce the modified Langer transformation
\begin{align}\label{eq:Langer-mod}
	\eta^{La}(Y)=\eta_r^{La}(Y)+\mathrm i\eta_i^{La},
\end{align}
where $\eta_i^{La}=-U_s'(Y_c)^{-1}c_i$ and 
\begin{align}\label{eq:Langer_r-mod}
	\eta_r^{La}(Y)=\chi\Big(\frac{Y-Y_c}{\kappa^{-1} M}\Big)\eta_{in}^{La}(Y)+\Big(1-\chi\Big(\frac{Y-Y_c}{\kappa^{-1} M}\Big)\Big)\eta_{out}^{La}(Y)
\end{align}
with $\chi(\cdot)$ a smooth function in $\mathbb R_+$ satisfying $\chi(Y)\equiv1$ on $[0,1]$ and $\chi(Y)=0$ on $[2,+\infty)$, and $\kappa=|\varepsilon|^{-\f13}U_s'(Y_c)^\f13$. 

We introduce the modified Airy function
\begin{align*}
	&A_1(Y)=-\mathrm i e^{-2\th_0 \mathrm i}|\varepsilon|^{-\f23}U_s'(Y_c)^{-\f13} Ai(e^{\mathrm i(\frac{\pi}{6}-\th_0)}\kappa\eta^{La}(Y)),\\
	&A_2(Y)=2\pi  Ai(e^{\mathrm i(\frac{5\pi}{6}-\th_0)}\kappa\eta^{La}(Y)).
\end{align*}
We define 
\begin{align*}
	&A_1(1,Y)=-\int_Y^{+\infty}A_1(Z)dZ,\qquad A_1(2,Y)=-\int_Y^{+\infty}A_1(1,Z)dZ,\\
	&A_2(1,Y)=\int_0^{Y}A_2(Z)dZ,\qquad A_2(2,Y)=\int_0^{Y}A_2(1,Z)dZ.
\end{align*}
By the properties of Airy function and the above definitions, we can derive that for $j=1,2$,
\begin{align}\label{eq:Airy-Err-app}
	\varepsilon(\partial_Y^2-\alpha^2)A_j(Y)-(U_s-c)A_j(Y)=Err_1(Y)A_j+Err_2(Y)\partial_YA_j,
\end{align}
where
\begin{align*}
	Err_1(Y)=U_s'(Y_c)\eta^{La}(\partial_Y\eta^{La})^2-\varepsilon\alpha^2-(U_s-c),\quad Err_2(Y)=\frac{\varepsilon\partial_Y^2\eta^{La}(Y)}{\partial_Y\eta^{La}(Y)}.
\end{align*}
We shall see that $Err_1$ and $Err_2$ are small errors actually.
Moreover, by the property of Airy function, we find that for any $z\in\mathbb C$,
\begin{align*}
	Ai(z)(Ai(e^{\mathrm i\frac{2\pi}{3}}z))'-Ai'(z)Ai(e^{\mathrm i\frac{2\pi}{3}}z)=\frac{1}{2\pi}e^{-\mathrm i\frac{\pi}{6}},
\end{align*}
which implies 
\begin{align}\label{eq:Airy-wron}
	A_1(Y)\partial_Y A_2(Y)-\partial_Y A_1(Y)A_2(Y)=\varepsilon^{-1}\partial_Y\eta^{La}(Y).
\end{align}

For given source term $e^{\vartheta Y}F\in L^\infty$, we define
\begin{align}\label{eq:airy-w-app}
	w_{app}(Y)=&A_1(Y)\int_0^YA_2(Z)\partial_Y\eta^{La}(Z)^{-1} F(Z)dZ\\
	&+A_2(Y)\int_Y^{+\infty}A_1(Z)\partial_Y\eta^{La}(Z)^{-1}F(Z)dZ,\nonumber
\end{align} 
and 
\begin{align}\label{eq:airy-psi-app}
	\psi_{app}(Y)=\int_Y^{+\infty}A(Y')\int_{Y'}^{+\infty}w_{app}(Z)dZ.
\end{align}
Then we find that $w_{app}$ is a solution to the following equations
\begin{align}\label{eq:Airy-app-Fs}
	\begin{split}
		&\varepsilon(\partial_Y^2-\alpha^2)w_{app}-(U_s-c)w_{app}=F+Err_1w_{app}+Err_2\partial_Yw_{app},\\
		&\pa_Y(A^{-1}\pa_Y)\psi_{app}(Y)=w_{app}(Y),\quad\quad\lim_{Y\to\infty}\psi_{app}(Y)=\lim_{Y\to\infty}w_{app}(Y)=0,
	\end{split}
\end{align}
where $A(Y)=1-m^2(U_s-c)^2$. We will show that \eqref{eq:Airy-app-Fs} is an appropriate approximation of \eqref{eq:Airy-eq}.

In what follows, we decompose $\mathbb R_+\cup\{0\}=\mathcal N^-\cup\mathcal N\cup\mathcal N^+,$ where
\begin{align}
	&\mathcal N^+=\{Y\geq Y_c:|\kappa\eta^{La}(Y)|\geq 3M\},\quad\mathcal N=\{Y:|\kappa\eta^{La}(Y)|\leq 3M\}\label{def: N^+},\\
	&\mathcal N^-=\{0\leq Y\leq Y_c:|\kappa\eta^{La}(Y)|\geq 3M\}\label{def: N^-}.
\end{align}

 \subsection{Estimates of the Airy function}
 
 Since we introduce the modified  Langer transformation involved in the construction of $w_{app}$, we need to obtain some useful estimates about $\eta^{La}(Y)$ and the modified Airy function. 
\begin{lemma}\label{lem:est-eta}
	Let $|\alpha|,|\varepsilon|,|c|\ll1$ and $c_r>0$. Let $0<L\leq 1$ be a small number in Lemma \ref{lem:classical-Langer}. Then $\eta_r^{La}(Y)\in C^{\infty}(\mathbb R_+)$ satisfies that following properties: 
	
    \begin{enumerate}
    	\item For any $|Y-Y_c|\leq L$, 
    	      \begin{align*}
	           |\eta_r^{La}(Y)-(Y-Y_c)|\leq C|Y-Y_c|^2,
              \end{align*}
        	  and for any $|Y-Y_c|\geq L$,
        	  \begin{align*}
	           C^{-1}(1+|Y-Y_c|)^\f23\leq |\eta_r^{La}(Y)|\leq C(1+|Y-Y_c|)^\f23.
              \end{align*}
    	\item There exits $0<C_1\leq C_2<+\infty$ such that for any $Y\geq 0$,
    	      \begin{align*}
    	      	C_1(1+|Y-Y_c|)^{-\f13}\leq |\partial_Y \eta_{r}^{La}(Y)|\leq C_2(1+|Y-Y_c|)^{-\f13}.
    	      \end{align*}
    	      Moreover, for any $|Y-Y_c|\leq L$,
    	      \begin{align*}
    	      	|\partial_Y \eta_{r}^{La}(Y)-1|\leq C|Y-Y_c|.
    	      \end{align*}
    	\item For any $Y\geq 0$,
              \begin{align*}
	           |\partial_Y^2\eta_{r}^{La}(Y)|\leq C(1+|Y-Y_c|)^{-\f43}.
              \end{align*}
    	\end{enumerate}
    	
	\end{lemma}

\begin{remark}
From the above Lemma, we know that the real part of the modified Langer transformation $\eta_r^{La}(Y)$ behaves like $Y-Y_c$ if $|Y-Y_c|\leq L$. Such region is quite larger than the thickness of sublayer $|\varepsilon|^{\f13}$. It means that the Langer transformation is nearly linear in the sublayer. The imaginary part of $\eta^{La}(Y)$ is a constant, which implies that $\partial_Y\eta^{La}(Y)\sim1$ in the sublayer. 
\end{remark}

\begin{proof}
We first show the results about the bounds of $\eta^{La}_r(Y)$. According to Lemma \ref{lem:classical-Langer}, we know that for any $0<|Y-Y_c|\leq L$ ($c_r\ll  L\leq 1$ is a small constant in Lemma \ref{lem:classical-Langer}),
\begin{align}\label{eq:airy-eta-out}
	|\eta_{out}^{La}(Y)-(Y-Y_c)|\leq C|Y-Y_c|^2.
\end{align}

By the definition \eqref{eq:Langer-mod}, we know that for any $|Y-Y_c|\leq \kappa^{-1} M$,
\begin{align*}
	\eta_r^{La}(Y)=Y-Y_c,
\end{align*}
and for any $|Y-Y_c|\in[\kappa^{-1} M,2\kappa^{-1} M]$,
\begin{align*}
	\eta_r^{La}=&\chi\Big(\frac{Y-Y_c}{\kappa^{-1} M}\Big)(Y-Y_c)+\Big(1-\chi\Big(\frac{Y-Y_c}{\kappa^{-1} M}\Big)\Big)\eta_{out}^{La}(Y)\\
	=&(Y-Y_c)+\Big(1-\chi\Big(\frac{Y-Y_c}{\kappa^{-1} M}\Big)\Big)(\eta_{out}^{La}(Y)-(Y-Y_c)),
\end{align*}
and for any $|Y-Y_c|\geq 2\kappa^{-1} M$,
\begin{align*}
	\eta_r^{La}(Y)=\eta_{out}^{La}(Y)=(Y-Y_c)+(\eta_{out}^{La}(Y)-(Y-Y_c)).
\end{align*}
Since $\kappa^{-1} M\ll 1$, we obtain that for any $|Y-Y_c|\leq L$,
\begin{align*}
	|\eta_r^{La}(Y)-(Y-Y_c)|\leq C|Y-Y_c|^2.
\end{align*}

Again by Lemma \ref{lem:classical-Langer}, we know that for any $|Y-Y_c|\geq L$,
\begin{align*}
C^{-1}(1+|Y-Y_c|)^\f23\leq	|\eta_{out}^{La}(Y)|\leq C(1+|Y-Y_c|)^\f23,
\end{align*}
from which and the definition \eqref{eq:Langer_r-mod}, we infer that for any $|Y-Y_c|\geq L$,
\begin{align*}
C^{-1}(1+|Y-Y_c|)^\f23	\leq|\eta_r^{La}(Y)|\leq C(1+|Y-Y_c|)^\f23.
\end{align*}	
This proves the first statement of the lemma. 

Now we turn to show the results for $\partial_Y\eta_r^{La}$. By the definition of $\partial_Y\eta_r^{La}(Y)$, we know that 
\begin{align*}
	\partial_Y\eta^{La}_r(Y)=&\chi\Big(\frac{Y-Y_c}{\kappa^{-1} M}\Big)+\Big(1-\chi\Big(\frac{Y-Y_c}{\kappa^{-1} M}\Big)\Big)\partial_Y\eta_{out}^{La}(Y)\\
	&+(\kappa^{-1} M)^{-1}\chi'\Big(\frac{Y-Y_c}{\kappa^{-1} M}\Big)(Y-Y_c-\eta_{out}^{La}(Y))\\
	=&1+\Big(1-\chi\Big(\frac{Y-Y_c}{\kappa^{-1} M}\Big)\Big)(\partial_Y\eta_{out}^{La}(Y)-1)\\
	&+(\kappa^{-1} M)^{-1}\chi'\Big(\frac{Y-Y_c}{\kappa^{-1} M}\Big)(Y-Y_c-\eta_{out}^{La}(Y)),
\end{align*}
which gives
\begin{align*}
	|\partial_Y\eta_r^{La}(Y)-1|\leq C|\partial_Y\eta_{out}^{La}(Y)-1|+C(\kappa^{-1}M)^{-1}\Big|\chi'\Big(\frac{Y-Y_c}{\kappa^{-1} M}\Big)\Big||\eta_{out}^{La}(Y)-(Y-Y_c)|.
\end{align*}
By Lemma \ref{lem:classical-Langer}, we have
\begin{align}\label{eq:airy-eta-out1}
	\begin{split}
	&|\partial_Y \eta_{out}^{La}(Y)-1|\leq C|Y-Y_c|,\qquad |Y-Y_c|\leq L,
	\end{split}
\end{align}
which along with  \eqref{eq:airy-eta-out}  show that for any $|Y-Y_c|\leq 2\kappa^{-1} M$,
\begin{align*}
	|\partial_Y\eta_r^{La}(Y)-1|\leq C|Y-Y_c|.
\end{align*}

For any $|Y-Y_c|\geq L$, we know that 
\begin{align*}
	C_2(1+|Y-Y_c|)^{-\f23}\geq (\pa_Y\eta_{out}^{La})^2(Y)=\frac{U_s(Y)-c_r}{U_s'(Y_c)\eta_{out}^{La}(Y)}\geq C_1(1+|Y-Y_c|)^{-\f23}.
\end{align*} 
And for any $|Y-Y_c|\geq 2\kappa^{-1}M$, $\partial_Y\eta_r^{La}(Y)=\partial_Y\eta_{out}^{La}(Y)$. Thus, we deduce that 
\begin{align*}
		&|\partial_Y\eta_{r}^{La}(Y)-1|\leq C|Y-Y_c|,\quad |Y-Y_c|\leq L,\\
	&C_1(1+|Y-Y_c|)^{-\f13}\leq |\partial_Y\eta_{r}^{La}(Y)|\leq C_2(1+|Y-Y_c|)^{-\f13},\quad|Y-Y_c|\geq L.
\end{align*}
This finishes the proof second statement of the lemma. 

Next we show the results for $\partial_Y^2\eta_r^{La}$. We notice that 
\begin{align*}
	\partial_Y^2\eta_r^{La}(Y)=&(\kappa^{-1} M)^{-2}\chi''\Big(\frac{Y-Y_c}{\kappa^{-1} M} \Big)\big( (Y-Y_c)-\eta_{out}^{La}(Y) \big)\\
	&+2(\kappa^{-1} M)^{-1}\chi'\Big(\frac{Y-Y_c}{\kappa^{-1} M} \Big)( 1-\partial_Y \eta_{out}^{La}(Y) )\\
	&+\Big(1-\chi\Big(\frac{Y-Y_c}{\kappa^{-1} M} \Big)\Big)\partial_Y^2\eta_{out}^{La}(Y).
\end{align*}
By Lemma \ref{lem:classical-Langer}, we know that for any $|Y-Y_c|>0$,
\begin{align*}
	|\partial_Y^2\eta_{out}^{La}(Y)|\leq C(1+|Y-Y_c|)^{-\f43}.
\end{align*}
Therefore, for $|Y-Y_c|\leq 2\kappa^{-1}M$, we get by \eqref{eq:airy-eta-out} and \eqref{eq:airy-eta-out1} that
\begin{align*}
	|\partial_Y^2\eta_r^{La}(Y)|\leq C|\varepsilon|^{-\f23}|Y-Y_c|^2+C|\varepsilon|^{-\f13}|Y-Y_c|+C\leq C.
\end{align*}
For $|Y-Y_c|>2\kappa^{-1}M$, we know that 
\begin{align*}
	|\partial_Y^2\eta_r^{La}(Y)|=|\partial_Y^2\eta_{out}^{La}(Y)|\leq C(1+|Y-Y_c|)^{-\f43}.
\end{align*}
Therefore, we obtain that for any $Y\geq 0$,
\begin{align*}
	|\partial_Y^2\eta_{r}^{La}(Y)|\leq C(1+|Y-Y_c|)^{-\f43}.
\end{align*}

\end{proof}

Based on the above lemma, we show the following estimates about the errors $Err_1(Y)$ and $Err_2(Y)$.

\begin{lemma}\label{lem:err1-err2}
	Let $|\alpha|,|\varepsilon|,|c|\ll1$ and $c_r>0$, $|c_i|\leq C|\varepsilon|^\f13$. Let  $0<L\leq 1$ be a small number in Lemma \ref{lem:classical-Langer}. Then it holds that
\begin{align*}
	&|Err_1(Y)|\leq C|\eta_r^{La}(Y)|(|c_i|+|\eta_r^{La}(Y)|)+C|\e||\al|^2,\qquad |Y-Y_c|\leq 2\kappa^{-1} M\\
		&|Err_1(Y)|\leq C|c_i||\eta_r^{La}(Y)|+C|\e||\al|^2,\qquad 2\kappa^{-1} M\leq |Y-Y_c|\leq L,\\
		&|Err_1(Y)|\leq C|c_i|,\qquad|Y-Y_c|\geq L,
\end{align*}
and 
\begin{align*}
	|Err_2(Y)|\leq C|\varepsilon|.
\end{align*}
\end{lemma}

\begin{proof}
	 We first show the estimates of $Err_1(Y)$. Recall that 
	\begin{align*}
		Err_1(Y)=U_s'(Y_c)\eta^{La}(\partial_Y\eta^{La})^2-\varepsilon\alpha^2-(U_s-c).
	\end{align*}
	 By the definition of $\eta^{La}(Y)$, we have
	\begin{align*}
		\eta^{La}(Y)=\chi\Big(\frac{Y-Y_c}{\kappa^{-1} M}\Big)\eta_{in}^{La}(Y)+\Big(1-\chi\Big(\frac{Y-Y_c}{\kappa^{-1} M}\Big)\Big)\eta_{out}^{La}(Y)-\mathrm i U_s'(Y_c)^{-1}c_i,
	\end{align*}
	which gives
	\begin{align*}
		Err_1(Y)=&\chi\Big(\frac{Y-Y_c}{\kappa^{-1} M}\Big)\left( U_s'(Y_c)(Y-Y_c)(\pa_Y\eta^{La})^2- (U_s-c_r)\right)\\
		&+\Big(1-\chi\Big(\frac{Y-Y_c}{\kappa^{-1} M}\Big)\Big)(U_s'(Y_c)\eta_{out}^{La}(Y)(\partial_Y\eta^{La})^2-(U_s-c_r))\\
		&+\mathrm i(1-(\partial_Y\eta^{La})^2)c_i-\e \al^2.
	\end{align*}
	Then by Lemma \ref{lem:est-eta} and \eqref{eq: eta_out^La}, we obtain that for $|Y-Y_c|\leq 2\kappa^{-1} M$,
	\begin{align}\label{eq:eta-Err1}
		|Err_1(Y)|\leq C|Y-Y_c|^2+C|c_i||Y-Y_c|+C|\e||\al|^2.
	\end{align}
	
	For $|Y-Y_c|\geq 2\kappa^{-1} M$, we know that 
	\begin{align*}
		\eta^{La}(Y)=\eta_{out}^{La}(Y)-\mathrm i U_s'(Y_s)^{-1}c_i,
	\end{align*}
	which along with the fact $(\partial_Y\eta_{out}^{La}(Y))^2U_s'(Y_c)\eta_{out}^{La}(Y)=(U_s-c_r)$ implies that 
	\begin{align*}
		Err_1(Y)=&U_s'(Y_c)\eta_{out}^{La}(Y)\partial_Y\eta^{La}_{out}(Y)^2-(U_s-c_r)+ic_i(1-\partial_Y\eta^{La}(Y)^2)-\e \al^2\\
		=&\mathrm  i c_i(1-\partial_Y\eta^{La}(Y)^2)-\e \al^2.
	\end{align*}
From Lemma \ref{lem:est-eta}, we deduce that for any $Y\in\{Y\geq 0:|Y-Y_c|\geq 2\kappa^{-1} M,|Y-Y_c|\leq L\}$,
	\begin{align*}
		|Err_1(Y)|\leq C|c_i||Y-Y_c|+C|\e||\al|^2,
	\end{align*}
	and for any $Y\in\{Y\geq 0: |Y-Y_c|\ge L\}$,
	\begin{align*}
		|Err_1(Y)|\leq C|c_i|+C|\e||\al|^2\leq C|c_i|,
	\end{align*}
	which along with \eqref{eq:eta-Err1} show that 
	\begin{align*}
	&|Err_1(Y)|\leq C|\eta_r^{La}(Y)|(|c_i|+|\eta_r^{La}(Y)|)+C|\e||\al|^2,\qquad |Y-Y_c|\leq 2\kappa^{-1} M\\
		&|Err_1(Y)|\leq C|c_i||\eta_r^{La}(Y)|+C|\e||\al|^2,\qquad 2\kappa^{-1} M\leq |Y-Y_c|\leq L,\\
		&|Err_1(Y)|\leq C|c_i|,\qquad|Y-Y_c|\geq L.
\end{align*}

Next we show the estimates of $Err_2(Y)$. Recall that $Err_2(Y)=\frac{\varepsilon\partial_Y^2\eta^{La}(Y)}{\partial_Y\eta^{La}(Y)}$. By Lemma \ref{lem:est-eta}, we know that for any $|Y-Y_c|\leq L$,
\begin{align*}
	\left|\frac{\varepsilon\partial_Y^2\eta^{La}(Y)}{\partial_Y\eta^{La}(Y)}\right|\leq C|\varepsilon|.
\end{align*}
which implies that for any $|Y-Y_c|\leq L$, $|Err_2(Y)|\leq C|\varepsilon|.$ For $|Y-Y_c|\geq L$, by Lemma \ref{lem:est-eta} again, we have
\begin{align*}
	|\partial_Y\eta^{La}(Y)|\geq C(1+|Y-Y_c|)^{-\f13}\text{ and }|\partial_Y^2\eta^{La}(Y)|\leq C(1+|Y-Y_c|)^{-\f43},
\end{align*}
which implies 
\begin{align*}
	|Err_2(Y)|\leq C|\varepsilon|.
\end{align*}

\end{proof}

Finally, we show the estimates of the Green function of the approximate Airy equation \eqref{eq:Airy-app-Fs}.
\begin{lemma}\label{lem:A1A2}
	Let $0<\delta_0\ll1\leq M$ be the constants in Lemma \ref{lem:Airy-p1}.  Suppose $|\kappa\eta_i^{La}|<\delta_0$  and $|\alpha|,|\varepsilon|, |c|\ll1$. Then for any $0\leq Z\leq Y$, there exists $\gamma_0>0$ such that for $k=0,1,2$ and $j=0,1,2$,
\begin{align*}
		|\partial_Y^k&A_1(Y)A_2(j,Z)|\\
		\leq& C|\varepsilon|^{-\frac{2+k-j}{3}}\mathcal M(k,Y)\mathcal M(-j,Z)e^{-\gamma_0|\varepsilon|^{-\f13}|\eta^{La}(Y)-\eta^{La}(Z)|(|\kappa\eta^{La}(Y)|^\f12+|\kappa\eta^{La}(Z)|^\f12)},
	\end{align*}
	and 
	\begin{align*}
		|\partial_Y^k&A_2(Z)A_1(j,Y)|\\
		\leq& C|\varepsilon|^{-\frac{2+k-j}{3}}\mathcal M(k,Z)\mathcal M(-j,Y) e^{-\gamma_0|\varepsilon|^{-\f13}|\eta^{La}(Y)-\eta^{La}(Z)|(|\kappa\eta^{La}(Y)|^\f12+|\kappa\eta^{La}(Z)|^\f12)},
	\end{align*}
	where $\mathcal M(l;Y)=|\partial_Y\eta^{La}(Y)|^l|\kappa\eta^{La}(Y)|^{-\f14+\f l2}$ for $Y\in \mathcal{N}^{+}\cup \mathcal{N}^{-}$ and $\mathcal M(l;Y)=M^{-\f14+\f l2}$ for $Y\in \mathcal{N}$.
\end{lemma}

\begin{proof}
   By the definition of $A_1(Y)$ and $A_2(Y)$,  we notice that for any $Y\geq 0,$ and $k=0,1$,
     \begin{align*}
     	\partial_Y^kA_1(Y)=&-\mathrm i e^{-2\th_0\mathrm i}|\varepsilon|^{-\f23}U_s'(Y_c)^{-\f13}\left(e^{\mathrm i(\frac{\pi}{6}-\th_0)} \kappa\partial_Y\eta^{La}(Y)\right)^k\partial_Y^kAi\left(e^{\mathrm i(\frac{\pi}{6}-\th_0)} \kappa\eta^{La}(Y)\right),\\
     	\partial_Y^kA_2(Y)=&2\pi \left(e^{\mathrm i(\frac{5\pi}{6}-\th_0)} \kappa\partial_Y\eta^{La}(Y)\right)^k\partial_Y^kAi\left(e^{\mathrm i(\frac{5\pi}{6}-\th_0)} \kappa\eta^{La}(Y)\right),
     \end{align*}
     and 
     \begin{align*}
     	\partial_Y^2A_1(Y)=&-\mathrm i e^{-2\th_0\mathrm i}|\varepsilon|^{-\f23}U_s'(Y_c)^{-\f13}\left(e^{\mathrm i(\frac{\pi}{6}-\th_0)} \kappa\partial_Y\eta^{La}(Y)\right)^2\partial_Y^2Ai\left(e^{\mathrm i(\frac{\pi}{6}-\th_0)} \kappa\eta^{La}(Y)\right)\\
     	&+e^{-\mathrm i(\frac{\pi}{3}+3\th_0)}|\varepsilon|^{-1}\partial_Y^2\eta^{La}(Y)\partial_YAi\left(e^{\mathrm i(\frac{\pi}{6}-\th_0)} \kappa\eta^{La}(Y)\right),\\
     	\partial_Y^2A_1(Y)=&2\pi\left(e^{\mathrm i(\frac{5\pi}{6}-\th_0)} \kappa\partial_Y\eta^{La}(Y)\right)^2\partial_Y^2Ai\left(e^{\mathrm i(\frac{5\pi}{6}-\th_0)} \kappa\eta^{La}(Y)\right)\\
     	&+2\pi e^{\mathrm i(\frac{5\pi}{6}-\th_0)}\kappa\partial_Y^2\eta^{La}(Y)\partial_YAi\left(e^{\mathrm i(\frac{5\pi}{6}-\th_0)} \kappa\eta^{La}(Y)\right).
     \end{align*}

     \no\textbf{Case 1. $Y,Z\in\mathcal N$ and $0\leq Z\leq Y$.}\smallskip
     
      By Lemma \ref{lem:airy-langer-asy} and the properties of Airy functions,  for any $0\leq Z\leq Y$ and for $k=0,1$ and $j=0,1,2$, we have
	\begin{align*}
	&\left|\partial_Y^{k}A_1(Y)A_2(j,Z) \right|+\left|\partial_Y^{k}A_2(Z)A_1(j,Y) \right|\leq C|\varepsilon|^{\frac{j-2-k}{3}}|\partial_Y\eta^{La}(Y)|^k\leq C|\varepsilon|^{\frac{j-2-k}{3}},
	\end{align*}
and
 \begin{align*}
    	&\Big|\partial_Y^{2}A_1(Y)A_2(j,Z)\Big|+\Big|\partial_Y^{2}A_2(Z)A_1(j,Y)\Big|\\
    	&\leq C|\varepsilon|^{\frac{j-4}{3}}|\partial_Y\eta^{La}(Y)|^2+C|\varepsilon|^{\frac{j-3}{3}}|\partial_Y^2\eta^{La}(Y)|\leq C|\varepsilon|^{\frac{j-4}{3}}.
    \end{align*}
 where we used the  fact  for any $Y\in\mathcal N$, $\partial_Y\eta^{La}\sim1$.  Hence, for any $k=0,1,2$ and $j=0,1,2$,
    \begin{align}\nonumber
    	\left|\partial_Y^{k}A_1(Y)A_2(j,Z) \right|+\left|\partial_Y^{k}A_2(Z)A_1(j,Y) \right|\leq C|\varepsilon|^{\frac{j-2-k}{3}}.
    \end{align}
    
	\no\textbf{Case 2. $Y,Z\in\mathcal N^+\cup\mathcal N^-$ and $0\leq Z\leq Y$.} \smallskip

According to \eqref{eq:airy-decay}, Lemma \ref{lem:Airy-green-decay} and Lemma \ref{lem:airy-langer-asy}, there exists $\gamma_1>0$ such that  for $k=0,1$ and $j=0,1,2$,
	\begin{align*}
		\Big|\partial_Y^{k}A_1(Y)A_2(j,Z)\Big|\leq& C|\varepsilon|^{\frac{j-2-k}{3}}\frac{|\partial_Y\eta^{La}(Y)|^k|\kappa\eta^{La}(Y)|^{-\f14+\f k2}}{|\partial_Z\eta^{La}(Z)|^j|\kappa\eta^{La}(Z)|^{\f14+\f j2}}\Big|e^{-\f23((e^{\mathrm i(\frac{\pi}{6}-\th_0)}\kappa\eta^{La}(Y))^\f32-(e^{\mathrm i(\frac{5\pi}{6}-\th_0)}\kappa\eta^{La}(Z))^\f32)}\Big|	\\
		\leq&C|\varepsilon|^{\frac{j-2-k}{3}}\frac{|\partial_Y\eta^{La}(Y)|^k|\kappa\eta^{La}(Y)|^{-\f14+\f k2}}{|\partial_Z\eta^{La}(Z)|^j|\kappa\eta^{La}(Z)|^{\f14+\f j2}}e^{-\gamma_1|\varepsilon|^{-\f13}|\eta^{La}(Y)-\eta^{La}(Z)|(|\kappa\eta^{La}(Y)|^\f12+|\kappa\eta^{La}(Z)|^\f12)}.
    \end{align*}
	Similarly, we can  obtain that for any $Y,Z\in\mathcal N^+\cup\mathcal N^-$ and $Z\leq Y$,
		\begin{align*}
		&\Big|\partial_Y^{k}A_2(Z)A_1(j,Y)\Big|\\
		&\leq C|\varepsilon|^{\frac{j-2-k}{3}}\frac{|\partial_Z\eta^{La}(Z)|^k|\kappa\eta^{La}(Z)|^{-\f14+\f k2}}{|\partial_Y\eta^{La}(Y)|^j|\kappa\eta^{La}(Y)|^{\f14+\f k2}}e^{-\gamma_1|\varepsilon|^{-\f13}|\eta^{La}(Y)-\eta^{La}(Z)|(|\kappa\eta^{La}(Y)|^\f12+|\kappa\eta^{La}(Z)|^\f12)}.
	\end{align*}
For $k=2$ and $j=0,1,2$, we have
    \begin{align}\label{eq:airy-langer-green1}
    	\begin{split}
    				\Big|\partial_Y^2&A_1(Y)A_2(j,Z)\Big|\\
    				\leq&C|\varepsilon|^{\frac{j-4}{3}}\frac{|\partial_Y\eta^{La}(Y)|^2|\kappa\eta^{La}(Y)|^{\f34}}{|\partial_Z\eta^{La}(Z)|^j|\kappa\eta^{La}(Z)|^{\f14+\f j2}}\Big|e^{-\f23((e^{\mathrm i(\frac{\pi}{6}-\th_0)}\kappa\eta^{La}(Y))^\f32-(e^{\mathrm i(\frac{5\pi}{6}-\th_0)}\kappa\eta^{La}(Z))^\f32)}\Big|\\
		&+C|\varepsilon|^{\frac{j-3}{3}}	\frac{|\partial_Y^2\eta^{La}(Y)||\kappa\eta^{La}(Y)|^{\f14}}{|\partial_Y\eta^{La}(Z)|^j|\kappa\eta^{La}(Z)|^{\f14+\f j2}}\Big|e^{-\f23((e^{\mathrm i(\frac{\pi}{6}-\th_0)}\kappa\eta^{La}(Y))^\f32-(e^{\mathrm i(\frac{5\pi}{6}-\th_0)}\kappa\eta^{La}(Z))^\f32)}\Big|\\
		\leq&C|\varepsilon|^{\frac{j-4}{3}}\frac{|\partial_Y\eta^{La}(Y)|^2|\kappa\eta^{La}(Y)|^{\f34}}{|\partial_Z\eta^{La}(Z)|^j|\kappa\eta^{La}(Z)|^{\f14+\f j2}}e^{-\gamma_0|\varepsilon|^{-\f13}|\eta^{La}(Y)-\eta^{La}(Z)|(|\kappa\eta^{La}(Y)|^\f12+|\kappa\eta^{La}(Z)|^\f12)}\\
		&+C|\varepsilon|^{\frac{j-3}{3}}	\frac{|\partial_Y^2\eta^{La}(Y)||\kappa\eta^{La}(Y)|^{\f14}}{|\partial_Z\eta^{La}(Z)|^j|\kappa\eta^{La}(Z)|^{\f14+\f j2}}e^{-\gamma_0|\varepsilon|^{-\f13}|\eta^{La}(Y)-\eta^{La}(Z)|(|\kappa\eta^{La}(Y)|^\f12+|\kappa\eta^{La}(Z)|^\f12)}.
    	\end{split}
    \end{align}
 
By Lemma \ref{lem:est-eta}, we know that for any $Y,Z\in\mathcal N^+\cup\mathcal N^-$,
	\begin{align*}
		|\partial_Y^2\eta^{La}(Y)|\leq C(1+|Y-Y_c|)^{-\f43}\leq C|\partial_Y\eta^{La}(Y)|,\quad |\kappa\eta^{La}(Y)|^\f14\leq |\kappa\eta^{La}(Y)|^\f34,
	\end{align*}
	which along with \eqref{eq:airy-langer-green1} implies 
	\begin{align*}
		\Big|\partial_Y^2&A_1(Y)A_2(j,Z)\Big|\\
		\leq&C|\varepsilon|^{\frac{j-4}{3}}\frac{|\partial_Y\eta^{La}(Y)|^2|\kappa\eta(Y)|^{\f34}}{|\partial_Z\eta^{La}(Z)|^j|\kappa\eta(Z)|^{\f14+\f j2}}e^{-\gamma_1|\varepsilon|^{-\f13}|\eta^{La}(Y)-\eta^{La}(Z)|(|\kappa\eta^{La}(Y)|^\f12+|\kappa\eta^{La}(Z)|^\f12)}.
	\end{align*}
	Similarly, we can show that for any $Y,Z\in\mathcal N^+\cup\mathcal N^-$ and $0\leq Z\leq Y$,
	\begin{align*}
		&\Big|\partial_Y^2 A_2(Z)A_1(j,Y)\Big|\\
		&\leq C|\varepsilon|^{\frac{j-4}{3}}\frac{|\partial_Z\eta^{La}(Z)|^2|\kappa\eta^{La}(Z)|^{\f34}}{|\partial_Y\eta^{La}(Y)|^j|\kappa\eta^{La}(Y)|^{\f14+\f j2}}e^{-\gamma_1|\varepsilon|^{-\f13}|\eta^{La}(Y)-\eta^{La}(Z)|(|\kappa\eta^{La}(Y)|^\f12+|\kappa\eta^{La}(Z)|^\f12)}.
	\end{align*}
	
Summing up, we conclude that for $k, j=0,1,2$ and $Y,Z\in\mathcal N^+\cup\mathcal N^-$ with $0\leq Z\leq Y$,
	\begin{align*}
		&\Big|\partial_Y^{k}A_1(Y)A_2(j,Z)\Big|\\
		&\leq C|\varepsilon|^{\frac{j-2-k}{3}}\frac{|\partial_Y\eta^{La}(Y)|^k|\kappa\eta^{La}(Y)|^{-\f14+\f k2}}{|\partial_Z\eta^{La}(Z)|^j|\kappa\eta^{La}(Z)|^{\f14+\f j2}}e^{-\gamma_1|\varepsilon|^{-\f13}|\eta^{La}(Y)-\eta^{La}(Z)|(|\kappa\eta^{La}(Y)|^\f12+|\kappa\eta^{La}(Z)|^\f12)},
	\end{align*}
    and
		\begin{align*}
		&\Big|\partial_Y^{k}A_2(Z)A_1(j,Y)\Big|\\
		&\leq C|\varepsilon|^{\frac{j-2-k}{3}}\frac{|\partial_Z\eta^{La}(Z)|^k|\kappa\eta^{La}(Z)|^{-\f14+\f k2}}{|\partial_Y\eta(Y)|^j|\kappa\eta^{La}(Y)|^{\f14+\f j2}}e^{-\gamma_1|\varepsilon|^{-\f13}|\eta^{La}(Y)-\eta^{La}(Z)|(|\kappa\eta^{La}(Y)|^\f12+|\kappa\eta^{La}(Z)|^\f12)}.
	\end{align*}
	
	\no\textbf{Case 3. $Y\in\mathcal N^+$ and $Z\in\mathcal N$.}\smallskip
	
By Lemma \ref{lem:airy-langer-asy}, \ref{lem:Airy-green-decay} and \eqref{eq:airy-decay} again, for $k=0,1,2$ and $j=0,1,2$, we have
	\begin{align*}
		|\partial_Y^kA_1(Y)A_2(j,Z)|\leq C|\varepsilon|^{\frac{j-2-k}{3}}|\partial_Y\eta^{La}(Y)|^k|\kappa\eta^{La}(Y)|^{-\f14+\f k2}|e^{-\f23(e^{\mathrm i(\frac{\pi}{6}-\th_0)}\kappa\eta^{La}(Y))^\f32}|.
	\end{align*}
	Since $Y\in \mathcal N^+$ and $|\kappa\eta_i^{La}|<\delta_0$, there exists $\gamma_2>0$ such that 
	\begin{align*}
		|e^{-\f23(e^{\mathrm i(\frac{\pi}{6}-\theta_0)}\kappa\eta^{La}(Y))^\f32}|\leq& Ce^{-\gamma_2|\kappa\eta^{La}(Y)|^\f32}.
	\end{align*}
This shows that for $k=0,1,2$ and $j=0,1,2$,
	\begin{align*}
		|\partial_Y^kA_1(Y)A_2(j,Z)|\leq& C|\varepsilon|^{\frac{j-2-k}{3}}|\partial_Y\eta^{La}(Y)|^k|\kappa\eta^{La}(Y)|^{-\f14+\f k2}e^{-\gamma_2|\kappa\eta^{La}(Y)|^\f32}\\
		\leq&C|\varepsilon|^{\frac{j-2-k}{3}}e^{-\gamma_3|\kappa\eta^{La}(Y)|^\f32}\leq C|\varepsilon|^{\frac{j-2-k}{3}}e^{-\gamma_3(|\kappa\eta^{La}(Y)|^\f32-|\kappa\eta^{La}(Z)|^\f32)}\\
		\leq&C|\varepsilon|^{\frac{j-2-k}{3}} e^{-\gamma_4|\varepsilon|^{-\f13}|\eta^{La}(Y)-\eta^{La}(Z)|(|\kappa\eta^{La}(Y)|^\f12+|\kappa\eta^{La}(Z)|^\f12)}.
	\end{align*}
    By a similar argument, for $k,j=0,1,2$, we have
    \begin{align*}
    	|\partial_Y^kA_2(Z)A_1(j,Y)|\leq C|\varepsilon|^{\frac{j-2-k}{3}} e^{-\gamma_4|\varepsilon|^{-\f13}|\eta^{La}(Y)-\eta^{La}(Z)|(|\kappa\eta^{La}(Y)|^\f12+|\kappa\eta^{La}(Z)|^\f12)}.
    \end{align*}	
    
	\no\textbf{Case 4. $Y\in\mathcal N$ and $Z\in\mathcal N^-$. } \smallskip

By Lemma \ref{lem:airy-langer-asy}, \ref{lem:Airy-green-decay} and \eqref{eq:airy-decay},  for $k=0,1,2$ and $j=0,1,2$, we have
	\begin{align*}
		\Big|\partial_Y^{k}A_1(Y)A_2(j,Z)\Big|\leq& C|\varepsilon|^{\frac{j-2-k}{3}}|\partial_Y\eta^{La}(Z)|^{-j}|\kappa\eta^{La}(Z)|^{-\f14-\f j2}e^{-\gamma_2|\kappa\eta^{La}(Z)|^\f32}.
	\end{align*}
On the other hand, by Lemma \ref{lem:est-eta},  for any $Z\in\mathcal{N}^-$,
	\begin{align*}
	|\partial_Y\eta^{La}(Z)|^{-1}\leq (1+|Z-Y_c|)^{\f13}\leq C|\eta^{La}(Z)|^{\f12}.
	\end{align*}
This shows that for any $Y\in\mathcal N$ and $Z\in\mathcal N^-$,
	\begin{align*}
		\Big|\partial_Y^{k}A_1(Y)A_2(j,Z)\Big|\leq& C|\varepsilon|^{\frac{j-2-k}{3}}e^{-\gamma_2|\kappa\eta^{La}(Z)|^\f32}\\
		\leq& C|\varepsilon|^{\frac{j-2-k}{3}} e^{-\gamma_4|\varepsilon|^{-\f13}|\eta^{La}(Y)-\eta^{La}(Z)|(|\kappa\eta^{La}(Y)|^\f12+|\kappa\eta^{La}(Z)|^\f12)}.
	\end{align*}
	Similarly, for $k,j=0,1,2$, we have
	\begin{align*}
		\Big|\partial_Y^{k}A_2(Z)A_1(j,Y)\Big|\leq C|\varepsilon|^{\frac{j-2-k}{3}} e^{-\gamma_4|\varepsilon|^{-\f13}|\eta^{La}(Y)-\eta^{La}(Z)|(|\kappa\eta^{La}(Y)|^\f12+|\kappa\eta^{La}(Z)|^\f12)}.
	\end{align*}
	
We conclude our lemma by summing up the above four cases.
	\end{proof}
	
\subsection{Approximate solution to the Airy equation} 
Let $w_{app}$ and $\psi_{app}$ be defined as \eqref{eq:airy-w-app} and \eqref{eq:airy-psi-app}. Then $w_{app}$ and $\psi_{app}$ satisfy
\begin{align}\label{eq:ariy-app-goodf}
	\begin{split}
		&\varepsilon(\partial_Y^2-\alpha^2)w_{app}-(U_s-c)w_{app}=F+Err_1w_{app}+Err_2\partial_Yw_{app},\\
		&\partial_Y(A(Y)^{-1}\pa_Y)\psi_{app}(Y)=w_{app}(Y),\quad\quad\lim_{Y\to\infty}\psi_{app}(Y)=\lim_{Y\to\infty}w_{app}(Y)=0.
	\end{split}
\end{align}

Now we show the estimates of $w_{app}$ with the source term $F\sim(U_s-c)^{-1}$ in the sublayer. 

\begin{proposition}\label{prop:airy-gf2}
		Let $\delta_0>0$ be a constant in Lemma \ref{lem:Airy-p1} and  $ c_i\geq c_0 |\e|^\f12$. Suppose $-\delta_0<\mathrm{Im}(\kappa\eta^{La})<\delta_0$, $(\alpha,c)\in\mathbb H_1$ and  $e^{\vartheta Y}(U_s-c)F\in L^{\infty}$. Then it holds  that
\begin{enumerate}
\item For $0<\vartheta\leq 3\eta_0$,
\begin{align*}
&|e^{\vartheta Y}(U_s-c)^2w_{app}(Y)|+|\e|^\f23|e^{\vartheta Y}(U_s-c)\pa_Yw_{app}(Y)|\\
&\qquad\qquad+|\e||e^{\vartheta Y}(U_s-c)\pa_Y^2w_{app}(Y)|
\leq C|\log c_i|\|e^{\vartheta Y}(U_s- c)F\|_{L^\infty}.
\end{align*}
Moreover, for $0<c_i\leq C|\varepsilon|^\f13$, we have
\begin{align*}
 |e^{\vartheta Y}(U_s-c)Err_1(Y)w_{app}(Y)|+&|e^{\vartheta Y}(U_s-c)Err_2(Y)\partial_Yw_{app}(Y)|\\
 \leq& C|\varepsilon|^\f13 |\log c_i| \|e^{\vartheta Y}(U_s- c)F\|_{L^\infty}.
\end{align*}

\item For $1\lesssim\vartheta\leq 3\eta_0$,
\begin{align*}
&|e^{\vartheta Y}(U_s-c)\partial_Y\psi_{app}(Y)|\leq C(1+|c||\varepsilon|^{-\f13}) |\log c_i|\|e^{\vartheta Y}(U_s-c)F\|_{L^\infty}   ,\\
&|e^{\vartheta Y}\psi_{app}(Y)|\leq C(1+|c||\varepsilon|^{-\f13}) |\log c_i|^2\|e^{\vartheta Y}(U_s-c)F\|_{L^\infty}.
\end{align*}	
\end{enumerate}
\end{proposition}

Before presenting the proof, we introduce some notations and estimates used frequently. For any $Y\geq 0$, we define
\begin{align*}
	&\mathcal N_{near}(Y)=\big\{Z\in\mathcal N^-\cup\mathcal N^+ :|\kappa(\eta^{La}(Z)-\eta^{La}(Y))|\leq 1\big\},\\
	&\mathcal N_{far}(Y)=\big\{Z\in\mathcal N^-\cup\mathcal N^+ :|\kappa(\eta^{La}(Z)-\eta^{La}(Y))|>1\big\},
\end{align*}
where the definitions of $\mathcal N^+$ and $\mathcal N^-$ are given in \eqref{def: N^+}-\eqref{def: N^-}. 

For $Z, Y\geq 0$, we have
\begin{align*}
	e^{\vartheta|Y-Z|}\leq Ce^{C|\eta_r^{La}(Y)^\f32-\eta_r^{La}(Z)^\f32|}\leq& Ce^{C|\eta_r^{La}(Y)-\eta_r^{La}(Z)|(|\eta_r^{La}(Y)|^\f12+|\eta_r^{La}(Z)|^\f12)}\\
	\leq& Ce^{C|\varepsilon|^{\f12} |\varepsilon|^{-\f13}|\eta^{La}(Y)-\eta^{La}(Z)|(|\kappa\eta^{La}(Y)|^\f12+|\kappa\eta^{La}(Z)|^\f12)},
\end{align*}
which implies 
\begin{align}\label{eq:ariy-app-gf-decay1}
	\begin{split}
		&e^{-\gamma_0|\varepsilon|^{-\f13}|\eta^{La}(Y)-\eta^{La}(Z)|(|\kappa\eta^{La}(Y)|^\f12+|\kappa\eta^{La}(Z)|^\f12)}e^{\vartheta|Y-Z|}\\
		&\quad\leq e^{-C|\varepsilon|^{-\f13}|\eta^{La}(Y)-\eta^{La}(Z)|(|\kappa\eta^{La}(Y)|^\f12+|\kappa\eta^{La}(Z)|^\f12)(1+|\varepsilon|^\f12)}\\
		&\quad\leq e^{-C|\varepsilon|^{-\f13}|\eta^{La}(Y)-\eta^{La}(Z)|(|\kappa\eta^{La}(Y)|^\f12+|\kappa\eta^{La}(Z)|^\f12)}.
	\end{split}
\end{align}

The proof of Proposition \ref{prop:airy-gf2} is rather long and split into the following three subsections.

\subsubsection{Estimates of $w_{app}(Y)$} We consider three cases.\medskip

\no\textbf{Case 1. $|Y-Y_c|\geq L$. } In this case, we first notice that 
\begin{align}\label{eq:airy-app-f-sim}
	|w_{app}(Y)|\sim|(U_s-c)^2w_{app}(Y)|.
\end{align}
Thus, it is enough to show the control of $w_{app}(Y)$. By the definition of $w_{app}$ and Lemma \ref{lem:A1A2}, for any $|Y-Y_c|\geq L$, we have
	\begin{align}\label{eq:airy-app-f-Y1}
		\begin{split}
		|w_{app}(Y)|\leq& C|\varepsilon|^{-\f23}\int_{Z\in\mathcal N_{near}(Y)}|\kappa\eta^{La}(Y)|^{-\f14}|\kappa\eta^{La}(Z)|^{-\f14}|\partial_Z\eta^{La}(Z)|^{-1} \\
		&\qquad\times e^{-\gamma_0|\varepsilon|^{-\f13}|\eta^{La}(Y)-\eta^{La}(Z)|(|\kappa\eta^{La}(Y)|^\f12+|\kappa\eta^{La}(Z)|^\f12)}|F(Z)|dZ\\
		&+C|\varepsilon|^{-\f23}\int_{Z\in\mathcal N_{far}(Y)}|\kappa\eta^{La}(Y)|^{-\f14}|\kappa\eta^{La}(Z)|^{-\f14}|\partial_Z\eta^{La}(Z)|^{-1} \\
		&\qquad\times e^{-\gamma_0|\varepsilon|^{-\f13}|\eta^{La}(Y)-\eta^{La}(Z)|(|\kappa\eta^{La}(Y)|^\f12+|\kappa\eta^{La}(Z)|^\f12)}|F(Z)|dZ\\
		&+C|\varepsilon|^{-\f23}\int_{Z\in\mathcal N}|\kappa\eta^{La}(Y)|^{-\f14}e^{-\gamma_0|\varepsilon|^{-\f13}|\eta^{La}(Y)-\eta^{La}(Z)|(|\kappa\eta^{La}(Y)|^\f12+|\kappa\eta^{La}(Z)|^\f12)}|F(Z)|dZ\\
		=&|\e|^{-\f23}(I_1+I_2+I_3).
		\end{split}
	\end{align}
	
For $Z\in\mathcal N_{near}(Y)$ and  $|Y-Y_c|\geq L$, we get by  Lemma \ref{lem:est-eta} that
\begin{align*}
|\kappa\eta^{La}(Y)|^{-\f14}|\kappa\eta^{La}(Z)|^{-\f14}|\partial_Z\eta^{La}(Z)|^{-1}\leq&C|\kappa\eta^{La}(Z)|^{-\f12}|\eta^{La}(Z)|^\f12\leq C\kappa^{-\f12}\leq C|\e|^\f16,
\end{align*}
and
\begin{align}\label{est: exp-near}
e^{-\gamma_0|\varepsilon|^{-\f13}|\eta^{La}(Y)-\eta^{La}(Z)|(|\kappa\eta^{La}(Y)|^\f12+|\kappa\eta^{La}(Z)|^\f12)}\leq e^{-C|\varepsilon|^{-\f12}|\eta^{La}(Y)-\eta^{La}(Z)|},
\end{align}
which along with \eqref{eq:ariy-app-gf-decay1}  show that 
\begin{align*}
|I_1|\leq&C|\varepsilon|^{\f16}\int_{Z\in\mathcal N_{near}(Y)}e^{-\gamma_0|\varepsilon|^{-\f13}|\eta^{La}(Y)-\eta^{La}(Z)|(|\kappa\eta^{La}(Y)|^\f12+|\kappa\eta^{La}(Z)|^\f12)}|F(Z)|dZ\\
\leq& C|\e|^\f16 e^{-\vartheta Y}\|e^{\vartheta Z}(U_s-c)F\|_{L^\infty}\int_{Z\in\mathcal N_{near}(Y)}e^{-C|\varepsilon|^{-\f12}|\eta^{La}(Y)-\eta^{La}(Z)|} dZ\\
\leq&C|\e|^\f23 e^{-\vartheta Y}\|e^{\vartheta Y}(U_s-c)F\|_{L^\infty},
\end{align*}
where in the last step we used the fact that 
\begin{align*}
\int_{Z\in\mathcal N_{near}(Y)}e^{-C|\varepsilon|^{-\f12}|\eta^{La}(Y)-\eta^{La}(Z)|} dZ\leq C|\e|^\f12.
\end{align*}

For $Z\in\mathcal N_{far}(Y)$,  we get by  Lemma \ref{lem:est-eta} that
\begin{align*}
|\kappa\eta^{La}(Y)|^{-\f14}|\kappa\eta^{La}(Z)|^{-\f14}|\partial_Z\eta^{La}(Z)|^{-1}\leq&C\kappa^{-\f12}|\eta^{La}(Z)|^\f14\leq C\kappa^{-\f34}|\kappa \eta^{La}(Z)|^{\f14},
\end{align*}
and 
\begin{align}\label{est: exp-far}
e^{-C|\varepsilon|^{-\f13}|\eta^{La}(Y)-\eta^{La}(Z)|(|\kappa\eta^{La}(Y)|^\f12+|\kappa\eta^{La}(Z)|^\f12)}\leq e^{-C(|\kappa\eta^{La}(Y)|^\f12+|\kappa\eta^{La}(Z)|^\f12)},
\end{align}
from which, we infer that 
\begin{align*}
|I_2|\leq&Ce^{-\vartheta Y}|\varepsilon|^{\frac{1}{4}}\int_{Z\in\mathcal N_{far}(Y)}|\kappa\eta^{La}(Z)|^\f14e^{-C|\varepsilon|^{-\f13}|\eta^{La}(Y)-\eta^{La}(Z)|(|\kappa\eta^{La}(Y)|^\f12+|\kappa\eta^{La}(Z)|^\f12)}|e^{\vartheta Z}F(Z)|dZ\\
\leq&Ce^{-\vartheta Y}|\varepsilon|^{\frac{1}{4}}\int_{Z\in\mathcal N_{far}(Y)}|\kappa\eta^{La}(Z)|^\f14e^{-C(|\kappa\eta^{La}(Y)|^\f12+|\kappa\eta^{La}(Z)|^\f12)}|U_s(Z)-c|^{-1}|e^{\vartheta Z}(U_s(Z)-c)F(Z)|dZ\\
\leq&Ce^{-\vartheta Y}|\varepsilon|^{\frac{1}{4}}\int_{Z\in\mathcal N_{far}(Y)}|\kappa\eta^{La}(Z)|^\f14e^{-C(|\e|^{-\f16}+|\kappa\eta^{La}(Z)|^\f12)}|\eta^{La}(Z)|^{-1}|e^{\vartheta Z}(U_s(Z)-c)F(Z)|dZ\\
\leq& C|\e|^{N}e^{-\vartheta Y}\|e^{\vartheta Y}(U_s-c)F\|_{L^\infty}
\end{align*}
for any $N>0$.

For $Z\in\mathcal N$ and any $|Y-Y_c|\geq L$, we get by  Lemma \ref{lem:est-eta} that
\begin{align*}
|\kappa\eta^{La}(Y)|^{-\f14}\leq& C|\kappa|^{-\f14}\leq C|\e|^{\f1{12}},
\end{align*}
and 
\begin{align}\label{est: exp-N}
e^{-C|\varepsilon|^{-\f13}|\eta^{La}(Y)-\eta^{La}(Z)|(|\kappa\eta^{La}(Y)|^\f12+|\kappa\eta^{La}(Z)|^\f12)}\leq&e^{-C |\e|^{-\f13}|\kappa|^{\f12}}\leq e^{-C|\e|^{-\f12}},
\end{align}
from which, we infer that
\begin{align*}
|I_3|
\leq&C|\varepsilon|^{\frac{1}{12}}e^{-\vartheta Y}e^{-C|\varepsilon|^{-\f12}}\|(U_s- c)F\|_{L^\infty(\mathcal N)}\int_{Z\in\mathcal N}|U_s- c|^{-1}dZ\\
\leq& C|\varepsilon|^{\frac{1}{12}}e^{-C|\varepsilon|^{-\f12}}|\log|c_i|\|(u- c)F\|_{L^\infty(\mathcal N)}\\
\leq&C|\e|^{N}e^{-\vartheta Y}\|e^{\vartheta Y}(U_s-c)F\|_{L^\infty},
\end{align*}
for any $N> 0$. Therefore, we conclude that for $|Y-Y_c|\geq L$,
	\begin{align*}
		|e^{\vartheta Y}w_{app}(Y)|\leq C \|e^{\vartheta Y}(U_s-c)F\|_{L^\infty},
	\end{align*}
	which along with \eqref{eq:airy-app-f-sim} implies that for any $|Y-Y_c|\geq L$,
	\begin{align}\label{eq:airy-app-f-c1}
		|e^{\vartheta Y}(U_s-c)^2w_{app}(Y)|\leq C \|e^{\vartheta Y}(U_s-c)F\|_{L^\infty}.
	\end{align}
	
	\no\textbf{Case 2. $|\kappa\eta^{La}(Y)|\geq 3M+1$ and $|Y-Y_c|\leq L$.} We denote $\mathcal N_1=\{Y\geq 0:|\kappa\eta^{La}(Y)|\geq 3M+1, |Y-Y_c|\leq L\}$. Notice that $\mathcal N_1$ is a subset of $\mathcal N^+\cup\mathcal N^-$ and $\mathcal N\subset\mathcal N_{far}(Y)$ for any $Y\in\mathcal N_1$.	
	By Lemma \ref{lem:est-eta}, we know that for any $|Y-Y_c|\leq L$, $|\eta^{La}(Y)|\sim|U_s-c|$.
	 Again by Lemma \ref{lem:A1A2},  we deduce that for any $Y\in\mathcal N_1$,
	 \begin{align}\label{eq:airy-app-f-Y2}
	 	\begin{split}
	 	&|(U_s- c)^2w_{app}(Y)|\\
	 	&\leq C|\kappa\eta^{La}(Y)|^2 \int_{Z\in\mathcal N_{near}(Y)}|\kappa\eta^{La}(Y)|^{-\f14}|\kappa\eta^{La}(Z)|^{-\f14}  |\partial_Z\eta^{La}(Z)|^{-1}\\
		&\qquad\times e^{-\gamma_0|\varepsilon|^{-\f13}|\eta^{La}(Y)-\eta^{La}(Z)|(|\kappa\eta^{La}(Y)|^\f12+|\kappa\eta^{La}(Z)|^\f12)}|F(Z)|dZ\\
		&\quad+C|\kappa\eta^{La}(Y)|^2  \int_{Z\in\mathcal N_{far}(Y)}|\kappa\eta^{La}(Y)|^{-\f14}|\kappa\eta^{La}(Z)|^{-\f14} |\partial_Z\eta^{La}(Z)|^{-1}\\
		&\qquad\times e^{-\gamma_0|\varepsilon|^{-\f13}|\eta^{La}(Y)-\eta^{La}(Z)|(|\kappa\eta^{La}(Y)|^\f12+|\kappa\eta^{La}(Z)|^\f12)}|F(Z)|dZ\\
		&\quad+C|\kappa\eta^{La}(Y)|^2 \int_{Z\in\mathcal N}|\kappa\eta^{La}(Y)|^{-\f14}|\partial_Z\eta^{La}(Z)|^{-1}\\
		&\qquad\times e^{-\gamma_0|\varepsilon|^{-\f13}|\eta^{La}(Y)-\eta^{La}(Z)|(|\kappa\eta^{La}(Y)|^\f12+|\kappa\eta^{La}(Z)|^\f12)}|F(Z)|dZ\\
		&=I_4+I_5+I_6.
	 	\end{split}
	 \end{align}
	 
For any $Y\in\mathcal N_1$ and $Z\in\mathcal N_{near}(Y)$, we get by  Lemma \ref{lem:est-eta} that
\begin{align*}
|\kappa\eta^{La}(Y)|^2 |\kappa\eta^{La}(Y)|^{-\f14}|\kappa\eta^{La}(Z)|^{-\f14}  |\partial_Z\eta^{La}(Z)|^{-1}|U_s(Z)-c|^{-1}\leq& C|\kappa\eta^{La}(Y)|^{\f32}|\eta^{La}(Z)|^{-1}\\
\leq& C|\e|^{-\f13}|\kappa\eta^{La}(Y)|^{\f12},
\end{align*}
and 
\begin{align*}
\int_{Z\in\mathcal N_{near}(Y)}|\kappa\eta^{La}(Y)|^{\f12}e^{-C|\varepsilon|^{-\f13}|\eta^{La}(Y)-\eta(Z)||\kappa\eta^{La}(Y)|^{\f12}}dZ\leq C|\e|^\f13,
\end{align*}
from which  and  \eqref{eq:ariy-app-gf-decay1}, we infer that 
\begin{align}\label{eq:airy-app-f-Y2-1}
\begin{split}
|I_4|\leq& C|\varepsilon|^{-\f13} e^{-\vartheta Y}\|e^{\vartheta Y}(U_s-c)F\|_{L^\infty}\\
&\qquad\times\int_{Z\in\mathcal N_{near}(Y)}|\kappa\eta^{La}(Y)|^{\f12}e^{-C|\varepsilon|^{-\f13}|\eta^{La}(Y)-\eta(Z)||\kappa\eta^{La}(Y)|^{\f12}}dZ\\
\leq& Ce^{-\vartheta Y}\|e^{\vartheta Y}(U_s-c)F\|_{L^\infty}.
\end{split}
\end{align}
 For any $Y\in\mathcal N_1$ and $Z\in \mathcal{N}_{far}(Y)$, we have
\begin{align*}
&|\kappa\eta^{La}(Y)|^2 |\kappa\eta^{La}(Y)|^{-\f14}|\kappa\eta^{La}(Z)|^{-\f14}  |\partial_Z\eta^{La}(Z)|^{-1}|U_s(Z)-c|^{-1}\\
&\qquad\leq C|\varepsilon|^{-\f13}|\kappa\eta^{La}(Z)|^\f74(1+|\eta^{La}(Z)|^\f12).
\end{align*}
Then we get by \eqref{eq:ariy-app-gf-decay1} and \eqref{est: exp-far}  that
	\begin{align}\label{eq:airy-app-f-Y2-2}
		\begin{split}
		|I_5|\leq& C|\varepsilon|^{-\f13} e^{-\vartheta Y}\|e^{\vartheta Y}(U_s- c)F\|_{L^\infty}|\kappa\eta^{La}(Y)|^{\f7 4}e^{-C|\kappa\eta^{La}(Y)|^\f12}\\
		&\qquad\times\int_{Z\in\mathcal N_{far}(Y)}(1+|\eta^{La}(Z)|^\f12)e^{-C|\kappa\eta^{La}(Z)|^\f12}dZ\\
		\leq& C e^{-\vartheta Y}\|e^{\vartheta Y}(U_s- c)F\|_{L^\infty}.
		\end{split}
	\end{align}
	
For any $Y\in\mathcal N_1$ and $Z\in \mathcal{N}$, $|\varepsilon|^{-\f13}|\eta^{La}(Y)-\eta^{La}(Z)|\geq 1$. Then we infer that
\begin{align*}
|I_6|
\leq& Ce^{-\vartheta Y}|\kappa\eta^{La}(Y)|^{\f74}e^{-C|\kappa\eta^{La}(Y)|^\f12}\|(U_s- c)F\|_{L^\infty(\mathcal N)}\int_{Z\in\mathcal N}|U_s(Z)- c|^{-1}dZ\\
\leq& Ce^{-\vartheta Y}\|e^{\vartheta Y}(U_s- c)F\|_{L^\infty}|\log c_i|,
\end{align*}
	which along with \eqref{eq:airy-app-f-Y2}, \eqref{eq:airy-app-f-Y2-1} and \eqref{eq:airy-app-f-Y2-2} deduces that for any $Y\in\mathcal N_1$,		\begin{align}\label{eq:airy-app-f-c2}
\begin{split}
|e^{\vartheta Y}(U_s-& c)^2w_{app}(Y)|\leq C|\log c_i|\|e^{\vartheta Y}(U_s- c)F\|_{L^\infty}.
		\end{split}
	\end{align}
	
	\no\textbf{Case 3. $|\kappa\eta^{La}(Y)|\leq 3M+1$.} We denote $\mathcal N_2=\{Y:|\kappa\eta^{La}(Y)|\leq 3M+1\}$. By Lemma \ref{lem:A1A2}, we have that for any $Y\in\mathcal N_2$ ,
	\begin{align}\label{eq:airy-app-f-Y3}
		\begin{split}
		|(U_s-& c)^2w_{app}(Y)|\\
		\leq& C\int_{Z\in\mathcal N_{far}(Y)}
		|\kappa\eta^{La}(Z)|^{-\f14}|\partial_Z\eta^{La}(Z)|^{-1} \\
		&\qquad\times e^{-\gamma_0|\varepsilon|^{-\f13}|\eta^{La}(Y)-\eta^{La}(Z)|(|\kappa\eta^{La}(Y)|^\f12+|\kappa\eta^{La}(Z)|^\f12)}|F(Z)|dZ\\
		&+C\int_{Z\in\mathcal N\cup\mathcal N_{near}(Y)}|F(Z)|dZ=I_7+I_8.
		\end{split}
	\end{align}
For any $Y\in\mathcal N_2$ and $Z\in \mathcal{N}_{far}(Y)$, we have
\begin{align*}
|\kappa\eta^{La}(Z)|^{-\f14}|\partial_Z\eta^{La}(Z)|^{-1} |U_s(Z)-c|^{-1}\leq&{C|\varepsilon|^{-\f13} |\kappa\eta^{La}(Z)|^{-\f14}(1+|\eta^{La}(Z)|^\f12)\leq C|\e|^{-\f13},}
\end{align*}
which along with \eqref{eq:ariy-app-gf-decay1} and \eqref{est: exp-far} gives 
\begin{align}\label{eq:airy-app-f-Y3-1}
\begin{split}
|I_7|\leq& C|\varepsilon|^{-\f13}e^{-\vartheta Y}\|e^{\vartheta Y}(U_s- c)F\|_{L^\infty}\int_{Z\in\mathcal N_{far}(Y)}(1+|\eta^{La}(Z)|^\f12)e^{-C|\kappa\eta^{La}(Z)|^\f12}dZ\\
\leq& Ce^{-\vartheta Y}\|e^{\vartheta Y}(U_s- c)F\|_{L^\infty}.
\end{split}
\end{align}

	For any $Y\in\mathcal N_2$ and $Z\in\mathcal N\cup\mathcal N_{near}(Y)$, it is easy to see that
\begin{align}\nonumber
|I_8|\leq&Ce^{-\vartheta Y}\|e^{\vartheta Y}(U_s- c)F\|_{L^\infty}|\log c_i|,
\end{align}
	which along with \eqref{eq:airy-app-f-Y3} and \eqref{eq:airy-app-f-Y3-1} implies that for any $Y\in\mathcal N_2$ 
\begin{align}\label{eq:airy-app-f-4}
\begin{split}
|e^{\vartheta Y}(U_s-& c)^2w_{app}(Y)|
\leq C|\log c_i|\|e^{\vartheta Y}(U_s- c)F\|_{L^\infty}.
\end{split}
\end{align}

Summing up the above three cases, we conclude the following estimate 
\begin{align}\label{eq:airy-app-w_app-1}
|e^{\vartheta Y}(U_s-c)^2w_{app}(Y)|
\leq C|\log c_i|\|e^{\vartheta Y}(U_s- c)F\|_{L^\infty}.
\end{align}

\subsubsection{Estimates of $\partial_Y w_{app}(Y)$ and  $\partial_Y^2 w_{app}$}
We consider three cases.\medskip

    \no\textbf{Case 1. $|Y-Y_c|\geq L$. } 
    Again in this case, it is enough to show the control of $\partial_Yw_{app}(Y)$. By Lemma \ref{lem:A1A2}, we have
    \begin{align}\label{eq:airy-app-f-dw1}
    	\begin{split}
    	|\partial_Y w_{app}(Y)|\leq& C|\varepsilon|^{-1}\Bigg(\int_{Z\in\mathcal N_{near}(Y)}+\int_{Z\in\mathcal N_{far}(Y)}\Bigg)
	|\kappa\eta^{La}(Y)|^{\f14}|\partial_Y\eta^{La}(Y) ||\kappa\eta^{La}(Z)|^{-\f14}|\partial_Z\eta^{La}(Z)|^{-1} \\
		&\qquad\times e^{-\gamma_0|\varepsilon|^{-\f13}|\eta^{La}(Y)-\eta^{La}(Z)|(|\kappa\eta^{La}(Y)|^\f12+|\kappa\eta^{La}(Z)|^\f12)}|F(Z)|dZ\\
				&+C|\varepsilon|^{-1}\int_{Z\in\mathcal N}|\kappa\eta^{La}(Y)|^{\f14}|\partial_Y\eta^{La}(Y)| |\partial_Z\eta^{La}(Z)|^{-1}\\
				&\qquad\times e^{-\gamma_0|\varepsilon|^{-\f13}|\eta^{La}(Y)-\eta^{La}(Z)|(|\kappa\eta^{La}(Y)|^\f12+|\kappa\eta^{La}(Z)|^\f12)}|F(Z)|dZ\\
				=&|\e|^{-1}(II_1+II_2+II_3).
    	\end{split}
    \end{align}
    
For any $|Y-Y_c|\geq L$ and $Z\in \mathcal{N}_{near}(Y)$, we get by  Lemma \ref{lem:est-eta} that
\begin{align*}
|\kappa\eta^{La}(Y)|^{\f14}|\partial_Y\eta^{La}(Y) ||\kappa\eta^{La}(Z)|^{-\f14}|\partial_Z\eta^{La}(Z)|^{-1}|U_s(Z)-c|^{-1}\leq C,
\end{align*}
and for $Z\in \mathcal{N}_{far}(Y)$, we have
\begin{align*}
&|\kappa\eta^{La}(Y)|^{\f14}|\partial_Y\eta^{La}(Y) | |\kappa\eta^{La}(Z)|^{-\f14}|\partial_Z\eta^{La}(Z)|^{-1}
|U_s(Z)-c|^{-1}\\
&\quad\leq C|\kappa\eta^{La}(Y)|^{\f14}|\eta^{La}(Y)|^{-\f12}|\varepsilon|^{-\f13}(1+|\eta^{La}(Z)|^\f12)\\
&\quad\leq C|\varepsilon|^{-\f{5}{12}}(1+|\eta^{La}(Z)|^\f12),
\end{align*}
and for $Z\in \mathcal{N}$, we have 
\begin{align*}
&|\kappa\eta^{La}(Y)|^{\f14}|\partial_Y\eta^{La}(Y) |\partial_Z\eta^{La}(Z)|^{-1}|e^{-C|\varepsilon|^{-\f13} |\eta^{La}(Y)-\eta^{La}(Z)|(|\kappa\eta^{La}(Y)|^\f12+|\kappa\eta^{La}(Z)|^\f12)}\\
&\quad\leq C|\kappa\eta^{La}(Y)|^{\f14} e^{-C|\varepsilon|^{-\f13}(|\kappa\eta^{La}(Y)|^\f12+|\kappa\eta^{La}(Z)|^\f12)}\leq Ce^{-|\e|^{-\f13}}.
\end{align*}
Then we apply \eqref{eq:ariy-app-gf-decay1}, \eqref{est: exp-near}, \eqref{est: exp-far} to obtain
    \begin{align}\label{eq:airy-app-f-dw1-1}
    	\begin{split}
    	|II_1|\leq& Ce^{-\vartheta Y} \|e^{\vartheta Y}(U_s-c)F\|_{L^\infty}\int_{Z\in\mathcal N_{near}(Y)}e^{-C|\varepsilon|^{-\f12}|\eta^{La}(Y)-\eta^{La}(Z)||\eta^{La}(Y)|^\f12}dZ\\
		\leq& C|\varepsilon|^\f12 e^{-\vartheta Y} \|e^{\vartheta Y}(U_s-c)F\|_{L^\infty},
    	\end{split}
    \end{align}
    and
    \begin{align}\label{eq:airy-app-f-dw1-2}
    	\begin{split}
    		|II_2|\leq& C|\e|^{-\f{5}{12}} e^{-\vartheta Y} \|e^{\vartheta Y}(U_s-c)F\|_{L^\infty}\int_{Z\in\mathcal N_{far}(Y)}(1+|\eta^{La}(Z)|^\f12)e^{-C(|\kappa\eta^{La}(Y)|^\f12+|\kappa\eta^{La}(Z)|^\f12)} dZ\\
		\leq& C|\varepsilon|^{N}e^{-\vartheta Y} \|e^{\vartheta Y}(U_s-c)F\|_{L^\infty},
    	\end{split}
    \end{align}
 and
 	\begin{align*}
		|II_3|
		\leq&Ce^{-\vartheta Y} \|e^{\vartheta Y}(U_s-c)F\|_{L^\infty}e^{-C|\varepsilon|^\f13}\int_{\mathcal N}|U_s(Z)- c|^{-1}dZ\\
		\leq&C|\e|^{N}e^{-\vartheta Y} \|e^{\vartheta Y}(U_s-c)F\|_{L^\infty}|\log c_i|,
	\end{align*}
	for any $N>0$,
	which along with \eqref{eq:airy-app-f-dw1} and \eqref{eq:airy-app-f-dw1-1}  deduce that for any $|Y-Y_c|\geq L$,
	\begin{align}\label{eq:airy-app-f-dwe-1}
		|e^{\vartheta Y}(U_s-c)\partial_Yw_{app}(Y)|\leq C|\varepsilon|^{-\f23} \|e^{\vartheta Y}(U_s-c)F\|_{L^\infty}|\log c_i|.
	\end{align}
	
	\no\textbf{Case 2. $|\kappa\eta^{La}(Y)|\geq 3M+1$ and $|Y-Y_c|\leq L$.}  By Lemma \ref{lem:A1A2}, for any $Y\in\mathcal N_1$,
	\begin{align}\label{eq:airy-app-f-dw2}
		\begin{split}
		&|(U_s- c)\partial_Y w_{app}(Y)|\\
		&\leq C|\varepsilon|^{-\frac{2}{3}}|\kappa\eta^{La}(Y)|\Bigg(\int_{Z\in\mathcal N_{near}(Y)}+\int_{Z\in\mathcal N_{far}(Y)}\Bigg)|\kappa\eta^{La}(Y)|^{\f14}|\partial_Y\eta^{La}(Y)|\kappa\eta^{La}(Z)|^{-\f14} |\partial_Z\eta^{La}(Z)|^{-1}\\
		&\qquad\times e^{-\gamma_0|\varepsilon|^{-\f13}|\eta^{La}(Y)-\eta^{La}(Z)|(|\kappa\eta^{La}(Y)|^\f12+|\kappa\eta^{La}(Z)|^\f12)}|F(Z)|dZ\\
		&\qquad+C|\varepsilon|^{-\frac{2}{3}}|\kappa\eta^{La}(Y)|\int_{Z\in\mathcal N}|\kappa\eta^{La}(Y)|^{\f14}|\partial_Y\eta^{La}(Y)||\partial_Z\eta^{La}(Z)|^{-1} \\
		&\qquad\times e^{-\gamma_0|\varepsilon|^{-\f13}|\eta^{La}(Y)-\eta^{La}(Z)|(|\kappa\eta^{La}(Y)|^\f12+|\kappa\eta^{La}(Z)|^\f12)}|F(Z)|dZ=|\e|^{-\f23}(II_4+II_5+II_6).
		\end{split}
	\end{align}
	
For $Y\in \mathcal{N}_1$ and $Z\in \mathcal{N}_{near}(Y)$, we have
\begin{align*}
|\kappa\eta^{La}(Y)||\kappa\eta^{La}(Y)|^{\f14}|\partial_Y\eta^{La}(Y)|\kappa\eta^{La}(Z)|^{-\f14} |\partial_Z\eta^{La}(Z)|^{-1}|U_s(Z)-c|^{-1}\leq C|\e|^{-\f13},
\end{align*}
and for $Y\in \mathcal{N}_1$ and $Z\in \mathcal{N}_{far}(Y)$, 
\begin{align*}
&|\kappa\eta^{La}(Y)||\kappa\eta^{La}(Y)|^{\f14}|\partial_Y\eta^{La}(Y)||\kappa\eta^{La}(Z)|^{-\f14} |\partial_Z\eta^{La}(Z)|^{-1}|U_s(Z)-c|^{-1}\\
&\leq C|\e|^{-\f13}|\kappa\eta^{La}(Y)|^{\f54}(1+|\eta^{La}(Z)|^\f12),
\end{align*} 
and for $Y\in \mathcal{N}_1$ and $Z\in \mathcal{N}$, 
\begin{align*}
&|\kappa\eta^{La}(Y)||\kappa\eta^{La}(Y)|^{\f14}|\partial_Y\eta^{La}(Y)| |\partial_Z\eta^{La}(Z)|^{-1}\leq C|\kappa\eta^{La}(Y)|^{\f54},\\
&|\varepsilon|^{-\f13}|\eta^{La}(Y)-\eta^{La}(Z)|\geq 1.
\end{align*}
Then we apply \eqref{eq:ariy-app-gf-decay1}, \eqref{est: exp-far} and a similar process as in $I_4-I_6$ to obtain
	\begin{align}\label{eq:airy-app-f-dw2-1}
		\begin{split}
		|II_4|\leq&C|\e|^{-\f13}e^{-\vartheta Y} \|e^{\vartheta Y}(U_s-c)F\|_{L^\infty}\int_{Z\in\mathcal N_{near}(Y)}e^{-C|\varepsilon|^{-\f13}|\eta^{La}(Y)-\eta^{La}(Z)|}dZ\\
			\leq& Ce^{-\vartheta Y} \|e^{\vartheta Y}(U_s-c)F\|_{L^\infty},
		\end{split}
	\end{align}
	and
	\begin{align}\label{eq:airy-app-f-dw2-2}
		\begin{split}
			|II_5|\leq& C|\varepsilon|^{-\f13}e^{-\vartheta Y} \|e^{\vartheta Y}(U_s-c)F\|_{L^\infty}|\kappa\eta^{La}(Y)|^\f54\\
		&\qquad\times\int_{Z\in\mathcal N_{far}(Y)}(1+|\eta^{La}(Z)|^\f12)e^{-C(|\kappa\eta^{La}(Y)|^\f12+|\kappa\eta^{La}(Z)|^\f12)}dZ\\
			\leq& Ce^{-\vartheta Y} \|e^{\vartheta Y}(U_s-c)F\|_{L^\infty},
		\end{split}
	\end{align} 
	and
		\begin{align*}
		|II_6|\leq& Ce^{-\vartheta Y} \|e^{\vartheta Y}(U_s-c)F\|_{L^\infty}|\kappa\eta^{La}(Y)|^\f54e^{-C|\kappa\eta^{La}(Y)|^\f12}\int_{Z\in\mathcal N}|U_s(Z)- c|^{-1}dZ\\
		\leq &C e^{-\vartheta Y} \|e^{\vartheta Y}(U_s-c)F\|_{L^\infty}|\log c_i|,
	\end{align*}
    which along with \eqref{eq:airy-app-f-dw2} deduce that for any $Y\in \mathcal N_1$,
    \begin{align}\label{eq:airy-app-f-dwe-2}
		|e^{\vartheta Y}(U_s- c)\partial_Yw_{app}(Y)|\leq C|\varepsilon|^{-\frac{2}{3}} \|e^{\vartheta Y}(U_s-c)F\|_{L^\infty}.
	\end{align}
	
	\no\textbf{Case 3. $|\kappa\eta^{La}(Y)|\leq 3M+1$.} By Lemma \ref{lem:A1A2}, we have 
	\begin{align*}
		|(U_s- c)\partial_Y w_{app}(Y)|\leq& C|\varepsilon|^{-\frac{2}{3}}\int_{Z\in\mathcal N_{far}(Y)}|\kappa\eta^{La}(Z)|^{-\f14}|\partial_Z\eta^{La}(Z)|^{-1} \\
		&\times e^{-\gamma_0|\varepsilon|^{-\f13}|\eta^{La}(Y)-\eta^{La}(Z)|(|\kappa\eta(Y)|^\f12+|\kappa\eta^{La}(Z)|^\f12)}|F(Z)|dZ\\
		&+C|\varepsilon|^{-\frac{2}{3}}\int_{Z\in\mathcal N\cup\mathcal N_{near}(Y)}|F(Z)|dZ.\\
		\leq&C|\varepsilon|^{-\frac{2}{3}}|\log c_i| e^{-\vartheta Y} \|e^{\vartheta Y}(U_s-c)F\|_{L^\infty},
	\end{align*}
	which along with \eqref{eq:airy-app-f-dwe-1} and \eqref{eq:airy-app-f-dwe-2} deduces the following estimate 
	\begin{align}\label{eq:airy-app-w_app-2}
		|e^{\vartheta Y}(U_s- c)\partial_Yw_{app}(Y)|\leq C|\varepsilon|^{-\frac{2}{3}} |\log c_i|\|e^{\vartheta Y}(U_s-c)F\|_{L^\infty}.
	\end{align}
   
Using the equation of $w_{app}$ in \eqref{eq:ariy-app-goodf} and along with estimates \eqref{eq:airy-app-w_app-1} and \eqref{eq:airy-app-w_app-2}, it is easy to see that
\begin{align*}
|\e||e^{\vartheta Y}(U_s- c)\pa_Y^2w_{app}|\leq C|\log c_i|\|e^{\vartheta Y}(U_s-c)F\|_{L^\infty}.
\end{align*}

	By Lemma \ref{lem:err1-err2}, \eqref{eq:airy-app-f-c1}, \eqref{eq:airy-app-f-c2} and \eqref{eq:airy-app-f-4}, and using the fact $|U_s-c|\geq c_i\geq c_0|\e|^\f12$ and $c_i\leq C|\varepsilon|^\f13$, we obtain
    \begin{align*}
    	|e^{\vartheta Y}(U_s- c)Err_1(Y)w_{app}(Y)|\leq C|\e|^\f13|\log c_i|\|e^{\vartheta Y}(U_s- c)F\|_{L^\infty}.
    \end{align*}	
Moreover, by Lemma \ref{lem:err1-err2}, we obtain 
	\begin{align*}
		|e^{\vartheta Y}(U_s- c)Err_2(Y)\partial_Yw_{app}(Y)|\leq C|\varepsilon|^{\f13}|\log c_i|\|e^{\vartheta Y}(U_s-c)F\|_{L^\infty}.
	\end{align*}

\subsubsection{Estimates of $\partial_Y\psi_{app}(Y)$ and $\psi_{app}(Y)$}

Here we restrict that $1\lesssim\vartheta\leq 3\eta_0$.  By the definition of $\psi_{app}(Y)$, we know that 
\begin{align*}
A^{-1}\partial_Y\psi_{app}(Y)=-\int_Y^{+\infty}w_{app}(Z)dZ,
\end{align*}
from which, we get by the formula of $w_{app}$ and integration by parts  that 
\begin{align}\nonumber
	\begin{split}
		A^{-1}\partial_Y\psi_{app}(Y)=&A_1(1,Y)\int_0^YA_2(Z)(\partial_Z\eta^{La}(Z))^{-1}F(Z)dZ\\
		&+A_2(1,Y)\int_Y^{+\infty}A_1(Z)(\partial_Z\eta^{La}(Z))^{-1}F(Z)dZ\\
	&+\int_Y^{+\infty}[A_1(1,Z)A_2(Z)-A_2(1,Z)A_1(Z)](\partial_Z\eta^{La}(Z))^{-1}F(Z)dZ.
	\end{split}
\end{align}
Therefore, we get by \eqref{est: A} that
\begin{align}\label{eq:airy-app-gf1-ppsi1}
	\begin{split}
		|(U_s(Y)- c)\partial_Y\psi_{app}(Y)|\leq &\Bigg(\left|(U_s- c) A_1(1,Y)\int_0^YA_2(Z)(\partial_Z\eta^{La}(Z))^{-1}F(Z)dZ\right|\\
		&+\left|(U_s- c) A_2(1,Y)\int_Y^{+\infty}A_1(Z)(\partial_Z\eta^{La}(Z))^{-1}F(Z)dZ\right|\Bigg)\\
		&+\left|(U_s- c) \int_Y^{+\infty}[A_1(1,Z)A_2(Z)-A_2(1,Z)A_1(Z)](\partial_Z\eta^{La}(Z))^{-1}F(Z)dZ\right|\\
		=&H_1+H_2.
	\end{split}
\end{align}

\no\textbf{Case 1. $|Y-Y_c|\geq L$}.  As in the same argument in $I_1-I_3$, we have
\begin{align}\label{eq:airy-app-gf1-ppsi2}
\begin{split}
	H_1\leq& C|\varepsilon|^{-\f13} e^{-\vartheta Y}\|e^{\vartheta Y}(U_s-c)F\|_{L^\infty}\Bigg(\int_{Z\in\mathcal N_{near}(Y)}+\int_{Z\in\mathcal N_{far}(Y)}\Bigg)|\kappa\eta(Y)|^{-\f34}|\kappa\eta^{La}(Z)|^{-\f14}\\
		&\qquad\times|\partial_Y\eta^{La}(Y)|^{-1} ||\partial_Z\eta^{La}(Z)|^{-1}|U_s(Z)-c|^{-1} e^{-C|\varepsilon|^{-\f13}|\eta^{La}(Y)-\eta^{La}(Z)|(|\kappa\eta^{La}(Y)|^\f12+|\kappa\eta^{La}(Z)|^\f12)}dZ\\
		&+ C|\varepsilon|^{-\f13} e^{-\vartheta Y}\|e^{\vartheta Y}(U_s- c)F\|_{L^\infty}\\
		&\qquad\times\int_{Z\in\mathcal N}|\kappa\eta^{La}(Y)|^{-\f34} e^{-C|\varepsilon|^{-\f13}|\eta^{La}(Y)-\eta^{La}(Z)|(|\kappa\eta^{La}(Y)|^\f12+|\kappa\eta^{La}(Z)|^\f12)}|U_s- c|^{-1}dZ\\
		\leq& C|\varepsilon|^{\f13} e^{-\vartheta Y}\|e^{\vartheta Y}(U_s-c)F\|_{L^\infty},
		\end{split}
	\end{align}
where we used  \eqref{est: exp-near}, \eqref{est: exp-far} and
\begin{align*}
|\kappa\eta^{La}(Y)|^{-\f34}|\kappa\eta^{La}(Z)|^{-\f14}|\partial_Y\eta^{La}(Y)|^{-1} ||\partial_Z\eta^{La}(Z)|^{-1}|U_s(Z)-c|^{-1}\leq& C|\e|^\f13,\quad Z\in \mathcal{N}_{near}(Y),
\end{align*}
and 
\begin{align*}
&|\kappa\eta^{La}(Y)|^{-\f34}|\kappa\eta^{La}(Z)|^{-\f14}|\partial_Y\eta^{La}(Y)|^{-1} ||\partial_Z\eta^{La}(Z)|^{-1}|U_s(Z)-c|^{-1}	\\
&\leq C|\varepsilon|^{\f14-\f13}|\eta^{La}(Y)|^{-\f34}|\kappa\eta^{La}(Z)|^{-\f14}|\eta^{La}(Y)|^{\f12}(1+|\eta^{La}(Z)|^\f12)\leq C|\kappa\eta^{La}(Y)|^\f14(1+|\eta^{La}(Z)|^\f12),\quad Z\in\mathcal N_{far}(Y),
\end{align*}
and
	\begin{align*}
		&\int_{Z\in\mathcal N}|\kappa\eta^{La}(Y)|^{-\f34}e^{-C|\varepsilon|^{-\f13}|\eta^{La}(Y)-\eta^{La}(Z)|(|\kappa\eta^{La}(Y)|^\f12+|\kappa\eta^{La}(Z)|^\f12)}|U_s- c|^{-1}dZ\\
		&\quad\leq Ce^{-|\varepsilon|^{-\f13}|\kappa\eta^{La}(Y)|^\f12}\int_{Z\in\mathcal N}|U_s- c|^{-1}dZ\leq C|\varepsilon|^N.
	\end{align*}
For  $|Y-Y_c|\geq L$, we have
\begin{align*}
	H_2\leq& C|\varepsilon|^{-\f13}\int_Y^{+\infty}|\kappa\eta^{La}(Z)|^{-1}|\partial_Z\eta^{La}(Z)|^{-2}|F(Z)|dZ\\
	\leq& Ce^{-\vartheta Y}\|e^{\vartheta Y}F\|_{L^\infty(\{|Y-Y_c|\geq L\})}\int_Y^{+\infty}e^{-\vartheta(Z-Y)}dZ\leq Ce^{-\vartheta Y}\|e^{\vartheta Y}(U_s-c)F\|_{L^\infty},
\end{align*}
which along with \eqref{eq:airy-app-gf1-ppsi1} and \eqref{eq:airy-app-gf1-ppsi2} implies that for $|Y-Y_c|\geq L$,
\begin{align*}
	|e^{\vartheta Y}(U_s(Y)- c)\partial_Y\psi_{app}(Y)|\leq C\|e^{\vartheta Y}(U_s- c)F\|_{L^\infty}.
\end{align*}

\no\textbf{Case 2. $|\kappa\eta^{La}(Y)|\geq 3M+1$ and $|Y-Y_c|\leq L$.} In this case,  using the fact $|U_s-c|\leq C|Y-Y_c|\leq C|\e|^\f13|\kappa \eta^{La}(Y)|$, we get
\begin{align}\label{eq:airy-app-gf1-ppsi3}
\begin{split}
	H_1\leq& Ce^{-\vartheta Y}\|e^{\vartheta Y}(U_s- c)F\|_{L^\infty}\\
		&\qquad\times\Bigg(\int_{Z\in\mathcal N_{near}(Y)}+\int_{Z\in\mathcal N_{far}(Y)}\Bigg)|\kappa\eta^{La}(Y)|^\f14|\kappa\eta^{La}(Z)|^{-\f14}|\partial_Z\eta^{La}(Z)|^{-1}|U_s(Z)- c|^{-1}\\
		&\qquad\times e^{-C|\varepsilon|^{-\f13}|\eta^{La}(Y)-\eta^{La}(Z)|(|\kappa\eta^{La}(Y)|^\f12+|\kappa\eta^{La}(Z)|^\f12)}dZ\\
		&\qquad+Ce^{-\vartheta Y}\|e^{\vartheta Y}(U_s- c)F\|_{L^\infty}\\
		&\qquad\times\int_{Z\in\mathcal N}|\kappa\eta^{La}(Y)|^\f14e^{-C|\varepsilon|^{-\f13}|\eta^{La}(Y)-\eta^{La}(Z)|(|\kappa\eta^{La}(Y)|^\f12+|\kappa\eta^{La}(Z)|^\f12)}|U_s(Z)- c|^{-1}dZ\\
		\leq &C|\log c_i| e^{-\vartheta Y}\|e^{\vartheta Y}(U_s- c)F\|_{L^\infty},
		\end{split}
\end{align}
where we used the facts that
\begin{align*}
|\kappa\eta^{La}(Y)|^\f14|\kappa\eta^{La}(Z)|^{-\f14}|\partial_Z\eta^{La}(Z)|^{-1}|U_s(Z)- c|^{-1}\leq& C\kappa|\kappa \eta^{La}(Z)|^{-1}\leq C|\e|^{-\f13}, \quad Z\in \mathcal{N}_{near}(Y),
\end{align*}
and 
\begin{align*}
	&|\kappa\eta^{La}(Y)|^\f14|\kappa\eta^{La}(Z)|^{-\f14}|\partial_Z\eta^{La}(Z)|^{-1}|U_s(Z)- c|^{-1}\\
	&\qquad\leq C|\varepsilon|^{-\f13}|\kappa\eta^{La}(Y)|^\f14(1+|\eta^{La}(Z)|^\f12),\quad Z\in\mathcal N_{far}(Y),
\end{align*}
and $|\varepsilon|^{-\f13}|\eta^{La}(Y)-\eta^{La}(Z)|\geq 1$ so that 
\begin{align*}
	&\int_{Z\in\mathcal N}|\kappa\eta^{La}(Y)|^\f14e^{-C|\varepsilon|^{-\f13}|\eta^{La}(Y)-\eta^{La}(Z)|(|\kappa\eta^{La}(Y)|^\f12+|\kappa\eta^{La}(Z)|^\f12)}|U_s(Z)- c|^{-1}dZ\\
	&\qquad\leq C|\kappa\eta^{La}(Y)|^\f14e^{-C|\kappa\eta^{La}(Y)|^\f12}\int_{\mathcal N}|U_s(Z)- c|^{-1}dZ\leq C|\log c_i|.
\end{align*}

Notice that for any $Y\in\{Y:|\kappa\eta^{La}(Y)|\geq 3M+1, Y\geq Y_c\}\cup\{Y:|Y-Y_c|\leq L\}$,
\begin{align*}
	H_2\leq& C\int_Y^{+\infty}|\kappa\eta^{La}(Z)|^{-1}|\partial_Z\eta^{La}(Z)|^{-2}|F(Z)|dZ
	\leq Ce^{-\vartheta Y}\|e^{\vartheta Y}(U_s-c)F\|_{L^\infty},
\end{align*}
and for any  $Y\in\{Y:|\kappa\eta^{La}(Y)|\geq 3M+1, Y\leq Y_c\}\cup\{Y:|Y-Y_c|\leq L\}$,
\begin{align*}
	H_2\leq& C\left|(U_s-c) \int_{Z\in\mathcal N^+}[A_1(1,Z)A_2(Z)-A_2(1,Z)A_1(Z)]\partial_Z\eta^{La}(Z)^{-1}F(Z)dZ\right|\\
	&\qquad+C\left|(U_s- c) \int_{Z\in\mathcal N}[A_1(1,Z)A_2(Z)-A_2(1,Z)A_1(Z)]\partial_Z\eta^{La}(Z)^{-1}F(Z)dZ\right|\\
	&\qquad+C\left|(U_s-c) \int_{Y}^{Y_-}[A_1(1,Z)A_2(Z)-A_2(1,Z)A_1(Z)]\partial_Z\eta^{La}(Z)^{-1}F(Z)dZ\right|\\
	\leq& C\|e^{\vartheta Y}(U_s-c)F\|_{L^\infty}\left(1+|c||\varepsilon|^{-\f13}\int_{Z\in\mathcal N}|U_s(Z)- c|^{-1}dZ\right)\\
	&\qquad+C| c||\e|^{-\f13}\int_Y^{Y_-}|\kappa\eta^{La}(Z)|^{-1}|\partial_Z\eta^{La}(Z)|^{-2}|F(Z)|dZ\\
	\leq& C(1+| c||\varepsilon|^{-\f13})|\log c_i| e^{-\vartheta Y}\|e^{\vartheta Y}(U_s-c)F\|_{L^\infty}.
\end{align*}
Thus, we conclude that for any $Y\in\big\{Y:|\kappa\eta^{La}(Y)|\geq 3M+1\}\cup\{Y:|Y-Y_c|\leq L\big\}$,
\begin{align*}
	H_2\leq& C(1+| c||\varepsilon|^{-\f13})|\log c_i| \|e^{\vartheta Y}(U_s-c)F\|_{L^\infty},
\end{align*}
which along with \eqref{eq:airy-app-gf1-ppsi1} and \eqref{eq:airy-app-gf1-ppsi3} shows that for any $Y\in\{Y:|\kappa\eta(Y)|\geq 3M+1\}\cup\{Y:|Y-Y_c|\leq L\}$,
\begin{align*}
		|e^{\vartheta Y}(U_s(Y)- c)\partial_Y\psi_{app}(Y)|\leq C(1+| c||\varepsilon|^{-\f13})|\log c_i|\|e^{\vartheta Y}(U_s- c)F\|_{L^\infty}.
\end{align*}

\no\textbf{Case 3. $Y\in\{Y:|\kappa\eta^{La}(Y)|\leq 3M+1\}$.} A similar argument as in Case 2 shows that 
\begin{align}\label{eq:airy-app-gf1-ppsi4}
\begin{split}
		H_1\leq& Ce^{-\vartheta Y}\|e^{\vartheta Y}(U_s- c)F\|_{L^\infty}\int_{Z\in\{Z:|\kappa\eta^{La}(Z)|\geq 3M+2\}}|\kappa\eta^{La}(Z)|^{-\f14}|\partial_Z\eta^{La}(Z)|^{-1}\\
		&\qquad\times e^{-C|\varepsilon|^{-\f13}|\eta^{La}(Y)-\eta^{La}(Z)||\kappa\eta^{La}(Z)|^\f12}|U_s(Z)- c|^{-1}dZ\\
		&+Ce^{-\vartheta Y}\|e^{\vartheta Y}(U_s- c)F\|_{L^\infty}\int_{Z\in\mathcal N}|U_s- c|^{-1}dZ\\
		\leq&C|\log c_i|e^{-\vartheta Y}\|e^{\vartheta Y}(U_s- c)F\|_{L^\infty}.
		\end{split}
\end{align}
Notice that 
\begin{align*}
	H_2\leq& \left|(U_s- c) \int_{Y_+}^{+\infty}[A_1(1,Z)A_2(Z)-A_2(1,Z)A_1(Z)]\partial_Z\eta(Z)^{-1}F(Z)dZ\right|\\
	&\qquad+\left|(U_s- c) \int_{\mathcal N}[A_1(1,Z)A_2(Z)-A_2(1,Z)A_1(Z)]\partial_Z\eta(Z)^{-1}F(Z)dZ\right|\\
	\leq& Ce^{-\vartheta Y}\|e^{\vartheta Y}(U_s- c)F\|_{L^\infty}+C\|e^{\vartheta Y}(U_s- c)F\|_{L^\infty}\int_{\mathcal N}|U_s(Z)-c|^{-1}dZ\\
	\leq& C|\log  c_i| e^{-\vartheta Y}\|e^{\vartheta Y}(U_s- c)F\|_{L^\infty},
\end{align*}
which along with \eqref{eq:airy-app-gf1-ppsi1} and \eqref{eq:airy-app-gf1-ppsi4} implies that for any $Y\in\{Y:|\kappa\eta(Y)|\leq M+1\}$,
\begin{align*}
		|e^{\vartheta Y}(U_s(Y)- c)\partial_Y\psi_{app}(Y)|\leq C|\log c_i|e^{-\vartheta Y}\|e^{\vartheta Y}(U_s-c)F\|_{L^\infty}.
\end{align*}

Summing up, we conclude that for any $Y\geq 0$,
\begin{align}\label{eq:airy-app-gh1-ppsie}
	|e^{\vartheta Y}(U_s(Y)-c)\partial_Y\psi_{app}(Y)|\leq C(1+| c||\varepsilon|^{-\f13})|\log c_i|\|e^{\vartheta Y}(U_s- c)F\|_{L^\infty}.
\end{align}
By \eqref{eq:airy-app-gh1-ppsie} and $\psi_{app}=-\int_Y^{+\infty}\partial_Y\psi_{app}(Z)dZ$, we can obtain that for any $Y\geq 0$,
\begin{align*}
	|e^{\vartheta Y}\psi_{app}(Y)|\leq C(1+|c||\varepsilon|^{-\f13})|\log c_i|^2\|e^{\vartheta Y}(U_s- c)F\|_{L^\infty}.
\end{align*}

\medskip

In a similar way as in Proposition \ref{prop:airy-gf2},  we can show that 

\begin{proposition}\label{prop:airy-gf1}
		Let $\delta_0>0$ be a constant in Lemma \ref{lem:Airy-p1} and  $ c_i\geq c_0 |\e|^\f12$. Suppose $-\delta_0<\mathrm{Im}(\kappa\eta^{La})<\delta_0$, $(\alpha,c)\in\mathbb H_1$ and $e^{\vartheta Y}F\in L^{\infty}$. Then it holds that
\begin{enumerate}
\item  For $0<\vartheta\leq 3\eta_0$,
\begin{align*}
&|e^{\vartheta Y}(U_s-c)w_{app}(Y)|+|\e|^\f23|e^{\vartheta Y}\pa_Yw_{app}(Y)|+|\e||e^{\vartheta Y}\pa_Y^2w_{app}(Y)|
\leq C\|e^{\vartheta Y}F\|_{L^\infty}.
\end{align*}
Moreover, for $0<c_i\leq C|\varepsilon|^\f13$, we have
\begin{align*}
 |e^{\vartheta Y}Err_1(Y)w_{app}(Y)|+|e^{\vartheta Y}Err_2(Y)\partial_Yw_{app}(Y)|\leq C|\varepsilon|^\f13 \|e^{\vartheta Y}F\|_{L^\infty}.
\end{align*}	
\item For $1\lesssim\vartheta\leq 3\eta_0$,
\begin{align*}
&|e^{\vartheta Y}(U_s-c)\partial_Y\psi_{app}(Y)|+|e^{\vartheta Y}\psi_{app}(Y)|\leq C |\log c_i|\|e^{\vartheta Y}F\|_{L^\infty} .
\end{align*}

\end{enumerate}
\end{proposition}

\subsection{Non-homogeneous Airy equation}
In this subsection, we construct a solution to the non-homogeneous Airy equation:
\begin{align}\label{eq:airy-equation}
	\begin{split}
		&\varepsilon(\partial_Y^2-\alpha^2)w-(U_s-c)w=F,\quad\La\psi(Y)=w(Y),\\
		&\lim_{Y\to\infty}\psi(Y)=\lim_{Y\to\infty} w(Y)=0.
	\end{split}
\end{align}
We introduce the notation
\begin{align*}
Airy[\cdot]=\varepsilon(\partial_Y^2-\alpha^2)-(U_s-c),\quad \La=\partial_Y(A^{-1}\partial_Y)-\alpha^2.
\end{align*}

\begin{proposition}\label{pro: Airy-1}
Let $\delta_0>0$ be a constant in Lemma \ref{lem:Airy-p1} and  $c_0|\e|^\f12\leq c_i\leq C|\e|^\f13$. Suppose $-\delta_0<\mathrm{Im}(\kappa\eta^{La})<\delta_0$ and $(\alpha,c)\in\mathbb H_1$. Then it holds that
\begin{itemize}
\item  Take $\th>0$. For $F$ such that  $e^{\vartheta Y}F\in L^{\infty}$, there exists a solution $(w,\psi)$ to \eqref{eq:airy-equation} satisfying 
\begin{enumerate}
\item  for $0<\vartheta\leq 3\eta_0$,
\begin{align*}
&|\varepsilon|\|e^{\vartheta Y}\partial^2_Y w\|_{L^\infty}+	|\varepsilon|^{\f23}\|e^{\vartheta Y}\partial_Y w\|_{L^\infty}+\|e^{\vartheta Y}(U_s- c)w\|_{L^\infty}\leq C \|e^{\vartheta Y}F\|_{L^\infty}.
\end{align*}
\item for $1\lesssim\vartheta\leq 3\eta_0$,
\begin{align*}
&\|e^{\vartheta Y}\psi\|_{L^\infty}+\|(U_s- c)e^{\vartheta Y}\partial_Y\psi\|_{L^\infty}\leq C|\log c_i|\|e^{\vartheta Y}F\|_{L^\infty}.
\end{align*}
\end{enumerate}

\item Take $\th>0$. For $F$ such that  $e^{\vartheta Y}(U_s- c)F\in L^{\infty}$, there exists a solution $(w,\psi)$ to \eqref{eq:airy-equation} satisfying 
\begin{enumerate}
\item  for $0<\vartheta\leq3\eta_0$:
\begin{align*}
&|\varepsilon|\|e^{\vartheta Y}(U_s-c)\partial_Y^2 w\|_{L^\infty}+|\varepsilon|^\f23\|e^{\vartheta Y}(U_s-c)\partial_Yw\|_{L^\infty}+\|e^{\vartheta Y}(U_s-c)^2w\|_{L^\infty}\\
&\quad\leq C|\log c_i|\|e^{\vartheta Y}(U_s- c)F\|_{L^\infty}.
\end{align*}
\item for $1\lesssim\vartheta\leq 3\eta_0$,
\begin{align*}
&\|e^{\vartheta Y}\psi\|_{L^\infty}+\|e^{\vartheta Y}(U_s-c)\partial_Y\psi\|_{L^\infty}\leq C(1+|c||\varepsilon|^{-\f13})|\log c_i| ^2\|e^{\vartheta Y}(U_s-c)F\|_{L^\infty}.
\end{align*}
\end{enumerate}\end{itemize}
\end{proposition}
\begin{remark}
The  results in Proposition \ref{pro: Airy-1} are also valid when we replace $w$ by $\pa_Y^2 \psi$ according to the relation $\La \psi=\om.$
\end{remark}

\begin{proof}
Since the proof for the above two kinds of source terms  $F$ is similar, we only show the harder case, i.e., $F\sim(U_s- c)^{-1}$.
We shall construct the solution $(w,\psi)$ to \eqref{eq:airy-equation} via the iteration. 

For given source term $F$ with $e^{\vartheta Y}(U_s- c)F\in L^{\infty}$, we define $w^{(0)}(Y)=w_{app}(Y)$, where $w_{app}(Y)\in W^{2,\infty}(\mathbb{R}_+)$ is the solution to 
\begin{align*}
\varepsilon(\partial_Y^2-\alpha^2)w_{app}-(U_s-c)w_{app}=F+Err_1w_{app}+Err_2\partial_Yw_{app}
\end{align*}
constructed in Proposition \ref{prop:airy-gf2}. Then for any $j\in\mathbb{N}_+$, we define that $w^{(j)}(Y)$ is the solution constructed in Proposition \ref{prop:airy-gf2} with the source term $-Err_1w^{(j-1)}(Y)-Err_2\partial_Y w^{(j-1)}(Y)$. 
According to Proposition \ref{prop:airy-gf2}, we obtain 
\begin{align*}
&\|e^{\vartheta Y}(U_s-c)^2w^{(j)}\|_{L^\infty}+|\varepsilon|^\f23\|e^{\vartheta Y}(U_s-c)\partial_Yw^{(j)}\|_{L^\infty}+|\varepsilon|\|e^{\vartheta Y}(U_s-c)\partial_Y^2 w^{(j)}\|_{L^\infty}\\
&\leq C|\log c_i|(\|e^{\vartheta Y}(U_s-c)Err_1w^{(j-1)}\|_{L^\infty}+\|e^{\vartheta Y}(U_s-c)Err_2\partial_Y w^{(j-1)}\|_{L^\infty})\\
&\leq
\left\{
\begin{aligned}
&C|\varepsilon|^\f13|\log c_i| \|e^{\vartheta Y}(U_s-c)F\|_{L^\infty},\quad j=1,\\
&C|\varepsilon|^\f13|\log c_i|(\|e^{\vartheta Y}(U_s-c)Err_1w^{(j-2)}\|_{L^\infty}+\|e^{\vartheta Y}(U_s-c)Err_2\partial_Y w^{(j-2)}\|_{L^\infty}),\quad j\geq2.
\end{aligned}
\right.
\end{align*}
Therefore, for any $j\geq 1$, we take $|\e|$ small enough so that $|\varepsilon|^\f13|\log c_i|\leq \f12$ to obtain
\begin{align*}
&\|e^{\vartheta Y}(U_s-c)^2w^{(j)}\|_{L^\infty}+|\varepsilon|^\f23\|e^{\vartheta Y}(U_s-c)\partial_Yw^{(j)}\|_{L^\infty}+|\varepsilon|\|e^{\vartheta Y}(U_s-c)\partial_Y^2 w^{(j)}\|_{L^\infty}\\
&\quad\leq C\left(|\varepsilon|^\f13|\log c_i|\right)^j \|e^{\vartheta Y}(U_s-c)F\|_{L^\infty}\leq C2^{-j} \|e^{\vartheta Y}(U_s-c)F\|_{L^\infty}.
\end{align*}
Now we define
\begin{align*}
w(Y)=\sum_{j=0}^\infty w^{(j)}(Y).
\end{align*}
From the above results, we infer that 
\begin{align}\label{eq:airy-gf2-we}
	\begin{split}
		&\|e^{\vartheta Y}(U_s-c)^2w\|_{L^\infty}+|\varepsilon|^\f23\|e^{\vartheta Y}(U_s-c)\partial_Yw\|_{L^\infty}+|\varepsilon|\|e^{\vartheta Y}(U_s-c)\partial_Y^2 w\|_{L^\infty}\\
\quad&\leq C|\log c_i|\|e^{\vartheta Y}(U_s-c)F\|_{L^\infty}+ C\|e^{\vartheta Y}(U_s-c)F\|_{L^\infty}\\
\quad&\leq C|\log c_i|\|e^{\vartheta Y}(U_s-c)F\|_{L^\infty}.
	\end{split}
\end{align}

Now we turn to construct $\psi(Y)$. Here we restrict $1\lesssim\vartheta\leq 3\eta_0$. We first recall that 
\begin{align*}
	\psi_{app}(Y)=\int_{Y}^{+\infty}A(Y')\int_{Y'}^{+\infty}w_{app}(Z)dZdY'.
\end{align*}
By Proposition \ref{prop:airy-gf2}, we know that
\begin{align}\label{eq:airy-gf2-se}
	&\|e^{\vartheta Y}\psi_{app}\|_{L^\infty}+\|e^{\vartheta Y}(U_s-c)\partial_Y\psi_{app}\|_{L^\infty}\\
	&\qquad\leq C(1+|c||\varepsilon|^{-\f13})|\log c_i|^2\|e^{\vartheta Y}(U_s-c)F\|_{L^\infty}.\nonumber
\end{align}
Moreover, we  have 
\begin{align*}
	\La\psi_{app}(Y)=w(Y)+(w_{app}(Y)-w(Y))-\alpha^2\psi_{app}(Y).
\end{align*}
Hence, to construct $\psi(Y)$, we need to shrink the error: $(w_{app}(Y)-w(Y))-\alpha^2\psi_{app}(Y)$. For this purpose, we define 
\begin{align*}
	\psi^{(1)}(Y)=\int_Y^{+\infty}A(Y')\int_{Y'}^{+\infty}(w_{app}(Z)-w(Z))-\alpha^2\psi_{app}(Z)dZdY',
\end{align*}
from which, we directly have
\begin{align*}
	\partial_Y\psi^{(1)}(Y)=-A(Y)\int_{Y}^{+\infty}(w_{app}(Z)-w(Z))-\alpha^2\psi_{app}(Z)dZ.
\end{align*}
Therefore, by the proof of \eqref{eq:airy-gf2-we} and \eqref{eq:airy-gf2-se}, we obtain 
\begin{align*}
	\|e^{\vartheta Y}\partial_Y\psi^{(1)}\|_{L^\infty}\leq C(|\varepsilon|^\f13+|\alpha|)|\log c_i|^2\|e^{\vartheta Y}(U_s- c)F\|_{L^\infty},
\end{align*}
which implies  
\begin{align*}
	\|e^{\vartheta Y}\psi^{(1)}\|_{L^\infty}+\|e^{\vartheta Y}\partial_Y\psi^{(1)}\|_{L^\infty}\leq C(|\varepsilon|^\f13+|\alpha|)|\log c_i|^2\|e^{\vartheta Y}(U_s- c)F\|_{L^\infty}.
\end{align*}
Now for any $j\geq 2$, we define 
\begin{align*}
	\psi^{(j)}(Y)=-\alpha^2\int_Y^{+\infty}A(Y')\int_{Y'}^{+\infty}\psi^{(j-1)}(Z)dZdY'.
\end{align*}
Then we obtain that for any $j\geq 2$,
\begin{align*}
	\|e^{\vartheta Y}\psi^{(j)}\|_{L^\infty}+\|e^{\vartheta Y}\partial_Y\psi^{(j)}\|_{L^\infty}&\leq C|\alpha|^2\|e^{\vartheta Y}\psi^{(j-1)}\|_{L^\infty}\\
	&\leq C|\alpha|^{2(j-1)}(|\varepsilon|^\f13+|\alpha|)|\log c_i|^2\|e^{\vartheta Y}(U_s- c)F\|_{L^\infty}.
\end{align*}
Then we define 
\begin{align*}
	\psi(Y)=\psi_{app}(Y)+\sum_{j=1}^\infty\psi^{(j)}.
\end{align*}
From the above arguments, we infer that $\La\psi(Y)=w(Y)$ and 
\begin{align*}
\|e^{\vartheta Y}\psi\|_{L^\infty}+\|e^{\vartheta Y}(U_s- c)\partial_Y\psi\|_{L^\infty}\leq C(1+|c||\varepsilon|^{-\f13})|\log c_i|^2\|e^{\vartheta Y}(U_s- c)F\|_{L^\infty},
\end{align*}
and 
\begin{align*}
&\|e^{\vartheta Y}(\psi-\psi_{app})\|_{L^\infty}+\|e^{\vartheta Y}(U_s- c)\partial_Y(\psi-\psi_{app})\|_{L^\infty}\\
&\quad\leq C(|\varepsilon|^\f13+|\alpha|)|\log c_i|^2\|e^{\vartheta Y}F\|_{L^\infty}.
\end{align*}

The proof is completed.
\end{proof}

\subsection{Homogeneous Airy equation}
In this subsection, we construct a non-trivial solution to the homogeneous Airy equation:
\begin{align}\label{eq:airy-homo}
	\begin{split}
			&\varepsilon(\partial_Y^2-\alpha^2)w_a-(U_s-c)w_a=0,\quad  \La\psi_a(Y)=w_a(Y),\\
		&\lim_{Y\to\infty}\psi_a(Y)=\lim_{Y\to\infty}w_a(Y)=0.
	\end{split}
\end{align}

 We denote 
	\begin{align*}
&\tilde{\mathcal A}(1,Y)=-A(Y)\int_Y^{+\infty}Ai(e^{\mathrm i(\frac{\pi}{6}-\th_0)}\kappa\eta^{La}(Z))dZ,\\
&\tilde{\mathcal A}(2,Y)=-\int_Y^{+\infty}\tilde{\mathcal A}(1,Z)dZ.
	\end{align*}
Then we define
	\begin{align}\label{def: (psi_a^0, w_a^0)}
		w_a^{(0)}(Y)=\frac{Ai(e^{\mathrm i(\frac{\pi}{6}-\th_0)}\kappa\eta^{La}(Y))}{Ai(e^{\mathrm i(\frac{\pi}{6}-\th_0)}\kappa\eta^{La}(0))},\quad \psi_a^{(0)}(Y)=\frac{\tilde{\mathcal A}(2,Y)}{Ai(e^{\mathrm i(\frac{\pi}{6}-\th_0)}\kappa\eta^{La}(0))},
	\end{align}
	so that $\pa_Y(A^{-1}\pa_Y)\psi_a^{(0)}=w_a^{(0)}$. By \eqref{eq:Airy-Err-app}, we know that 
	\begin{align*}
		Airy[w_a^{(0)}(Y)]=Err_1(Y)w_a^{(0)}(Y)+Err_2(Y)\partial_Y w^{(0)}_a(Y).
	\end{align*}
To eliminate the above errors, we define $w_{a,err}(Y)$ to be a solution constructed in Proposition  \ref{pro: Airy-1} with source term $-Err_1(Y)w_a^{(0)}(Y)-Err_2(Y)\partial_Y w^{(0)}_a(Y)$.

	\begin{proposition}\label{pro: (psi_a, w_a)}
		Let $\delta_0>0$ be a constant in Lemma \ref{lem:Airy-p1} and  $c_0|\e|^\f12\leq c_i\leq C|\e|^\f13$. 
		Suppose  that$-\delta_0<\mathrm{Im}(\kappa\eta^{La})<\delta_0$ and $(\alpha,c)\in\mathbb H_1$. Then there exists a solution $\psi_a(Y)\in W^{4,+\infty}$ to \eqref{eq:airy-homo} such that for any $0<\vartheta\leq 3\eta_0$,
	   	\begin{align*}
&\|e^{\vartheta Y}(U_s-c)^2(w_a-w_a^{(0)})\|_{L^\infty}+|\varepsilon|^\f23\|e^{\vartheta Y}(U_s-c)\partial_Y(w_a-w_a^{(0)})\|_{L^\infty}\\
&\quad+|\varepsilon|\|e^{\vartheta Y}(U_s-c)\partial_Y^2 (w_a-w_a^{(0)})\|_{L^\infty}\leq C|\log c_i|(|\e|+|c|^2|\varepsilon|^\f13+|c||\varepsilon|^\f23),
\end{align*}
and   for $1\lesssim\vartheta\leq 3\eta_0$,
\begin{align*}
	&\|e^{\vartheta Y}(\psi_a-\psi_{a}^{(0)})\|_{L^\infty}+\|e^{\vartheta Y}(U_s-c)\partial_Y(\psi_a-\psi_a^{(0)})\|_{L^\infty}\leq C|\log c_i|^2(|\e|+|c|^2|\varepsilon|^\f13+|c||\varepsilon|^\f23).
\end{align*}	
Moreover,  it holds that
\begin{align*}
	\frac{\psi_a(0)}{\partial_Y\psi_a(0)}	=&\frac{\tilde{\mathcal A}(2,0)}{\tilde{\mathcal A}(1,0)}+\mathcal O(|\log c_i|^2 (|\e|^{\f23}+|c|^2+|c||\varepsilon|^\f13)).
\end{align*}
\end{proposition}
\begin{proof}
	By the definition of $\tilde{\mathcal A}(2, Y)$ and the properties of Airy function, we obtain 
	\begin{align*}
		\partial_Y(A^{-1}\pa_Y)\psi_a^{(0)}(Y)=w_a^{(0)}(Y),
	\end{align*}
	and 
	\begin{align*}
		\varepsilon(\partial_Y^2-\alpha^2)w_a^{(0)}-(U_s-c)w_a^{(0)}=Err_1w_a^{(0)}+Err_2\partial_Yw_a^{(0)}.
	\end{align*}
To construct a solution to \eqref{eq:airy-homo}, we need to construct a solution $(w_{a,err},\psi_{a,err})=(w_a-w_a^{(0)},\psi_a-\psi_a^{(0)})$ to the following equation:
	\begin{align}\label{eq:airy-homo-err}
		\begin{split}
			&\varepsilon(\partial_Y^2-\alpha^2)w_{a,err}-(U_s-c)w_{a,err}=-Err_1w_a^{(0)}-Err_2\partial_Yw_a^{(0)},\\
			&\La\psi_{a,err}(Y)=w_{a,err}(Y)+\alpha^2\psi_{a}^{(0)}(Y) ,\\
		&\lim_{Y\to\infty}\psi_{a,err}(Y)=\lim_{Y\to\infty}w_{a,err}(Y)=0.
		\end{split}
	\end{align}

Notice that for $k=0,1$,
\begin{align*}
	&\left|\partial_Y^kw_a^{(0)}(Y)\right|\sim |\varepsilon|^{-\frac{k}{3}}|\partial_Y^kAi(e^{\mathrm i(\frac{\pi}{6}-\th_0)}\kappa\eta^{La}(Y))Ai(e^{\mathrm i(\frac{5\pi}{6}-\th_0)}\kappa\eta^{La}(0))|,\\
	&\left|\partial_Y^k\psi_a^{(0)}(Y)\right|\sim |\varepsilon|^{\frac{2-k}{3}}|\tilde{\mathcal A}(2-k,Y)Ai(e^{\mathrm i\frac{5\pi}{6}}\kappa\eta^{La}(0))|.
\end{align*}
Due to $0\in\mathcal N^-$,  we get by Lemma \ref{lem:A1A2} that for any $k=0,1$,
	\begin{align*}
		\left|\partial_Y^kw_a^{(0)}(Y)\right|\leq& C|\varepsilon|^{-\frac{k}{3}}|\kappa\eta^{La}(Y)|^{-\f14+\f k2}|\kappa\eta^{La}(0)|^{\f14}|\pa_Y\eta^{La}(Y)|^{k}\\
		&\qquad \times e^{-C|\varepsilon|^{-\f13}|\eta^{La}(Y)-\eta^{La}(0)|(|\kappa\eta^{La}(Y)|^\f12+|\kappa\eta^{La}(0)|^\f12)},\\
		\left|\partial_Y^k\psi_a^{(0)}(Y)\right|\leq& C|\varepsilon|^{\frac{2-k}{3}}|\kappa\eta^{La}(Y)|^{-\f54+\f k2}|\kappa\eta^{La}(0)|^{\f14}|\pa_Y\eta^{La}(Y)|^{-(2-k)}\\
		&\qquad\times e^{-C|\varepsilon|^{-\f13}|\eta^{La}(Y)-\eta^{La}(0)|(|\kappa\eta^{La}(Y)|^\f12+|\kappa\eta^{La}(0)|^\f12)}.
	\end{align*}
	Moreover,  by Lemma \ref{lem:pri-Airy-decay} and \ref{lem:airy-0} , we have
	\begin{align*}
		&|\psi_a^{(0)}(0)|\sim|\varepsilon|^\f23|\kappa\eta^{La}(0)|^{-1}\sim|\varepsilon||c|^{-1},\\
		&|\partial_Y\psi_a^{(0)}(0)|\sim |\varepsilon|^\f13|\kappa\eta^{La}(0)|^{-\f12}\sim|\varepsilon|^\f12|c|^{-\f12}.
	\end{align*}
	
 With the above estimates in hand, we can show the estimates of the solution to \eqref{eq:airy-homo-err}. Let us first show the control of $Err_1 w_a^{(0)}(Y)$ and $Err_2\partial_Yw_a^{(0)}(Y)$.  For any $Y-Y_c\geq L$, we first notice that
    \begin{align*}
    	\left|\partial_Y^kw_a^{(0)}(Y)\right|\leq C|\varepsilon|^{-\frac{k}{3}}e^{-C|\varepsilon|^{-\f12}Y}\leq Ce^{-L|\varepsilon|^{-\f12}}e^{-\vartheta Y},\quad k=0,1,
    \end{align*}
    which along with Lemma \ref{lem:err1-err2} and $c_i\leq C|\e|^\f13$  implies that for any $Y-Y_c\geq L$,
    \begin{align}\label{eq:airy-homo-source1}
    	\begin{split}
    		&|(U_s- c)Err_1w_a^{(0)}(Y)|\leq C|\varepsilon|^\f13|w_a^{(0)}(Y)|\leq C|\varepsilon|^\f13 e^{-L|\varepsilon|^{-\f12}}e^{-\vartheta Y},\\
    	&|(U_s- c)Err_2\partial_Y w_a^{(0)}(Y)|\leq C|\varepsilon|e^{-L|\varepsilon|^{-\f12}}e^{-\vartheta Y}.
    	\end{split}
    \end{align}
    For any $Y\in\mathcal N^+\cap\{Y:|Y-Y_c|\leq L\}$, we notice that 
    \begin{align*}
    	\left|\partial_Y^kw_a^{(0)}(Y)\right|\leq C|\varepsilon|^{-\frac{k}{3}}e^{-C|\kappa\eta^{La}(Y)|^\f12},\quad k=0,1.
    \end{align*}
    Then by Lemma \ref{lem:err1-err2}, for any $Y\in\mathcal N^+\cap\{Y:|Y-Y_c|\leq L\}$, we have
    \begin{align}\label{eq:airy-homo-source2}
    	\begin{split}
    		|(U_s- c)Err_1w_a^{(0)}(Y)|\leq& C|\varepsilon|^\f13|\eta^{La}(Y)|^2|w_a^{(0)}(Y)|\leq C|\varepsilon||\kappa\eta^{La}(Y)|^2e^{-C|\kappa\eta^{La}(Y)|^\f12}\\
    	\leq&C|\varepsilon|e^{-\vartheta Y},
    	\end{split}
    \end{align}
    and 
    \begin{align}\label{eq:airy-homo-source3}
    	\begin{split}
    		|(U_s- c)Err_2\partial_Y w_a^{(0)}(Y)|\leq C|\varepsilon||\kappa\eta^{La}(Y)|e^{-C|\kappa\eta^{La}(Y)|^\f12}\leq C|\varepsilon|e^{-\vartheta Y}.
    	\end{split}
    \end{align}
    For any $Y\in\mathcal N^-\cup\mathcal N$, $|U_s- c|\sim|\eta^{La}(Y)|\leq |\eta^{La}(0)|$. For any $Y\in(\mathcal N^-\cup\mathcal N)\cap\{Y:|Y|\geq |\varepsilon|^\f13\}$, we have that for $k=0,1$
    \begin{align*}
    	|\partial_Y^k w_a^{(0)}(Y)|\leq C|\varepsilon|^{-\f k3}e^{-C|\kappa\eta^{La}(0)|^\f12}.
    \end{align*}
    Therefore, we obtain
    \begin{align}\label{eq:airy-homo-source4}
    		|(U_s- c)Err_1w_a^{(0)}(Y)|\leq& C|\varepsilon|^\f13|\eta^{La}(0)|^2|w_a^{(0)}(Y)|\\
		\leq& C|\varepsilon||\kappa\eta^{La}(0)|^2e^{-C|\kappa\eta^{La}(0)|^\f12}\leq C|\varepsilon|,\nonumber
    \end{align}
    and 
    \begin{align}\label{eq:airy-homo-source5}
    	|(U_s- c)Err_2\partial_Y w_a^{(0)}(Y)|\leq C|\varepsilon|^\f23|\eta^{La}(0)|e^{-C|\kappa\eta^{La}(0)|^\f12}\leq C|\varepsilon|.
    \end{align}
   
 Summing up  \eqref{eq:airy-homo-source1}, \eqref{eq:airy-homo-source2}, \eqref{eq:airy-homo-source3}, \eqref{eq:airy-homo-source4} and \eqref{eq:airy-homo-source5}, we conclude that for any $|Y|\geq |\varepsilon|^\f13$,
    \begin{align}\nonumber
    	\begin{split}
    		&|(U_s- c)Err_1w_a^{(0)}(Y)|\leq C|\varepsilon| e^{-\vartheta Y},\\
    	&|(U_s- c)Err_2\partial_Y w_a^{(0)}(Y)|\leq C|\varepsilon|e^{-\vartheta Y}.
    	\end{split}
    \end{align}
    For $Y\in[0,|\varepsilon|^\f13]$, we have 
    \begin{align*}
    	&|(U_s- c)Err_1w_a^{(0)}(Y)|\leq C|\varepsilon|^\f13|c|^2 e^{-\vartheta Y},\\
    	&|(U_s- c)Err_2\partial_Y w_a^{(0)}(Y)|\leq C|\varepsilon|^\f23|c|e^{-\vartheta Y}.
    \end{align*}
    As a result, we obtain 
    \begin{align*}
    	&\|e^{\vartheta Y}(U_s-c)Err_1w_a^{(0)}\|_{L^\infty}\leq C|\varepsilon|^\f13(|\varepsilon|^\f23+|c|^2),\\
    	&\|e^{\vartheta Y}(U_s-c)Err_2\partial_Y w_a^{(0)}\|_{L^\infty}\leq C|\varepsilon|^\f23(|\varepsilon|^\f13+|c|).
    \end{align*}
    Then it follows from Proposition \ref{pro: Airy-1} that 
   	\begin{align*}
&\|e^{\vartheta Y}(U_s-c)^2w_{a,err}\|_{L^\infty}+|\varepsilon|^\f23\|e^{\vartheta Y}(U_s-c)\partial_Yw_{a,err}\|_{L^\infty}+|\varepsilon|\|e^{\vartheta Y}(U_s-c)\partial_Y^2 w_{a,err}\|_{L^\infty}\\
&\quad\leq C|\log c_i|(|\e|+|c|^2|\varepsilon|^\f13+|c||\varepsilon|^\f23),
\end{align*}
and
\begin{align*}
\|e^{\vartheta Y}(U_s-c)\pa_Y\psi_{a,err}\|_{L^\infty}+\|e^{\vartheta Y}\psi_{a,err}\|_{L^\infty}\leq C|\log c_i|^2(|\e|+|c|^2|\varepsilon|^\f13+|c||\varepsilon|^\f23).
\end{align*}
In particular, we have 
\begin{align*}
&|\psi_a(0)-\psi_a^{(0)}(0)|\leq C|\log c_i|^2(|\e|+|c|^2|\varepsilon|^\f13+|c||\varepsilon|^\f23)\ll \psi_a^{(0)}(0),\\
&|\pa_Y\psi_a(0)-\pa_Y\psi_a^{(0)}(0)|\leq C|c|^{-1}|\log c_i|^2(|\e|+|c|^2|\varepsilon|^\f13+|c||\varepsilon|^\f23)\ll \pa_Y\psi_a^{(0)}(0).
\end{align*}
Thus, we arrive at
\begin{align*}
\frac{\psi_a(0)}{\partial_Y\psi_a(0)}	=&\frac{\tilde{\mathcal A}(2,0)}{\tilde{\mathcal A}(1,0)}+\mathcal O(|\log c_i|^2 (|\e|^{\f23}+|c|^2+|c||\varepsilon|^\f13)).
\end{align*}

This completes the proof of the proposition.
\end{proof}

\section{The Orr-Sommerfeld equation} 

To solve the quasi-incompressible system \eqref{eq:LCNS-OS}, we solve the following Orr-Sommerfeld(OS) type equation with general source term $F$:
\begin{align}\label{eq: OS-F}
\varepsilon\Lambda(\partial_Y^2-\alpha^2)\phi-(U_s-c)\Lambda\phi+\partial_Y(A^{-1}\partial_Y U_s)\phi=F.
\end{align}
We introduce the operators
\begin{align*}
OS[\cdot]=&\varepsilon\Lambda(\partial_Y^2-\alpha^2)-(U_s-c)\Lambda+\partial_Y(A^{-1}\partial_Y U_s),\\
Ray[\cdot]=&(U_s-c)\Lambda-\partial_Y(A^{-1}\partial_Y U_s),\\
Airy[\cdot]=&\e(\pa_Y^2-\al^2)-(U_s-c).
\end{align*}

We develop an Airy-Airy-Rayleigh iteration to constructe the solution to the Orr-Sommerfeld equation \eqref{eq: OS-F}. More precisely,  we start our iteration with 
\begin{align*}
\phi^{(1)}=\psi^{(1,0)}+\psi^{(1,1)}+\varphi^{(1)},
\end{align*}
where $\psi^{1,0},\psi^{(1,1)}$ and $\varphi^{(1)}$ solve the following equations respectively,
\begin{align*}
	&Airy[\Lambda\psi^{(1,0)}]=F,\\
	&Airy[\Lambda\psi^{(1,1)}]=-\varepsilon[\Lambda,\partial_Y^2]\psi^{(1,0)}-\partial_Y(A^{-1}\partial_YU_s)\psi^{(1,0)},\\
	&Ray[\varphi^{(1)}]=\varepsilon[\Lambda,\partial_Y^2]\psi^{(1,1)}+\partial_Y(A^{-1}\partial_YU_s)\psi^{(1,1)}.
\end{align*}
Then we obtain 
\begin{align*}
	OS[\phi^{(1)}]=OS[\psi^{(1,0)}]+OS[\psi^{(1,1)}]+OS[\varphi^{(1)}]=\varepsilon\Lambda(\partial_Y^2-\alpha^2)\varphi^{(1)}+F.
\end{align*}
Now we define the iteration scheme inductively. For any $j\geq 1$, 
$$
\phi^{(j+1)}=\psi^{(j+1,0)}+\psi^{(j+1,1)}+\varphi^{(j+1)},
$$
 where $\psi^{(j+1,0)},\psi^{(j+1,1)}$ and $\varphi^{(j+1)}$ solve
\begin{align}\nonumber
	\begin{split}
	&Airy[\Lambda\psi^{(j+1,0)}]=-\varepsilon\Lambda(\partial_Y^2-\alpha^2)\varphi^{(j)},\\
	& Airy[\Lambda\psi^{(j+1,1)}]=-\varepsilon[\Lambda,\partial_Y^2]\psi^{(j+1,0)}-\partial_Y(A^{-1}\partial_YU_s)\psi^{(j+1,0)}\\
	&Ray[\varphi^{(j+1)}]=\varepsilon[\Lambda,\partial_Y^2]\psi^{(j+1,1)}+\partial_Y(A^{-1}\partial_YU_s)\psi^{(j+1,1)}.
	\end{split}
\end{align}
Then we have 
\begin{align*}
	OS\Big[\sum_{j=1}^N\phi^{(j)}\Big]=\varepsilon\Lambda(\partial_Y^2-\alpha^2)\varphi^{(N)}+F.
\end{align*}
Formally, $\phi=\sum_{j=1}^{\infty}\phi^{(j)}$ satisfies \eqref{eq: OS-F}.
\begin{remark}
Compared with the Airy-Rayleigh iteration for the construction of the solution to the classical Orr-Sommerfeld equation in the incompressible case, the Airy-Airy-Rayleigh iteration is the key to avoid the loss of derivatives caused by the commutator $[\Lambda,\partial_Y^2]$. 
\end{remark}

We introduce the following functional spaces
\begin{align}
\|f\|_{\mathcal{X}}=&|\e|\|(U_s-c)\pa_Y^4 f\|_{L^\infty_{\eta}}+|\e|^\f23\|(U_s-c)\pa_Y^3 f\|_{L^\infty_{\eta}}+\|(U_s-c)^2\pa_Y^2 f\|_{L^\infty_{\eta}}\label{def: norm-X_eta}\\
\nonumber
&+\|(U_s-c)\pa_Y f\|_{L^\infty_{\eta}}+\| f\|_{L^\infty_{\eta}},\\
\|f\|_{\mathcal{Y}}=&|\e|^\f23\|(U_s-c)\pa_Y^4 f\|_{L^\infty_{\eta}}+|\e|^\f13\|(U_s-c)\pa_Y^3 f\|_{L^\infty_{\eta}}+\|(U_s-c)\pa_Y^2 f\|_{L^\infty_{\eta}}\label{def: norm-Y_eta}\\
\nonumber
&+\|\pa_Y f\|_{L^\infty_{\eta}}+\| f\|_{L^\infty_{\eta}},
\end{align}
where
\begin{align}
\|f\|_{L^\infty_\eta}=\|e^{\eta Y} f\|_{L^\infty},\nonumber
\end{align}
for a fixed constant $\eta\in (0, \eta_0)$. We also denote 
\begin{align*}
\|f\|_{W^{s,\infty}_\eta}=\sum_{k=0}^s\|\pa_Y^k f\|_{L^\infty_\eta}.
\end{align*}
It is easy to see that
\begin{align}\label{relation: norm}
\|f\|_{\mathcal{X}} \leq C\|f\|_{\mathcal{Y}}.
\end{align}

In this section, we focus on the following regime: $(\alpha,c)\in \mathbb H$ with 
\begin{align}\label{def:Hset}
	\mathbb H=\mathbb H_1\cap\big\{(\alpha,c)\in\mathbb C^2:|\alpha|\sim|c|\sim|\varepsilon|^\f13, c_0|\varepsilon|^\f13\leq c_i\leq c_1|\varepsilon|^\f13 \text{ with }0<c_0<c_1\ll1\big\}.
\end{align}
For any $(\alpha,c)\in\mathbb H$,  we have
\begin{align}\label{est: e^f13-(U_s-c)}
|\e|^\f13\leq \f{|\e|^\f13}{c_i}|U_s-c|\leq C|U_s-c|,\quad Y\geq 0,
\end{align} 

\subsection{Non-homogeneous OS equation}

Let us first show some key estimates related to the Airy-Airy-Rayleigh iteration. For convenience, we define
\begin{align}
&Airy[\La \psi]=F_{\psi},\label{eq-Airy: psi}\\
&Airy[\La \widetilde{\psi}]=-\e[\La, \pa_Y^2]\psi-\pa_Y(A^{-1}\pa_Y U_s)\psi=F_{\widetilde{\psi}},\label{eq-Airy: tilde psi}\\
&Ray[\varphi]=\e[\La, \pa_Y^2]\widetilde{\psi}+\pa_Y(A^{-1}\pa_Y U_s)\widetilde{\psi}=F_{\varphi}.\label{eq-Rayleigh: varphi}
\end{align}
Then it holds that
\begin{align*}
OS[\psi+\widetilde{\psi}+\varphi]=F_{\psi}+\e\La (\pa_Y^2-\al^2)\varphi.
\end{align*}

\begin{lemma}\label{lem-iteration: (psi, tilde psi, varphi)}
For any $(U_s-c)F_{\psi}\in L^{\infty}_\eta$, there exists a solution $(\psi, \widetilde{\psi}, \varphi)\in\mathcal X^2\times\mathcal Y$ to \eqref{eq-Airy: psi}-\eqref{eq-Rayleigh: varphi}  satisfying 
\begin{align*}
\|\psi\|_{\mathcal{X}}\leq&C|\log c_i|^2\|(U_s-c)F_{\psi}\|_{L^\infty_\eta},\\
\|\widetilde{\psi}\|_{\mathcal{X}}\leq& C|\log c_i|^4\|(U_s-c)F_{\psi}\|_{L^\infty_\eta}\\
\|\varphi\|_{\mathcal{Y}}\leq& C|\log c_i|^4\|(U_s-c)F_{\psi}\|_{L^\infty_\eta}.
\end{align*}

\end{lemma}

\begin{proof}
According to Proposition \ref{pro: Airy-1}, we have
\begin{align}
\|\psi\|_{\mathcal{X}}\leq&C|\log c_i|^2\|(U_s-c)F_{\psi}\|_{L^\infty_\eta},\nonumber\\
\|\widetilde{\psi}\|_{\mathcal{X}}\leq&C|\log c_i|^2\|(U_s-c)F_{\widetilde{\psi}}\|_{L^\infty_\eta}.\label{est: tilde psi_X-1}
\end{align}
By the definition of $F_{\widetilde{\psi}}$ in \eqref{eq-Airy: tilde psi} and 
\begin{align}\label{eq: commutator-1}
[\La ,\pa_Y^2]f=&-2\pa_Y(A^{-1})\pa_Y^3 f-3\pa_Y^2(A^{-1})\pa_Y^2f-\pa_Y^3(A^{-1})\pa_Y f,
\end{align}
we obtain 
\begin{align}\label{est: (U_s-c)F_tilde psi}
\|(U_s-c)F_{\widetilde{\psi}}\|_{L^\infty_\eta}\leq&\|\e(U_s-c)[\La, \pa_Y^2]\psi\|_{L^\infty_\eta}+\|(U_s-c)\pa_Y(A^{-1}\pa_Y U_s)\psi\|_{L^\infty_\eta}\\
\nonumber
\leq&C|\e|\|(U_s-c)\pa_Y(A^{-1})\pa_Y^3\psi\|_{L^\infty_\eta}+C|\e|\|(U_s-c)\pa_Y^2(A^{-1})\pa_Y^2\psi\|_{L^\infty_\eta}\\
\nonumber
&+C|\e|\|(U_s-c)\pa_Y^3(A^{-1})\pa_Y\psi\|_{L^\infty_\eta}+C\|(U_s-c)\psi\|_{L^\infty_\eta}\\
\nonumber
\leq&C|\e|\|(U_s-c)^2\pa_Y^3\psi\|_{L^\infty_\eta}+C|\e|\|(U_s-c)\pa_Y^2\psi\|_{L^\infty_\eta}\\
\nonumber
&+C|\e|\|(U_s-c)\pa_Y\psi\|_{L^\infty_\eta}+C\|\psi\|_{L^\infty_\eta}\\
\nonumber
\leq&C\|\psi\|_{\mathcal{X}},
\end{align}
from which and \eqref{est: tilde psi_X-1}, we infer that 
\begin{align}\label{est: tilde psi-X}
\|\widetilde{\psi}\|_{\mathcal{X}}\leq&C|\log c_i|^2\|\psi\|_{\mathcal{X}}\leq C|\log c_i|^4\|(U_s-c)F_{\psi}\|_{L^\infty_\eta}.
\end{align}

Next we show the existence and estimates of $\varphi$. By Proposition \ref{pro: phi_non} and a similar argument  as in \eqref{est: (U_s-c)F_tilde psi}, we have
\begin{align*}
&\|\varphi\|_{L^\infty_\eta}+\|\pa_Y\varphi\|_{L^\infty_\eta}+\|(U_s-c)\pa_Y^2\varphi\|_{L^\infty_\eta}\leq C|\log c_i|\|F_{\varphi}\|_{L^\infty_\eta}\\
&\leq C|\log c_i||\e|\Big(\|(U_s-c)\pa_Y^3\widetilde{\psi}\|_{L^\infty_\eta}+\|\pa_Y^2\widetilde{\psi}\|_{L^\infty_\eta}+\|\pa_Y\widetilde{\psi}\|_{L^\infty_\eta}\Big)+C|\log c_i|\|\widetilde{\psi}\|_{L^\infty_\eta}\\
&\leq C|\e||\log c_i|\Big(|\e|^{-\f23}\|\widetilde{\psi}\|_{\mathcal{X}}+c_i^{-2}\|\widetilde{\psi}\|_{\mathcal{X}}+c_i^{-1}\|\widetilde{\psi}\|_{\mathcal{X}}\Big)+C|\log c_i|\|\widetilde{\psi}\|_{\mathcal{X}}\\
&\leq C|\log c_i|\|\widetilde{\psi}\|_{\mathcal{X}}
\leq C|\log c_i|^3\|\psi\|_{\mathcal{X}},
\end{align*}
where we used \eqref{est: e^f13-(U_s-c)} and  \eqref{est: tilde psi-X}. 

To estimate $\pa_Y^3\varphi$, we  first take $\pa_Y$ to the commutator $[\La, \pa_Y^2]f$ and get by \eqref{eq: commutator-1} that
\begin{align*}
\pa_Y[\La, \pa_Y^2]f=&-2\pa_Y(A^{-1})\pa_Y^4 f-5\pa_Y^2(A^{-1})\pa_Y^3f-4\pa_Y^3(A^{-1})\pa_Y^2 f-\pa_Y^4(A^{-1})\pa_Yf,
\end{align*}
which gives
\begin{align*}
&|\e|^\f13\|\pa_Y F_{\varphi}\|_{L^\infty_\eta}\\
&\leq |\e|^\f13\|\e\pa_Y[\La, \pa_Y^2]\widetilde{\psi}\|_{L^\infty_\eta}+|\e|^\f13\|\pa_Y\big(\pa_Y(A^{-1}\pa_Y U_s)\widetilde{\psi}\big)\|_{L^\infty_\eta}\\
&\leq C|\e|^\f13\|\e(U_s-c)\pa_Y^4\widetilde{\psi}\|_{L^\infty_\eta}+C|\e|^\f43\|(\pa_Y^3, \pa_Y^2, \pa_Y)\widetilde{\psi}\|_{L^\infty_\eta}+C|\e|^\f13\|(\pa_Y\widetilde{\psi},\widetilde{\psi})\|_{L^\infty_\eta}\\
&\leq C|\e|^\f13\|\e(U_s-c)\pa_Y^4\widetilde{\psi}\|_{L^\infty_\eta}+C|\e|^\f43(|\e|^{-\f23}|c_i|^{-1}+|c_i|^{-2}+|c_i|^{-1})\|\widetilde{\psi}\|_{\mathcal{X}}+C\|\widetilde{\psi}\|_{\mathcal{X}}\\
&\leq C|\e|^\f13\|\e(U_s-c)\pa_Y^4\widetilde{\psi}\|_{L^\infty_\eta}+C\|\widetilde{\psi}\|_{\mathcal{X}}\leq C|\e|^\f13\|\widetilde{\psi}\|_{\mathcal{X}}+C\|\widetilde{\psi}\|_{\mathcal{X}}\leq C\|\widetilde{\psi}\|_{\mathcal{X}}\leq C|\log c_i|^2\|\psi\|_{\mathcal{X}}.
\end{align*}
Then by Proposition \ref{pro: phi_non}, we obtain 
\begin{align}\nonumber
|\e|^\f13\|(U_s-c)\pa_Y^3\varphi\|_{L^\infty_\eta}\leq& C|\e|^\f13\|\pa_Y F_{\varphi}\|_{L^\infty_\eta}+C|\log c_i|\|F_{\varphi}\|_{L^\infty_\eta}
\leq C|\log c_i|^3\|\psi\|_{\mathcal{X}}.
\end{align}

We turn to  the estimate of $\pa_Y^4\varphi$. Taking $\pa_Y^2$ to the commutator $[\La, \pa_Y^2]f$, we get by \eqref{eq: commutator-1} that
\begin{align*}
\pa_Y^2[\La, \pa_Y^2]f=&-2\pa_Y(A^{-1})\pa_Y^5 f-7\pa_Y^2(A^{-1})\pa_Y^4f-9\pa_Y^3(A^{-1})\pa_Y^3 f\\
&-5\pa_Y^4(A^{-1})\pa_Y^2 f-\pa_Y^5(A^{-1})\pa_Yf,
\end{align*}
which along with \eqref{est: e^f13-(U_s-c)} implies that 
\begin{align*}
&|\e|^\f23\|\pa_Y^2 F_{\varphi}\|_{L^\infty_\eta}\\
&\leq C|\e|^\f23\|\e\pa_Y^2[\La, \pa_Y^2]\widetilde{\psi}\|_{L^\infty_\eta}+C|\e|^\f23\|\pa_Y^2\big(\pa_Y(A^{-1}\pa_Y U_s)\widetilde{\psi}\big)\|_{L^\infty_\eta}\\
&\leq C|\e|^\f23\|\e(U_s-c)\pa_Y^5 \widetilde{\psi}\|_{L^\infty_\eta}+C|\e|^\f23\|\e(\pa_Y^4,\pa_Y^3,\pa_Y^2,\pa_Y)\widetilde{\psi}\|_{L^\infty_\eta}+C\|\widetilde{\psi}\|_{\mathcal{X}}\\
&\leq C|\e|^\f23\|\e(U_s-c)\pa_Y^5 \widetilde{\psi}\|_{L^\infty_\eta}+C|\e|^\f23\|\e\pa_Y^4\widetilde{\psi}\|_{L^\infty_\eta}\\
&\quad+C|\e|^\f53(|\e|^{-\f23}c_i^{-1}+c_i^{-2}+c_i^{-1})\|\widetilde{\psi}\|_{\mathcal{X}}+C\|\widetilde{\psi}\|_{\mathcal{X}}\\
&\leq C|\e|^\f23\|\e(U_s-c)\pa_Y^5 \widetilde{\psi}\|_{L^\infty_\eta}+C|\e|^\f13\|\e(U_s-c)\pa_Y^4\widetilde{\psi}\|_{L^\infty_\eta}+C\|\widetilde{\psi}\|_{\mathcal{X}}\\
&\leq C|\e|^\f23\|\e(U_s-c)\pa_Y^5 \widetilde{\psi}\|_{L^\infty_\eta}+C|\log c_i|^2\|\psi\|_{\mathcal{X}}.
\end{align*}
In order to estimate $\pa_Y^5\widetilde{\psi}$, we show the control of   $|\e|^\f23(U_s-c)\pa_YF_{\widetilde{\psi}}$  firstly. A direct calculation gives 
\begin{align*}
\pa_YF_{\widetilde{\psi}}=&-\pa_Y\big(\e[\La, \pa_Y^2] \psi+\pa_Y(A^{-1}\pa_Y U_s)\psi\big)\\
=&2\e\pa_Y(A^{-1})\pa_Y^4 \psi+5\e\pa_Y^2(A^{-1})\pa_Y^3\psi+4\e\pa_Y^3(A^{-1})\pa_Y^2 \psi
+\e\pa_Y^4(A^{-1})\pa_Y\psi\\
&-\pa_Y(A^{-1}\pa_Y U_s)\pa_Y\psi-\pa_Y^2(A^{-1}\pa_Y U_s)\psi,
\end{align*}
which implies 
\begin{align*}
&|\e|^\f23\|(U_s-c)\pa_YF_{\widetilde{\psi}}\|_{L^\infty_\eta}\\
&\leq C|\e|^\f23\|\e(U_s-c)^2\pa_Y^4\psi\|_{L^\infty_\eta}+C|\e|^\f53\|(U_s-c)(\pa_Y^3, \pa_Y^2, \pa_Y)\psi\|_{L^\infty_\eta}+C|\e|^\f23\|\psi\|_{\mathcal{X}}\\
&\leq C|\e|^\f23 \|\psi\|_{\mathcal{X}}+C|\e|^\f53(|\e|^{-\f23}+|c_i|^{-1}+1)\|\psi\|_{\mathcal{X}}+C|\e|^\f23\|\psi\|_{\mathcal{X}}\\
&\leq C|\e|^\f23\|\psi\|_{\mathcal{X}}\leq C|\e|^\f23|\log c_i|^2\|(U_s-c)F_{\psi}\|_{L^\infty_\eta}.
\end{align*}
Therefore, using the equation of $\widetilde{\psi}$ in \eqref{eq-Airy: tilde psi},  the above estimate along with \eqref{est: tilde psi-X} gives
\begin{align*}
&|\e|^\f23\|\e(U_s-c)\pa_Y^5 \widetilde{\psi}\|_{L^\infty_\eta}\\
&\leq C|\e|^\f23\|(U_s-c)\pa_YF_{\widetilde{\psi}}\|_{L^\infty_\eta}+C|\e|^\f23\|\e(U_s-c)^2\pa_Y^4\widetilde{\psi}\|_{L^\infty_\eta}\\
&\quad+C|\e|^\f23\|\e(U_s-c)(\pa_Y^3,\pa_Y^2,\pa_Y)\widetilde{\psi}\|_{L^\infty_\eta}+C|\e|^\f23\|(U_s-c)^2\pa_Y\La\widetilde{\psi}\|_{L^\infty_\eta}\\
&\quad+C|\e|^\f23\|(U_s-c)\La\widetilde{\psi}\|_{L^\infty_\eta}\\
&\leq C|\e|^\f23\|(U_s-c)\pa_YF_{\widetilde{\psi}}\|_{L^\infty_\eta}+C\|\widetilde{\psi}\|_{\mathcal{X}}\\
&\leq C|\e|^\f23|\log c_i|^2\|(U_s-c)F_{\psi}\|_{L^\infty_\eta}+C|\log c_i|^2\|\psi\|_{\mathcal{X}}.
\end{align*}
Finally, we apply Proposition \ref{pro: phi_non} to obtain
\begin{align*}
|\e|^\f23\|(U_s-c)\pa_Y^4\varphi\|_{L^\infty_\eta}\leq&C|\e|^\f23\|\pa_Y^2 F_{\varphi}\|_{L^\infty_\eta}+C|\e|^\f13\|\pa_Y F_{\varphi}\|_{L^\infty_\eta}+C|\log c_i|\|F_{\varphi}\|_{L^\infty_\eta}\\
\leq&C|\e|^\f23\|\e(U_s-c)\pa_Y^5 \widetilde{\psi}\|_{L^\infty_\eta}+C\|\widetilde{\psi}\|_{\mathcal{X}}+C|\log c_i|^2\|\psi\|_{\mathcal{X}}\\
\leq&C|\e|^\f23|\log c_i|^2\|(U_s-c)F_{\psi}\|_{L^\infty_\eta}+C|\log c_i|^2\|\psi\|_{\mathcal{X}}.
\end{align*}
This shows that
\begin{align}\nonumber
\|\varphi\|_{\mathcal{Y}}\leq&C|\e|^\f23 |\log c_i|^2\|(U_s-c)F_{\psi}\|_{L^\infty_\eta}+C|\log c_i|^2\|\psi\|_{\mathcal{X}}\\
\nonumber
\leq&C|\log c_i|^4\|(U_s-c)F_{\psi}\|_{L^\infty_\eta}.
\end{align}

\end{proof}

Now we are in a position to construct the solution to the Orr-Sommerfeld equation \eqref{eq: OS-F}.

 \begin{proposition}\label{lem: OS-nonhomo}
For any $F$ with  $(U_s-c)F\in L^\infty_\eta$, there exists a solution  $\phi\in\mathcal X$  to \eqref{eq: OS-F} satisfying 
 \begin{align*}
 \|\phi\|_{\mathcal{X}}\leq C|\log c_i|^4 \|(U_s-c)F\|_{L^\infty_\eta}.
 \end{align*}

 \end{proposition}
 
\begin{proof}
We use Airy-Airy-rayleigh iteration introduced at the beginning of this section.  For $j=1$, let $(\psi^{(1,0)}, \psi^{(1,1)}, \varphi^{(1)})$ be the solution of the system \eqref{eq-Airy: psi}-\eqref{eq-Rayleigh: varphi} with 
\begin{align}
F_{\psi}=F_{\psi^{(1,0)}}=F.\nonumber
\end{align}
For $j\ge 2$, let $(\psi^{(j,0)}, \psi^{(j,1)}, \varphi^{(j)})$ be the solution of the system \eqref{eq-Airy: psi}-\eqref{eq-Rayleigh: varphi} with 
\begin{align}
F_{\psi}=F_{\psi^{(j,0)}}=-\e\La (\pa_Y^2-\al^2)\varphi^{(j-1)}.\nonumber
\end{align}
Then we define
\begin{align}
\phi^{(j)}=\psi^{(j,0)}+\psi^{(j,1)}+\varphi^{(j)},\quad j\geq 1.\nonumber
\end{align}

It follows from \eqref{est: e^f13-(U_s-c)} that for $j\geq 2$,
\begin{align}
\|(U_s-c)F_{\psi^{(j,0)}}\|_{L^\infty_\eta}\leq&C|\e|\|(U_s-c)\La (\pa_Y^2-\al^2)\varphi^{(j-1)}\|_{L^\infty_\eta}\nonumber\\
\nonumber
\leq&C|\e|\|(U_s-c)(\pa_Y^4,\pa_Y^3)\varphi^{(j-1)} \|_{L^\infty_\eta}\\
\nonumber
&+C|\e|\al^2\|(U_s-c)(\pa_Y^2,\pa_Y,\al^2)\varphi^{(j-1)} \|_{L^\infty_\eta}\\
\nonumber
\leq& C|\e|^\f13\|\varphi^{(j-1)}\|_{\mathcal{Y}},
\end{align}
from which and Lemma \ref{lem-iteration: (psi, tilde psi, varphi)}, we infer that
\begin{align*}
&\|\psi^{(j,0)}\|_{\mathcal{X}}\leq C|\e|^\f13|\log c_i|^2 \|\varphi^{(j-1)}\|_{\mathcal{Y}},\\
&\|\psi^{(j,1)}\|_{\mathcal{X}}\leq C|\e|^\f13|\log c_i|^4 \|\varphi^{(j-1)}\|_{\mathcal{Y}},\\
&\|\varphi^{(j)}\|_{\mathcal{Y}}\leq C|\e|^\f13|\log c_i|^4 \|\varphi^{(j-1)}\|_{\mathcal{Y}}.
\end{align*}
These estimates  along with \eqref{relation: norm}  show that  for $j\geq 2$,
 \begin{align*}
\|\phi^{(j)}\|_{\mathcal{X}}\leq &\|\psi^{(j,0)}\|_{\mathcal{X}}+\|\psi^{(j,1)}\|_{\mathcal{X}}+\|\varphi^{(j)}\|_{\mathcal{X}}\\
\leq&C|\e|^\f13|\log c_i|^4\|\varphi^{(j-1)}\|_{\mathcal{Y}}\\
\leq&C(|\e|^\f13|\log c_i|^4)^{j-1}\|\varphi^{(1)}\|_{\mathcal{Y}}.
\end{align*}

 By the smallness assumptions on $\al,~|c|,~|\e|$, which guarantee  $|\e|^\f13|\log c_i|^4\leq \f12$, we obtain
\begin{align}\nonumber
\sum_{j=2}^\infty\|\phi^{(j)}\|_{\mathcal{X}}
\leq&\sum_{j=2}^\infty(|\e|^\f13|\log c_i|^4)^{j-1}\|\varphi^{(1)}\|_{\mathcal{Y}}
\leq C|\e|^\f13|\log c_i|^4\|\varphi^{(1)}\|_{\mathcal{Y}}.
\end{align}
By Lemma \ref{lem-iteration: (psi, tilde psi, varphi)} again, we have
\begin{align*}
&\|\psi^{(1,0)}\|_{\mathcal{X}}\leq C|\log c_i|^2 \|(U_s-c)F\|_{L^\infty_\eta},\\
&\|\psi^{(1,1)}\|_{\mathcal{X}}\leq C|\log c_i|^4\|(U_s-c)F\|_{L^\infty_\eta},\\
&\|\varphi^{(1)}\|_{\mathcal{Y}}\leq C|\log c_i|^4 \|(U_s-c)F\|_{L^\infty_\eta},
\end{align*}
which yield that
\begin{align*}
\|\phi^{(1)}\|_{\mathcal{X}}\leq&\|\psi^{(1,0)}\|_{\mathcal{X}}+\|\psi^{(1,1)}\|_{\mathcal{X}}+C\|\varphi^{(1)}\|_{\mathcal{Y}}\leq C|\log c_i|^4 \|(U_s-c)F\|_{L^\infty_\eta}.
\end{align*}
Then $\phi=\sum_{j=1}^{+\infty}\phi^{(j)}$ is the solution to \eqref{eq: OS-F} with 
\begin{align*}
\|\phi\|_{\mathcal{X}}\leq& \|\phi^{(1)}\|_{\mathcal{X}}+\sum_{j=2}^\infty\|\phi^{(j)}\|_{\mathcal{X}}\\
\leq&C|\log c_i|^4 \|(U_s-c)F\|_{L^\infty_\eta}+C|\e|^\f13|\log c_i|^4\|\varphi^{(1)}\|_{\mathcal{Y}}\\
\leq& C|\log c_i|^4 \|(U_s-c)F\|_{L^\infty_\eta}.
\end{align*}
\end{proof}

\subsection{Homogenous OS equation}

We first construct the slow mode $\phi_s$ to the homogeneous OS equation $OS[\phi]=0$ around the solution $\varphi_{Ray}$ constructed in Proposition \ref{pro: varphi_Ray(0)} to the homogeneous Rayleigh equation. From $\phi_s$ we can construct a slow solution to the homogeneous quasi-incompressible system \eqref{eq:LCNS-OS}.
According to Proposition \ref{pro: varphi_Ray(0)}, we know that 
\begin{align*}
\varphi_{Ray}=\Big(\frac{1-m^2(1-c)^2}{(1-c)^2}\Big)(U_s-c)e^{-\beta Y}+\varphi_{Ray,R}
\end{align*}
with 
\begin{align}\label{est: varphi_R-Y}
\|\varphi_{Ray,R}\|_{\mathcal{Y}}\leq C|\al| |\log c_i|.
\end{align}
Putting $\varphi_{Ray}$ into the operator $OS[\cdot]$, we obtain 
\begin{align*}
OS[\varphi_{Ray}]=&\e\La (\pa_Y^2-\al^2)\varphi_{Ray}=\e\La (\pa_Y^2-\al^2)\varphi_{Ray,0}+\e\La (\pa_Y^2-\al^2)\varphi_{Ray,R},
\end{align*}
where
\begin{align*}
	\varphi_{Ray,0}=\Big(\frac{1-m^2(1-c)^2}{(1-c)^2}\Big)(U_s-c)e^{-\beta Y}.
\end{align*}

\begin{lemma}\label{lem: OS-nonhomo-s}
There exists a solution $\phi_s\in W^{4,\infty}$ to $OS[\phi]=0$ satisfying  
 \begin{align*}
 \phi_s=&\varphi_{Ray}+\widetilde{\phi}_s,
  \end{align*}
where
 \begin{align*}
\|\widetilde{\phi}_s\|_{\mathcal{X}}\leq& C|\e|^\f13(|\e|^\f23+|\al|)|\log c_i|^5.
 \end{align*}

 \end{lemma}
\begin{proof}

Let $(\psi^{(1,0)}, \psi^{(1,0)}, \varphi^{(1)})$ be the solution of the system \eqref{eq-Airy: psi}-\eqref{eq-Rayleigh: varphi} with 
\begin{align}
F_{\psi}=F_{\psi^{(1,0)}}=-\e\La (\pa_Y^2-\al^2)\varphi_{Ray}.\nonumber
\end{align}
By  \eqref{est: e^f13-(U_s-c)} and \eqref{est: varphi_R-Y}, we have
\begin{align*}
\|(U_s-c)F_{\psi^{(1,0)}}\|_{L^\infty_\eta}\leq&\|(U_s-c)\e\La (\pa_Y^2-\al^2)\varphi_{Ray,0}\|_{L^\infty_\eta}+\|(U_s-c)\e\La (\pa_Y^2-\al^2)\varphi_{Ray,R}\|_{L^\infty_\eta}\\
\leq& C|\e|+C|\e|\|(U_s-c)(\pa_Y^4,\pa_Y^3,\pa_Y^2, \al^2)\varphi_{Ray,R}\|_{L^\infty_\eta}\\
\leq&C|\e|+C|\e|^\f13\|\varphi_{Ray,R}\|_{\mathcal{Y}}\\
\leq&C|\e|+C|\e|^\f13|\al| |\log c_i|\leq C|\e|^\f13(|\e|^\f23+|\al|)|\log c_i|.
\end{align*}
Then we get by Lemma \ref{lem-iteration: (psi, tilde psi, varphi)}  that
\begin{align*}
&\|\psi^{(1,0)}\|_{\mathcal{X}}\leq C|\e|^\f13(|\e|^\f23+|\al|)|\log c_i|^3,\\
&\|\psi^{(1,1)}\|_{\mathcal{X}}\leq C|\e|^\f13(|\e|^\f23+|\al|)|\log c_i|^5,\\
&\|\varphi^{(1)}\|_{\mathcal{Y}}\leq C|\e|^\f13(|\e|^\f23+|\al|)|\log c_i|^5.
\end{align*}
Let $\phi^{(1)}=\psi^{(1,0)}+\psi^{(1,1)}+\varphi^{(1)}$, which holds(by \eqref{est: e^f13-(U_s-c)})
\begin{align*}
\|\phi^{(1)}\|_{\mathcal{X}}\leq &\|\psi^{(1,0)}\|_{\mathcal{X}}+\|\psi^{(1,1)}\|_{\mathcal{X}}+C\|\varphi^{(1)}\|_{\mathcal{Y}}\\
\leq&C|\e|^\f13(|\e|^\f23+|\al|)|\log c_i|^5,
\end{align*}
 and 
\begin{align*}
OS[\varphi_{Ray}+\phi^{(1)}]=\e\La (\pa_Y^2-\al^2)\varphi^{(1)}.
\end{align*}
 
Let $\phi^{(2)}$ be the solution of $OS[\phi]=F_{\psi}$ with 
\begin{align}
F_{\psi}=F_{\psi^{(2,0)}}=-\e\La (\pa_Y^2-\al^2)\varphi^{(1)}.\nonumber
\end{align}
By Proposition \ref{lem: OS-nonhomo}, we get 
 \begin{align*}
\|\phi^{(2)}\|_{\mathcal{X}}\leq&C|\e|^\f13|\log c_i|^4\|\varphi^{(1)}\|_{\mathcal{Y}}.
\end{align*}

We define $\widetilde{\phi}_s=\phi^{(1)}+\phi^{(2)}$. By the smallness assumptions on $\al,~|c|,~|\e|$, which guarantee $|\e|^\f13|\log c_i|^4\leq \f12$, we obtain
\begin{align}\label{est: tilde phi_s-X}
\|\widetilde{\phi}_s\|_{\mathcal{X}}
\leq&\|\phi^{(1)}\|_{\mathcal{X}}+\|\phi^{(2)}\|_{\mathcal{X}}\\
\nonumber
\leq&\|\phi^{(1)}\|_{\mathcal{X}}+C|\e|^\f13|\log c_i|^4\|\varphi^{(1)}\|_{\mathcal{Y}}\\
\nonumber
\leq&C|\e|^\f13(|\e|^\f23+|\al|)|\log c_i|^5.
\end{align}
Then $\phi_s=\varphi_{Ray}+\widetilde{\phi}_s$ is a solution to $OS[\phi]=0$.
\end{proof}

Next we construct the fast mode $\phi_f$ to the homogeneous OS equation, which is a perturbation of the solution $\psi_a$ to the homogeneous Airy equation constructed in Proposition  \ref{pro: (psi_a, w_a)}.  From $\phi_f$ we can construct a fast solution to the homogeneous quasi-incompressible system \eqref{eq:LCNS-OS}
By Proposition  \ref{pro: (psi_a, w_a)},  the main part of $\psi_a$ is given by $\psi_a^{(0)}$, which is defined by the Airy functions as \eqref{def: (psi_a^0, w_a^0)}. 

According to the properties of Airy function(see also the proof of Proposition \ref{pro: (psi_a, w_a)}),  we have the following estimates:
due to $0\in \mathcal{N}^{-}$, it holds that for any $Y\geq 0$
\begin{align}\label{est: pa_Y^k psi_a^(0)-2}
|\pa_Y^k\psi_a^{(0)}(Y)|\leq&C|\e|^{\f 23-\f{k}{3}} e^{-C|\e|^{-\f13}|\eta^{La}(Y)-\eta^{La}(0)|(|\kappa \eta^{La}(Y)|^\f12+|\kappa \eta^{La}(0)|^\f12)},\quad k=0,1,2,3,4,
\end{align}
which along with \eqref{eq:ariy-app-gf-decay1} implies that for any $s\geq 0$, 
\begin{align}\label{est: (U_s-c)^s pa_Y^k psi_a^(0)}
|e^{\eta Y}(U_s-c)^s\pa_Y^k\psi_a^{(0)}|\leq C|\e|^{\f 23-\f k3+\f s3},\quad Y\geq 0.
\end{align}

We know that
\begin{align*}
OS[\psi_a]=\e[\La, \pa_Y^2]\psi_a+\pa_Y(A^{-1}\pa_Y U_s)\psi_a.
\end{align*}
So, we will construct $\widetilde{\phi}_f$ satisfying
\begin{align*}
OS[\widetilde{\phi}_f]=-\e[\La, \pa_Y^2]\psi_a-\pa_Y(A^{-1}\pa_Y U_s)\psi_a.
\end{align*}

\begin{lemma}\label{lem: OS-nonhomo-f}
There exists a solution $\phi_f\in W^{4,\infty}$ to $OS[\phi]=0$ satisfying  
 \begin{align*}
 \phi_f=&\psi_a+\widetilde{\phi}_f,
  \end{align*}
where
 \begin{align}\nonumber
\|\widetilde{\phi}_f\|_{\mathcal{X}}\leq C|\e|^\f23|\log c_i|,
\end{align}
and
\begin{align*}
&|\widetilde{\phi}_f(0)|\leq C|\e|^{\f23}(|\e|^\f13+|c|+\al)|\log c_i|^3\ll |\psi_a(0)|,\\
&|\pa_Y\widetilde{\phi}_f(0)|\leq C|\e|^{\f23}|\log c_i|^3\ll |\pa_Y\psi_a(0)|.
\end{align*}

 \end{lemma}
 
\begin{proof}
We define $\psi_f^{(0)}$ to be the solution to the following Airy equation:
\begin{align}\label{eq: psi_f^(0)}
Airy[\La \psi_f^{(0)}]=-\e[\La, \pa_Y^2]\psi_a=F_{\psi_f^{(0)}}.
\end{align}
Using  Proposition \ref{pro: (psi_a, w_a)} and  \eqref{est: (U_s-c)^s pa_Y^k psi_a^(0)}, we  obtain
\begin{align*}
\|(U_s-c)F_{\psi_f^{(0)}}\|_{L^\infty_\eta}\leq& C|\e|\big(\|(U_s-c)^2\pa_Y^3\psi_a^{(0)}\|_{L^\infty_\eta}+\|(U_s-c)(\pa_Y^2,\pa_Y)\psi_a^{(0)}\|_{L^\infty_\eta})\\
&+C|\e|(\|(U_s-c)^2\pa_Y^3(\psi_a-\psi_a^{(0)})\|_{L^\infty_\eta}+\|(U_s-c)(\pa_Y^2,\pa_Y)(\psi_a-\psi_a^{(0)})\|_{L^\infty_\eta}\big)\\
\leq&C|\e|^\f43.
\end{align*}
Then by Proposition \ref{pro: Airy-1} (2), we get
\begin{align}\label{est: psi_f^(0)-X}
\|\psi_f^{(0)}\|_{\mathcal{X}}\leq C|\log c_i|^2\|(U_s-c)F_{\psi_f^{(0)}}\|_{L^\infty_\eta}\leq C|\e|^\f43 |\log c_i|^2.
\end{align}
 
 In the following, we define $\varphi_f^{(0)}=\varphi_f^{(0,0)}+\varphi_f^{(0,1)}+\varphi_f^{(0,2)}$, where $\varphi_f^{(0,k)},~k=0,1,2$ is the solution to the following Rayleigh equation respectively:
 \begin{align*}
 &Ray[\varphi_f^{(0,0)}]+\al^2(U_s-c)\varphi_f^{(0,0)}=\pa_Y(A^{-1}\pa_Y U_s)\psi_a^{(0)}=F_{\varphi_f^{(0,0)}},\\
 &Ray[\varphi_f^{(0,1)}]=\al^2(U_s-c)\varphi_f^{(0,0)}=F_{\varphi_f^{(0,1)}},\\
 &{Ray[\varphi_f^{(0,2)}]=\pa_Y(A^{-1}\pa_Y U_s)(\psi_a-\psi_a^{(0)})+\e[\La, \pa_Y^2]\psi_f^{(0)}+\partial_Y(A^{-1}\partial_YU_s)\psi_f^{(0)} =F_{\varphi_f^{(0,2)}}.}
 \end{align*}

According to \eqref{est: (U_s-c)^s pa_Y^k psi_a^(0)} and the fact $|\pa_Y(A^{-1}\pa_Y U_s)|\leq Ce^{-\eta_0 Y}$, it holds that
\begin{align*}
&\|F_{\varphi_f^{(0,0)}}\|_{L^\infty_\eta}\leq C\|\psi_a^{(0)}\|_{L^\infty_\eta}\leq C|\e|^\f23,\\
&\|\pa_YF_{\varphi_f^{(0,0)}}\|_{L^\infty_\eta}\leq C\|\pa_Y\psi_a^{(0)}\|_{L^\infty_\eta}+C\|\psi_a^{(0)}\|_{L^\infty_\eta}\leq C|\e|^\f13,\\
&\|\pa_Y^2F_{\varphi_f^{(0,0)}}\|_{L^\infty_\eta}\leq C\|\pa_Y^2\psi_a^{(0)}\|_{L^\infty_\eta}+C\|\pa_Y\psi_a^{(0)}\|_{L^\infty_\eta}+C\|\psi_a^{(0)}\|_{L^\infty_\eta}\leq C.
\end{align*}
Since  the estimates in Proposition \ref{pro: phi_non} are also valid for $\al=0$ in the definition of $Ray[\cdot ]$ , we obtain
\begin{align}\label{est: varphi^(0,0)-Y}
\|\varphi_f^{(0,0)}\|_{\mathcal{Y}}\leq&C|\log c_i|\|F_{\varphi_f^{(0,0)}}\|_{L^\infty_\eta}+C|\e|^\f13\|\pa_YF_{\varphi_f^{(0,0)}}\|_{L^\infty_\eta}+C|\e|^\f23\|\pa_Y^2F_{\varphi_f^{(0,0)}}\|_{L^\infty_\eta}\\
\leq& C|\e|^\f23|\log c_i|.\nonumber
\end{align}

Next we give a more precise estimate for $\varphi_f^{(0,0)}(0)$, which has the formula: 
\begin{align*}
\varphi_f^{(0,0)}(0)=&-c\int_0^{+\infty}\f{A(Y)}{(U_s-c)^2}\int_{Y}^{+\infty} \pa_Y(A^{-1}\pa_Y U_s)\psi_a^{(0)}dY'dY\\
=&-c\int_{Y_0}^{+\infty}\f{A(Y)}{(U_s-c)^2}\int_{Y}^{+\infty} \pa_Y(A^{-1}\pa_Y U_s)\psi_a^{(0)}dY'dY\\
&-c\int_0^{Y_0}\f{A(Y)}{(U_s-c)^2}\int_{Y}^{+\infty} \pa_Y(A^{-1}\pa_Y U_s)\psi_a^{(0)}dY'dY=I_1+I_2.
\end{align*}
Due to $|\pa_Y(A^{-1}\pa_Y U_s)|\leq Ce^{-\eta_0 Y}$ and $|U_s-c|\geq C^{-1}>0$, we get by \eqref{est: pa_Y^k psi_a^(0)-2} that
\begin{align*}
|I_1|\leq&Ce^{-\eta_0 Y_0}\int_{Y_0}^{+\infty}|e^{\eta_0 Y} \psi_a^{(0)}|dY\\
\leq &C\int_{Y_0}^{+\infty} |\e|^\f23e^{-C|\e|^{-\f13}|\eta^{La}(Y)-\eta^{La}(0)|(|\kappa \eta^{La}(Y)|^\f12+|\kappa \eta^{La}(0)|^\f12)} dY, \\
\leq&C|\e|.
\end{align*}
For $I_2$, we apply Lemma \ref{lem: integral} with $G(Y)=A(Y)\int_{Y}^{+\infty} \pa_Y(A^{-1}\pa_Y U_s)\psi_a^{(0)}dY'$ and $Y=0$ to infer that
\begin{align*}
|I_2|\leq&C|G(0)|+C|c||\log c_i|(\|G\|_{L^\infty}+\|\pa_YG\|_{L^\infty}).
\end{align*}
Thus, we focus on the estimates of $G$ and $\pa_Y G$. As in the estimate of $I_1$, we have
\begin{align*}
|G(Y)|\leq& C\int_0^{+\infty}|e^{\eta_0 Y} \psi_a^{(0)}|dY\\
\leq &C\int_{0}^{+\infty} |\e|^\f23e^{-C|\e|^{-\f13}|\eta(Y)-\eta(0)|(|\kappa \eta(Y)|^\f12+|\kappa \eta(0)|^\f12)} dY\leq C|\e|,
\end{align*}
and
\begin{align*}
|\pa_Y G|\leq C|A\pa_Y(A^{-1}\pa_Y U_s)\psi_a^{(0)} |+C|A'(0)|\int_0^{+\infty}|e^{\eta_0 Y} \psi_a^{(0)}|dY
\leq C|\e|^\f23+C|\e|\leq C|\e|^\f23.
\end{align*}
Therefore, we get
\begin{align*}
|I_2|\leq C|\e|+C|c||\e|^\f23|\log c_i|\leq C|\e|^{\f23}(|\e|^\f13+|c|)|\log c_i|.
\end{align*}
This along with the estimate of $I_1$ shows that
\begin{align}\label{est: psi_f^(0,0)(0)}
|\varphi_f^{(0,0)}(0)|\leq C|\e|^{\f23}(|\e|^\f13+|c|)|\log c_i|\ll |\psi_a(0)|.
\end{align}

For $\varphi_f^{(0,1)}$, we get by Proposition \ref{pro: phi_non} that
\begin{align}\label{est: varphi^(0,1)-Y}
\|\varphi_f^{(0,1)}\|_{\mathcal{Y}}\leq &C|\log c_i|\|F_{\varphi_f^{(0,1)}}\|_{L^\infty_\eta}+C|\e|^\f13\|\pa_YF_{\varphi_f^{(0,1)}}\|_{L^\infty_\eta}+C|\e|^\f23\|\pa_Y^2F_{\varphi_f^{(0,1)}}\|_{L^\infty_\eta}\\
\nonumber
\leq&C\al^2(|\log c_i|\|\varphi_f^{(0,0)}\|_{L^\infty_\eta}+|\e|^\f13\|(U_s-c)\pa_Y\varphi_f^{(0,0)}\|_{L^\infty_\eta}\\
\nonumber
&+|\e|^\f23\|(U_s-c)\pa_Y^2\varphi_f^{(0,0)}\|_{L^\infty_\eta}+|\e|^\f23\|\pa_Y\varphi_f^{(0,0)}\|_{L^\infty_\eta})\\
\nonumber
\leq&C\al^2|\log c_i|\|\varphi_f^{(0,0)}\|_{\mathcal{Y}}\leq C\al^2|\e|^\f23|\log c_i|^2.
\end{align}

 For $\varphi_f^{(0,2)}$, we first give the estimate of $F_{\varphi_f^{(0,2)}}$. By Proposition \ref{pro: (psi_a, w_a)}, we have
 \begin{align*}
 \|F_{\varphi_f^{(0,2)}}\|_{L^\infty_\eta}\leq&C\|\psi_a-\psi_a^{(0)}\|_{L^\infty_\eta}+C\|\e[\La, \pa_Y^2]\psi_f^{(0)}\|_{L^\infty_\eta}+\|\partial_Y(A^{-1}\partial_YU_s)\psi_f^{(0)}\|_{L^\infty_\eta}\\
 \leq&C\|\psi_a-\psi_a^{(0)}\|_{L^\infty_\eta}+C|\e|\|(U_s-c)\pa_Y^3\psi_f^{(0)}\|_{L^\infty_\eta}+C|\e|\|(\pa_Y^2, \pa_Y)\psi_f^{(0)}\|_{L^\infty_\eta}\\
 &+C\|\psi_f^{(0)}\|_{L^\infty_\eta}\leq C|\e|+C\|\psi_f^{(0)}\|_{\mathcal{X}}\leq C|\e|,\\
\end{align*}
and
\begin{align*}
 |\e|^\f13\|\pa_YF_{\varphi_f^{(0,2)}}\|_{L^\infty_\eta}\leq&C|\e|^\f13\|\pa_Y(\psi_a-\psi_a^{(0)})\|_{L^\infty_\eta}+C|\e|^\f13\|\psi_a-\psi_a^{(0)}\|_{L^\infty_\eta}+C|\e|^\f13\|\e\pa_Y[\La, \pa_Y^2]\psi_f^{(0)}\|_{L^\infty_\eta}\\
 &+C|\e|^\f13\|\partial^2_Y(A^{-1}\partial_YU_s)\psi_f^{(0)}\|_{L^\infty_\eta}+C|\e|^\f13\|\partial_Y(A^{-1}\partial_YU_s)\pa_Y\psi_f^{(0)}\|_{L^\infty_\eta}\\
 \leq&C|\e|+C|\e|^\f13(\|\e(U_s-c)\pa_Y^4\psi_f^{(0)}\|_{L^\infty_\eta}+\|\e(\pa_Y^3, \pa_Y^2)\psi_f^{(0)}\|_{L^\infty_\eta}+\|(\partial_Y,1)\psi_f^{(0)}\|_{L^\infty_\eta})\\
 \leq&C|\e|+C\|\psi_f^{(0)}\|_{\mathcal{X}},\\
\end{align*}
and
\begin{align*}
 |\e|^\f23\|\pa_Y^2F_{\varphi_f^{(0,2)}}\|_{L^\infty_\eta}\leq&C|\e|^\f23\|(\pa_Y^2, \pa_Y, 1)(\psi_a-\psi_a^{(0)})\|_{L^\infty_\eta}+C|\e|^\f23\|\e\pa_Y^2[\La, \pa_Y^2]\psi_f^{(0)}\|_{L^\infty_\eta}\\
 &+C|\e|^\f23\|(\pa_Y^2, \pa_Y, 1)\psi_f^{(0)}\|_{L^\infty_\eta}\\
 \leq&C|\e||\log c_i|^2+C|\e|^\f23(\|\e(U_s-c)\pa_Y^5\psi_f^{(0)}\|_{L^\infty_\eta}+\|\e(\pa_Y^4, \pa_Y^3)\psi_f^{(0)}\|_{L^\infty_\eta}+\|(\partial_Y^2,\partial_Y,1)\psi_f^{(0)}\|_{L^\infty_\eta} )\\
 \leq&C|\e|^\f23(\|\e(U_s-c)\pa_Y^5\psi_f^{(0)}\|_{L^\infty_\eta}+C|\e||\log c_i|^2+C\|\psi_f^{(0)}\|_{\mathcal{X}}.
 \end{align*}
 
By using the equation of $\psi_f^{(0)}$ in \eqref{eq: psi_f^(0)}, we have
\begin{align*}
|\e|^\f23\|\e(U_s-c)\pa_Y^5\psi_f^{(0)}\|_{L^\infty_\eta}\leq&C|\e|^\f23\|(U_s-c)\pa_YF_{\psi_f^{(0)}}\|_{L^\infty_\eta}+C|\e|^\f23\|\e(U_s-c)^2\pa_Y^4\psi_f^{(0)}\|_{L^\infty_\eta}\\
&+C|\e|^\f23\|\e(U_s-c)(\pa_Y^3,\pa_Y^2,\pa_Y)\psi_f^{(0)}\|_{L^\infty_\eta}+C|\e|^\f23\|(U_s-c)^2\pa_Y\La\psi_f^{(0)}\|_{L^\infty_\eta}\\
&+C|\e|^\f23\|(U_s-c)\La\psi_f^{(0)}\|_{L^\infty_\eta}\\
\leq&C|\e|^\f23\|(U_s-c)\pa_YF_{\psi_f^{(0)}}\|_{L^\infty_\eta}+C\|\psi_f^{(0)}\|_{\mathcal{X}},
\end{align*}
and
\begin{align*}
|\e|^\f23\|(U_s-c)\pa_YF_{\psi_f^{(0)}}\|_{L^\infty_\eta}\leq&C|\e|^\f23\|\e(U_s-c)^2\pa_Y^4\psi_a\|_{L^\infty_\eta}+C|\e|^\f53\|(U_s-c)(\pa_Y^3, \pa_Y^2, \pa_Y)\psi_a\|_{L^\infty_\eta}\\
\leq&C|\e|^\f53.
\end{align*}
Thus, we conclude that
 \begin{align*}
 |\e|^\f13\|\pa_YF_{\varphi_f^{(0,2)}}\|_{L^\infty_\eta}\leq&C|\e|+C\|\psi_f^{(0)}\|_{\mathcal{X}}\leq C|\e|,\\
  |\e|^\f23\|\pa_Y^2F_{\varphi_f^{(0,2)}}\|_{L^\infty_\eta}\leq&C|\e||\log c_i|^2+C\|\psi_f^{(0)}\|_{\mathcal{X}}\leq C|\e||\log c_i|^2,
 \end{align*}
from which and Proposition \ref{pro: phi_non}, we infer that
\begin{align}\label{est: varphi^(0,2)-Y}
\|\varphi_f^{(0,2)}\|_{\mathcal{Y}}\leq &C|\log c_i|\|F_{\varphi_f^{(0,2)}}\|_{L^\infty_\eta}+C|\e|^\f13\|\pa_YF_{\varphi_f^{(0,2)}}\|_{L^\infty_\eta}+C|\e|^\f23\|\pa_Y^2F_{\varphi_f^{(0,2)}}\|_{L^\infty_\eta}\\
\nonumber
\leq&C|\e| |\log c_i|^2.
\end{align}

Summing up \eqref{est: varphi^(0,0)-Y}, \eqref{est: varphi^(0,1)-Y} and \eqref{est: varphi^(0,2)-Y}, we arrive at
\begin{align}\nonumber
\|\varphi_f^{(0)}\|_{\mathcal{Y}}\leq &C|\e|^\f23|\log c_i|.
\end{align}
This along with \eqref{est: psi_f^(0,0)(0)} implies that
\begin{align}
&|\varphi_f^{(0)}(0)|\leq C|\e|^{\f23}(|\e|^\f13+|c|+|\al|)|\log c_i|\ll |\psi_a(0)|,\label{est: varphi_f^(0)(0)}\\
&|\pa_Y\varphi_f^{(0)}(0)|\leq C|\e|^{\f23}|\log c_i|\ll |\pa_Y\psi_a(0)|.\label{est: pa_Y varphi_f^(0)(0)}
\end{align}
Moreover, we have
\begin{align*}
OS[\phi_f^{(0)}+\varphi_f^{(0)}]=\e\La(\pa_Y^2-\al^2)\varphi_f^{(0)}.
\end{align*}

Let $\phi^{(1)}_f$ be the solution of $OS[\phi]=F_\psi$ with 
\begin{align}
F_{\psi}=-\e\La (\pa_Y^2-\al^2)\varphi^{(0)}_f.\nonumber
\end{align}
By Proposition \ref{lem: OS-nonhomo}, we get 
\begin{align*}
\|\phi^{(1)}_f\|_{\mathcal{X}}\leq& C|\e|^\f13|\log c_i|^4\|\varphi_f^{(0)}\|_{\mathcal{Y}}.
\end{align*}
We define
\begin{align}\label{def: tilde phi_f}
\widetilde{\phi}_f=\psi_f^{(0)}+\varphi_f^{(0)}+\phi^{(1)}_f,
\end{align}
which holds
\begin{align}\label{est: tilde phi_f-X}
\|\widetilde{\phi}_f\|_{\mathcal{X}}\leq \|\psi_f^{(0)}\|_{\mathcal{X}}+\|\varphi_f^{(0)}\|_{\mathcal{Y}}+C|\e||\log c_i|^4\|\varphi_f^{(0)}\|_{\mathcal{Y}}\leq C|\e|^\f23|\log c_i|.
\end{align}
This along with \eqref{est: varphi_f^(0)(0)} and \eqref{est: pa_Y varphi_f^(0)(0)} implies that
\begin{align}
&|\widetilde{\phi}_f(0)|\leq C|\e|^{\f23}(|\e|^\f13+|c|+|\al|)|\log c_i|^3\ll |\psi_a(0)|,\label{est: tilde phi_f(0)}\\
&|\pa_Y\widetilde{\phi}_f(0)|\leq C|\e|^{\f23}|\log c_i|^3\ll |\pa_Y\psi_a(0)|.\label{est: tilde pa_Y phi_f(0)}
\end{align}

Thus, $\phi_f=\psi_a+\widetilde{\phi}_f$ is a solution of $OS[\phi_f]=0$.
\end{proof}

\section{The quasi-incompressible system}

In this section, we solve the following quasi-incompressible system:
\begin{align}\label{eq:LCNS-OS-1}
	\left\{
	\begin{aligned}
		&\mathrm i\alpha(U_s-c)\rho+\mathrm i\alpha u+\partial_Y v=0,\\
		&\sqrt{\nu}(\partial_Y^2-\alpha^2)[u+(U_s-c)\rho]-\mathrm i\alpha(U_s-c)u-v\partial_Y U_s-\mathrm i\alpha m^{-2}\rho=f_u,\\
		&\sqrt{\nu}(\partial_Y^2-\alpha^2)v-\mathrm i\alpha(U_s-c)v-m^{-2}\partial_Y\rho=f_v.
	\end{aligned}
	\right.
\end{align}
As we mentioned in Section 2, the  above system mainly corresponds to the incompressible part of the system \eqref{eq:LCNS-Y}.  Moreover, the slow and fast solutions to the system  \eqref{eq:LCNS-OS-1} constructed in this section are the main part of the slow and fast solutions to the system \eqref{eq:LCNS-Y}.

We denote by $\mathcal{L}_{Q}$ the operator on the left hand side of \eqref{eq:LCNS-OS-1}.  We introduce the  stream function $\phi(Y)$ defined by 
\begin{align*}
	\partial_Y\phi(Y)=u+(U_s-c)\rho,\quad -\mathrm i\alpha\phi(Y)=v.
\end{align*}
Then the quasi-incompressible system can be reduced to the following Orr-Sommerfeld type system:
\begin{align}\label{eq: OS}
	\left\{
	\begin{aligned}
		&\varepsilon\Lambda(\partial_Y^2-\alpha^2)\phi-(U_s-c)\Lambda\phi+\partial_Y(A^{-1}\partial_Y U_s)\phi=\Om(f_u, f_v),\\
		&\rho=m^2A^{-1}(Y)\left(\varepsilon(\partial_Y^2-\alpha^2)\partial_Y\phi-(U_s-c)\partial_Y\phi+\partial_Y U_s\phi-\frac{f_u}{\mathrm i\alpha}\right),
	\end{aligned}
	\right.
\end{align}
where 
\begin{align}\label{def: Om(f_u, f_v)}
\Om(f_u, f_v)=\f{1}{\mathrm i\al}\pa_Y(A^{-1}f_u)-f_v.
\end{align}

Let us give some formulas about the density and its derivatives. First of all, we know that
\begin{align}\label{def: rho}
\rho=m^2A^{-1}\Big(\e(\pa_Y^2-\al^2)\pa_Y\phi-(U_s-c)\pa_Y\phi+\pa_Y U_s \phi-\frac{f_u}{\mathrm i\alpha}\Big).
\end{align}
Then with $\phi$ and $\rho$, we can recover the velocity $u,v$ by
\begin{align}\label{eq:relation-uv}
u=\partial_Y\phi(Y)-(U_s-c)\rho,\quad v=-\mathrm{i}\al\phi.
\end{align}
Taking $\pa_Y$ to the equation \eqref{def: rho}, we obtain
\begin{align*}
\pa_Y\rho=-A^{-1}\pa_YA\rho+m^2A^{-1}\Big(\e(\pa_Y^2-\al^2)\pa_Y^2\phi-(U_s-c)\pa_Y^2\phi+\pa_Y^2 U_s \phi-\frac{\partial_Yf_u}{\mathrm i\alpha}\Big).
\end{align*}
On the other hand, we may write
\begin{align*}
&\e\pa_Y^2(\pa_Y^2-\al^2)\phi=\e A\La(\pa_Y^2-\al^2)\phi+A\e\al^2(\pa_Y^2-\al^2)\phi+\e A^{-1}\pa_Y A\pa_Y(\pa_Y^2-\al^2)\phi,\\
&(U_s-c)\pa_Y^2\phi=(U_s-c)(A\La \phi+A\al^2\phi+A^{-1}\pa_YA \pa_Y\phi),\\
&\pa_Y^2 U_s\phi=A\pa_Y(A^{-1}\pa_Y U_s)\phi+A^{-1}\pa_Y A \pa_Y U_s \phi.
\end{align*}
Then we find that
\begin{align*}
&\e(\pa_Y^2-\al^2)\pa_Y^2\phi-(U_s-c)\pa_Y^2\phi+\pa_Y^2 U_s \phi\\
&=A\Om(f_u, f_v)+A^{-1}\pa_Y A(m^{-2}A\rho+\f{f_u}{\mathrm i\al})+A\e\al^2(\pa_Y^2-\al^2)\phi-A(U_s-c)\al^2\phi\\
&=A\Om(f_u, f_v)+m^{-2}\pa_Y A\rho+A\e\al^2(\pa_Y^2-\al^2)\phi-A(U_s-c)\al^2\phi+A^{-1}\pa_Y A \f{f_u}{\mathrm i\al},
\end{align*}
from which and \eqref{def: Om(f_u, f_v)}-\eqref{def: rho}, we infer that  
\begin{align}
\pa_Y\rho=&-A^{-1}\pa_YA\rho+m^2A^{-1}\Big(A\Om(f_u, f_v)+m^{-2}\pa_Y A\rho\label{def: pa_Y rho}\\
\nonumber
&+A\e\al^2(\pa_Y^2-\al^2)\phi-A(U_s-c)\al^2\phi-\frac{\partial_Yf_u}{\mathrm i\alpha}+A^{-1}\pa_Y A \f{f_u}{\mathrm i\al}\Big)\\
\nonumber
=&m^2\Om(f_u, f_v)+m^2\al^2(\e(\pa_Y^2-\al^2)\phi-(U_s-c)\phi)-\frac{m^2}{\mathrm i\alpha }\pa_Y(A^{-1} f_u)\\
\nonumber
=&m^2\al^2(\e(\pa_Y^2-\al^2)\phi-(U_s-c)\phi)-m^2f_v,\\
\pa_Y^2\rho=&m^2\al^2(\e(\pa_Y^2-\al^2)\pa_Y\phi-(U_s-c)\pa_Y\phi-\pa_Y U_s\phi)-m^2\pa_Y f_v \label{def: pa_Y^2 rho}\\
\nonumber
=&\al^2A\rho-2m^2\al^2 \pa_Y U_s\phi-m^2\pa_Y f_v-\mathrm i m^2\al f_u.
\end{align}

 Our goal of this section is as follows.
 
\begin{itemize}
	\item Construct a decay solution $(\rho,u,v)$ to the quasi-incompressible system \eqref{eq:LCNS-OS-1} with decay source term by solving the non-homogenous Orr-Sommerfeld system \eqref{eq: OS}. Such solution and its estimates will be used in the iteration during the construction of the solutions to the linearized CNS \eqref{eq:LCNS-Y}.
	
	\item Construct two linearly independent decay solutions $(\rho_s,u_s,v_s)$ and $(\rho_f,u_f,v_f)$ to \eqref{eq:LCNS-OS-1}, whose stream functions $\phi_s$ and $\phi_f$ are the slow mode and fast mode to the Orr-Sommerfeld equation respectively. These two solutions are the main part of two linearly independent solution to \eqref{eq:LCNS-Y}, from which we can obtain the dispersion relation.
\end{itemize}

In this section, we focus on the regime: $(\alpha,c)\in \mathbb H$ defined by \eqref{def:Hset}.

\subsection{Non-homogeneous quasi-incompressible system}

In this subsection, we solve the non-homogeneous quasi-incompressible system via solving the OS system \eqref{eq: OS}.

\begin{proposition}\label{pro: (rho, u,v)-OS-nonhomo}
 For any $f_u\in W^{1,\infty}_{\eta}$ and $f_v\in L^\infty_\eta$, there exists a solution $(\rho,u,v)\in W^{2,\infty}_\eta$ to \eqref{eq:LCNS-OS-1} satisfying 
\begin{align*}
&\|\rho\|_{L^\infty_\eta}\leq C|\log c_i|^4 \|(U_s-c)\Om(f_u, f_v)\|_{L^\infty_\eta}+C|\alpha|^{-1}\|f_u\|_{L^\infty_\eta},\\
&\|\pa_Y\rho\|_{L^\infty_\eta}\leq C|\al|^2|\log c_i|^4 \|(U_s-c)\Om(f_u, f_v)\|_{L^\infty_\eta}+C\|f_v\|_{L^\infty_\eta},\\
	&\|(U_s-c)\pa_Y^2\rho\|_{L^\infty_\eta}\leq C|\al|^2 |\log c_i|^4 \|(U_s-c)\Om(f_u, f_v)\|_{L^\infty_\eta}+C\|\pa_Y f_v\|_{L^\infty_\eta}+C|\al|\|f_u\|_{L^\infty_\eta}.
\end{align*}
and
\begin{align*}
\|u\|_{L^\infty_\eta}+|\e|^\f13\|\pa_Yu\|_{L^\infty_\eta}\leq& C|\e|^{-\f13}|\log c_i|^4 \|(U_s-c)\Om(f_u, f_v)\|_{L^\infty_\eta}\\
&+C|\alpha|^{-1}\|f_u\|_{L^\infty_\eta}+C|\e|^\f13\|f_v\|_{L^\infty_\eta},
\end{align*}
\begin{align*}
	&\|v\|_{L^\infty_\eta}+|\e|^\f13\|\pa_Yv\|_{L^\infty_\eta}+|\varepsilon|^\f23\|\partial_Y^2v\|_{L^\infty_\eta} \leq C|\al| \log c_i|^4 \|(U_s-c)\Om(f_u, f_v)\|_{L^\infty_\eta},
\end{align*}
and 
\begin{align*}
	\|\partial_Y^2 u\|_{L^\infty_{\eta}}\leq& C|\e|^{-1}|\log c_i|^4 \|(U_s-c)\Om(f_u, f_v)\|_{L^\infty_\eta}\\
	&+C|\alpha|^{-1}\|f_u\|_{L^\infty_\eta}+C\|f_v\|_{L^\infty_\eta}+C\|\pa_Yf_v\|_{L^\infty_\eta}.
\end{align*}
Moreover, we have
\begin{align*}
\|\dv_\al(u,v)\|_{L^\infty_\eta}+\|\pa_Y\dv_\al(u,v)\|_{L^\infty_\eta}\leq& C|\al| |\log c_i|^4\|(U_s-c)\Om(f_u, f_v)\|_{L^\infty_\eta}\\
&+C\|f_u\|_{L^\infty_\eta}+C|\al|\|f_v\|_{L^\infty_\eta},
\end{align*}
\begin{align*}
	\|\pa_Y^2\dv_\al(u,v)\|_{L^\infty_\eta}\leq&C|\al| |\log c_i|^4 \|(U_s-c)\Om(f_u, f_v)\|_{L^\infty_\eta}\\
&\quad+C\|f_u\|_{L^\infty_\eta}+C|\al|\|f_v\|_{L^\infty_\eta}+C|\al|\|\pa_Yf_v\|_{L^\infty_\eta}.
\end{align*}

\end{proposition}
\begin{proof}
According to \eqref{def: rho}, we have
\begin{align*}
\|\rho\|_{L^\infty_\eta}\leq& C\|\e(\pa_Y^2-\al^2)\pa_Y\phi\|_{L^\infty_\eta}+C\|(U_s-c)\pa_Y\phi\|_{L^\infty_\eta}+C\|\pa_Y^2 U_s \phi\|_{L^\infty_\eta}+|\alpha|^{-1}\|f_u\|_{L^\infty_\eta}\\
\leq&C|\e|c_i^{-1}\|(U_s-c)\pa_Y^3\phi\|_{L^\infty_\eta}+C(1+|\e||\al|^2 c_i^{-1})\|(U_s-c)\pa_Y\phi\|_{L^\infty_\eta}+C\| \phi\|_{L^\infty_\eta}\\
&+C|\alpha|^{-1}\|f_u\|_{L^\infty_\eta}\\
\leq& C\|\phi\|_{\mathcal{X}}+C|\alpha|^{-1}\|f_u\|_{L^\infty_\eta}\leq C|\log c_i|^4 \|(U_s-c)\Om(f_u, f_v)\|_{L^\infty_\eta}+C|\alpha|^{-1}\|f_u\|_{L^\infty_\eta},
\end{align*}
where we used Proposition \ref{lem: OS-nonhomo} in the last step. 

Similarly, by \eqref{def: pa_Y rho} and \eqref{def: pa_Y^2 rho}, we  have
\begin{align*}
\|\pa_Y\rho\|_{L^\infty_\eta}\leq&C|\al|^2|\e|\|\pa_Y^2\phi\|_{L^\infty_\eta}+C|\al|^2\|\phi\|_{L^\infty_\eta}+C\|f_v\|_{L^\infty_\eta}\\
\leq& C|\al|^2\|\phi\|_{\mathcal{X}}+C\|f_v\|_{L^\infty_\eta}\\
\leq& C|\al|^2|\log c_i|^4 \|(U_s-c)\Om(f_u, f_v)\|_{L^\infty_\eta}+C\|f_v\|_{L^\infty_\eta},
\end{align*}
and 
\begin{align*}
\|(U_s-c)\pa_Y^2\rho\|_{L^\infty_\eta}\leq&C|\al|^2\|\rho\|_{L^\infty_\eta}+C|\al|^2\|\phi\|_{L^\infty_\eta}+C\|(U_s-c)\pa_Y f_v\|_{L^\infty_\eta}+C|\al|\|(U_s-c) f_u\|_{L^\infty_\eta}\\
\leq&C|\al|^2 \|\phi\|_{\mathcal{X}}+C\|\pa_Y f_v\|_{L^\infty_\eta}+C|\al|\|f_u\|_{L^\infty_\eta}\\
\leq&C|\al|^2 |\log c_i|^4 \|(U_s-c)\Om(f_u, f_v)\|_{L^\infty_\eta}+C\|\pa_Y f_v\|_{L^\infty_\eta}+C|\al|\|f_u\|_{L^\infty_\eta}.
\end{align*}

Thanks to the relation \eqref{eq:relation-uv}, we can deduce from the above estimates and Proposition \ref{lem: OS-nonhomo}  that
\begin{align*}
\|u\|_{L^\infty_\eta}\leq& \|\pa_Y\phi\|_{L^\infty_\eta}+C\|\rho\|_{L^\infty_\eta}\leq C c_i^{-1}\|\phi\|_{\mathcal{X}}+C\|\phi\|_{\mathcal{X}}+C|\alpha|^{-1}\|f_u\|_{L^\infty_\eta}\\
\leq&C|\e|^{-\f13}|\log c_i|^4 \|(U_s-c)\Om(f_u, f_v)\|_{L^\infty_\eta}+C|\alpha|^{-1}\|f_u\|_{L^\infty_\eta},
\end{align*}
and 
\begin{align*}
\|\pa_Yu\|_{L^\infty_\eta}\leq&\|\pa_Y^2\phi\|_{L^\infty_\eta}+C\|\rho\|_{L^\infty_\eta}+C\|(U_s-c)\pa_Y\rho\|_{L^\infty_\eta}\\
\leq& C|\e|^{-\f23}\|\phi\|_{\mathcal{X}}+C|\log c_i|^4 \|(U_s-c)\Om(f_u, f_v)\|_{L^\infty_\eta}+C|\alpha|^{-1}\|f_u\|_{L^\infty_\eta}+C\|f_v\|_{L^\infty_\eta}\\
\leq&C|\e|^{-\f23}|\log c_i|^4 \|(U_s-c)\Om(f_u, f_v)\|_{L^\infty_\eta}+C|\alpha|^{-1}\|f_u\|_{L^\infty_\eta}+C\|f_v\|_{L^\infty_\eta},
\end{align*}
and 
\begin{align*}
\|\partial_Y^2u\|_{L^\infty_\eta}\leq& \|\partial_Y^3\phi\|_{L^\infty_\eta}+C\|\rho\|_{L^\infty_\eta}+C\|\partial_Y\rho\|_{L^\infty}+\|(U_s-c)\partial_Y^2\rho\|_{L^\infty_\eta}\\
\leq&C|\varepsilon|^{-1} \|\phi\|_{\mathcal{X}}+C\|(\rho,\pa_Y\rho)\|_{L^\infty_\eta}+\|(U_s-c)\partial_Y^2\rho\|_{L^\infty_\eta}\\
\leq&C|\e|^{-1}|\log c_i|^4 \|(U_s-c)\Om(f_u, f_v)\|_{L^\infty_\eta}\\
&+C|\alpha|^{-1}\|f_u\|_{L^\infty_\eta}+C\|f_v\|_{L^\infty_\eta}+C\|\pa_Yf_v\|_{L^\infty_\eta}.
\end{align*}
and
\begin{align*}
&\|v\|_{L^\infty_\eta}\leq C|\al| \|\phi\|_{\mathcal{X}}\leq C|\al|\log c_i|^4 \|(U_s-c)\Om(f_u, f_v)\|_{L^\infty_\eta},\\
&\|\partial_Y v\|_{L^\infty_\eta}\leq C|\al| \|\pa_Y\phi\|_{L^\infty_\eta}\leq C|\al| |\e|^{-\f13}\|\phi\|_{\mathcal{X}}\leq C|\al| |\e|^{-\f13}\log c_i|^4 \|(U_s-c)\Om(f_u, f_v)\|_{L^\infty_\eta},\\
&\|\partial^2_Y v\|_{L^\infty_\eta}\leq C|\al| \|\pa^2_Y\phi\|_{L^\infty_\eta}\leq C|\al| |\e|^{-\f23}\|\phi\|_{\mathcal{X}}\leq C|\al| |\e|^{-\f23}\log c_i|^4 \|(U_s-c)\Om(f_u, f_v)\|_{L^\infty_\eta}.
\end{align*}
Thanks to $\dv_\al(u,v)=-\mathrm i\al (U_s-c)\rho$, we finally obtain
\begin{align*}
&\|\dv_\al(u,v)\|_{L^\infty_\eta}\leq C|\al|\|\rho\|_{L^\infty_\eta}\leq C|\al| \log c_i|^4 \|(U_s-c)\Om(f_u, f_v)\|_{L^\infty_\eta}+C\|f_u\|_{L^\infty_\eta} ,
\end{align*}
\begin{align*}
	\|\pa_Y\dv_\al(u,v)\|_{L^\infty_\eta}\leq& C|\al|\|(\rho, (U_s-c)\pa_Y\rho)\|_{L^\infty_\eta}\\
\leq& C|\al| |\log c_i|^4\|(U_s-c)\Om(f_u, f_v)\|_{L^\infty_\eta}+C\|f_u\|_{L^\infty_\eta}+C|\al|\|f_v\|_{L^\infty_\eta},
\end{align*}
and 
\begin{align*}
\|\pa_Y^2\dv_\al(u,v)\|_{L^\infty_\eta}&\leq C|\al|\|(\rho, \pa_Y\rho)\|_{L^\infty_\eta}+C|\al|\|(U_s-c)\pa_Y^2\rho\|_{L^\infty_\eta}\\
&\leq C|\al| |\log c_i|^4 \|(U_s-c)\Om(f_u, f_v)\|_{L^\infty_\eta}\\
&\quad+C\|f_u\|_{L^\infty_\eta}+C|\al|\|f_v\|_{L^\infty_\eta}+C|\al|\|\pa_Yf_v\|_{L^\infty_\eta}.
\end{align*}

The proof is completed. 
\end{proof}

\subsection{Homogeneous quasi-incompressible system} 
In this subsection, we construct two linearly independent solutions $(\rho_s,u_s,v_s)$ (slow solution) and $(\rho_f,u_f,v_f)$ (fast solution) to the homogeneous quasi-incompressible system via constructing the slow mode $\phi_s$ and fast mode $\phi_f$ to the homogeneous OS equation.

\begin{proposition}\label{pro: (rho, u,v)-OS-homo}
Let $\beta=\beta_r+i\beta_i$ be given by \eqref{def: beta}. Then there exists  a slow solution $(\rho_s,u_s,v_s)$ and a fast solution $(\rho_f,u_f,v_f)$ belonging to $W^{2,\infty}_{\beta_r}$ of the homogeneous  quasi-incompressible system  satisfying 
\begin{itemize}

\item For the density $\rho_s$ and $\rho_f$, we have
\begin{align*}
&|\al| |\rho_s|+|\pa_Y\rho_s|+|\pa_Y^2\rho_s|\leq C|\al|^2 e^{-\beta_r Y},\\
&|\al|^2\|\rho_f\|_{L^\infty_\eta}+\|\pa_Y\rho_f\|_{L^\infty_\eta}+\|\pa_Y^2\rho_f\|_{L^\infty_\eta}\leq C|\al|^2|\e|^\f23|\log c_i|.
\end{align*}

\item For the velocity, we have
\begin{align*}
&\|\dv_\al(u_s, v_s)\|_{L^\infty}+\|\pa_Y\dv_\al(u_s, v_s)\|_{L^\infty}+\|\pa_Y^2\dv_\al(u_s, v_s)\|_{L^\infty}\leq C|\al|^2,\\
&\|\dv_\al(u_f, v_f)\|_{L^\infty_\eta}+
\|\pa_Y\dv_\al(u_f, v_f)\|_{L^\infty_\eta}+\|\pa_Y^2\dv_\al(u_f, v_f)\|_{L^\infty_\eta}\leq C|\al| |\e|^\f23|\log c_i|.
\end{align*}

\item For the boundary value of the velocity, we have the asymptotic expansions:
\begin{align*}
u_s(0)
=&(1-m^2)U_s'(0)+\mathcal{O}(|\al| |\log c_i|^4 ),\\
v_s(0)
=&-\mathrm i \al(-(1-m^2)c+\f{\beta}{U_s'(Y_c)}+\mathcal{O}((|\al| |\e|^\f13+|\al|^2+|c|^2) |\log c_i|^4 )),\\
u_f(0)=&=\pa_Y \psi_a(0)+\mathcal{O}(|\e|^{\f23}|\log c_i|^3),\\
v_f(0)=&-\mathrm i \al(\psi_a(0)+\mathcal{O}(|\e|^{\f23}(|\e|^\f13+|c|+|\al|)|\log c_i|^3  )).
\end{align*}

\end{itemize}

\end{proposition}

\begin{proof} 
Recall that the slow mode $\phi_s$ of the OS equation constructed in Lemma \ref{lem: OS-nonhomo-s} satisfies 
\begin{align}
\phi_s=\varphi_{Ray}+\widetilde{\phi}_s=\varphi_{Ray,0}+\varphi_{Ray,R}+\widetilde{\phi}_s,\nonumber
\end{align}
where
\begin{align*}
\varphi_{Ray,0}=\frac{A_\infty}{(1-c)^2}(U_s-c)e^{-\beta Y},
\end{align*}
and $\varphi_{Ray,R}$ and $\widetilde{\phi}_s$ hold the estimates in \eqref{est: varphi_R-Y} and \eqref{est: tilde phi_s-X}.
We denote
\begin{align*}
\rho(\phi)=m^2A^{-1}\Big(\e(\pa_Y^2-\al^2)\pa_Y\phi-(U_s-c)\pa_Y\phi+\pa_Y U_s \phi\Big).
\end{align*}

Using the fact that
\begin{align*}
-(U_s-c)\pa_Y^2\varphi_{Ray,0}+\pa_Y^2 U_s \varphi_{Ray,0}=-\beta^2 (U_s-c)\varphi_{Ray,0}-2\beta\pa_Y U_s \varphi_{Ray,0},
\end{align*}
we have
\begin{align*}
|\rho(\varphi_{Ray,0})|=&\Big|m^2A^{-1}\Big(\e(\pa_Y^2-\al^2)\pa_Y\varphi_{Ray,0}-(U_s-c)\pa_Y\varphi_{Ray,0}+\pa_Y U_s \varphi_{Ray,0}\Big)\Big|\\
\leq&C|\al| e^{-\beta_rY},
\end{align*}
and by  \eqref{est: varphi_R-Y} and \eqref{est: tilde phi_s-X},
\begin{align*}
\|\rho(\varphi_{Ray,R})\|_{L^\infty_\eta}\leq&C|\e|\|(\pa_Y^2-\al^2)\pa_Y\varphi_{Ray,R}\|_{L^\infty_\eta}+C\|(U_s-c)\pa_Y\varphi_{Ray,R}\|_{L^\infty_\eta}+C\|\varphi_{Ray,R}\|_{L^\infty_\eta}\\
\leq&C\|\varphi_{Ray,R}\|_{\mathcal{Y}}\leq C|\al| |\log c_i|,\\
\|\rho(\widetilde{\phi}_s)\|_{L^\infty_\eta}\leq&C|\e|\|(\pa_Y^2-\al^2)\pa_Y\widetilde{\phi}_s\|_{L^\infty_\eta}+C\|(U_s-c)\pa_Y\widetilde{\phi}_s\|_{L^\infty_\eta}+C\|\widetilde{\phi}_s\|_{L^\infty_\eta}\\
\leq&C\|\widetilde{\phi}_s\|_{\mathcal{X}}\leq C|\e|^\f13(|\e|^\f23+|\al|)|\log c_i|^5.
\end{align*}
Then we conclude that
\begin{align}\label{est: rho_s}
|\rho_s|=|\rho(\phi_s)|\leq&C|\al|e^{-\beta_r Y}.
\end{align}

Using the formula \eqref{def: pa_Y rho}, $OS[\phi_s]=0$ $(f_u=f_v=0)$ and a similar argument as above, we deduce that
\begin{align*}
e^{\beta_rY}|\pa_Y\rho(\phi_s)|\leq&C|\al|^2e^{\beta_rY}(|\e(\pa_Y^2-\al^2)\varphi_{Ray, 0}|+|(U_s-c)\varphi_{Ray, 0}|)+C|\al|^2(\|\varphi_{Ray, R}\|_{\mathcal{Y}}+\|\widetilde{\phi}_s\|_{\mathcal{X}})\\
\leq&  C|\al|^2+C|\al|^3|\log c_i|+C|\al|^3|\e|^\f13|\log c_i|^5,
\end{align*}
which gives 
\begin{align}\nonumber
|\pa_Y\rho_s|=|\pa_Y\rho(\phi_s)|\leq&C|\al|^2 e^{-\beta_r Y}.
\end{align}
Similarly, by \eqref{def: pa_Y^2 rho} and  \eqref{est: rho_s}, we have
\begin{align}
&|\pa_Y^2\rho_s|\leq C|\al|^2e^{-\beta_r Y}.\nonumber
\end{align}

Next we give the estimate of $\rho_f=\rho(\phi_f)$. By the construction of $\phi_f$ in Lemma \ref{lem: OS-nonhomo-f}, we know that
\begin{align*}
\phi_f=\psi_a+\widetilde{\phi}_f,
\end{align*}
where $\psi_a$ is the solution to the homogeneous Airy equation satisfying 
\begin{align}
\|\psi_a \|_{\mathcal{X}}\leq C|\e|^\f23,\nonumber
\end{align}
and $\widetilde{\phi}_f$ given by \eqref{def: tilde phi_f} holds the estimate \eqref{est: tilde phi_f-X}. So, we have 
\begin{align}
\|\phi_f\|_{\mathcal{X}}\leq C|\e|^\f23|\log c_i|.\nonumber
\end{align}

Using the formulas \eqref{def: rho}, \eqref{def: pa_Y rho} and \eqref{def: pa_Y^2 rho} for $f_u=f_v=0$, it is easy to see that 
\begin{align*}
\|\rho_f\|_{L^\infty_\eta}\leq &C\|\e(\pa_Y^2-\al^2)\pa_Y\phi_f\|_{L^\infty_\eta}+\|(U_s-c)\pa_Y\phi_f\|_{L^\infty_\eta}+C\|\phi_f\|_{L^\infty_\eta}\\
\nonumber
\leq& C\|\phi_f\|_{\mathcal{X}}\leq C|\e|^\f23|\log c_i|,\\
\|\pa_Y\rho_f\|_{L^\infty_\eta}\leq &C|\al|^2\big(\|\e(\pa_Y^2-\al^2)\phi_f\|_{L^\infty_\eta}+\|(U_s-c)\phi_f\|_{L^\infty_\eta}\big)\\
\nonumber
\leq& C|\al|^2\|\phi_f\|_{\mathcal{X}}\leq C\al^2|\e|^\f23|\log c_i|,\\
\|\pa_Y^2\rho_f\|_{L^\infty_\eta}\leq &C|\al|^2\big(\|\rho_f\|_{L^\infty_\eta}+\|\phi_f\|_{L^\infty_\eta}\big)\\
\nonumber
\leq& C|\al|^2\|\phi_f\|_{\mathcal{X}}\leq C|\al|^2|\e|^\f23|\log c_i|.
\end{align*}

Using the relation
\begin{align}
u=\pa_Y\phi-(U_s-c)\rho,\quad v=-\mathrm i \al \phi,\nonumber
\end{align}
we get by Proposition \ref{pro: varphi_Ray(0)} that
\begin{align*}
u_s(0)=&\pa_Y\phi_s(0)+c\rho_s(0)=\pa_Y \varphi_{Ray}(0)+\mathcal{O}(|\al| |\log c_i|^4 )+\mathcal{O}(|c||\al|)\\
\nonumber
=&(1-m^2)U_s'(0)+\mathcal{O}(|\al| |\log c_i|^4 ),\\
v_s(0)=&-\mathrm i \al\phi_s(0)=-\mathrm i |\al|\big(\varphi_{Ray}(0)+\mathcal{O}(|\al| |\e|^\f13 |\log c_i|^4 \big)\\
\nonumber
=&-\mathrm i \al\big(-(1-m^2)c+\f{\beta}{U_s'(Y_c)}+\mathcal{O}((|\al| |\e|^\f13+|\al|^2+|c|^2) |\log c_i|^4\big).
\end{align*}
By using \eqref{est: tilde phi_f(0)}-\eqref{est: tilde pa_Y phi_f(0)}, we have
\begin{align*}
u_f(0)=&\pa_Y\phi_f(0)+c\rho_f(0)=\pa_Y \psi_a(0)+\mathcal{O}(|\e|^{\f23}|\log c_i|^3),\\
v_f(0)=&-\mathrm i \al\phi_f(0)=-\mathrm i \al\big(\psi_a(0)+\mathcal{O}(|\e|^{\f23}(|\e|^\f13+|c|+|\al|)|\log c_i|^3 )\big).
\end{align*}

By using the equation 
\begin{align*}
\dv_\al(u, v)+\mathrm i\al(U_s-c)\rho=0,
\end{align*}
it is easy to derive that
\begin{align*}
&\|\dv_\al(u_s, v_s)\|_{L^\infty}\leq C|\al|\|\rho_s\|_{L^\infty}\leq C|\al|^2,\\
&\|\pa_Y\dv_\al(u_s, v_s)\|_{L^\infty}\leq C|\al|\|\pa_Y U_s\rho_s\|_{L^\infty}+C|\al|\|\pa_Y\rho_s\|_{L^\infty}\leq C|\al|^2,\\
&\|\pa_Y^2\dv_\al(u_s, v_s)\|_{L^\infty}\leq C|\al|\|\pa_Y^2 U_s\rho_s\|_{L^\infty}+C|\al|\|\pa_Y U_s\pa_Y\rho_s\|_{L^\infty}+C|\al|\|\pa_Y^2\rho_s\|_{L^\infty}\leq C|\al|^2
\end{align*}
and
\begin{align*}
&\|\dv_\al(u_f, v_f)\|_{L^\infty}\leq C|\al|\|\rho_f\|_{L^\infty}\leq C|\al| |\e|^\f23|\log c_i|,\\
&\|\pa_Y\dv_\al(u_f, v_f)\|_{L^\infty}\leq C|\al|\|(\rho_f,\pa_Y \rho_f)\|_{L^\infty}\leq C|\al| |\e|^\f23|\log c_i|,\\
&\|\pa_Y^2\dv_\al(u_f, v_f)\|_{L^\infty}\leq C|\al|\|(\rho_f, \pa_Y\rho_f, \pa_Y^2\rho_f)\|_{L^\infty}
\leq C|\al| |\e|^\f23|\log c_i|.
\end{align*}

The proof is completed.
\end{proof}

\section{The quasi-compressible system}

To estimate the compressible part, we introduce the following quasi-compressible system:
\begin{align}\label{eq: Stokes}
\left\{
\begin{aligned}
&\mathrm i\al (U_s-c)\rho+\dv_\al(u, v)=0,\\
&\sqrt{\nu}(\partial_Y^2-\alpha^2)u+\mathrm i\alpha\lambda\sqrt{\nu}\mathrm{div}_\alpha(u,v)-\mathrm i\alpha(U_s-c)u\\
&\qquad\quad-(\mathrm i\alpha m^{-2}+\sqrt{\nu}\partial_Y^2 U_s)\rho=q_1,\\
&\sqrt{\nu}(\partial_Y^2-\alpha^2)v+\lambda\sqrt{\nu}\partial_Y\mathrm{div}_\alpha(u,v)-\mathrm i\alpha(U_s-c)v\\
&\qquad\quad-\int_Y^\infty \mathrm i\al \pa_Y U_s v dY'-m^{-2}\partial_Y\rho=q_2.
\end{aligned}
\right.
\end{align}
This system is used to reduce the error of the quasi-incompressible system \eqref{eq:LCNS-OS-1} during an incompressible-compressible iteration process.\smallskip

 We introduce a new term $\int_Y^\infty \mathrm i\al \pa_Y U_s v dY'$ in the third equation of \eqref{eq: Stokes}, which is different from the Stokes system used in Yang and Zhang's work \cite{YZ}. With this term, we could decouple the equation of density, which is an elliptic equation. Moreover, we can derive the expression for density. This method could help us to construct spatial mode when $\al\in \mathbb{C}$.
\medskip

We denote by $\mathcal{L}_{S}$  the operator on the left hand side of  \eqref{eq: Stokes}.
In this section, we always assume that $(\alpha,c)\in\mathbb H$ with $\mathbb H$ defined by \eqref{def:Hset} and  $\beta=\beta_r+i\beta_i$ given by \eqref{def: beta}.  

  \begin{proposition}\label{pro: stokes}
For  any $q_1,\partial_Yq_2\in L^\infty_{\beta_r}$, there exists a solution $(\rho,u,v)\in (W^{2,\infty}_{\f{\beta_r}{2}})^3$
to the system \eqref{eq: Stokes} satisfying 
\begin{align*}
		&|\alpha|^2\|\rho\|_{L^\infty_{\f{\beta_r}{2}}}+|\alpha|\|\partial_Y \rho\|_{L^\infty_{\f{\beta_r}{2}}}+\|\partial_Y^2\rho\|_{L^\infty_{\f{\beta_r}{2}}}\leq C|\alpha|\|q_1\|_{L^\infty_{\beta_r}}+C\|\partial_Y q_2\|_{L^\infty_{\beta_r}},\\
		&|\varepsilon|\|\partial_Y^2u\|_{L^\infty_{\f{\beta_r}{2}}}+|\varepsilon|^\f23\|\partial_Yu\|_{L^\infty_{\f{\beta_r}{2}}}+\|(U_s-c)u\|_{L^\infty_{\f{\beta_r}{2}}}
	\leq C|\alpha|^{-1}\|q_1\|_{L^\infty_{\beta_r}}+C|\alpha|^{-2}\|\partial_Y q_2\|_{L^\infty_{\beta_r}},
\end{align*}
and
\begin{align*}
	&|\varepsilon|^\f23\|\partial_Y^2v\|_{L^\infty_{\f{\beta_r}{2}}}+|\varepsilon|^\f13\|\partial_Yv\|_{L^\infty_{\f{\beta_r}{2}}}+|\varepsilon|^\f13\|v\|_{L^\infty_{\f{\beta_r}{2}}}\leq C\|(q_1,q_2)\|_{L^\infty_{\beta_r}}+C|\alpha|^{-1}\|\partial_Y q_2\|_{L^\infty_{\beta_r}}.
\end{align*}
\end{proposition}

This proposition is a direct corollary of the following Proposition \ref{prop:density-st} and Proposition \ref{prop:velocity-st}.\smallskip
  
  We  sketch our main ideas to construct the solution to the system \eqref{eq: Stokes}.
  \begin{enumerate}
  	\item    Multiplying $\mathrm{i}\al$ on both sides of the second equation in \eqref{eq: Stokes} and taking $\pa_Y$ on the third equation of \eqref{eq: Stokes}, we use the fact $\dv_\al(u, v)=\mathrm i\al u+\pa_Y v$ to obtain
 \begin{align*}
 (1+\la)\sqrt\nu (\pa_Y^2-\al^2)\dv_\al(u, v)&-\mathrm i\al (U_s-c) \dv_\al(u, v)-m^{-2}(\pa_Y^2-\al^2)\rho\\
 &-\mathrm i\al \sqrt\nu \pa_Y^2 U_s\rho=\mathrm i\al q_1+\pa_Yq_2.
 \end{align*}
Using the first equation in \eqref{eq: Stokes},  the above equation can be written as
\begin{align}\label{eq: rho-Stokes}
-\mathrm i\al \sqrt\nu (1+\la)(\pa_Y^2-\al^2)((U_s-c)\rho)&-\al^2(U_s-c)^2 \rho-m^{-2}(\pa_Y^2-\al^2)\rho\\
\nonumber
&-\mathrm i\al \sqrt\nu \pa_Y^2 U_s\rho=\mathrm i\al q_1+\pa_Yq_2.
\end{align}
Then we apply the method of Green function and iteration to construct the density $\rho$ to the equation \eqref{eq: rho-Stokes}.

   \item We apply Proposition \ref{pro: Airy-1} to construct the tangent velocity $u$, since the second equation in \eqref{eq: Stokes} is in the form of Airy equation with source terms in terms of $q_1,\rho$ and $\dv_{\alpha}(u,v)$:
\begin{align*}
	\varepsilon(\partial_Y^2-\alpha^2)u-(U_s-c)u=&(\mathrm i\alpha)^{-1}q_1+m^{-2}\rho+\varepsilon(\partial_Y^2U_s)\rho-\lambda\sqrt{\nu}\dv_{\alpha}(u,v)\\
	=&(\mathrm i\alpha^{-1})q_1+\big(m^{-2}+\varepsilon\partial_Y^2U_s+\mathrm i\lambda\sqrt{\nu}\alpha(U_s-c)\big)\rho.
\end{align*}

\item We can recover $\partial_Yv,\partial_Y^2 v$ from $\dv_{\alpha}(u,v)$ and $u$. The vertical velocity $v$ can be recovered from the last equation in \eqref{eq: Stokes}. Indeed, we can write it as
   \begin{align}\label{eq: v_st-1}
&(U_s-c)v+\int_Y^\infty \pa_Y (U_s-1) v dY'\\
\nonumber
&=\e(\partial_Y^2-\alpha^2)v+\lambda\e\partial_Y\mathrm{div}_\alpha(u,v)-m^{-2}(\mathrm i\al)^{-1}\partial_Y\rho-(\mathrm i\al)^{-1}q_2.
\end{align}
For the left hand side of \eqref{eq: v_st-1}, we get by integration by parts that
\begin{align*}
&(U_s-c)v+\int_Y^\infty \pa_Y (U_s-1) v dY'\\
&=(U_s-c)v-\int_Y^\infty  (U_s-1) \pa_Yv dY'-(U_s-1)v\\
&=(1-c)v-\int_Y^\infty  (U_s-1) \pa_Yv dY'.
\end{align*}
Then we may rewrite \eqref{eq: v_st-1} as follows
\begin{align*}
(1-c+\e\al^2)v=&\int_Y^\infty  (U_s-1) \pa_Yv dY'+\e\pa_Y^2 v+\lambda\e\partial_Y\mathrm{div}_\alpha(u,v)\\
&-m^{-2}(\mathrm i\al)^{-1}\partial_Y\rho-(\mathrm i\al)^{-1}q_2,
\end{align*}
 in which all of the terms on the right hand side have been obtained in the previous steps.
  \end{enumerate}

\subsection{Solving the density equation} In this subsection, we solve the following density equation 
\begin{align}\label{eq:Stokes-rho}
	\begin{split}
-\mathrm i\al \sqrt\nu (1+\la)(\pa_Y^2-\al^2)((U_s-c)\rho)&-\al^2(U_s-c)^2 \rho-m^{-2}(\pa_Y^2-\al^2)\rho\\
&-\mathrm i\al \sqrt\nu \pa_Y^2 U_s\rho=\mathrm i\al q_1+\pa_Yq_2.
	\end{split}
\end{align}
Notice that 
\begin{align*}
	m^{-2}(\partial_Y^2-\alpha^2)\rho+\alpha^2(U_s-c)^2\rho=&m^{-2}\partial_Y^2\rho-m^{-2}\alpha^2(1-m^2(U_s-c)^2)\rho\\
	=&m^{-2}\partial_Y^2\rho-m^{-2}\alpha^2A(Y)\rho\\
	=&m^{-2}(\partial_Y^2-\beta^2)\rho-m^{-2}\alpha^2(A(Y)-A_{\infty})\rho.
\end{align*}
Then we can  rewrite \eqref{eq:Stokes-rho} into the following form
\begin{align}\label{eq:rho-la}
	\begin{split}
			(\partial_Y^2-\beta^2)\rho=&\alpha^2(A(Y)-A_{\infty})\rho-\mathrm i m^2\al \sqrt\nu (1+\la)(\pa_Y^2-\al^2)((U_s-c)\rho)\\
	&-\mathrm i\alpha m^2\sqrt{\nu}\partial_Y^2U_s\rho-\mathrm i\alpha m^2q_1-m^2\partial_Y q_2.
	\end{split}
\end{align}

We construct a decay solution $\rho$ to \eqref{eq:rho-la} by the iteration. We first define 
\begin{align*}
	\rho^{(0)}(Y)=\frac{-1}{2\beta}\left(e^{-\beta Y}\int_0^{Y}e^{\beta Z}F(q_1,q_2)(Z)dZ+e^{\beta Y}\int_{Y}^{+\infty}e^{-\beta Z}F(q_1,q_2)(Z)dZ\right),
\end{align*}
where $F(q_1,q_2)(Y)=-\mathrm i\alpha m^2q_1-m^2\partial_Y q_2$. According to the above definition, we have 
\begin{align*}
		(\partial_Y^2-\beta^2)\rho^{(0)}=-\mathrm i\alpha m^2q_1-m^2\partial_Y q_2.
\end{align*}
For $j\ge 0$, we define $\rho^{(j+1)}$ as follows
\begin{align*}
	\rho^{(j+1)}(Y)=\frac{-1}{2\beta}\left(e^{-\beta Y}\int_0^{Y}e^{\beta Z}G_j(Z) dZ+e^{\beta Y}\int_{Y}^{+\infty}e^{-\beta Z}G_j (Z)dZ\right),
\end{align*}
where
\begin{align*}
	G_j(Z)=&\alpha^2(A(Y)-A_{\infty})\rho^{(j)}_{st}+\Big(-\mathrm i m^2\al \sqrt\nu (1+\la)(\pa_Y^2-\al^2)((U_s-c)\rho^{(j)})-\mathrm i\alpha m^2\sqrt{\nu}\partial_Y^2U_s\rho^{(j)}\Big)\\
	=&G_{j,1}(Z)+G_{j,2}(Z).
\end{align*}
Formally, $\rho(Y)=\sum_{j=0}^\infty\rho^{(j)}(Y)$ is a solution to \eqref{eq:rho-la}. 

We denote 
\begin{align*}
	f(Y)=\frac{-1}{2\beta}\left(e^{-\beta Y}\int_0^{Y}e^{\beta Z}g(Z) dZ+e^{\beta Y}\int_{Y}^{+\infty}e^{-\beta Z}g (Z)dZ\right).
\end{align*}
We show some useful esitmates about $f(Y)$ in the following lemma.

\begin{lemma}\label{lem:density-st}
	Let $\beta\in\mathbb C$ with $\beta_r> 0$ and $|\beta_i|\leq C\beta_r$. Then there holds 
	\begin{itemize}
		\item For any $g(Y)\in L^\infty_{\f{\beta_r}{2}}$, we have
		      \begin{align*}
		      	|\beta|^2\|f\|_{L^\infty_\f{\beta_r}{2}}+|\beta|\|\partial_Y f\|_{L^\infty_\f{\beta_r}{2}}+\|\partial_Y^2 f\|_{L^\infty_\f{\beta_r}{2}}\leq C\|g\|_{L^\infty_{\f{\beta_r}{2}}}.
		      \end{align*}
		\item For any $g(Y)\in L^\infty_{\eta}$ with $\eta>2\beta_r$, we have
		      \begin{align*}
		      	|\beta|\|f\|_{L^\infty_{\beta_r}}+\|\partial_Y f\|_{L^\infty_{\beta_r}}+\|\partial_Y^2 f\|_{L^\infty_{\beta_r}}\leq C\|g\|_{L^\infty_\eta}.
		      \end{align*}
	\end{itemize}
\end{lemma}
\begin{proof}
	We first show the first statement of this lemma. By the definition of $f(Y)$, we have  
	\begin{align*}
		|e^{\frac{\beta_r}{2}Y}f(Y)|\leq& C|\beta|^{-1}\int_0^{+\infty}e^{-\beta_r|Y-Z|}e^{\frac{\beta_r}{2}|Y-Z|}|e^{\frac{\beta_r Z}{2}}g(Z)|dZ\\
		\leq&C|\beta|^{-2}\|g\|_{L^\infty_{\f{\beta_r}{2}}}.
	\end{align*}
	By a direct calculation, we know that 
	\begin{align*}
		\partial_Yf(Y)= \frac{1}{2}\left(e^{-\beta Y}\int_0^{Y}e^{\beta Z}g(Z) dZ-e^{\beta Y}\int_{Y}^{+\infty}e^{-\beta Z}g (Z)dZ\right).
	\end{align*}
	Then we have 
	\begin{align*}
		|e^{\frac{\beta_r}{2}Y}\partial_Y f(Y)|\leq& C\int_0^{+\infty}e^{-\beta_r|Y-Z|}e^{\frac{\beta_r}{2}|Y-Z|}|e^{\frac{\beta_r Z}{2}}g(Z)|dZ\\
		\leq&C|\beta|^{-1}\|g\|_{L^\infty_{\f{\beta_r}{2}}}.
	\end{align*}
	On the other hand, by using $\partial_Y^2f=\beta^2 f+g$, we have
	\begin{align*}
		\|\partial_Y^2 f\|_{L^\infty_\f{\beta_r}{2}}\leq C\|g\|_{L^\infty_{\f{\beta_r}{2}}}.
	\end{align*}
	
	Now we turn to show the second statement. Again by the definition of $f$, we have
	\begin{align*}
		|e^{\beta_r Y}f(Y)|\leq&|\beta|^{-1}\int_0^Ye^{\beta_r Z}|g(Z)|dZ+|\beta|^{-1}\int_Y^{+\infty}e^{2\beta_r Z}|g(Z)|dZ\\
		\leq&|\beta|^{-1}\int_0^{Y}e^{-(\eta-\beta_r)Z}|e^{\eta Z}g(Z)|dZ+|\beta|^{-1}\int_Y^{+\infty}e^{-(\eta-2\beta_r)Z}|e^{\eta Z}g(Z)|dZ\\
		\leq&C|\beta|^{-1}\|g\|_{L^\infty_\eta}.
	\end{align*}
	Similarly, we can obtain 
	\begin{align*}
		\|\partial_Y f\|_{L^\infty_{\beta_r}}+\|\partial_Y^2f\|_{L^\infty_{\beta_r}}\leq C\|g\|_{L^\infty_\eta}.
	\end{align*}
	
\end{proof}
	
\begin{proposition}\label{prop:density-st}
 For  any $q_1,\partial_Yq_2\in L^\infty_{\beta_r}$, there exists a solution $\rho\in W^{2,\infty}_{\f{\beta_r}{2}}$ satisfying 
	\begin{align*}
		|\alpha|^2\|\rho\|_{L^\infty_{\f{\beta_r}{2}}}+|\alpha|\|\partial_Y \rho\|_{L^\infty_{\f{\beta_r}{2}}}+\|\partial_Y^2\rho\|_{L^\infty_{\f{\beta_r}{2}}}\leq C|\alpha|\|q_1\|_{L^\infty_{\beta_r}}+C\|\partial_Y q_2\|_{L^\infty_{\beta_r}}.
	\end{align*}
	
\end{proposition}

\begin{proof}
	We construct the solution $\rho$ by the iteration. We define 
	\begin{align*}
	\rho^{(0)}(Y)=\frac{-1}{2\beta}\left(e^{-\beta Y}\int_0^{Y}e^{\beta Z}F(q_1,q_2)(Z)dZ+e^{\beta Y}\int_{Y}^{+\infty}e^{-\beta Z}F(q_1,q_2)(Z)dZ\right),
\end{align*}
where $F(q_1,q_2)(Y)=-\mathrm i\alpha m^2q_1-m^2\partial_Y q_2$. Hence, by Lemma \ref{lem:density-st},  we obtain 
\begin{align*}
		|\alpha|^2\|\rho^{(0)}\|_{L^\infty_{\f{\beta_r}{2}}}+|\alpha|\|\partial_Y \rho^{(0)}\|_{L^\infty_{\f{\beta_r}{2}}}+\|\partial_Y^2\rho^{(0)}\|_{L^\infty_{\f{\beta_r}{2}}}\leq C|\alpha|\|q_1\|_{L^\infty_{\beta_r}}+C\|\partial_Y q_2\|_{L^\infty_{\beta_r}}.
	\end{align*}
Recall that for any $j\in\mathbb N$,
\begin{align*}
	\rho^{(j+1)}(Y)=&\frac{-1}{2\beta}\left(e^{-\beta Y}\int_0^{Y}e^{\beta Z}G_{j,1}(Z) dZ+e^{\beta Y}\int_{Y}^{+\infty}e^{-\beta Z}G_{j,1} (Z)dZ\right)\\
	&+\frac{-1}{2\beta}\left(e^{-\beta Y}\int_0^{Y}e^{\beta Z}G_{j,2}(Z) dZ+e^{\beta Y}\int_{Y}^{+\infty}e^{-\beta Z}G_{j,2} (Z)dZ\right),
\end{align*}
where
\begin{align*}
	\|G_{j,1}\|_{L^\infty_\eta}\leq& C|\alpha|^2\|\rho_{st}^{(j)}\|_{L^\infty_{\f{\beta_r}{2}}},\\
	\|G_{j,2}\|_{L^\infty_{\f{\beta_r}{2}}}\leq&C|\alpha|\nu^\f12(\|\rho_{st}^{(j)}\|_{L^\infty_{\f{\beta_r}{2}}}+\|\partial_Y \rho_{st}^{(j)}\|_{L^\infty_{\f{\beta_r}{2}}}+\|\partial_Y^2\rho_{st}^{(j)}\|_{L^\infty_{\f{\beta_r}{2}}}).
\end{align*}
Therefore, we apply Lemma \ref{lem:density-st} to obtain 
\begin{align*}
	&\|\rho_{st}^{(j+1)}\|_{L^\infty_{\f{\beta_r}{2}}}\leq C|\alpha|\|\rho_{st}^{(j)}\|_{L^\infty_{\f{\beta_r}{2}}}+C|\al|^{-1}\nu^\f12(\|\partial_Y \rho_{st}^{(j)}\|_{L^\infty_{\f{\beta_r}{2}}}+\|\partial_Y^2\rho_{st}^{(j)}\|_{L^\infty_{\f{\beta_r}{2}}}),\\ 
	&\|\partial_Y \rho_{st}^{(j+1)}\|_{L^\infty_{\f{\beta_r}{2}}}+\|\partial_Y^2\rho_{st}^{(j+1)}\|_{L^\infty_{\f{\beta_r}{2}}}\leq C(|\alpha|^2\|\rho_{st}^{(j)}\|_{L^\infty_{\f{\beta_r}{2}}}+\nu^\f12\|\partial_Y \rho_{st}^{(j)}\|_{L^\infty_{\f{\beta_r}{2}}}+\nu^\f12\|\partial_Y^2\rho_{st}^{(j)}\|_{L^\infty_{\f{\beta_r}{2}}}),
\end{align*}
which implies 
	\begin{align*}
		&|\alpha|\| \rho_{st}^{(j+1)}\|_{L^\infty_{\f{\beta_r}{2}}}+\|\partial_Y \rho_{st}^{(j+1)}\|_{L^\infty_{\f{\beta_r}{2}}}+\|\partial_Y^2\rho_{st}^{(j+1)}\|_{L^\infty_{\f{\beta_r}{2}}}\\
		&\qquad\leq C|\alpha|(|\alpha|\| \rho_{st}^{(j)}\|_{L^\infty_{\f{\beta_r}{2}}}+\|\partial_Y \rho_{st}^{(j)}\|_{L^\infty_{\f{\beta_r}{2}}}+\|\partial_Y^2\rho_{st}^{(j)}\|_{L^\infty_{\f{\beta_r}{2}}}).
	\end{align*} 
This finished the proof by performing the iteration due to the smallness of $|\al|$.	
\end{proof}

\subsection{Constructing the velocity $(u,v)$}

In this subsection, we construct the velocity $(u,v)$ to the system \eqref{eq: Stokes} once the density $\rho$ is constructed in Proposition \ref{prop:density-st}.
\begin{proposition}\label{prop:velocity-st}
Let $\rho$ be the density constructed in Proposition \ref{prop:density-st} for  any given $q_1,q_2,\partial_Yq_2\in L^\infty_{\beta_r}$, there exists $u,v\in W^{2,\infty}_{\beta_r/2}$ such that $(\rho,u,v)$ is a solution to the system \eqref{eq: Stokes} satisfying 
	\begin{align*}
	|\varepsilon|\|\partial_Y^2u\|_{L^\infty_{\f{\beta_r}{2}}}+|\varepsilon|^\f23\|\partial_Yu\|_{L^\infty_{\f{\beta_r}{2}}}+\|(U_s-c)u\|_{L^\infty_{\f{\beta_r}{2}}}
	\leq&C|\alpha|^{-1}\|q_1\|_{L^\infty_{\beta_r}}+C|\alpha|^{-2}\|\partial_Y q_2\|_{L^\infty_{\beta_r}},
\end{align*}
and
\begin{align*}
	&|\varepsilon|^\f23\|\partial_Y^2v\|_{L^\infty_{\f{\beta_r}{2}}}+|\varepsilon|^\f13\|\partial_Yv\|_{L^\infty_{\f{\beta_r}{2}}}+|\varepsilon|^\f13\|v\|_{L^\infty_{\f{\beta_r}{2}}}\leq C\|(q_1,q_2)\|_{L^\infty_{\beta_r}}+C|\alpha|^{-1}\|\partial_Y q_2\|_{L^\infty_{\beta_r}}.
\end{align*}
 \end{proposition}

\begin{proof}
	By Proposition \ref{prop:density-st}, we know that 
	\begin{align*}
		|\alpha|^2\|\rho\|_{L^\infty_{\f{\beta_r}{2}}}+|\alpha|\|\partial_Y \rho\|_{L^\infty_{\f{\beta_r}{2}}}+\|\partial_Y^2\rho\|_{L^\infty_{\f{\beta_r}{2}}}\leq C|\alpha|\|q_1\|_{L^\infty_{\beta_r}}+C\|\partial_Y q_2\|_{L^\infty_{\beta_r}},
	\end{align*}
	which along with the first equation in \eqref{eq: Stokes} deduces that 
	\begin{align}
		|\alpha|\sum_{k=0}^2\|\pa_Y^{k}\dv_{\al}(u, v)\|_{L^\infty_{\f{\beta_r}{2}}}\leq C|\alpha|\|q_1\|_{L^\infty_{\beta_r}}+C\|\partial_Y q_2\|_{L^\infty_{\beta_r}}.
	\end{align}
	
	Now we construct the tangent velocity $u$. Recall that $u$ satisfies the following equation 
	\begin{align*}
	\varepsilon(\partial_Y^2-\alpha^2)u-(U_s-c)u
=&(\mathrm i\alpha^{-1})q_1+(m^{-2}+\varepsilon\partial_Y^2U_s+\mathrm i\lambda\sqrt{\nu}\alpha(U_s-c))\rho.
\end{align*}
By Proposition \ref{pro: Airy-1}, we obtain 
\begin{align*}
	|\varepsilon|\|\partial_Y^2u\|_{L^\infty_{\f{\beta_r}{2}}}+|\varepsilon|^\f23\|\partial_Yu\|_{L^\infty_{\f{\beta_r}{2}}}+\|(U_s-c)u\|_{L^\infty_{\f{\beta_r}{2}}}\leq& C|\alpha|^{-1}\|q_1\|_{L^\infty_{\f{\beta_r}{2}}}+C\|\rho\|_{L^\infty_{\f{\beta_r}{2}}}\\
	\leq&C|\alpha|^{-1}\|q_1\|_{L^\infty_{\beta_r}}+C|\alpha|^{-2}\|\partial_Y q_2\|_{L^\infty_{\beta_r}}.
\end{align*}

Now we turn to show the control of $v$. Firstly, we notice that 
\begin{align*}
&\pa_Y v=-\mathrm i\al(u+(U_s-c)\rho),\\
&\pa_Y^2 v=-\mathrm i\al\big(\pa_Yu+(U_s-c)\pa_Y\rho+U_s' \rho\big).
\end{align*}
Then we obtain 
\begin{align*}
	&|\varepsilon|^\f23\|\partial_Y^2v\|_{L^\infty_{\f{\beta_r}{2}}}+|\varepsilon|^\f13\|\partial_Yv\|_{L^\infty_{\f{\beta_r}{2}}}\leq C\|q_1\|_{L^\infty_{\beta_r}}+C|\alpha|^{-1}\|\partial_Y q_2\|_{L^\infty_{\beta_r}}.
\end{align*}
On the other hand, we know that $v$ satisfies the following equation
\begin{align*}
(1-c+\e\al^2)v=&\int_Y^\infty  (U_s-1) \pa_YvdY'+\e\pa_Y^2 v+\lambda\e\partial_Y\mathrm{div}_\alpha(u,v)\\
&-m^{-2}(\mathrm i\al)^{-1}\partial_Y\rho-(\mathrm i\al)^{-1}q_2,
\end{align*}
which implies 
\begin{align*}
\|v\|_{L^\infty_{\f{\beta_r}{2}}}\leq& C\|\partial_Yv\|_{L^\infty_{\f{\beta_r}{2}}}+C|\varepsilon|\|\partial_Y^2v\|_{L^\infty_{\f{\beta_r}{2}}}+C|\varepsilon|\|\pa_Y\dv_{\al}(u, v)\|_{L^\infty_{\f{\beta_r}{2}}}\\
&+C|\alpha|^{-1}\|\partial_Y\rho\|_{L^\infty_{\f{\beta_r}{2}}}+|\alpha|^{-1}\|q_2\|_{L^\infty_{\f{\beta_r}{2}}}\\
\leq& C\|(|\varepsilon|^{-\f13}q_1,|\alpha|^{-1}q_2)\|_{L^\infty_{\beta_r}}+|\varepsilon|^{-\f13}|\alpha|^{-1}\|\partial_Y q_2\|_{L^\infty_{\beta_r}}.
\end{align*}

The proof is completed.
\end{proof}

\section{The full linearized CNS system}
In this section, we construct two linearly independent solutions to the  homogeneous linearized CNS system:
\begin{align}\label{eq:LCNS-Y1}
	\left\{
	\begin{aligned}
		&\mathrm i\alpha(U_s-c)\rho+\mathrm{div}_\alpha(u,v)=0,\\
		&\sqrt{\nu}(\partial_Y^2-\alpha^2)u+\mathrm i\alpha\lambda\sqrt{\nu}\mathrm{div}_\alpha(u,v)-\mathrm i\alpha(U_s-c)u\\
		&\qquad\qquad-(\mathrm i\alpha m^{-2}+\sqrt{\nu}\partial_Y^2 U_s)\rho-v\partial_Y U_s=0,\\
		&\sqrt{\nu}(\partial_Y^2-\alpha^2)v+\lambda\sqrt{\nu}\partial_Y\mathrm{div}_\alpha(u,v)-\mathrm i\alpha(U_s-c)v-m^{-2}\partial_Y\rho=0.
	\end{aligned}
	\right.
\end{align}
We denote by $\mathcal L$ the linear operator of the above system.  Recall that $\mathcal L_{Q}$ and $\mathcal L_{S}$ are the linear operators of the system \eqref{eq:LCNS-OS-1} and the system \eqref{eq: Stokes} respectively.\smallskip

We introduce main idea of the construction:

\begin{enumerate}
	\item Let $(\rho_{os}, u_{os}, v_{os})$ be either the slow or fast solution to homogeneous quasi-incompressible Navier-Stokes system constructed in Proposition \ref{pro: (rho, u,v)-OS-homo}.  By direct computations, we have
\begin{align*}
\mathcal{L}(\rho_{os}, u_{os}, v_{os})=&(0,-\sqrt{\nu}(\pa_Y^2-\al^2)((U_s-c)\rho_{os})+\mathrm i\la \sqrt{\nu} \al\dv_\al(u_{os}, v_{os})\\
&-\sqrt{\nu}\pa_Y^2 U_s\rho_{os}, \la \sqrt{\nu} \pa_Y\dv_\al(u_{os}, v_{os}))
\end{align*}
We denote $(\rho^{(1)}_{os},u^{(1)}_{os},v^{(1)}_{os})=(\rho_{os}, u_{os}, v_{os})$. Then we introduce $(\rho^{(1)}_{st}, u^{(1)}_{st}, v^{(1)}_{st})$ to  be the solution to the system \eqref{eq: Stokes} constructed in Proposition \ref{pro: stokes} with the source term $(0,q^{(1)}_1,q^{(1)}_2)$:
\begin{align*}
q^{(1)}_1=&\sqrt{\nu}(\pa_Y^2-\al^2)((U_s-c)\rho^{(1)}_{os})-\mathrm i\la \sqrt{\nu} \al\dv_\al(u^{(1)}_{os}, v^{(1)}_{os})+\sqrt{\nu}\pa_Y^2 U_s\rho^{(1)}_{os}\\
=&-\e(\pa_Y^2-\al^2)\dv_\al(u^{(1)}_{os}, v^{(1)}_{os})-\mathrm i\la \sqrt{\nu} \al\dv_\al(u^{(1)}_{os}, v^{(1)}_{os})+\sqrt{\nu}\pa_Y^2 U_s\rho^{(1)}_{os},\\
q^{(1)}_2=&-\la \sqrt{\nu} \pa_Y\dv_\al(u^{(1)}_{os}, v^{(1)}_{os}).
\end{align*}
Here $\sqrt\nu=|\e||\al|$. Then we define the initial step of the iteration by $(\rho^{(1)},u^{(1)},v^{(1)})=(\rho^{(1)}_{os},u^{(1)}_{os},v^{(1)}_{os})+(\rho^{(1)}_{st},u^{(1)}_{st},v^{(1)}_{st})$. It is obvious that 
\begin{align*}
	\mathcal L(\rho^{(1)},u^{(1)},v^{(1)})=\Big(0,-\pa_Y U_sv^{(1)}_{st},-\int_Y^\infty \mathrm i\al \pa_Y U_s v^{(1)}_{st} dY'\Big).
\end{align*}

\item We define our iterative scheme inductively. For any $j\geq 1$, we define that $(\rho_{os}^{(j+1)},u_{os}^{(j+1)},v_{os}^{(j+1)})$ is the solution to the quasi-incompressible system \eqref{eq:LCNS-OS-1} constructed in Proposition \ref{pro: (rho, u,v)-OS-nonhomo} with source term $(0,f^{(j)}_u,f^{(j)}_v)$, where $f^{(j)}_u=\pa_Y U_sv^{(j)}_{st}, ~f_v^{(j)}=\int_Y^\infty \mathrm i\al \pa_Y U_s v^{(j)}_{st} dY'$. According to such definition, we find that 
   \begin{align*}
   	\mathcal L\Big(\sum_{l=1}^j\rho^{(l)}+\rho^{(j+1)}_{os},&\sum_{l=1}^ju^{(l)}+u^{(j+1)}_{os},\sum_{l=1}^jv^{(l)}+v_{os}^{(j+1)}\Big)\\
   	=&\big(0,-\sqrt{\nu}(\pa_Y^2-\al^2)((U_s-c)\rho_{os}^{(j+1)})+\mathrm i\la \sqrt{\nu} \al\dv_\al(u_{os}^{(j+1)}, v_{os}^{(j+1)})\\
&-\sqrt{\nu}\pa_Y^2 U_s\rho_{os}^{(j+1)}, \la \sqrt{\nu} \pa_Y\dv_\al(u_{os}^{(j+1)}, v_{os}^{(j+1)})\big).
   \end{align*}
  Then as in the initial step, we use the system \eqref{eq: Stokes} to shrink the above errors. In details, we introduce $(\rho^{(j+1)}_{st}, u^{(j+1)}_{st}, v^{(j+1)}_{st})$ to be the solution to the system \eqref{eq: Stokes} constructed in Proposition \ref{pro: stokes} with the source term $(0,q^{(j+1)}_1,q^{(j+1)}_2)$:
\begin{align*}
	q^{(j+1)}_1=&-\e(\pa_Y^2-\al^2)\dv_\al(u^{(j+1)}_{os}, v^{(j+1)}_{os})-i\la \sqrt{\nu} \al\dv_\al(u^{(j+1)}_{os}, v^{(j+1)}_{os})+\sqrt{\nu}\pa_Y^2 U_s\rho^{(j+1)}_{os},\\
q^{(j+1)}_2=&-\la \sqrt{\nu} \pa_Y\dv_\al(u^{(j+1)}_{os}, v^{(j+1)}_{os}).
\end{align*}
Hence, we define the $(j+1)$th step of the iteration by $(\rho^{(j+1)},u^{(j+1)},v^{(j+1)})=(\rho^{(j+1)}_{os},u^{(j+1)}_{os},v^{(j+1)}_{os})+(\rho^{(j+1)}_{st},u^{(j+1)}_{st},v^{(j+1)}_{st})$. It is obvious that 
\begin{align*}
	\mathcal L\Big(\sum_{l=1}^{j+1}\rho^{(l)},\sum_{l=1}^{j+1}u^{(l)},\sum_{l=1}^{j+1}v^{(l)}\Big)=\Big(0,-\pa_Y U_sv^{(j+1)}_{st},-\int_Y^\infty\mathrm i\al \pa_Y U_s v^{(j+1)}_{st} dY'\Big).
\end{align*}

\item By the above iteration, we formally obtain that $\mathcal L(\rho,u,v)=0$ with 
     \begin{align*}
     	\rho=\sum_{j=1}^{+\infty}\rho^{(j)},\quad u=\sum_{j=1}^{+\infty}u^{(j)},\quad v=\sum_{j=1}^{+\infty}v^{(j)}.
     \end{align*}
     
 \end{enumerate}
 
  In this section, we always assume that $(\alpha,c)\in\mathbb H$ with $\mathbb H$ defined by \eqref{def:Hset} and $\beta=\beta_r+\mathrm i\beta_i$ be given by \eqref{def: beta}.  
To realize the above iteration scheme rigorously,  we introduce the space $\mathcal Z$ for the density defined by
     \begin{align*}
     	\mathcal Z=\big\{f\in W^{1,\infty}:\|f\|_{\mathcal Z}<+\infty\big\},
     \end{align*}
     where 
     \begin{align*}
     	\|f\|_{\mathcal Z}=\|f\|_{L^\infty_{\f{\beta_r}{2}}}+\|\partial_Y f\|_{L^\infty_{\f{\beta_r}{2}}}.
     \end{align*}
   For the velocity, we take the space $\tilde{\mathcal X}$ defined by
     \begin{align*}
     	\tilde{\mathcal X}=\big\{f\in W^{2,\infty}:\|f\|_{\tilde{\mathcal X}}<+\infty\big\},
     \end{align*}
     where
     \begin{align*}
     	\|f\|_{\tilde{\mathcal X}}=|\varepsilon|^\f13\|\partial_Y^2 f\|_{L^\infty_{\f{\beta_r}{2}}}+\|\partial_Y f\|_{L^\infty_{\f{\beta_r}{2}}}+\| f\|_{L^\infty_{\f{\beta_r}{2}}}.
     \end{align*}

 We first give the estimates for the above iteration sequence. 
 
\begin{lemma}\label{prop:main-iteration}
If $(\rho^{(j)}_{st},u^{(j)}_{st},v^{(j)}_{st})$ belongs to $\mathcal Z\times\tilde{\mathcal X}^2$, then we have
\begin{align*}
	&|\varepsilon|^\f23\|\rho^{(j+1)}_{os}\|_{\mathcal Z}+|\varepsilon|\|u_{os}^{(j+1)}\|_{\tilde{\mathcal X}}+|\varepsilon|^\f13\|v_{os}^{(j+1)}\|_{\tilde{\mathcal X}}\leq C|\log c_i|^4\|v_{st}^{(j)}\|_{\tilde{\mathcal X}},\\
	&\|\rho^{(j+1)}_{st}\|_{\mathcal Z}+|\varepsilon|^\f13\|u_{st}^{(j+1)}\|_{\tilde{\mathcal X}}+|\varepsilon|^{-\f13}\|v_{st}^{(j+1)}\|_{\tilde{\mathcal X}}\leq C|\log c_i|^4\|v_{st}^{(j)}\|_{\tilde{\mathcal X}},
\end{align*}
and
	\begin{align*}
	|\varepsilon|^\f23\|u_{os}^{(j+1)}\|_{L^\infty}+\|v_{os}^{(j+1)}\|_{L^\infty}\leq C|\log c_i|^4\big(\|v^{(j)}_{st}\|_{L^\infty}+\|\partial_Y v^{(j)}_{st}\|_{L^\infty}\big).
\end{align*}
\end{lemma}

\begin{proof}
  
Recall that $(\rho_{os}^{(j+1)},u_{os}^{(j+1)},v_{os}^{(j+1)})$ is the solution to 
	\begin{align*}
			\left\{
	\begin{aligned}
		&\mathrm i\alpha(U_s-c)\rho_{os}+\mathrm i\alpha u_{os}+\partial_Y v_{os}=0,\\
		&\sqrt{\nu}(\partial_Y^2-\alpha^2)[u_{os}+(U_s-c)\rho_{os}]-\mathrm i\alpha(U_s-c)u_{os}-v_{os}\partial_Y U_s-\mathrm i\alpha m^{-2}\rho_{os}=f_u,\\
		&\sqrt{\nu}(\partial_Y^2-\alpha^2)v_{os}-\mathrm i\alpha(U_s-c)v_{os}-m^{-2}\partial_Y\rho_{os}=f_v,
	\end{aligned}
	\right.
	\end{align*}
	with $f_u=\pa_Y U_sv^{(j)}_{st}, ~f_v=\int_Y^\infty \mathrm i\al \pa_Y U_s v^{(j)}_{st} dY'$. 
	By Proposition \ref{pro: (rho, u,v)-OS-nonhomo}, we obtain 
	\begin{align*}
\|\rho^{(j+1)}_{os}\|_{L^\infty_\eta}+\|\pa_Y\rho^{(j+1)}_{os}\|_{L^\infty_\eta}\leq& C|\log c_i|^4 \|(U_s-c)\Om(f_u, f_v)\|_{L^\infty_\eta}\\
&+C|\alpha|^{-1}\|f_u\|_{L^\infty_\eta}+C\|f_v\|_{L^\infty_\eta},
\end{align*}
and
\begin{align*}
\|u_{os}^{(j+1)}\|_{L^\infty_\eta}+|\e|^\f13\|\pa_Yu_{os}^{(j+1)}\|_{L^\infty_\eta}\leq& C|\e|^{-\f13}|\log c_i|^4 \|(U_s-c)\Om(f_u, f_v)\|_{L^\infty_\eta}+C|\alpha|^{-1}\|f_u\|_{L^\infty_\eta}\\
&+C|\e|^\f13\|f_v\|_{L^\infty_\eta},
\end{align*}
\begin{align*}
\|\partial_Y^2 u_{os}^{(j+1)}\|_{L^\infty_\eta}\leq& C|\e|^{-1}|\log c_i|^4 \|(U_s-c)\Om(f_u, f_v)\|_{L^\infty_\eta}\\
&+C|\alpha|^{-1}\|f_u\|_{L^\infty_\eta}+C\|f_v\|_{L^\infty_\eta}+C\|\pa_Yf_v\|_{L^\infty_\eta},
\end{align*}
and
\begin{align*}
	&\|v_{os}^{(j+1)}\|_{L^\infty_\eta}+|\e|^\f13\|\pa_Yv_{os}^{(j+1)}\|_{L^\infty_\eta}+|\varepsilon|^\f23\|\partial_Y^2v_{os}^{(j+1)}\|_{L^\infty_\eta} \leq C|\al| \log c_i|^4 \|(U_s-c)\Om(f_u, f_v)\|_{L^\infty_\eta}.
\end{align*}

Moreover, we have 
\begin{align*}
&\|\dv_\al(u_{os}^{(j+1)},v_{os}^{(j+1)})\|_{L^\infty_\eta}+\|\pa_Y\dv_\al(u_{os}^{(j+1)},v_{os}^{(j+1)})\|_{L^\infty_\eta}\\
&\leq C|\al| \log c_i|^4 \|(U_s-c)\Om(f_u, f_v)\|_{L^\infty_\eta}+C\|f_u\|_{L^\infty_\eta}+C|\al|\|f_v\|_{L^\infty_\eta},
\end{align*}
\begin{align*}
	\|\pa_Y^2\dv_\al(u_{os}^{(j+1)},v_{os}^{(j+1)})\|_{L^\infty_\eta}\leq& C|\al| |\log c_i|^4 \|(U_s-c)\Om(f_u, f_v)\|_{L^\infty_\eta}\\
	&+C\|f_u\|_{L^\infty_\eta}+C|\al|\|f_v\|_{L^\infty_\eta}+C|\al|\|\pa_Yf_v\|_{L^\infty_\eta}.
\end{align*}
Here 
\begin{align*}
&f_u=\pa_Y U_sv^{(j)}_{st},\quad f_v=\int_Y^\infty \mathrm i\al \pa_Y U_s v^{(j)}_{st} dY',\\
&\Om(f_u, f_v)=\f{1}{\mathrm i\al}\pa_Y(A^{-1}\pa_Y U_sv^{(j)}_{st})-\int_Y^\infty\mathrm i\al \pa_Y U_s v^{(j)}_{st} dY'.
\end{align*}
On the other hand, we have 
\begin{align*}
&\|(U_s-c)\Omega(f_u,f_v)\|_{L^\infty_\eta}+|\alpha|^{-1}\|f_u\|_{L^\infty_\eta}\leq C|\alpha|^{-1}\|v_{st}^{(j)}\|_{L^\infty}+C|\alpha|^{-1}\|\partial_Yv^{(j)}_{st}\|_{L^\infty},
\end{align*}
and
\begin{align*}
\|f_v\|_{L^\infty_\eta}+\|\pa_Yf_v\|_{L^\infty_\eta}\leq C|\al|\|v_{st}^{(j)}\|_{L^\infty}
\end{align*}
Therefore, we obtain 
\begin{align*}
\|\rho^{(j+1)}_{os}\|_{L^\infty_\eta}+\|\pa_Y\rho^{(j+1)}_{os}\|_{L^\infty_\eta}\leq&C|\varepsilon |^{-\f13}|\log c_i|^4\big(\|v^{(j)}_{st}\|_{L^\infty}+\|\partial_Y v^{(j)}_{st}\|_{L^\infty}\big),
\end{align*}
and 
\begin{align*}
&\|u_{os}^{(j+1)}\|_{L^\infty_\eta}+|\e|^\f13\|\pa_Yu_{os}^{(j+1)}\|_{L^\infty_\eta}+|\varepsilon|^\f23\|\partial_Y^2u_{os}^{(j+1)}\|_{L^\infty_\eta}\\
&\quad\leq C|\varepsilon |^{-\f23} |\log c_i|^4 (\|v^{(j)}_{st}\|_{L^\infty}+\|\partial_Y v^{(j)}_{st}\|_{L^\infty}),
\end{align*}
and
\begin{align*}
	\|v_{os}^{(j+1)}\|_{L^\infty_\eta}+|\e|^\f13\|\pa_Yv_{os}^{(j+1)}\|_{L^\infty_\eta}+|\varepsilon|^\f23\|\partial_Y^2v_{os}^{(j+1)}\|_{L^\infty_\eta} \leq&C  |\log c_i|^4\big(\|v^{(j)}_{st}\|_{L^\infty}+\|\partial_Y v^{(j)}_{st}\|_{L^\infty}\big),
	\end{align*}
and
\begin{align*}
&\|\dv_\al(u_{os}^{(j+1)},v_{os}^{(j+1)})\|_{L^\infty_\eta}+\|\pa_Y\dv_\al(u_{os}^{(j+1)},v_{os}^{(j+1)})\|_{L^\infty_\eta}\leq C|\log c_i|^4\big(\|v^{(j)}_{st}\|_{L^\infty}+\|\partial_Y v^{(j)}_{st}\|_{L^\infty}\big),\\
&\|\pa_Y^2\dv_\al(u_{os}^{(j+1)},v_{os}^{(j+1)})\|_{L^\infty_\eta}\leq C|\log c_i|^4\big(\|v^{(j)}_{st}\|_{L^\infty}+\|\partial_Y v^{(j)}_{st}\|_{L^\infty}\big).
\end{align*}
As a result, we arrive at 
\begin{align*}
	&\|\rho_{os}^{(j+1)}\|_{\mathcal Z}\leq C|\varepsilon |^{-\f23} |\log c_i|^4\big(\|v^{(j)}_{st}\|_{L^\infty}+\|\partial_Y v^{(j)}_{st}\|_{L^\infty}\big),\\
	&\|u_{os}^{(j+1)}\|_{\tilde{\mathcal X}}\leq C|\varepsilon |^{-1} |\log c_i|^4\big(\|v^{(j)}_{st}\|_{L^\infty}+\|\partial_Y v^{(j)}_{st}\|_{L^\infty}\big),\\
	&\|v_{os}^{(j+1)}\|_{\tilde{\mathcal X}}\leq C  |\varepsilon|^{-\f13} |\log c_i|^4\big(\|v^{(j)}_{st}\|_{L^\infty}+\|\partial_Y v^{(j)}_{st}\|_{L^\infty}\big).
\end{align*}
This shows that
\begin{align*}
	|\varepsilon|^\f23\|\rho^{(j+1)}_{os}\|_{\mathcal Z}+|\varepsilon|\|u_{os}^{(j+1)}\|_{\tilde{\mathcal X}}+|\varepsilon|^\f13\|v_{os}^{(j+1)}\|_{\tilde{\mathcal X}}\leq C|\log c_i|^4\|v_{st}^{(j)}\|_{\tilde{\mathcal X}},
\end{align*}
and 
\begin{align*}
	|\varepsilon|^\f23\|u_{os}^{(j+1)}\|_{L^\infty}+\|v_{os}^{(j+1)}\|_{L^\infty}\leq C|\log c_i|^4\big(\|v^{(j)}_{st}\|_{L^\infty}+\|\partial_Y v^{(j)}_{st}\|_{L^\infty}\big).
\end{align*}

Now we turn to show the estimates about $(\rho_{st}^{(j+1)},u_{st}^{(j+1)},v_{st}^{j+1})$. By Proposition \ref{pro: stokes}, we obtain
\begin{align*}
		&|\alpha|^2\|\rho_{st}^{(j+1)}\|_{L^\infty_{\f{\beta_r}{2}}}+|\alpha|\|\partial_Y \rho_{st}^{(j+1)}\|_{L^\infty_{\f{\beta_r}{2}}}+\|\partial_Y^2\rho_{st}^{(j+1)}\|_{L^\infty_{\f{\beta_r}{2}}}\leq C|\alpha|\|q_1\|_{L^\infty_{\beta_r}}+C\|\partial_Y q_2\|_{L^\infty_{\beta_r}},\\
		&|\varepsilon|\|\partial_Y^2u_{st}^{(j+1)}\|_{L^\infty_{\f{\beta_r}{2}}}+|\varepsilon|^\f23\|\partial_Yu_{st}^{(j+1)}\|_{L^\infty_{\f{\beta_r}{2}}}+|\varepsilon|^\f13 \|u_{st}^{(j+1)}\|_{L^\infty_{\f{\beta_r}{2}}}
	\leq C|\alpha|^{-1}\|q_1\|_{L^\infty_{\beta_r}}+C|\alpha|^{-2}\|\partial_Y q_2\|_{L^\infty_{\beta_r}},
\end{align*}
and
\begin{align*}
	&|\varepsilon|^\f23\|\partial_Y^2v_{st}^{(j+1)}\|_{L^\infty_{\f{\beta_r}{2}}}+|\varepsilon|^\f13\|\partial_Yv_{st}^{(j+1)}\|_{L^\infty_{\f{\beta_r}{2}}}+|\varepsilon|^\f13\|v_{st}^{(j+1)}\|_{L^\infty_{\f{\beta_r}{2}}}\leq C\|(q_1,q_2)\|_{L^\infty_{\beta_r}}+C|\alpha|^{-1}\|\partial_Y q_2\|_{L^\infty_{\beta_r}},
\end{align*}
where 
\begin{align*}
	q_1=&-\e(\pa_Y^2-\al^2)\dv_\al(u^{(j+1)}_{os}, v^{(j+1)}_{os})-\mathrm i\la \sqrt{\nu} \al\dv_\al(u^{(j+1)}_{os}, v^{(j+1)}_{os})+\sqrt{\nu}\pa_Y^2 U_s\rho^{(j+1)}_{os},\\
q_2=&-\la \sqrt{\nu} \pa_Y\dv_\al(u^{(j+1)}_{os}, v^{(j+1)}_{os}).
\end{align*}
On the other hand, based on the estimates about $\dv_{\alpha}(u_{os}^{(j+1)},v_{os}^{(j+1)})$, we obtain 
\begin{align*}
	\|q_1\|_{L^\infty_{\beta_r}}\leq& C|\varepsilon|\|\pa_Y^2\dv_\al(u_{os}^{(j+1)},v_{os}^{(j+1)})\|_{L^\infty_\eta}
	+C|\varepsilon||\alpha|^2\|\dv_\al(u_{os}^{(j+1)},v_{os}^{(j+1)})\|_{L^\infty_\eta}\\
	&+C|\varepsilon||\alpha|\|\rho^{(j+1)}_{os}\|_{L^\infty_\eta}\\
	\leq&C|\varepsilon| |\log c_i|^4\big(\|v^{(j)}_{st}\|_{L^\infty}+\|\partial_Y v^{(j)}_{st}\|_{L^\infty}\big),\\
	\|q_2\|_{L^\infty_{\beta_r}}\leq&C|\varepsilon||\alpha|\|\partial_Y\dv_\al(u_{os}^{(j+1)},v_{os}^{(j+1)})\|_{L^\infty_\eta}\leq C|\varepsilon||\alpha||\log c_i|^4\big(\|v^{(j)}_{st}\|_{L^\infty}+\|\partial_Y v^{(j)}_{st}\|_{L^\infty}\big),
	\end{align*}
and 
\begin{align*}
	\|\partial_Y q_2\|_{L^\infty_{\beta_r}}\leq&C|\varepsilon||\alpha|\|\partial_Y^2\dv_\al(u_{os}^{(j+1)},v_{os}^{(j+1)})\|_{L^\infty_\eta}\\
	\leq&C|\varepsilon||\alpha| |\log c_i|^4\big(\|v^{(j)}_{st}\|_{L^\infty}+\|\partial_Y v^{(j)}_{st}\|_{L^\infty}\big).
\end{align*}
Therefore, we obtain 
\begin{align*}
	\|\rho_{st}^{(j+1)}\|_{L^\infty_{\f{\beta_r}{2}}}+\|\partial_Y \rho_{st}^{(j+1)}\|_{L^\infty_{\f{\beta_r}{2}}}\leq C|\varepsilon|^\f23|\log c_i|^4\big(\|v^{(j)}_{st}\|_{L^\infty}+\|\partial_Y v^{(j)}_{st}\|_{L^\infty}\big),
\end{align*}
and 
\begin{align*}
	|\varepsilon|^\f23\|\partial_Y^2u_{st}^{(j+1)}\|_{L^\infty_{\f{\beta_r}{2}}}+&|\varepsilon|^\f13\|\partial_Yu_{st}^{(j+1)}\|_{L^\infty_{\f{\beta_r}{2}}}+ \|u_{st}^{(j+1)}\|_{L^\infty_{\f{\beta_r}{2}}}\\
	\leq&C|\varepsilon|^\f13|\log c_i|^4\big(\|v^{(j)}_{st}\|_{L^\infty}+\|\partial_Y v^{(j)}_{st}\|_{L^\infty}\big),\\
\end{align*}
and
\begin{align*}
	|\varepsilon|^\f23\|\partial_Y^2v_{st}^{(j+1)}\|_{L^\infty_{\f{\beta_r}{2}}}&+|\varepsilon|^\f13\|\partial_Yv_{st}^{(j+1)}\|_{L^\infty_{\f{\beta_r}{2}}}+|\varepsilon|^\f13\|v_{st}^{(j+1)}\|_{L^\infty_{\f{\beta_r}{2}}}\\
	\leq& C|\varepsilon| |\log c_i|^4\big(\|v^{(j)}_{st}\|_{L^\infty}+\|\partial_Y v^{(j)}_{st}\|_{L^\infty}\big).
\end{align*}
Thus, we arrive at
\begin{align*}
	\|\rho^{(j+1)}_{st}\|_{\mathcal Z}+|\varepsilon|^\f13\|u_{st}^{(j+1)}\|_{\tilde{\mathcal X}}+|\varepsilon|^{-\f13}\|v_{st}^{(j+1)}\|_{\tilde{\mathcal X}}\leq C|\log c_i|^4\|v_{st}^{(j)}\|_{\tilde{\mathcal X}}.
\end{align*}

\end{proof}

Now we construct the slow mode of the linearized CNS system \eqref{eq:LCNS-Y1}.

\begin{proposition}\label{thm: slow-cns}
	Let $(\rho_s,u_s,v_s)$ be the slow solution to the homogenous quasi-incompressible system constructed in Proposition \ref{pro: (rho, u,v)-OS-homo}.  Then there exists a slow solution $(\rho^s,u^s,v^s)\in\mathcal Z\times\tilde{\mathcal X}\times\tilde{\mathcal X}$ to the system  \eqref{eq:LCNS-Y1} such that 
\begin{align*}
|\e|^{\f13}\|\rho^s-\rho_s\|_{\mathcal Z}+|\e|^{\f23}\|u^s-u_s\|_{\tilde{\mathcal X}}+\|v^s-v_s\|_{\tilde{\mathcal X}}\leq C |\varepsilon||\log c_i|^4.
\end{align*}
Moreover, we have	
	\begin{align*}
		u^s(0)
=&(1-m^2)U_s'(0)+\mathcal{O}(|\al| |\log c_i|^4 ),\\
v^s(0)
=&-\mathrm i \al\Big(-(1-m^2)c+\f{\beta}{U_s'(Y_c)}+\mathcal{O}((|\al| |\e|^\f13+|\al|^2+|c|^2) |\log c_i|^4)\Big).
\end{align*}
	
\end{proposition} 

\begin{proof}
We define $(\rho^{(1)},u^{(1)},v^{(1)})=(\rho^{(1)}_{os}+\rho^{(1)}_{st},u^{(1)}_{os}+u^{(1)}_{st},v^{(1)}_{os}+u^{(1)}_{st})$, where $(\rho_{os}^{(1)},u_{os}^{(1)},v_{os}^{(1)})=(\rho_s,u_s,v_s)$ and $(\rho_{st}^{(1)},u_{st}^{(1)},v_{st}^{(1)})$ is the solution to the system \eqref{eq: Stokes} constructed in Proposition \ref{pro: stokes} with the source term $(0,q^{(1)}_1,q^{(1)}_2)$:
\begin{align}
q^{(1)}_1=&\sqrt{\nu}(\pa_Y^2-\al^2)((U_s-c)\rho^{(1)}_{os})-\mathrm i\la \sqrt{\nu} \al\dv_\al(u^{(1)}_{os}, v^{(1)}_{os})+\sqrt{\nu}\pa_Y^2 U_s\rho^{(1)}_{os}\label{eq: q^(1)_1}\\
\nonumber
=&-\e(\pa_Y^2-\al^2)\dv_\al(u^{(1)}_{os}, v^{(1)}_{os})-\mathrm i\la \sqrt{\nu} \al\dv_\al(u^{(1)}_{os}, v^{(1)}_{os})+\sqrt{\nu}\pa_Y^2 U_s\rho^{(1)}_{os},\\
q^{(1)}_2=&-\la \sqrt{\nu} \pa_Y\dv_\al(u^{(1)}_{os}, v^{(1)}_{os}).\label{eq: q^(1)_2}
\end{align}

By Proposition \ref{pro: (rho, u,v)-OS-homo}, we obtain 
\begin{align*}
		\|q_1^{(1)}\|_{L^\infty_{\beta_r}}\leq& C|\varepsilon|\|\pa_Y^2\dv_\al(u_{os}^{(1)},v_{os}^{(1)})\|_{L^\infty_{\beta_r}}
	+C|\varepsilon||\alpha|^2\|\dv_\al(u_{os}^{(1)},v_{os}^{(1)})\|_{L^\infty_{\beta_r}}\\
	&+C|\varepsilon||\alpha|\|\rho^{(1)}_{os}\|_{L^\infty_{\beta_r}}\\
	\leq&C|\varepsilon||\alpha|^2,\\
	\|q_2^{(1)}\|_{L^\infty_{\beta_r}}+\|\partial_Y q_2^{(1)}\|_{L^\infty_{\beta_r}}\leq& C|\varepsilon||\alpha|\|\pa_Y^2\dv_\al(u_{os}^{(1)},v_{os}^{(1)})\|_{L^\infty_{\beta_r}}+C|\varepsilon||\alpha|\|\pa_Y\dv_\al(u_{os}^{(1)},v_{os}^{(1)})\|_{L^\infty_{\beta_r}}\\
	\leq&C|\varepsilon||\alpha|^3,
\end{align*}

which along with Proposition \ref{pro: stokes} implies 
\begin{align*}
		&\|\rho_{st}^{(1)}\|_{L^\infty_{\f{\beta_r}{2}}}+\|\partial_Y \rho_{st}^{(1)}\|_{L^\infty_{\f{\beta_r}{2}}}\leq C|\varepsilon||\alpha|\leq C|\varepsilon|^\f43 ,\\
		&|\varepsilon|^\f23\|\partial_Y^2u^{(1)}_{st}\|_{L^\infty_{\f{\beta_r}{2}}}+|\varepsilon|^\f13\|\partial_Yu^{(1)}_{st}\|_{L^\infty_{\f{\beta_r}{2}}}+\|u^{(1)}_{st}\|_{L^\infty_{\f{\beta_r}{2}}}
	\leq C|\varepsilon| ,
\end{align*}
and
\begin{align*}
	&|\varepsilon|^\f13\|\partial_Y^2v^{(1)}_{st}\|_{L^\infty_{\f{\beta_r}{2}}}+\|\partial_Yv^{(1)}_{st}\|_{L^\infty_{\f{\beta_r}{2}}}+\|v^{(1)}_{st}\|_{L^\infty_{\f{\beta_r}{2}}}\leq C|\varepsilon|^\f43 .
\end{align*}
In particular, we have 
\begin{align*}
	\|\rho_{st}^{(1)}\|_{\mathcal Z}+|\varepsilon|^\f23\|u_{st}^{(1)}\|_{\tilde{\mathcal X}}+\|v_{st}^{(1)}\|_{\tilde{\mathcal X}}\leq C|\varepsilon|^\f43,
\end{align*}
and 
\begin{align*}
	|u_{st}^{(1)}(0)|\leq C|\varepsilon|, \quad|v_{st}^{(1)}(0)|\leq C|\varepsilon|^\f43.
\end{align*}
Hence by using Lemma \ref{prop:main-iteration} repeatedly, we infer that for any $j\geq 2$,
\begin{align*}
	&\|v_{st}^{(j)}\|_{\tilde{\mathcal X}}\leq C|\varepsilon|^\f13|\log c_i|^4\|v_{st}^{(j-1)}\|_{\tilde{\mathcal X}}\leq C(|\varepsilon|^\f13|\log c_i|^4)^{j-2}\|v_{st}^{(1)}\|_{\tilde{\mathcal X}},\\
	&|\varepsilon|^\f23\|\rho^{(j)}_{os}\|_{\mathcal Z}+|\varepsilon|\|u_{os}^{(j)}\|_{\tilde{\mathcal X}}+|\varepsilon|^\f13\|v_{os}^{(j)}\|_{\tilde{\mathcal X}}\leq C|\log c_i|^4\|v_{st}^{(j-1)}\|_{\tilde{\mathcal X}}\leq C(|\varepsilon|^\f13|\log c_i|^4)^{j-2}\|v_{st}^{(1)}\|_{\tilde{\mathcal X}}|\log c_i|^4,
\end{align*}
and 
\begin{align*}
	\|\rho^{(j)}_{st}\|_{\mathcal Z}+|\varepsilon|^\f13\|u_{st}^{(j)}\|_{\tilde{\mathcal X}}\leq C|\log c_i|^4\|v_{st}^{(j-1)}\|_{\tilde{\mathcal X}}\leq C(|\varepsilon|^\f13|\log c_i|^4)^{j-2}\|v_{st}^{(1)}\|_{\tilde{\mathcal X}}|\log c_i|^4.
\end{align*}

Therefore, for $j\geq 2$, $(\rho^{(j)},u^{(j)},v^{(j)})=(\rho^{(j)}_{os}+\rho^{(j)}_{st},u^{(j)}_{os}+u^{(j)}_{st},v^{(j)}_{os}+u^{(j)}_{st})$ satisfies 
\begin{align*}
	&\|\rho^{(j)}\|_{\mathcal Z}\leq C|\e|^{-\f23}(|\varepsilon|^\f13|\log c_i|^4)^{j-2}\|v_{st}^{(1)}\|_{\tilde{\mathcal X}}|\log c_i|^4\leq C |\varepsilon|^\f23|\log c_i|^4\left(|\varepsilon|^\f13|\log c_i|^4\right)^{j-2},\\
	&\|u^{(j)}\|_{\tilde{\mathcal X}}\leq C|\e|^{-1}(|\varepsilon|^\f13|\log c_i|^4)^{j-2}\|v_{st}^{(1)}\|_{\tilde{\mathcal X}}|\log c_i|^4\leq C|\varepsilon|^\f13 |\log c_i|^4\left(|\varepsilon|^\f13|\log c_i|^4\right)^{j-2},\\
	&\|v^{(j)}\|_{\tilde{\mathcal X}}\leq C|\e|^{-\f13}(|\varepsilon|^\f13|\log c_i|^4)^{j-2}\|v_{st}^{(1)}\|_{\tilde{\mathcal X}}|\log c_i|^4\leq C |\varepsilon||\log c_i|^4\left(|\varepsilon|^\f13|\log c_i|^4\right)^{j-2},
\end{align*}
and 
\begin{align*}
	|u^{(j)}(0)|\leq C\left(|\varepsilon|^\f13|\log c_i|^4\right)^{j-2}|\varepsilon|^\f13|\log c_i|^4,\quad |v^{(j)}(0)|\leq C\left(|\varepsilon|^\f13|\log c_i|^4\right)^{j-2}|\varepsilon||\log c_i|^4.
\end{align*}
Thus, we may define
\begin{align*}
     	\rho^s=\sum_{j=1}^{+\infty}\rho^{(j)},\quad u^s=\sum_{j=1}^{+\infty}u^{(j)},\quad v^s=\sum_{j=1}^{+\infty}v^{(j)},
     \end{align*}
 which satisfies the desired properties by  Proposition \ref{pro: (rho, u,v)-OS-homo} and the above estimates.
 \end{proof}
 
Next we construct the fast mode of the linearized CNS system \eqref{eq:LCNS-Y1}.
 
\begin{proposition}\label{thm:fast-cns}
		Let $(\rho_f,u_f,v_f)$ be the fast solution to the homogenous quasi-incompressible system constructed in Proposition \ref{pro: (rho, u,v)-OS-homo}.  Then there exists a fast solution $(\rho^f,u^f,v^f)\in\mathcal Z\times\tilde{\mathcal X}\times\tilde{\mathcal X}$  to the system  \eqref{eq:LCNS-Y1} such that 
		\begin{align*}
|\e|^{\f13}\|\rho^f-\rho_f\|_{\mathcal Z}+|\e|^{\f23}\|u^f-u_f\|_{\tilde{\mathcal X}}+\|v^f-v_f\|_{\tilde{\mathcal X}}\leq C |\varepsilon|^{\f43}|\log c_i|^5.
\end{align*}
Moreover, we have
		\begin{align*}
			u^f(0)=&\pa_Y \psi_a(0)+\mathcal{O}(|\e|^{\f23}|\log c_i|^5),\\
v^f(0)=&-\mathrm i \al\big(\psi_a(0)+\mathcal{O}(|\e|^{\f23}(|\e|^\f13+|c|+|\al|)|\log c_i|^5)\big).
		\end{align*}
\end{proposition}  

\begin{proof}
Since the proof is similar to Proposition \ref{thm: slow-cns}, we only present a sketch of the proof. 

We define $(\rho^{(1)},u^{(1)},v^{(1)})=(\rho^{(1)}_{os}+\rho^{(1)}_{st},u^{(1)}_{os}+u^{(1)}_{st},v^{(1)}_{os}+u^{(1)}_{st})$, where $(\rho_{os}^{(1)},u_{os}^{(1)},v_{os}^{(1)})=(\rho_f,u_f,v_f)$ and $(\rho_{st}^{(1)},u_{st}^{(1)},v_{st}^{(1)})$ is the solution to the system \eqref{eq: Stokes} constructed in Proposition \ref{pro: stokes} with the source term $(0,q^{(1)}_1,q^{(1)}_2)$. The definition of $(q^{(1)}_1,q^{(1)}_2)$ is given in \eqref{eq: q^(1)_1} and \eqref{eq: q^(1)_2}.
By Proposition \ref{pro: (rho, u,v)-OS-homo}, we obtain 
\begin{align*}
		&\|q_1^{(1)}\|_{L^\infty_{\beta_r}}\leq C|\varepsilon|^\f53|\alpha||\log c_i|,\\
	&\|q_2^{(1)}\|_{L^\infty_{\beta_r}}+\|\partial_Y q_2^{(1)}\|_{L^\infty_{\beta_r}}\leq C|\varepsilon|^\f53|\alpha|^2|\log c_i|,
\end{align*}
which along with Proposition \ref{pro: stokes} implies 
\begin{align*}
		&\|\rho_{st}^{(1)}\|_{L^\infty_{\f{\beta_r}{2}}}+\|\partial_Y \rho_{st}^{(1)}\|_{L^\infty_{\f{\beta_r}{2}}}\leq C|\varepsilon|^\f53|\log c_i|,\\
		&|\varepsilon|^\f23\|\partial_Y^2u^{(1)}_{st}\|_{L^\infty_{\f{\beta_r}{2}}}+|\varepsilon|^\f13\|\partial_Yu^{(1)}_{st}\|_{L^\infty_{\f{\beta_r}{2}}}+\|u^{(1)}_{st}\|_{L^\infty_{\f{\beta_r}{2}}}
	\leq C|\varepsilon|^\f43|\log c_i|,\\
	&|\varepsilon|^\f13\|\partial_Y^2v^{(1)}_{st}\|_{L^\infty_{\f{\beta_r}{2}}}+\|\partial_Yv^{(1)}_{st}\|_{L^\infty_{\f{\beta_r}{2}}}+\|v^{(1)}_{st}\|_{L^\infty_{\f{\beta_r}{2}}}\leq C|\varepsilon|^\f53|\log c_i| .
\end{align*}
In particular, we have 
\begin{align*}
	\|\rho_{st}^{(1)}\|_{\mathcal Z}+|\varepsilon|^\f23\|u_{st}^{(1)}\|_{\tilde{\mathcal X}}+\|v_{st}^{(1)}\|_{\tilde{\mathcal X}}\leq C|\varepsilon|^\f53|\log c_i|,
\end{align*}
and 
\begin{align*}
	|u_{st}^{(1)}(0)|\leq C|\varepsilon|^{\f43}|\log c_i|, \quad|v_{st}^{(1)}(0)|\leq C|\varepsilon|^\f53|\log c_i|.
\end{align*}

Recall that for $j\ge 2$, $(\rho^{(j)},u^{(j)},v^{(j)})=(\rho^{(j)}_{os}+\rho^{(j)}_{st},u^{(j)}_{os}+u^{(j)}_{st},v^{(j)}_{os}+u^{(j)}_{st})$. By a similar calculation as in the proof of  Proposition \ref{thm: slow-cns},  it holds that for any $j\geq 2$,
\begin{align*}
	&\|\rho^{(j)}\|_{\mathcal Z}\leq C|\e|^{-\f23}(|\varepsilon|^\f13|\log c_i|^4)^{j-2}\|v_{st}^{(1)}\|_{\tilde{\mathcal X}}|\log c_i|^4\leq C |\varepsilon||\log c_i|^5\left(|\varepsilon|^\f13|\log c_i|^4\right)^{j-2},\\
	&\|u^{(j)}\|_{\tilde{\mathcal X}}\leq C|\e|^{-1}(|\varepsilon|^\f13|\log c_i|^4)^{j-2}\|v_{st}^{(1)}\|_{\tilde{\mathcal X}}|\log c_i|^4\leq C|\varepsilon|^\f23|\log c_i |^5\left(|\varepsilon|^\f13|\log c_i|^4\right)^{j-2},\\
	&\|v^{(j)}\|_{\tilde{\mathcal X}}\leq C|\e|^{-\f13}(|\varepsilon|^\f13|\log c_i|^4)^{j-2}\|v_{st}^{(1)}\|_{\tilde{\mathcal X}}|\log c_i|^4\leq C |\varepsilon|^{\f43}|\log c_i|^5\left(|\varepsilon|^\f13|\log c_i|^4\right)^{j-2},
\end{align*}
and 
\begin{align*}
	|u^{(j)}(0)|\leq C\left(|\varepsilon|^\f13|\log c_i|^4\right)^{j-2}|\varepsilon|^\f23|\log c_i|^5,\quad |v^{(j)}(0)|\leq C\left(|\varepsilon|^\f43|\log c_i|^4\right)^{j-2}|\varepsilon|^\f43|\log c_i|^5.
\end{align*}
Then we define
\begin{align*}
     	\rho^f=\sum_{j=1}^{+\infty}\rho^{(j)},\quad u^f=\sum_{j=1}^{+\infty}u^{(j)},\quad v^f=\sum_{j=1}^{+\infty}v^{(j)},
     \end{align*}
 which satisfies the desired properties by  Proposition \ref{pro: (rho, u,v)-OS-homo} and the above estimates.
 \end{proof}

\section{Dispersion relation and T-S waves}

In this section, we construct the solution $(\rho,u,v)$ to the homogeneous linearized compressible NS system:
\begin{align}
	\left\{
	\begin{aligned}
		&\mathrm i\alpha(U_s-c)\rho+\mathrm{div}_\alpha(u,v)=0,\\
		&\sqrt{\nu}(\partial_Y^2-\alpha^2)u+\mathrm i\alpha\lambda\sqrt{\nu}\mathrm{div}_\alpha(u,v)-\mathrm i\alpha(U_s-c)u\\
		&\qquad-(\mathrm i\alpha m^{-2}+\sqrt{\nu}\partial_Y^2 U_s)\rho-v\partial_Y U_s=0,\\
		&\sqrt{\nu}(\partial_Y^2-\alpha^2)v+\lambda\sqrt{\nu}\partial_Y\mathrm{div}_\alpha(u,v)-\mathrm i\alpha(U_s-c)v-m^{-2}\partial_Y\rho=0,
	\end{aligned}
	\right.
\end{align}
together with non-slip boundary condition
\begin{align*}
	u(0)=v(0)=\lim_{Y\to\infty}u(Y)=\lim_{Y\to\infty}v(Y)=0.
\end{align*}

In Proposition \ref{thm: slow-cns} and Proposition \ref{thm:fast-cns}, we have constructed two solutions $(\rho^s,u^s,v^s)$ and $(\rho^f,u^f,v^f)$ to the homogeneous linearized compressible NS system with the boundary condition at the infinity
\begin{align*}
	\lim_{Y\to\infty}u^s(Y)=\lim_{Y\to\infty}v^s(Y)=\lim_{Y\to\infty}u^f(Y)=\lim_{Y\to\infty}v^f(Y)=0.
\end{align*}
However, these two solutions do not match non-slip boundary condition. To this end, we find the solution with the form
\begin{align*}
	(\rho,u,v)=C_s(\rho^s,u^s,v^s)+C_f(\rho^f,u^f,v^f)
\end{align*}
with 
\begin{align*}
	C_su^s(0)+C_fu^f(0)=0\text{ and }C_sv^s(0)+C_fv^f(0)=0.
\end{align*}
The existence of $(C_s,C_f)$ is guaranteed by the following dispersion relation:
\begin{align}\label{eq:dr-vel}
	\frac{u^s(0)}{v^s(0)}=\frac{u^f(0)}{v^f(0)}.
\end{align}

First of all, to ensure the existence of solutions $(\rho^s,u^s,v^s)$ and $(\rho^f,u^f,v^f)$, we require $(c,\alpha)\in \mathbb H$, where we recall that
\begin{align*}
	\mathbb H=\mathbb H_1\cap\{(\alpha,c)\in\mathbb C^2:|\alpha|\sim|c|\sim|\varepsilon|^\f13, c_0|\varepsilon|^\f13\leq c_i\leq c_1|\varepsilon|^\f13 \text{ with }0<c_0<c_1\ll1\}
\end{align*}
with 
\begin{align*}
	\mathbb H_1=\big\{(\alpha,c)\in\mathbb C^2:(|\alpha|+|c|)|\log c_i|\ll1,\alpha_r\geq C_0|\alpha_i|,c_i\geq c_0|\varepsilon|^\f13\text{ and }0<c_0\ll1\ll C_0\big\}.
\end{align*}

According to Proposition \ref{thm: slow-cns} and Proposition \ref{thm:fast-cns},  we find that \eqref{eq:dr-vel} is equivalent to 
\begin{align}\label{eq:dr-beta}
	c-\frac{\beta}{U_s'(0)(1-m^2)} +\frac{U_s'(0)\tilde{\mathcal A}(2,0)}{\tilde{\mathcal A}(1,0)}+\mathcal R_d(\alpha,c)=0.
\end{align}
Notice that 
\begin{align*}
	A_{\infty}^\f12=(1-m^2(1-c)^2)^\f12=(1-m^2)^\f12+\mathcal O(|c|),
\end{align*}
which along with  the definition of $\beta=\alpha A^{\f12}_{\infty}$ shows that \eqref{eq:dr-beta} is equivalent to 
\begin{align}\label{eq:dr-alpha}
	c-\frac{\alpha}{U_s'(0)(1-m^2)^\f12} +\frac{U_s'(0)\tilde{\mathcal A}(2,0)}{\tilde{\mathcal A}(1,0)}+\mathcal R(\alpha,c)=0,
\end{align}
where $\mathcal R(\alpha,c)$ is smooth on $c$ and analytic on $\alpha$. Moreover, it holds that for any $(\alpha,c)\in\mathbb H$,
\begin{align}
	\begin{split}
		&|\mathcal R(\alpha,c)|\leq C|\varepsilon|^\f23|\log|\varepsilon||,\quad|\partial_{\alpha_r}\mathcal R(\alpha,c)|+|\partial_{\alpha_i}\mathcal R(\alpha,c)|\leq C,\\
		&|\partial_{c_r}\mathcal R(\alpha,c)|+|\partial_{c_i}\mathcal R(\alpha,c)|\leq C|\varepsilon|^\f13|\log|\varepsilon||.
	\end{split}
\end{align}

\begin{proposition}
	Let $m<1$ and $0<\nu\leq \nu(m)\ll1$. Suppose that $(c,\alpha)\in\mathbb H$.  There exist $1\ll A_0<B_0$ such that for any $\alpha\in\mathbb G=\{\alpha\in \mathbb C:\alpha_r=A\nu^\f18 \text{ with }A\in(A_0,B_0), |\alpha_i|\leq \gamma\alpha_r \text{ with }\gamma\ll1 \}$, the solution $c(\alpha)$ to \eqref{eq:dr-alpha} satisfies 
	\begin{align*}
	c_r=U'_s(0)^{-1}(1-m^2)^{-\f12}\al_r+o(\al),\quad	0<c_i-U_s'(0)^{-1}(1-m^2)^{-\f12}\alpha_i\sim A^{-1}\nu^\f18.
	\end{align*}
	In particular, for any $\alpha_r=A\nu^\f18$ with $A\in(A_0,B_0)$, there holds
	\begin{itemize}
		\item there exists  $\alpha^0=\alpha_r+\mathrm i\alpha^0_i\in\mathbb G$ with $\alpha_i^0<0$ such that 
		      \begin{align*}
		      	\alpha^0c(\alpha^0)\in\mathbb R\text{ and }|\alpha_i^0|\sim\nu^\f18.
		      \end{align*}
		\item there exists  $c(\alpha_r)$ such that $c_i(\alpha_r)>0$ and $|c_i(\alpha_r)|\sim\nu^\f18$.   
	\end{itemize}
\end{proposition}
\begin{proof}
Since $|\alpha|\sim A\nu^\f18$ and $|\varepsilon|=\nu^\f12|\alpha|^{-1}$, we know that $|\varepsilon|\sim A^{-1}\nu^\f38$.
We first rewrite \eqref{eq:dr-alpha} as follows
\begin{align*}
	c-\frac{\alpha}{U'_s(0)(1-m^2)^\f12}=-\frac{U_s'(0)\tilde{\mathcal A}(2,0)}{\tilde{\mathcal A}(1,0)}-\mathcal R(\alpha,c)=C_{Ai,r}+\mathrm iC_{Ai,i}.
\end{align*}
By Lemma \ref{lem:airy-0}, we know that 
\begin{align*}
	0<C_{Ai,r}\sim |\varepsilon|^{\f13}|\varepsilon^{-\f13}c_r|^{-\f12}=A^{-\f12}\nu^\f{3}{16}c_r^{-\f12}
\end{align*}
and
\begin{align*}
	0<C_{Ai,i}\sim |\varepsilon|^{\f13}|\varepsilon^{-\f13}c_r|^{-\f12}=A^{-\f12}\nu^\f{3}{16}c_r^{-\f12}.
\end{align*}
On the other hand, for large enough $A_0$ and $B_0$, we have that for any $\alpha_r\in(A_0\nu^\f18,B_0\nu^\f18)$,
\begin{align*}
	c_r=U'_s(0)^{-1}(1-m^2)^{-\f12}A\nu^\f18+C_{Ai,r}\sim A\nu^{\f18}+A^{-1}\nu^\f14c_r^{-\f12},
\end{align*}
which implies that $c_r\sim A\nu^\f18.$ Moreover, we obtain 
\begin{align*}
	0<c_i-U_s'(0)^{-1}(1-m^2)^{-\f12}\alpha_i\sim A^{-1}\nu^\f18.
\end{align*}
In particular, for the case of $\alpha_i=0$, we have 
\begin{align*}
	0<c_i\sim A^{-1}\nu^{\f18}.
\end{align*}
Now we suppose that $\alpha^0=\alpha_r+\mathrm i\alpha_i$ with $\alpha_i=-\gamma_0\alpha_r$. Then 
    \begin{align*}
    	c_i=-\gamma_0U'_s(0)^{-1}(1-m^2)^{-\f12}A\nu^\f18+C_{Ai,i}.
    \end{align*}
    On the other hand, we notice that $c(\alpha)\alpha\in\mathbb R$ is equivalent to 
    \begin{align*}
    	\frac{c_i}{c_r}=-\frac{\alpha_i}{\alpha_r}.
    \end{align*}
	Then we have 
	\begin{align*}
		\gamma_0=\frac{-\gamma_0U'_s(0)^{-1}(1-m^2)^{-\f12}A\nu^\f18+C_{Ai,i}}{U'_s(0)^{-1}(1-m^2)^{-\f12}A\nu^\f18+C_{Ai,r}},
	\end{align*}
	which implies 
	\begin{align*}
		\gamma_0=\frac{C_{Ai,i}}{2U'_s(0)^{-1}(1-m^2)^\f12A\nu^\f18+C_{Ai,r}}>0.
	\end{align*}
	Moreover, we notice that 
	\begin{align*}
		|\gamma_0|\sim A^{-2}\ll1.
	\end{align*}
	By taking $\alpha^0=A\nu^\f18-\mathrm i\gamma_0A\nu^\f18$, we finish the proof.
	\end{proof}

\begin{proposition}\label{prop:diss-existence}
	Let $m<1$ and $0<\nu\leq \nu(m)\ll1$. Suppose that $(c,\alpha)\in\mathbb H$.  There exist $1\ll A_0<B_0$ such that for any $\alpha\in\mathbb G$, there exists a solution $c(\alpha)$ to \eqref{eq:dr-alpha} with $(c(\alpha),\alpha)\in\mathbb H$. Moreover, $c(\alpha)$ depends on $\alpha$ continuously .
\end{proposition}

\begin{proof}
 We introduce 
 $$
 \mathbb F(\alpha_r,\alpha_i;c_r,c_i)=\big(F_r(\alpha_r,\al_i;c_r,c_i),F_i(\alpha_r,\alpha_i ;c_r,c_i)\big),
$$
  where
	\begin{align*}
		&F_r(\alpha_r, \al_i ;c_r,c_i)=\mathrm{Re}\left(c-U_s'(0)^{-1}(1-m^2)^{-\f12}\alpha+\frac{U_s'(0)\tilde{\mathcal A}(2,0)}{\tilde{\mathcal A}(1,0)}+\mathcal R(\alpha,c)\right)\\
		&F_i(\alpha_r,\al_i;c_r,c_i)=\mathrm{Im}\left(c-U_s'(0)^{-1}(1-m^2)^{-\f12}\alpha+\frac{U_s'(0)\tilde{\mathcal A}(2,0)}{\tilde{\mathcal A}(1,0)}+\mathcal R(\alpha,c)\right).
	\end{align*}
	
We need to show the Jacobian determinant of $F_r$ and $F_i$ does not equal zero for $(\alpha,c)\in \mathbb{ H}_2=\{(\alpha,c)\in \mathbb H:\alpha\in \mathbb G\}$. For this purpose, by Lemma \ref{lem:airy-0} and the definition of $\tilde{\mathcal A}(j,Y)$, we have that for $j=1,2$,
	\begin{align*}
		\frac{U_s'(0)\tilde{\mathcal A}(2,0)}{\tilde{\mathcal A}(1,0)}=\kappa^{-1}U_s'(0)\frac{\mathcal A(2,\kappa\eta^{La}(0))}{\mathcal A(1,\kappa\eta^{La}(0))}+\mathcal O(|c|^2).
	\end{align*}
	By our definition of $\eta^{La}(Y)$, we know that 
	\begin{align*}
		\eta^{La}(0;c_r,c_i)=-U_s'(0)^{-1}(c_r+\mathrm ic_i)+\mathcal O(|c|^2),
	\end{align*}
	and smoothly depends on $c_r$ and $c_i$. Therefore, we obtain 
	\begin{align}\label{eq:d-cr-ci}
		\begin{split}
				&\partial_{c_r}\left(\frac{U_s'(0)\tilde{\mathcal A}(2,0)}{\tilde{\mathcal A}(1,0)}\right)=-1+\frac{Ai(e^{\mathrm i(\frac{\pi}{6}-\theta_0)}\kappa\eta^{La}(0))\mathcal A(2,\kappa\eta^{La}(0))}{\mathcal A(1,\kappa\eta^{La}(0))^2}+\mathcal O(|c|),\\
		&\partial_{c_i }\left(\frac{U_s'(0)\tilde{\mathcal A}(2,0)}{\tilde{\mathcal A}(1,0)}\right)=-\mathrm i +\mathrm i\frac{Ai(e^{\mathrm i(\frac{\pi}{6}-\theta_0)}\kappa\eta^{La}(0))\mathcal A(2,\kappa\eta^{La}(0))}{\mathcal A(1,\kappa\eta^{La}(0))^2}+\mathcal O(|c|).
		\end{split}
	\end{align}
	Thus, we have 
	\begin{align*}
		&\partial_{c_r}F_r(\alpha_r,\al_i;c_r,c_i)=\mathrm{Re}\left(\frac{Ai(e^{\mathrm i(\frac{\pi}{6}-\theta_0)}\kappa\eta^{La}(0))\mathcal A(2,\kappa\eta^{La}(0))}{\mathcal A(1,\kappa\eta^{La}(0))^2}\right)+\mathcal O(|c|),\\
		&\partial_{c_r}F_i(\alpha_r,\al_i;c_r,c_i)=\mathrm{Im}\left(\frac{Ai(e^{\mathrm i(\frac{\pi}{6}-\theta_0)}\kappa\eta^{La}(0))\mathcal A(2,\kappa\eta^{La}(0))}{\mathcal A(1,\kappa\eta^{La}(0))^2}\right)+\mathcal O(|c|),\\
		&\partial_{c_r}F_i(\alpha_r,\al_i;c_r,c_i)=-\mathrm{Im}\left(\frac{Ai(e^{\mathrm i(\frac{\pi}{6}-\theta_0)}\kappa\eta^{La}(0))\mathcal A(2,\kappa\eta^{La}(0))}{\mathcal A(1,\kappa\eta^{La}(0))^2}\right)+\mathcal O(|c|),\\
		&\partial_{c_i}F(\alpha_r,\al_i;c_r,c_i)=\mathrm{Re}\left(\frac{Ai(e^{\mathrm i(\frac{\pi}{6}-\theta_0)}\kappa\eta^{La}(0))\mathcal A(2,\kappa\eta^{La}(0))}{\mathcal A(1,\kappa\eta^{La}(0))^2}\right)+\mathcal O(|c|).
	\end{align*}
Hence, the Jacobian determinant $J(F_r,F_i)(\alpha;c_r,c_i)$ of $F_r$ and $F_i$ satisfies 
	\begin{align*}
		J(F_r,F_i)(\alpha_r,\al_i;c_r,c_i)=&\left|\frac{Ai(e^{\mathrm i(\frac{\pi}{6}-\theta_0)}\kappa\eta^{La}(0))\mathcal A(2,\kappa\eta^{La}(0))}{\mathcal A(1,\kappa\eta^{La}(0))^2}\right|^2+\mathcal O(|c|).
	\end{align*}
And by Lemma \ref{lem:Airy-p1}, we know that for any $(\alpha,c)\in \mathbb H_2$
	\begin{align*}
		\left|\frac{Ai(e^{\mathrm i(\frac{\pi}{6}-\theta_0)}\kappa\eta^{La}(0))\mathcal A(2,\kappa\eta^{La}(0))}{\mathcal A(1,\kappa\eta^{La}(0))^2}\right|\sim1.
	\end{align*}
	Therefore, for $\nu\ll1$, we obtain that for any $(\alpha,c)\in\mathbb  H_2$,
    \begin{align}\label{eq:F-c-Jaco}
    	\begin{split}
    			J(F_r,F_i)(\alpha_r,\al_i;c_r,c_i)=&\left|\frac{Ai(e^{\mathrm i(\frac{\pi}{6}-\theta_0)}\kappa\eta^{La}(0))\mathcal A(2,\kappa\eta^{La}(0))}{\mathcal A(1,\kappa\eta^{La}(0))^2}\right|^2+\mathcal O(|c|)\\
		\geq&C(1-|c|)\geq \frac{1}{2}C>0.
    	\end{split}
    \end{align}
    
	On the other hand, since $F_r,F_i$ are both analytic on $\alpha$, for any fixed $\nu$, we have 
	\begin{align}\label{eq:F-alpha}
		|\partial_\alpha F_r(\alpha;c_r,c_i)|+|\partial_\alpha F_i(\alpha;c_r,c_i)|\leq C,\quad\forall \alpha\in\mathbb G .
	\end{align}
	
	We denote 
	\begin{align*}
		\mathbb A(c^{(j)})=\frac{U_s'(0)\tilde{\mathcal A}(2,0)}{\tilde{\mathcal A}(1,0)}\text{ with taking } c=c^{(j)}.
	\end{align*}
We now show that there exists a $(\al^0,c^0)\in\mathbb H_2$ such that $\mathbb F(\alpha^0;c_r^0,c_i^0)=0$. Let $\alpha=A\nu^\f18$ with $A\gg1$.  We take $c^{(0)}=U_s'(0)^{-1}(1-m^2)^{-\f12}\alpha$, and for $j\geq 1$,
	\begin{align*}
		c^{(j)}=U_s'(0)^{-1}(1-m^2)^{-\f12}\alpha-\mathbb A(c^{(j-1)})-\mathcal R(\alpha,c^{(j-1)}).
	\end{align*}
	By Lemma \ref{lem:airy-0}, we know that
	\begin{align*}
		\mathbb A(c^{(0)})=-A^{-1}e^{\mathrm i\frac{\pi}{4}}\nu^\f18+\mathcal O(A^{-\f43})\nu^\f18,
	\end{align*}
   which implies 
   \begin{align*}
   	&c^{(1)}=\nu^\f18\left(U_s'(0)(1-m^2)^{-\f12}A+A^{-1}e^{\mathrm i \frac{\pi}{4}}\right)+\mathcal O(A^{-\f43})\nu^\f18,\\
   	&\mathrm{Im}(c^{(1)})=\frac{\sqrt{2}}{2}A^{-1}\nu^{\f18}+\mathcal O(A^{-\f43})\nu^{\f18}.
   \end{align*}
   Hence, $(\al,c^{(1)})\in\mathbb H_2$.
   Hence, by the inductive argument and Lemma \ref{lem:airy-0}, we can obtain that for any $j\geq 1$, $(\alpha,c^{(j)})\in\mathbb H_2$. Moreover, again by Lemma \ref{lem:airy-0} and \eqref{eq:d-cr-ci}, we have
   \begin{align*}
   	\left|\partial_{c_r}\left(\frac{U_s'(0)\tilde{\mathcal A}(2,0)}{\tilde{\mathcal A}(1,0)}\right)\right|+\left|\partial_{c_i}\left(\frac{U_s'(0)\tilde{\mathcal A}(2,0)}{\tilde{\mathcal A}(1,0)}\right)\right|\leq CA^{-\f43}.
   \end{align*}
	Therefore, we obtain that for any $j\geq 1$,
	\begin{align*}
		|c^{(j+1)}-c^{(j)}|\leq& CA^{-\f43}|c^{(j)}-c^{(j-1)}|+C|\varepsilon\log|\varepsilon|| |c^{(j)}-c^{(j-1)}|\\
		\leq&CA^{-\f43}|c^{(j)}-c^{(j-1)}|.
	\end{align*}
	Since $A\gg1$, there exists $c^0$ such that $\lim_{j\to\infty}c^{(j)}=c^0$ and $c^0$ satisfies 
		\begin{align*}
		c^{0}=U_s'(0)^{-1}(1-m^2)^{-\f12}\alpha-\mathbb A(c^{0})-\mathcal R(\alpha,c^{(0)}).
	\end{align*}
	Hence,  $(\alpha,c_r^{(0)},c_i^{(0)})$ is a zero point of $\mathbb F(\alpha;c_r,c_i)$, which along with \eqref{eq:F-c-Jaco} and \eqref{eq:F-alpha} implies that for  $\alpha\in\mathbb G$, there exists a unique $c(\alpha)$ solving \eqref{eq:dr-alpha} with $(c(\alpha),\alpha)\in\mathbb H$. Moreover, $c(\alpha)$ depends on $\alpha$ continuously.
	\end{proof}

From the above two propositions, we conclude the following result.

\begin{theorem}
	Let $m<1$. Suppose that $U_s(Y)$ satisfies the structure assumption \eqref{eq:S-A}. Then there exist $0<\nu(m)\ll1$ and $1\ll A_0<B_0$ such that for any $\nu\leq \nu(m)$, the following statement holds true: for any $\alpha\in\mathbb G$, we can find $c(\alpha)$ such that there exists a solution $(\rho,u,v)\in W^{1,\infty}\times W^{2,\infty}\times W^{2,\infty}$ to  the system \eqref{eq:LCNS-Y} and \eqref{BC: u(0)=v(0)=0}. Moreover, such $c,\alpha$ satisfy the following properties
\begin{align*}
c_r=U'_s(0)^{-1}(1-m^2)^{-\f12}\al_r+o(\al),\quad	0<c_i-U_s'(0)^{-1}(1-m^2)^{-\f12}\alpha_i\sim A^{-1}\nu^\f18.
\end{align*}
	In particular, for any $\alpha_r=A\nu^\f18$ with $A\in(A_0,B_0)$ we have the following results:
	\begin{itemize}
		\item there exists  $\alpha^0=\alpha_r+\mathrm i\alpha^0_i\in\mathbb G$ with $\alpha_i^0<0$ such that 
		      \begin{align*}
		      	\alpha^0c(\alpha^0)\in\mathbb R\text{ and }|\alpha_i^0|\sim\nu^\f18.
		      \end{align*}
		\item there exists  $c(\alpha_r)$ such that $c_i(\alpha_r)>0$ and $|c_i(\alpha_r)|\sim\nu^\f18$.   
	\end{itemize}
	
\end{theorem}

\appendix

\section{The Airy function}\label{Sec:Airy}

 Let $Ai(y)$ be the Airy function, which is a nontrivial solution of $f''-yf=0$. We have the following asymptotic formula: for $|\arg z|\leq\pi-\delta$ with $\delta>0$ and $|z|\geq M$ for some large $M$,
 \begin{align}\label{eq:airy-decay}
 	\partial^k_z Ai(z)=\frac{1}{2\sqrt{\pi}}z^{-\f14+\f k2}e^{-\f23z^{\f32}}(1+R(z)),\quad R(z)=\mathcal O(z^{-\f32}),\quad k=0, 1, 2.
 \end{align}  
 
 We denote
 \begin{align*}
 	&\mathcal A(1,z)=-\int_z^{+\infty}Ai(e^{\mathrm i(\frac{\pi}{6}-\theta_0)}t)dt,\quad\mathcal A(2,z)=-\int_z^{+\infty}\mathcal A(1,t)dt,\\
 	&\mathcal B(1,z)=\int_0^zAi(e^{\mathrm i(\frac{5\pi}{6}-\theta_0)}t)dt,\quad\mathcal B(2,z)=\int_0^t\mathcal B(1,t)dt,
 \end{align*}
 and
\begin{align*}
&A_0(z)=\int_{\mathrm{e}^{{\mathrm{i}}(\pi/6-\theta_0)}z}^{+\infty}Ai(t)\mathrm{d}t =\mathrm{e}^{{i}(\pi/6-\theta_0)}\int_{z}^{+\infty}Ai(\mathrm{e}^{{\mathrm{i}}(\pi/6-\theta_0)}t)\mathrm{d}t=-e^{\mathrm i(\frac{\pi}{6}-\theta_0)}\mathcal A(1,z),
\end{align*}
where $|\theta_0|\ll1$.\smallskip

The following lemma comes from \cite{CLWZ}.

\begin{lemma}\label{lem:Airy-p1}
There exists $\gamma>0$ and $\delta_0>0$ so that for $\mathrm{Im}z\leq\delta_0$, we have
\begin{align}
\begin{split}
	&\left|\f{A_0'(z)}{A_0(z)}\right|\lesssim1+|z|^{\f12},\quad {\rm Re}\f{A_0'(z)}{A_0(z)}\leq\min(-1/3,-\gamma(1+|z|^{\f12})),
\end{split}
\end{align}
and 
\begin{align*}
	 \Big|\f{A_0''(z)}{A_0(z)}\Big|\le C(1+|z|).
\end{align*}
\end{lemma}

\begin{lemma}\label{lem:pri-Airy-decay}
Let $\delta_0$ be the constant in Lemma \ref{lem:Airy-p1} and $|\theta_0|\ll1$.  Assume that $\mathrm{Im}z\leq\delta_0$ and $|z|\geq M$. Then it holds that
\begin{align*}
	&\mathcal A(1,z)=-e^{-\mathrm i(\frac{\pi}{6}-\theta_0)}(e^{\mathrm i(\frac{\pi}{6}-\theta_0)}z)^{-\f34}e^{-\f23(e^{\mathrm i(\frac{\pi}{6}-\theta_0)}z )^\f32}(1+\mathcal R_1(z)),\\
	&\mathcal A(2,z)=e^{-\mathrm i(\frac{\pi}{3}-2\theta_0)}(e^{\mathrm i(\frac{\pi}{6}-\theta_0)}z)^{-\f54}e^{-\f23(e^{\mathrm i(\frac{\pi}{6}-\theta_0)}z )^\f32}(1+\mathcal R_2(z)),\\
	&\mathcal B(1,z)=-e^{-\mathrm i(\frac{5\pi}{6}-\theta_0)}(e^{\mathrm i(\frac{5\pi}{6}-\theta_0)}z)^{-\f34}e^{-\f23(e^{\mathrm i(\frac{5\pi}{6}-\theta_0)}z )^\f32}(1+\mathcal R_3(z)),\\
	&\mathcal B(2,z)=e^{\mathrm i(\frac{\pi}{3}+2\theta_0)}(e^{\mathrm i(\frac{5\pi}{6}-\theta_0)}z)^{-\f54}e^{-\f23(e^{\mathrm i(\frac{5\pi}{6}-\theta_0)}z )^\f32}(1+\mathcal R_4(z)),
\end{align*}
where $\mathcal R_i(z)=\mathcal O(z^{-\f32})$ for $i=1,2,3,4.$
\end{lemma}

Here we omit the proof of this lemma and refer to Lemma \ref{lem:airy-langer-asy} for the proof in more complicated cases.

\begin{lemma}\label{lem:Airy-green-decay}
	Let $\delta_0>0$ be the constant in Lemma \ref{lem:Airy-p1} and $|\theta_0|\ll1$.  Suppose $-\delta_0<\mathrm{Im}Z_0<\delta_0$, $\mathrm{Re}Z_0<0$.  There exist  $c_1>0, M_1>0$ such that if $|\mathrm{Re}(Y+Z_0)|>M_1\delta_0$, then for any $0\leq Z\leq Y$,
	\begin{align*}
		\left|e^{-\f23(e^{\mathrm i(\frac{\pi}{6}-\theta_0)}(Y+Z_0))^\f32 } e^{-\f23(e^{\mathrm i(\frac{5\pi}{6}-\theta_0)}(Z+Z_0))^\f32 }\right|\leq e^{-c_1|Y-Z|(|Y+Z_0|+|Z+Z_0)^\f12}.
	\end{align*}
\end{lemma}

\begin{proof}
  Notice that 
     \begin{align}\label{eq:Airy-large-decay}
    	\left|e^{-\f23(e^{\mathrm i(\frac{\pi}{6}-\theta_0)}(Y+Z_0))^\f32 } e^{-\f23(e^{\mathrm i(\frac{5\pi}{6}-\theta_0)}(Z+Z_0))^\f32 }\right|=e^{-\f23\mathrm{Re}((e^{\mathrm i(\frac{\pi}{6}-\theta_0)}(Y+Z_0))^\f32+(e^{\mathrm i(\frac{5\pi}{6}-\theta_0)}(Z+Z_0))^\f32)}.
    \end{align}
 Thus,  it is enough to focus on the estimate about $\mathrm{Re}\big((e^{\mathrm i(\frac{\pi}{6}-\theta_0)}(Y+Z_0))^\f32+(e^{\mathrm i(\frac{5\pi}{6}-\theta_0)}(Z+Z_0))^\f32\big)$.\smallskip
    
    \no\textbf{Case 1. $0\leq Z\leq Y\leq -\mathrm{Re}Z_0$.} In this case,  for large enough $M_1$,
	\begin{align*}
		&|(Z+Z_0)|\geq|(Y+Z_0)|, \quad|\mathrm{Re}(Z+Z_0)|\geq|\mathrm{Re}(Y+Z_0)|,\\
		&\mathrm{Im}(Y+Z_0)=\mathrm{Im}(Z+Z_0),    \\ 
		&\arg((e^{\mathrm i(\frac{\pi}{6}-\theta_0)}(Y+Z_0))^\f32)\in [\f\pi2+\xi_0,\pi-\xi_0] ,\\ &\arg((e^{\mathrm i(\frac{5\pi}{6}-\theta_0)}(Z+Z_0))^\f32))\in [-\f\pi2+\xi_0,-\xi_0],
	\end{align*}
    where $\xi_0\in(0,\f\pi4)$. Then we infer that for any $0\leq Z\leq Y\leq -\mathrm{Re}Z_0$,
    \begin{align*}
    	&\cos(\arg((e^{\mathrm i(\frac{\pi}{6}-\theta_0)}(Y+Z_0))^\f32))\leq0,\\
    	&-\cos(\arg((e^{\mathrm i(\frac{\pi}{6}-\theta_0)}(Z+Z_0))^\f32))=\cos(\arg((e^{\mathrm i(\frac{5\pi}{6}-\theta_0)}(Z+Z_0))^\f32)))\geq c_0>0,
    \end{align*}
    and
    \begin{align*}
    	\left|\cos(\arg((e^{\mathrm i(\frac{\pi}{6}-\theta_0)}(Z+Z_0))^\f32))\right|\geq \left|\cos(\arg((e^{\mathrm i\frac{\pi}{6}-\theta_0}(Y+Z_0)^\f32))\right|.
    \end{align*}
    Thus, we obtain
    \begin{align*}
    	&\mathrm{Re}((e^{\mathrm i(\frac{\pi}{6}-\theta_0)}(Y+Z_0))^\f32+(e^{\mathrm i(\frac{5\pi}{6}-\theta_0)}(Z+Z_0))^\f32)\\
    	&=\left|(Y+Z_0)\right|^\f32\cos(\arg((e^{\mathrm i(\frac{\pi}{6}-\theta_0)}(Y+Z_0))^\f32))+\left|(Z+Z_0)\right|^\f32\cos(\arg((e^{\mathrm i(\frac{5\pi}{6}-\theta_0)}(Z+Z_0))^\f32\\
    	&\geq\left|(Y+Z_0)\right|^\f32\cos(\arg((e^{\mathrm i(\frac{\pi}{6}-\theta_0)}(Z+Z_0))^\f32))+\left|(Z+Z_0)\right|^\f32\cos(\arg((e^{\mathrm i(\frac{5\pi}{6}-\theta_0)}(Z+Z_0))^\f32\\
    	&=\left|(Y+Z_0)\right|^\f32\cos(\arg((e^{\mathrm i(\frac{\pi}{6}-\theta_0)}(Z+Z_0))^\f32))-\left|(Z+Z_0)\right|^\f32\cos(\arg((e^{\mathrm i(\frac{\pi}{6}-\theta_0)}(Z+Z_0))^\f32\\
    	&\geq c_0(|Z+Z_0|^\f32-|Y+Z_0|^\f32)\geq {c}_0(|Z+Z_0|-|Y+Z_0|)(|Z+Z_0|+|Y+Y_0|)^\f12\\
    	&\geq c_0|Y-Z|(|Z+Z_0|+|Y+Y_0|)^\f12.
    \end{align*}	
	which along with \eqref{eq:Airy-large-decay} implies that for any $0\leq Z\leq Y\leq \mathrm{Re}(Z_0)$,
	\begin{align*}
		\left|e^{-\f23(e^{\mathrm i(\frac{\pi}{6}-\theta_0)}(Y+Z_0))^\f32 } e^{-\f23(e^{\mathrm i(\frac{5\pi}{6}-\theta_0)}(Z+Z_0))^\f32 }\right|\leq e^{-c_1|Y-Z|(|Y+Z_0|+|Z+Z_0)^\f12}.
	\end{align*}
	
	\no\textbf{Case 2. $0\leq Z\leq-\mathrm{Re}(Z_0)\leq Y$.} We first notice that 
	\begin{align*}
		\cos(\arg((e^{\mathrm i(\frac{5\pi}{6}-\theta_0)}(Z+Z_0))^\f32)))\ge 0,\quad\cos(\arg((e^{\mathrm i(\frac{\pi}{6}-\theta_0)}(Y+Z_0))^\f32))\geq0.
	\end{align*}
	As a result, we obtain 
	\begin{align*}
		&\mathrm{Re}((e^{\mathrm i(\frac{\pi}{6}-\theta_0)}(Y+Z_0))^\f32+(e^{\mathrm i(\frac{5\pi}{6}-\theta_0)}(Z+Z_0))^\f32)\\
		&=\left|(Y+Z_0)\right|^\f32\cos(\arg((e^{\mathrm i(\frac{\pi}{6}-\theta_0)}(Y+Z_0))^\f32))\\
		&\quad+\left|(Z+Z_0)\right|^\f32\cos(\arg((e^{\mathrm i(\frac{5\pi}{6}-\theta_0)}(Z+Z_0))^\f32\geq0.
	\end{align*}
	Moreover, by the assumption on $|Y+Z_0|$, we have
	\begin{align}\label{eq:airy-exp-decay}
		\cos(\arg((e^{\mathrm i(\frac{\pi}{6}-\theta_0)}(Y+Z_0))^\f32))\geq c_0.
	\end{align}
	Hence, for any $0\leq Z\leq -\mathrm{Re}(Z_0)$ and $Y\geq-\mathrm{Re}(Z_0)$, we have
	\begin{align*}
		&\mathrm{Re}((e^{\mathrm i(\frac{\pi}{6}-\theta_0)}(Y+Z_0))^\f32+(e^{\mathrm i(\frac{5\pi}{6}-\theta_0)}(Z+Z_0))^\f32)\geq c_0\left|(Y+Z_0)\right|^\f32.
	\end{align*}
	Then we obtain that for any $0\leq Z\leq -\mathrm{Re}(Z_0)$, $Y\geq-\mathrm{Re}(Z_0)$ and $|Z+Z_0|\leq |Y+Z_0|$,
	\begin{align*}
		\left|e^{-\f23(e^{\mathrm i(\frac{\pi}{6}-\theta_0)}(Y+Z_0))^\f32 } e^{-\f23(e^{\mathrm i(\frac{5\pi}{6}-\theta_0)}(Z+Z_0))^\f32 }\right|\leq& e^{-\f23 c_0|Y+Z_0|^\f32}\\
		\leq&e^{-c_1|Y-Z|(|Y+Z_0|+|Z+Z_0|)^\f12}.
	\end{align*}
	On the other hand, for any $0\leq Z\leq -\mathrm{Re}(Z_0)$, $Y\geq-\mathrm{Re}(Z_0)$ and $|Z+Z_0|\geq |Y+Z_0|$, we have 
	\begin{align*}
		\cos(\arg((e^{\mathrm i(\frac{5\pi}{6}-\theta_0)}(Z+Z_0))^\f32))\geq c_0,
	\end{align*}
	which implies 
	\begin{align*}
		&\mathrm{Re}((e^{\mathrm i(\frac{\pi}{6}-\theta_0)}(Y+Z_0))^\f32+(e^{\mathrm i(\frac{5\pi}{6}-\theta_0)}(Z+Z_0))^\f32)\geq c_0\left|(Z+Z_0)\right|^\f32.
	\end{align*}
	Therefore, we conclude that for any $0\leq Z\leq -\mathrm{Re}(Z_0)$, $Y\geq-\mathrm{Re}(Z_0)$ and $|Z+Z_0|\geq |Y+Z_0|$
	\begin{align*}
		\left|e^{-\f23(e^{\mathrm i(\frac{\pi}{6}-\theta_0)}(Y+Z_0))^\f32 } e^{-\f23(e^{\mathrm i(\frac{5\pi}{6}-\theta_0)}(Z+Z_0))^\f32 }\right|\leq& e^{-\f23c_0|Z+Z_0|^\f32}\\
		\leq&e^{-c_1|Y-Z|(|Y+Z_0|+|Z+Z_0|)^\f12}.
	\end{align*}
	
	\no\textbf{Case 3. $0<-\mathrm{Re}(Z_0)\leq Z\leq Y$.} In this case, we first notice that 
	\begin{align*}
    	&\cos(\arg((e^{\mathrm i(\frac{\pi}{6}-\theta_0)}(Y+Z_0))^\f32))>0,\\
    	&-\cos(\arg((e^{\mathrm i(\frac{\pi}{6}-\theta_0)}(Z+Z_0))^\f32))=\cos(\arg((e^{\mathrm i(\frac{5\pi}{6}-\theta_0)}(Z+Z_0))^\f32)))<0,
    \end{align*}
    and 
    \begin{align*}
    	|Z+Z_0|\leq |Y+Z_0|,\quad\cos(\arg((e^{\mathrm i(\frac{\pi}{6}-\theta_0)}(Y+Z_0))^\f32))\geq \cos(\arg((e^{\mathrm i(\frac{\pi}{6}-\theta_0)}(Z+Z_0))^\f32)),
    \end{align*}
	which implies 
	\begin{align*}
		&\mathrm{Re}((e^{\mathrm i(\frac{\pi}{6}-\theta_0)}(Y+Z_0))^\f32+(e^{\mathrm i(\frac{5\pi}{6}-\theta_0)}(Z+Z_0))^\f32)\\
    	&=\left|(Y+Z_0)\right|^\f32\cos(\arg((e^{\mathrm i(\frac{\pi}{6}-\theta_0)}(Y+Z_0))^\f32))+\left|(Z+Z_0)\right|^\f32\cos(\arg((e^{\mathrm i(\frac{5\pi}{6}-\theta_0)}(Z+Z_0))^\f32\\
    	&=\left|(Y+Z_0)\right|^\f32\cos(\arg((e^{\mathrm i(\frac{\pi}{6}-\theta_0)}(Y+Z_0))^\f32))-\left|(Z+Z_0)\right|^\f32\cos(\arg((e^{\mathrm i(\frac{\pi}{6}-\theta_0)}(Z+Z_0))^\f32\\
    	&\geq\left|(Y+Z_0)\right|^\f32\cos(\arg((e^{\mathrm i(\frac{\pi}{6}-\theta_0)}(Y+Z_0))^\f32))-\left|(Z+Z_0)\right|^\f32\cos(\arg((e^{\mathrm i(\frac{\pi}{6}-\theta_0)}(Y+Z_0))^\f32\\
    	&\geq(\left|(Y+Z_0)\right|^\f32-\left|(Z+Z_0)\right|^\f32)\cos(\arg((e^{\mathrm i(\frac{\pi}{6}-\theta_0)}(Y+Z_0))^\f32))>0.
	\end{align*}
	Therefore, we obtain that for any $-\mathrm{Re}(Z_0)\leq Z\le Y$,
	\begin{align*}
		&\mathrm{Re}((e^{\mathrm i(\frac{\pi}{6}-\theta_0)}(Y+Z_0))^\f32+(e^{\mathrm i(\frac{5\pi}{6}-\theta_0)}(Z+Z_0))^\f32)\geq c_0|Y-Z|(|Y+Z_0|+|Z+Z_0|)^\f12.
	\end{align*}
	Thus, we arrive at
	\begin{align*}
		\left|e^{-\f23(e^{\mathrm i(\frac{\pi}{6}-\theta_0)}(Y+Z_0))^\f32 } e^{-\f23(e^{\mathrm i(\frac{5\pi}{6}-\theta_0)}(Z+Z_0))^\f32 }\right|\leq Ce^{-c_1|Y-Z|(|Y+Z_0|+|Z+Z_0|)^\f12}.
	\end{align*}
	
	\end{proof}

\begin{lemma}\label{lem:airy-langer-asy}
Let $|\alpha|,|\varepsilon|,|c|\ll1$ and $c_r>0$. Suppose that $0<\delta_0\ll1\leq M$ are the constants in Lemma \ref{lem:Airy-p1}. Then there holds
\begin{enumerate}
	\item For any $Y\in\mathcal N^-\cup\mathcal N^+$,
      \begin{align*}
 		A_1(1,Y)=&\frac{e^{-\th_0 i}}{2\sqrt{\pi}}e^{\mathrm i\frac{\pi}{3}}|\varepsilon|^{-\f13}U_s'(Y_c)^{-\f23}\partial_Y\eta^{La}(Y)^{-1}\Big(e^{\mathrm i(\frac{\pi}{6}-\th_0)}\kappa\eta^{La}(Y) \Big)^{-\f34}\\
 		&\times e^{-\f23\big(e^{\mathrm i(\frac{\pi}{6}-\th_0)}\kappa\eta^{La}(Y)\big)^\f32}(1+\mathcal O(|\kappa\eta^{La}(Y)|^{-\f32})),
 	\end{align*}
 	\begin{align*}
 		A_1(2,Y)=&-\frac{1}{2\sqrt{\pi}}e^{\mathrm i\frac{\pi}{6}}U_s'(Y_c)^{-1}\partial_Y\eta^{La}(Y)^{-2}\Big(e^{\mathrm i(\frac{\pi}{6}-\th_0)}\kappa\eta^{La}(Y) \Big)^{-\f54}\\
 		&\times e^{-\f23\big(e^{\mathrm i(\frac{\pi}{6}-\th_0)}\kappa\eta^{La}(Y)\big)^\f32}(1+\mathcal O(|\kappa\eta^{La}(Y)|^{-\f32})),
 	\end{align*}
 	\begin{align*}
 		A_2(1,Y)=&-\sqrt{\pi}e^{-\mathrm i(\frac{5\pi}{6}-\th_0)}|\varepsilon|^\f13U_s'(Y_c)^{-\f13}\partial_Y\eta^{La}(Y)^{-1}\Big(e^{\mathrm i(\frac{5\pi}{6}-\th_0)}\kappa\eta^{La}(Y) \Big)^{-\f34}\\
 		&\times e^{-\f23\big(e^{\mathrm i(\frac{5\pi}{6}-\th_0)}\kappa\eta^{La}(Y)\big)^\f32}(1+\mathcal O(|\kappa\eta^{La}(Y)|^{-\f32})),
 	\end{align*}
 	 	\begin{align*}
 		A_2(2,Y)=&\sqrt{\pi}e^{\mathrm i(\frac{\pi}{3}+2\th_0)}|\varepsilon|^\f23U_s'(Y_c)^{-\f23}\partial_Y\eta^{La}(Y)^{-2}\Big(e^{\mathrm i(\frac{5\pi}{6}-\th_0)}\kappa\eta^{La}(Y) \Big)^{-\f54}\\
 		&\times e^{-\f23\big(e^{\mathrm i(\frac{5\pi}{6}-\th_0)}\kappa\eta^{La}(Y)\big)^\f32}(1+\mathcal O(|\kappa\eta^{La}(Y)|^{-\f32})).
 	\end{align*}

	\item For any $Y\in\mathcal N$,
	      \begin{align*}
	      	&|A_1(1,Y)|\leq C|\varepsilon|^{-\f13},\quad|A_1(2,Y)|\leq C,\\
	      	&|A_2(1,Y)|\leq C|\varepsilon|^\f13,\quad|A_2(2,Y)|\leq C|\varepsilon|^\f23.
	      \end{align*}
	      
\end{enumerate}
\end{lemma}

\begin{proof}
By the definition of $A_1(1,Y)$, we can write 
\begin{align}\label{formula: A_1(1,Y)}
	A_1(1,Y)=\mathrm i e^{-2\th_0 i}|\varepsilon|^{-\f23}U_s'(Y_c)^{-\f13}\int_Y^{+\infty}Ai(e^{\mathrm i(\frac{\pi}{6}-\th_0)}\kappa\eta^{La}(Z))dZ.
\end{align}

\no\textbf{Case 1. $Y\in\mathcal N^+$.} By the definition of $\eta^{La}(Y)$ and Lemma \ref{lem:est-eta}, we know   $\mathcal N^+\subset[Y_c+2\kappa^{-1}M,+\infty)$ and $\eta^{La}=\eta_{out}^{La}+i \eta_i^{La}$. By \eqref{eq:airy-decay}, we have that for any $Y\in\mathcal N^+$,
	\begin{align}\label{eq:airy-langer-L1}
		\begin{split}
			\int_Y^{+\infty}&Ai(e^{\mathrm i(\frac{\pi}{6}-\th_0)}\kappa\eta^{La}(Z))dZ\\
			=&\frac{1}{2\sqrt{\pi}}\int_Y^{+\infty} \Big(e^{\mathrm i(\frac{\pi}{6}-\th_0)}\kappa\eta^{La}(Z) \Big)^{-\f14}e^{-\f23\big(e^{\mathrm i(\frac{\pi}{6}-\th_0)}\kappa\eta^{La}(Z)\big)^\f32}(1+R(Z))dZ\\
			=&\frac{-1}{2\sqrt{\pi}}\int_Y^{+\infty}\Big(e^{\mathrm i(\frac{\pi}{6}-\th_0)}\kappa\eta^{La}(Z) \Big)^{-\f34}\Big(e^{\mathrm i(\frac{\pi}{6}-\th_0)}\kappa\partial_Z \eta(Z) \Big)^{-1}\partial_Z\Big(e^{-\f23\big(e^{\mathrm i(\frac{\pi}{6}-\th_0)}\kappa\eta^{La}(Z)\big)^\f32}\Big)dZ\\
			&+\frac{1}{2\sqrt{\pi}}\int_Y^{+\infty} \Big(e^{\mathrm i(\frac{\pi}{6}-\th_0)}\kappa\eta^{La}(Z) \Big)^{-\f14}e^{-\f23\big(e^{\mathrm i(\frac{\pi}{6}-\th_0)}\kappa\eta^{La}(Z)\big)^\f32}R(Z)dZ \\
			=&\frac{1}{2\sqrt{\pi}}\Big(e^{\mathrm i(\frac{\pi}{6}-\th_0)}\kappa\eta^{La}(Y) \Big)^{-\f34}\Big(e^{\mathrm i(\frac{\pi}{6}-\th_0)}\kappa\partial_Y \eta^{La}(Y) \Big)^{-1}e^{-\f23\big(e^{\mathrm i(\frac{\pi}{6}-\th_0)}\kappa\eta^{La}(Y)\big)^\f32}\\
			&+\Bigg(\frac{1}{2\sqrt{\pi}}\int_Y^{+\infty}\partial_Z\Big[\Big(e^{\mathrm i(\frac{\pi}{6}-\th_0)}\kappa\eta^{La}(Z) \Big)^{-\f34}\Big(e^{\mathrm i(\frac{\pi}{6}-\th_0)}\kappa\partial_Z \eta^{La}(Z) \Big)^{-1}\Big]e^{-\f23\big(e^{\mathrm i(\frac{\pi}{6}-\th_0)}\kappa\eta^{La}(Z)\big)^\f32}dZ\\
			&+\frac{1}{2\sqrt{\pi}}\int_Y^{+\infty} \Big(e^{\mathrm i(\frac{\pi}{6}-\th_0)}\kappa\eta^{La}(Z) \Big)^{-\f14}e^{-\f23\big(e^{\mathrm i(\frac{\pi}{6}-\th_0)}\kappa\eta^{La}(Z)\big)^\f32}R(Z)dZ\Bigg)\\
			=&\frac{1}{2\sqrt{\pi}}\Big(e^{\mathrm i(\frac{\pi}{6}-\th_0)}\kappa\eta^{La}(Y) \Big)^{-\f34}\Big(e^{\mathrm i(\frac{\pi}{6}-\th_0)}\kappa\partial_Y \eta^{La}(Y) \Big)^{-1}e^{-\f23\big(e^{\mathrm i(\frac{\pi}{6}-\th_0)}\kappa\eta^{La}(Y)\big)^\f32}+I_1.
		\end{split}
	\end{align}
Notice that for any $Y\in\mathcal N^+$,
    \begin{align}\label{eq:airy-langer-L2}
    	\begin{split}
    			I_1
		=&\frac{1}{2\sqrt{\pi}}\Big(e^{\mathrm i(\frac{\pi}{6}-\th_0)}\kappa\eta^{La}(Y) \Big)^{-\f34}\Big(e^{\mathrm i(\frac{\pi}{6}-\th_0)}\kappa\partial_Y \eta^{La}(Y) \Big)^{-1}e^{-\f23\big(e^{\mathrm i(\frac{\pi}{6}-\th_0)}\kappa\eta^{La}(Y)\big)^\f32}\mathcal R_1(Y),
    	\end{split}
    \end{align}
	where
	\begin{align}\label{def: R_1(Y)}
		\mathcal R_1(Y)=\int_Y^{+\infty}(\mathcal T(Y,Z)+\mathcal S(Y,Z)) e^{-\f23\big(e^{\mathrm i(\frac{\pi}{6}-\th_0)}\kappa\eta^{La}(Z)\big)^\f32+\f23\big(e^{\mathrm i(\frac{\pi}{6}-\th_0)}\kappa\eta^{La}(Y)\big)^\f32}dZ
	\end{align}
	with 
	\begin{align*}
		\mathcal T(Y,Z)=\Big(e^{\mathrm i(\frac{\pi}{6}-\th_0)}\kappa\eta^{La}(Y) \Big)^{\f34}e^{\mathrm i(\frac{\pi}{6}-\th_0)}\kappa\partial_Y \eta^{La}(Y)\partial_Z\Big[\Big(e^{\mathrm i(\frac{\pi}{6}-\th_0)}\kappa\eta^{La}(Z) \Big)^{-\f34}\Big(e^{\mathrm i(\frac{\pi}{6}-\th_0)}\kappa\partial_Z \eta^{La}(Z) \Big)^{-1}\Big],
	\end{align*}
	and 
	\begin{align*}
		\mathcal S(Y,Z)=\Big(e^{\mathrm i(\frac{\pi}{6}-\th_0)}\kappa\eta^{La}(Y) \Big)^{\f34}e^{\mathrm i(\frac{\pi}{6}-\th_0)}\kappa\partial_Y \eta^{La}(Y)\Big(e^{\mathrm i(\frac{\pi}{6}-\th_0)}\kappa\eta^{La}(Z) \Big)^{-\f14} R(Z).
	\end{align*}
	
	For any $Y\leq Z$ and $Y,Z\in\mathcal N^+$, we have $\eta^{La}=\eta_{out}^{La}+i \eta_i^{La}\sim \eta_{out}^{La}$,
	which along with Lemma \ref{lem:classical-Langer} implies that for any $Y\leq Z$ and $Y,Z\in\mathcal N^+$,
	\begin{align*}
		|\mathcal T(Y,Z)|\leq& |\eta_{out}^{La}(Y)|^\f34|\partial_Y\eta_{out}^{La}(Y)|\Big(|\eta_{out}^{La}(Z)|^{-\f74}+|\partial_Z^2\eta_{out}^{La}(Z)||\partial_Z\eta_{out}^{La}(Z)|^{-2}|\eta_{out}^{La}(Z)|^{-\f34} \Big)\\
		\leq&C |\eta_{out}^{La}(Y)|^\f34|\partial_Y\eta_{out}^{La}(Y)||\eta_{out}^{La}(Z)|^{-\f74}\leq C|\eta_{out}^{La}(Y)|^{-1}|\partial_{Y}\eta_{out}^{La}(Y)|,
	\end{align*}
where we used the facts that
\begin{align*}
&|\partial_Z^2\eta_{out}^{La}(Z)|\leq C(1+|Z-Y_c|)^{-\f43}\leq C|\eta_{out}^{La}(Z)|^{-2},\\
&|\partial_Z\eta_{out}^{La}(Z)|^{-2}\leq C(1+|Z-Y_c|)^{\f23}\leq C|\eta_{out}^{La}(Z)|.
\end{align*}
Similarly, we obtain that for any $Y\leq Z$ and $Y,Z\in\mathcal N^+$,
	\begin{align*}
		|\mathcal S(Y,Z)|\leq C|\eta_{out}^{La}(Y)|^{-1}|\partial_{Y}\eta_{out}^{La}(Y)|.
	\end{align*}
	Therefore, we obtain that for any $Y\in\mathcal N^{+}$,
	\begin{align}\label{eq:airy-langer-large}
		\begin{split}
			|\mathcal R_1(Y)|\leq& C|\eta_{out}^{La}(Y)|^{-1}|\partial_{Y}\eta_{out}^{La}(Y)|\int_Y^{+\infty}\left|e^{-\f23\big(e^{\mathrm i(\frac{\pi}{6}-\th_0)}\kappa\eta^{La}(Z)\big)^\f32+\f23\big(e^{\mathrm i(\frac{\pi}{6}-\th_0)}\kappa\eta^{La}(Y)\big)^\f32}\right|dZ.
		\end{split}
	\end{align}
 	On the other hand, by Lemma \ref{lem:Airy-green-decay}, for any $0\leq Y\leq Z$ and $Y,Z\in\mathcal N^{+}$, we have
 	\begin{align*}
 		&\int_Y^{+\infty}\left|e^{-\f23\big(e^{\mathrm i(\frac{\pi}{6}-\th_0)}\kappa\eta^{La}(Z)\big)^\f32+\f23\big(e^{\mathrm i(\frac{\pi}{6}-\th_0)}\kappa\eta^{La}(Y)\big)^\f32}\right|dZ\\
		&\leq C\int_Y^{+\infty}e^{-C_1\kappa|\eta^{La}(Y)-\eta^{La}(Z)|(|\kappa\eta^{La}(Y)|+|\kappa\eta^{La}(Z)|)^\f12}dZ\\
 		&\leq C\int_Y^{+\infty}e^{-C_2\kappa|\eta^{La}(Y)-\eta^{La}(Z)||\kappa\eta^{La}(Y)|^\f12}dZ\\
    	&\leq C\kappa^{-1}|\kappa\eta^{La}_{out}(Y)|^{-\f12}|\partial_Y\eta^{La}_{out}(Y)|^{-1}.
 	\end{align*}
 Therefore, according to \eqref{def: R_1(Y)}, we obtain that for any $Y\in\mathcal N^+$,
 	\begin{align*}
 		|\mathcal R_1(Y)|\leq C\kappa^{-\f32}|\eta_{out}^{La}(Y)|^{-1}|\partial_{Y}\eta_{out}^{La}(Y)||\eta^{La}_{out}(Y)|^{-\f12}|\partial_Y\eta^{La}_{out}(Y)|^{-1}\leq C|\kappa\eta^{La}(Y)|^{-\f32},
 	\end{align*}
 	which along with \eqref{eq:airy-langer-L1} and \eqref{eq:airy-langer-L2} deduces that for any $Y\in\mathcal N^+$,
 	\begin{align*}
 		&\int_Y^{+\infty}Ai(e^{\mathrm i(\frac{\pi}{6}-\th_0)}\kappa\eta^{La}(Z))dZ\\
 		&=\frac{1}{2\sqrt{\pi}}\Big(e^{\mathrm i(\frac{\pi}{6}-\th_0)}\kappa\eta^{La}(Y) \Big)^{-\f34}\Big(e^{\mathrm i(\frac{\pi}{6}-\th_0)}\kappa\partial_Y \eta^{La}(Y) \Big)^{-1}e^{-\f23\big(e^{\mathrm i(\frac{\pi}{6}-\th_0)}\kappa\eta^{La}(Y)\big)^\f32}(1+\mathcal R_1(Y)).
 	\end{align*}
 This proves the first statement of (1) for any $Y\in\mathcal N^+$ by using \eqref{formula: A_1(1,Y)}.\smallskip

 	\no\textbf{Case 2. $Y\in\mathcal N.$} In this case,  there exists $Y_+\sim Y_c+2M \kappa^{-1}$ such that $\mathcal N^{+}=[Y_+,+\infty)$. Hence,
 	\begin{align*}
 		\int_Y^{+\infty}&Ai(e^{\mathrm i(\frac{\pi}{6}-\th_0)}\kappa\eta^{La}(Z))dZ=\int_{Y}^{Y_+}Ai(e^{\mathrm i(\frac{\pi}{6}-\th_0)}\kappa\eta^{La}(Z))dZ+\int_{\mathcal N^+}Ai(e^{\mathrm i(\frac{\pi}{6}-\th_0)}\kappa\eta^{La}(Z))dZ.
 	\end{align*}
 	By the property of the Airy function, we know that for any $Z\in[Y,Y_+]$,
 	\begin{align*}
 		\Big|Ai(e^{\mathrm i(\frac{\pi}{6}-\th_0)}\kappa\eta^{La}(Z))\Big|\leq C.
 	\end{align*}
 	 Notice that 
 	\begin{align*}
 		|Y_+-Y_c|\leq C|\varepsilon|^\f13M,
 	\end{align*}
 	which implies that for any $Y\in\mathcal N$,
 	\begin{align*}
 		|Y_+-Y|\leq |Y_+-Y_c|+|Y_c-Y|\leq C|\varepsilon|^{\f13}M.
 	\end{align*}
 	Then we obtain that for any $Y\in\mathcal N$,
 	\begin{align*}
 		\Big|\int_{Y}^{Y_+}Ai(e^{\mathrm i(\frac{\pi}{6}-\th_0)}\kappa\eta^{La}(Z))dZ\Big|\leq C|\varepsilon|^\f13.
 	\end{align*}
 	According to results in case 1, we have 
 	\begin{align*}
 		&\Big|\int_{\mathcal N^+}Ai(e^{\mathrm i(\frac{\pi}{6}-\th_0)}\kappa\eta^{La}(Z))dZ\Big|\\
		&\leq C\left|\Big(e^{\mathrm i(\frac{\pi}{6}-\th_0)}\kappa\eta^{La}(Y_+) \Big)^{-\f34}\Big(e^{\mathrm i(\frac{\pi}{6}-\th_0)}\kappa\partial_Y \eta^{La}(Y_+) \Big)^{-1}e^{-\f23\big(e^{\mathrm i(\frac{\pi}{6}-\th_0)}\kappa\eta^{La}(Y_+)\big)^\f32}\right|.
 	\end{align*}
 	By Lemma \ref{lem:est-eta}, we know that 
 	\begin{align}\label{eq:airy-langer-mid}
 		\left|\Big(e^{\mathrm i(\frac{\pi}{6}-\th_0)}\kappa\eta^{La}(Y_+) \Big)^{-\f34}\Big(e^{\mathrm i(\frac{\pi}{6}-\th_0)}\kappa\partial_Y \eta^{La}(Y_+) \Big)^{-1}\right|\leq C|\varepsilon|^\f13.
 	\end{align}
 	Since $|\mathrm{Im}(\kappa\eta^{La}(Y_+) )|=|\kappa\eta^{La}_i|<\delta_0$ and $|\mathrm{Re}(\kappa\eta^{La}(Y_+))|\sim M\gg1$, we know that 
 	\begin{align*}
 		\arg\left(\left(e^{\mathrm i(\frac{\pi}{6}-\th_0)}\kappa\eta^{La}(Y_+) \right)^{\f32} \right)\in [\frac{\pi}{8},\frac{3\pi}{8} ],
 	\end{align*}
 	which implies  
 	\begin{align*}
 		\mathrm{Re}(\big(e^{\mathrm i(\frac{\pi}{6}-\th_0)}\kappa\eta^{La}(Y_+)\big)^\f32)>0.
 	\end{align*}
 	Therefore, we have 
 	\begin{align*}
 		\left|e^{-\f23\big(e^{\mathrm i(\frac{\pi}{6}-\th_0)}\kappa\eta^{La}(Y_+)\big)^\f32}\right|=e^{-\f23\mathrm{Re}(\big(e^{\mathrm i(\frac{\pi}{6}-\th_0)}\kappa\eta^{La}(Y_+)\big)^\f32)}\leq 1,
 	\end{align*}
 	which along with \eqref{eq:airy-langer-mid} implies that 
 	\begin{align*}
 		\Big|\int_{\mathcal N^+}Ai(e^{\mathrm i(\frac{\pi}{6}-\th_0)}\kappa\eta^{La}(Z))dZ\Big|\leq C|\varepsilon|^\f13.
 	\end{align*}
 This shows that for any $Y\in\mathcal N$,
 	\begin{align*}
 		|A_1(1,Y)|\leq C|\varepsilon|^{-\f13}.
 	\end{align*}
	
 	\no\textbf{Case 3. $Y\in\mathcal N^-$.} We first point out that if $|\kappa Y_c|\leq M$, then $\mathcal N^-=\emptyset$. Hence, we only need to consider the case of $\kappa Y_c>M$. Then there exists $Y_->0$ such that $\mathcal N^-=[0,Y_-]$ and $Y_-\sim Y_c-2\kappa^{-1}M\ll1$. Hence, we can write that for any $Y\in\mathcal N^-$,
 	\begin{align*}
 		\int_Y^{+\infty}Ai(e^{\mathrm i(\frac{\pi}{6}-\th_0)}\kappa\eta^{La}(Z))dZ=&\int_{Y}^{Y_-}Ai(e^{\mathrm i(\frac{\pi}{6}-\th_0)}\kappa\eta^{La}(Z))dZ+\int_{\mathcal N}Ai(e^{\mathrm i(\frac{\pi}{6}-\th_0)}\kappa\eta(Z))dZ\\
 		&+\int_{\mathcal N^+}Ai(e^{\mathrm i\frac{\pi}{6}}\kappa\eta^{La}(Z))dZ.
 	\end{align*}
 	According to results for the above two cases, we know that 
 	\begin{align*}
 		\left|\int_{\mathcal N}Ai(e^{\mathrm i(\frac{\pi}{6}-\th_0)}\kappa\eta^{La}(Z))dZ\right|+\left|\int_{\mathcal N^+}Ai(e^{\mathrm i(\frac{\pi}{6}-\th_0)}\kappa\eta^{La}(Z))dZ\right|\leq C|\varepsilon|^{\f13}.
 	\end{align*}
 	On the other hand, we notice that for any $Y\in\mathcal N^-$,
 	\begin{align*}
 		\int_{Y}^{Y_-}&Ai(e^{\mathrm i(\frac{\pi}{6}-\th_0)}\kappa\eta^{La}(Z))dZ\\
 		=&\frac{1}{2\sqrt{\pi}}\int_Y^{Y_-} \Big(e^{\mathrm i(\frac{\pi}{6}-\th_0)}\kappa\eta^{La}(Z) \Big)^{-\f14}e^{-\f23\big(e^{\mathrm i(\frac{\pi}{6}-\th_0)}\kappa\eta^{La}(Z)\big)^\f32}(1+R(Z))dZ\\
			=&\frac{-1}{2\sqrt{\pi}}\int_Y^{Y_-}\Big(e^{\mathrm i(\frac{\pi}{6}-\th_0)}\kappa\eta^{La}(Z) \Big)^{-\f34}\Big(e^{\mathrm i(\frac{\pi}{6}-\th_0)}\kappa\partial_Z \eta^{La}(Z) \Big)^{-1}\partial_Z\Big(e^{-\f23\big(e^{\mathrm i(\frac{\pi}{6}-\th_0)}\kappa\eta^{La}(Z)\big)^\f32}\Big)dZ\\
			&+\frac{1}{2\sqrt{\pi}}\int_Y^{Y_-} \Big(e^{\mathrm i(\frac{\pi}{6}-\th_0)}\kappa\eta^{La}(Z) \Big)^{-\f14}e^{-\f23\big(e^{\mathrm i(\frac{\pi}{6}-\th_0)}\kappa\eta^{La}(Z)\big)^\f32}R(Z)dZ \\
			=&\frac{1}{2\sqrt{\pi}}\Big(e^{\mathrm i(\frac{\pi}{6}-\th_0)}\kappa\eta^{La}(Y) \Big)^{-\f34}\Big(e^{\mathrm i(\frac{\pi}{6}-\th_0)}\kappa\partial_Y \eta^{La}(Y) \Big)^{-1}e^{-\f23\big(e^{\mathrm i(\frac{\pi}{6}-\th_0)}\kappa\eta^{La}(Y)\big)^\f32}\\
			&-\frac{1}{2\sqrt{\pi}}\Big(e^{\mathrm i(\frac{\pi}{6}-\th_0)}\kappa\eta^{La}(Y_-) \Big)^{-\f34}\Big(e^{\mathrm i(\frac{\pi}{6}-\th_0)}\kappa\partial_Y \eta^{La}(Y_-) \Big)^{-1}e^{-\f23\big(e^{\mathrm i(\frac{\pi}{6}-\th_0)}\kappa\eta^{La}(Y_-)\big)^\f32}   \\
			&+\Bigg(\frac{1}{2\sqrt{\pi}}\int_Y^{Y_-}\partial_Z\Big[\Big(e^{\mathrm i(\frac{\pi}{6}-\th_0)}\kappa\eta^{La}(Z) \Big)^{-\f34}\Big(e^{\mathrm i(\frac{\pi}{6}-\th_0)}\kappa\partial_Z \eta^{La}(Z) \Big)^{-1}\Big]e^{-\f23\big(e^{\mathrm i(\frac{\pi}{6}-\th_0)}\kappa\eta^{La}(Z)\big)^\f32}dZ\\
			&+\frac{1}{2\sqrt{\pi}}\int_Y^{Y_-} \Big(e^{\mathrm i(\frac{\pi}{6}-\th_0)}\kappa\eta^{La}(Z) \Big)^{-\f14}e^{-\f23\big(e^{\mathrm i(\frac{\pi}{6}-\th_0)}\kappa\eta^{La}(Z)\big)^\f32}R(Z)dZ\Bigg)\\
			=&\frac{1}{2\sqrt{\pi}}\Big(e^{\mathrm i(\frac{\pi}{6}-\th_0)}\kappa\eta^{La}(Y) \Big)^{-\f34}\Big(e^{\mathrm i(\frac{\pi}{6}-\th_0)}\kappa\partial_Y \eta^{La}(Y) \Big)^{-1}e^{-\f23\big(e^{\mathrm i(\frac{\pi}{6}-\th_0)}\kappa\eta^{La}(Y)\big)^\f32}+I_2+I_3.
 	\end{align*}
 	On the other hand, we have 
 	\begin{align*}
 		I_2
 		=&\frac{1}{2\sqrt{\pi}}\Big(e^{\mathrm i(\frac{\pi}{6}-\th_0)}\kappa\eta^{La}(Y) \Big)^{-\f34}\Big(e^{\mathrm i(\frac{\pi}{6}-\th_0)}\kappa\partial_Y \eta^{La}(Y) \Big)^{-1}e^{-\f23\big(e^{\mathrm i(\frac{\pi}{6}-\th_0)}\kappa\eta^{La}(Y)\big)^\f32}\mathcal R_2(Y),
 	\end{align*}
 	where
 	\begin{align*}
 		\mathcal R_2(Y)=\Big(\frac{\kappa\eta^{La}(Y)}{\kappa\eta^{La}(Y_-)}\Big)^{\f34}\frac{\partial_Y \eta^{La}(Y)}{\partial_Y\eta^{La}(Y_-)}e^{-\f23\big(e^{\mathrm i(\frac{\pi}{6}-\th_0)}\kappa\eta^{La}(Y_-)\big)^\f32+\f23\big(e^{\mathrm i(\frac{\pi}{6}-\th_0)}\kappa\eta^{La}(Y)\big)^\f32}.
 	\end{align*}
 	By the definition of $\eta^{La}(Y)$ and Lemma \ref{lem:Airy-green-decay}, we have that for any $Y\in\mathcal N^-$
 i.e. $Y\in[0,Y_-]$, 
    \begin{align*}
    	\left|\mathcal R_2(Y)\right|\leq C|\kappa\eta^{La}(Y)|^\f34e^{-C_1\kappa(\eta^{La}(Y)-\eta(Y_-))(|\kappa\eta^{La}(Y)|^\f12+|\kappa\eta^{La}(Y_-)|^\f12)}.
    \end{align*} 	
    If $Y\in\mathcal N^-$ such that $|\eta^{La}(Y)|\leq 2|\eta^{La}(Y_-)|\leq 2|\varepsilon|^\f13M $, then we have 
    \begin{align*}
    	|\mathcal R_2(Y)|\leq C\leq C|\kappa\eta^{La}(Y_-)|^{-\f32}\leq C|\kappa\eta^{La}(Y)|^{-\f32}.
    \end{align*} 
    If $Y\in\mathcal N^-$ such that $|\eta^{La}(Y)|\geq 2|\eta^{La}(Y_-)|$, then we have 
    \begin{align*}
    	|\kappa\eta^{La}(Y)|^\f34e^{-C_1\kappa(\eta^{La}(Y)-\eta^{La}(Y_-))(|\kappa\eta^{La}(Y)|^\f12+|\kappa\eta^{La}(Y_-)|^\f12)}\leq& C|\kappa\eta^{La}(Y)|^\f34e^{-C_3|\kappa\eta^{La}(Y)|^\f32}\\
	\leq& C|\kappa\eta^{La}(Y)|^{-\f32}.
    \end{align*}
    Therefore, we obtain that for any $Y\in\mathcal N^-$, 
    \beno
    |\mathcal R_2(Y)|\leq C|\kappa\eta^{La}(Y)|^{-\f32}.
    \eeno
    
 	By a similar argument in case 1, we can show  that there exists $\mathcal R_3(Y)$ such that 
 	\begin{align}
 		\begin{split}
 		I_3
		=&\frac{1}{2\sqrt{\pi}}\Big(e^{\mathrm i(\frac{\pi}{6}-\th_0)}\kappa\eta^{La}(Y) \Big)^{-\f34}\Big(e^{\mathrm i(\frac{\pi}{6}-\th_0)}\kappa\partial_Y \eta^{La}(Y) \Big)^{-1}e^{-\f23\big(e^{\mathrm i(\frac{\pi}{6}-\th_0)}\kappa\eta^{La}(Y)\big)^\f32}\mathcal R_3(Y),
 		\end{split}
 	\end{align}
 	where 
 	\begin{align*}
 		|\mathcal R_3(Y)|\leq C|\kappa\eta^{La}(Y)|^{-\f32}.
 	\end{align*}
 	Therefore, we obtain that for any $Y\in\mathcal N^-$, 
 		\begin{align*}
 		A_1(1,Y)=&\frac{e^{-\th_0 i}}{2\sqrt{\pi}}e^{\mathrm i\frac{\pi}{3}}|\varepsilon|^{-\f13}U_s'(Y_c)^{-\f23}\partial_Y\eta^{La}(Y)^{-1}\Big(e^{\mathrm i(\frac{\pi}{6}-\th_0)}\kappa\eta^{La}(Y) \Big)^{-\f34}\\
		&\times e^{-\f23\big(e^{\mathrm i(\frac{\pi}{6}-\th_0)}\kappa\eta^{La}(Y)\big)^\f32}\big(1+\mathcal R_2(Y)+\mathcal R_3(Y)\big)
 	\end{align*}
 	with 
 	\begin{align*}
 		|\mathcal R_2(Y)|+|\mathcal R_3(Y)|\leq C|\kappa\eta^{La}(Y)|^{-\f32}.
 	\end{align*}
	
 	By gathering the results for the above three cases, we conclude that for any $Y\in \mathcal N^-\cup\mathcal  N^+$, there exists $\mathcal R(Y)$ such that 
 		\begin{align*}
 		A_1(1,Y)=&\frac{e^{-\th_0 i}}{2\sqrt{\pi}}e^{\mathrm i\frac{\pi}{3}}|\varepsilon|^{-\f13}U_s'(Y_c)^{-\f23}\partial_Y\eta^{La}(Y)^{-1}\Big(e^{\mathrm i(\frac{\pi}{6}-\th_0)}\kappa\eta^{La}(Y) \Big)^{-\f34}\\
		&\times e^{-\f23\big(e^{\mathrm i(\frac{\pi}{6}-\th_0)}\kappa\eta^{La}(Y)\big)^\f32}(1+\mathcal R(Y)),
 	\end{align*}
 	where 
 	\begin{align*}
 			|\mathcal R(Y)|\leq C|\kappa\eta^{La}(Y)|^{-\f32},
 	\end{align*}
 	and for any $Y\in\mathcal N$,
 	\begin{align*}
 		|A_1(1,Y)|\leq C|\varepsilon|^{-\f13}.
 	\end{align*}

 	The other parts of the lemma can be proved in a similar way. Let‘s leave the proof to the interested readers. 
 	\end{proof}

\begin{lemma}\label{lem:airy-0}
	 Let $|\varepsilon|,|c|,|\theta_0|\ll1\ll\kappa$ and $c_r>0$. Suppose $|\kappa\eta_i^{La}|<\delta_0$. Then there holds
	 \begin{align*}
    	&\left|\tilde{\mathcal A}(1,0)-\kappa^{-1}\mathcal A(1,\kappa\eta^{La}(0))\right|\leq C\kappa^{-1}(|c|+\kappa^{-1}),\\
		&\left|\tilde{\mathcal A}(2,0)-\kappa^{-2}\mathcal A(2,\kappa\eta^{La}(0))\right|\leq C\kappa^{-2}(\kappa^{-1}+|c|),
	\end{align*}
	and
	\begin{align*}
    	\left|\frac{\tilde{\mathcal A}(2,0)}{\tilde{\mathcal A}(1,0)}-\kappa^{-1}\frac{\mathcal A(2,\kappa\eta^{La}(0))}{\mathcal A(1,\kappa\eta^{La}(0))}\right|\leq C\kappa^{-1}(\kappa^{-1}+|c|).
    \end{align*}
	 Moreover, if $|\kappa\eta^{La}(0)|>M$, then we have
	\begin{align*}
    	\frac{\tilde{\mathcal A}(2,0)}{\tilde{\mathcal A}(1,0)}
    	=&-e^{\mathrm i(\f{\pi}{4}-\f{\theta_0}{2})}|\varepsilon|^\f12c_r^{-\f12}+o(|\varepsilon|^\f12c_r^{-\f12}+|\varepsilon|^{\f13}),
    \end{align*}
    and
	 \begin{align*}
    	\frac{Ai(e^{\mathrm i(\frac{\pi}{6}-\theta_0)}\kappa\eta^{La}(0))\tilde{\mathcal A}(2,0)}{\tilde{\mathcal A}(1,0)^2}=1+\mathcal O(|\varepsilon|^\f12c_r^{-\f32}).
    \end{align*}
    \end{lemma}
\begin{proof}
Recall that 
\begin{align*}
		\tilde{\mathcal A}(1,0)=&-A(0)\int_0^{+\infty}Ai(e^{\mathrm i(\frac{\pi}{6}-\theta_0)}\kappa\eta^{La}(Z))dZ\\
		=&-\int_0^{+\infty}Ai(e^{\mathrm i(\frac{\pi}{6}-\theta_0)}\kappa\eta^{La}(Z))\eta^{La}(Z)'(\eta^{La}(Z)')^{-1}dZ+\mathcal{O}(|c|^2)\\
		=&\kappa^{-1}\mathcal A(1,\kappa\eta^{La}(0))\partial_Y\eta^{La}(0)^{-1}-\kappa^{-1}\int_0^{+\infty}\mathcal A(1,\kappa\eta^{La}(Z))\frac{\partial_Z^2\eta^{La}(Z)}{\partial_Z\eta^{La}(Z)^2}dZ+\mathcal{O}(\kappa^{-2}).
		\end{align*}
		By Lemma \ref{lem:est-eta}, we have
    \begin{align*}
    	\mathcal A(1,\kappa\eta^{La}(0))\partial_Y\eta^{La}(0)^{-1}=\mathcal A(1,\kappa\eta^{La}(0))+\mathcal O(|c|),
    \end{align*}
    and
    \begin{align*}
    	\left|\int_0^{+\infty}\mathcal A(1,\kappa\eta^{La}(Z))\frac{\partial_Z^2\eta^{La}(Z)}{\partial_Z\eta^{La}(Z)^2}dZ\right|\leq C\int_0^{+\infty}|\mathcal A(1,\kappa\eta^{La}(Z))|dZ\leq C\kappa^{-1}.
    \end{align*}
    Therefore, we obtain 
    \begin{align*}
    	\left|\tilde{\mathcal A}(1,0)-\kappa^{-1}\mathcal A(1,\kappa\eta^{La}(0))\right|\leq C\kappa^{-1}(|c|+\kappa^{-1}).
    \end{align*} 
    
	By the definition of $\tilde{\mathcal A}(2,Y)$, we have
		\begin{align*}
		&\tilde{\mathcal A}(2,0)=-\int_0^{+\infty}\tilde{\mathcal A}(1,Z)dZ\\
		&\quad=-\kappa^{-1}\int_0^{+\infty}\mathcal A(1,\kappa\eta^{La}(Z))\partial_Z\eta^{La}(Z)^{-1}dZ+\kappa^{-1}\int_0^{+\infty}A(Z)\int_Z^{\infty}\mathcal A(1,\kappa\eta^{La}(\xi))\frac{\partial_\xi^2\eta^{La}(\xi)}{\partial_\xi\eta^{La}(\xi)^2}d\xi dZ\\
		&\qquad+m^2\kappa^{-1}\int_0^{+\infty}(U_s-c)^2\mathcal A(1,\kappa\eta^{La}(Z))\partial_Z\eta^{La}(Z)^{-1}dZ\\
		&\quad=\kappa^{-2}\mathcal A(2,\kappa\eta^{La}(0))\partial_Z\eta^{La}(0)^{-2}-2\kappa^{-2}\int_0^{+\infty}\mathcal A(2,\kappa\eta^{La}(Z))\frac{\partial_Z^2\eta^{La}(Z)}{\partial_Z\eta^{La}(Z)^3}dZ\\
		&\qquad+\kappa^{-1}\int_0^{+\infty}A(Z)\int_Z^{\infty}\mathcal A(1,\kappa\eta^{La}(\xi))\frac{\partial_\xi^2\eta^{La}(\xi)}{\partial_\xi\eta^{La}(\xi)^2}d\xi dZ+\mathcal{O}(\kappa^{-3}).
	\end{align*}
Notice that
	\begin{align*}
		&\mathcal A(2,\kappa\eta^{La}(0))\partial_Z\eta^{La}(0)^{-2}=\mathcal A(2,\kappa\eta^{La}(Y))+\mathcal O(|c|),\\
		&\left|\int_0^{+\infty}\mathcal A(2,\kappa\eta^{La}(Z))\frac{\partial_Z^2\eta^{La}(Z)}{\partial_Z\eta^{La}(Z)^3}dZ\right|\leq C\kappa^{-1},
	\end{align*}
	and 
	\begin{align*}
		&\left|\int_0^{+\infty}\int_Z^{\infty}\mathcal A(1,\kappa\eta^{La}(\xi))\frac{\partial_\xi^2\eta^{La}(\xi)}{\partial_\xi\eta^{La}(\xi)^2}d\xi dZ\right|\\
		&\quad\leq \left|\int_0^{Y_+}\int_Z^{\infty}\mathcal A(1,\kappa\eta^{La}(\xi))\frac{\partial_\xi^2\eta^{La}(\xi)}{\partial_\xi\eta^{La}(\xi)^2}d\xi dZ\right|+\left|\int_{Y_+}^{+\infty}\int_Z^{\infty}\mathcal A(1,\kappa\eta^{La}(\xi))\frac{\partial_\xi^2\eta^{La}(\xi)}{\partial_\xi\eta^{La}(\xi)^2}d\xi dZ\right|\\
		&\quad\leq C\kappa^{-1}\int_0^{Y^+}dZ+C\kappa^{-1}\int_{Y_+}^{+\infty}e^{-C|\kappa\eta^{La}(Y)|}dY\\
		&\quad\leq C\kappa^{-2}.
	\end{align*}
	Therefore, we obtain 
	\begin{align*}
		\left|\tilde{\mathcal A}(2,0)-\kappa^{-2}\mathcal A(2,\kappa\eta^{La}(0))\right|\leq C\kappa^{-2}(\kappa^{-1}+|c|),
	\end{align*}
   which implies 
    \begin{align*}
    	\left|\frac{\tilde{\mathcal A}(2,0)}{\tilde{\mathcal A}(1,0)}-\kappa^{-1}\frac{\mathcal A(2,\kappa\eta^{La}(0))}{\mathcal A(1,\kappa\eta^{La}(0))}\right|\leq C\kappa^{-1}(\kappa^{-1}+|c|).
    \end{align*}
    
    If $|\kappa\eta^{La}(0)|\geq M$, then by Lemma \ref{lem:pri-Airy-decay}, we first have 
    \begin{align*}
    	\frac{\mathcal A(2,\kappa\eta^{La}(0))}{\mathcal A(1,\kappa\eta^{La}(0))}=-e^{\mathrm i(\f{\pi}{6}-\theta_0)}\left(e^{\mathrm i(\f{\pi}{6}-\theta_0)}\kappa\eta^{La}(0)\right)^{-\f12}(1+\mathcal R(\kappa\eta^{La}(0))),
    \end{align*}
    where 
    \begin{align*}
    	|\mathcal R(\kappa\eta^{La}(0))|\leq CM^{-\f32}\ll1.
    \end{align*}
Notice that 
    \begin{align*}
    	\left(e^{\mathrm i(\f{\pi}{6}-\theta_0)}\kappa\eta^{La}(0)\right)^{-\f12}=&\left(-e^{\mathrm i(\f{\pi}{6}-\theta_0)}\kappa Y_c\right)^{-\f12}+\mathcal O(\kappa^{-\f12}|c|^\f12)\\
    	=&|\kappa c_r|^{-\f12}U'_s(0)^\f12e^{\mathrm i(\f{5\pi}{12}+\f{\theta_0}{2})}+\mathcal O(\kappa^{-\f12}|c|^\f12).
    \end{align*}
    Therefore, we obtain 
    \begin{align*}
    	\frac{\mathcal A(2,\kappa\eta^{La}(0))}{\mathcal A(1,\kappa\eta^{La}(0))}=-e^{\mathrm i(\f{\pi}{4}-\f{\theta_0}{2})}U_s'(0)^\f13|\varepsilon|^\f16c_r^{-\f12}+o(|\varepsilon|^\f16c_r^{-\f12}).
    \end{align*}
   This shows that  for the case $|\kappa\eta^{La}(0)|\geq M$,
    \begin{align*}
    	\frac{\tilde{\mathcal A}(2,0)}{\tilde{\mathcal A}(1,0)}=&\kappa^{-1}\frac{\mathcal A(2,\kappa\eta^{La}(0))}{\mathcal A(1,\kappa\eta^{La}(0))}+o(\kappa^{-1})\\
    	=&-e^{\mathrm i(\f{\pi}{4}-\f{\theta_0}{2})}|\varepsilon|^\f12c_r^{-\f12}+o(|\varepsilon|^\f12c_r^{-\f12}+\kappa^{-1}).
    \end{align*}
    
    According to the above arguments, we find that
    \begin{align*}
    	\frac{Ai(e^{\mathrm i(\frac{\pi}{6}-\theta_0)}\kappa\eta^{La}(0))\tilde{\mathcal A}(2,0)}{\tilde{\mathcal A}(1,0)^2}=\frac{Ai(e^{\mathrm i(\frac{\pi}{6}-\theta_0)}\kappa\eta^{La}(0))\mathcal A(2,\kappa\eta^{La}(0))}{\mathcal A(1,\kappa\eta^{La}(0))^2}+\mathcal O(|\varepsilon|^{\f13}+|c|).
    \end{align*}
   And by Lemma \ref{lem:pri-Airy-decay}, we have 
    \begin{align*}
    	\frac{Ai(e^{\mathrm i(\frac{\pi}{6}-\theta_0)}\kappa\eta^{La}(0))\mathcal A(2,\kappa\eta^{La}(0))}{\mathcal A(1,\kappa\eta^{La}(0))^2}=1+\mathcal O(|\kappa c_r|^{-\f32}).
    \end{align*}
    Therefore, we arrive at 
    \begin{align*}
    	\frac{Ai(e^{\mathrm i(\frac{\pi}{6}-\theta_0)}\kappa\eta^{La}(0))\tilde{\mathcal A}(2,0)}{\tilde{\mathcal A}(1,0)^2}=1+\mathcal O(|\varepsilon|^\f12c_r^{-\f32}).
    \end{align*}

 \end{proof}

For  $\eta_{out}^{La}$ defined by \eqref{def: eta_out^La}, it is easy to show the following lemma.

\begin{lemma}\label{lem:classical-Langer}
Suppose that $U_s$ satisfies the structure assumption \eqref{eq:S-A}. Then there exists a positive constant $L\leq 1$ such that 
if $0<|Y-Y_c|\leq L$, then
\begin{align*}
|\eta_{out}^{La}(Y)-(Y-Y_c)|\leq C|Y-Y_c|^2,\quad|\partial_Y\eta_{out}^{La}(Y)-1|\leq C|Y-Y_c|,\quad |\partial_Y^2\eta_{out}^{La}(Y)|\leq C,
\end{align*}	
and if $|Y-Y_c|\geq L$, then
\begin{align*}
&C^{-1}|Y-Y_c|^\f23\leq |\eta_{out}^{La}(Y)|\leq C|Y-Y_c|^\f23,\\
&C^{-1}|Y-Y_c|^{-\f13}\leq |\partial_Y\eta_{out}^{La}(Y)|\leq C|Y-Y_c|^{-\f13},\\
&|\partial_Y^2\eta_{out}^{La}(Y)|\leq C|Y-Y_c|^{-\f43}.
\end{align*}
\end{lemma}

\section*{Acknowledgments}
Y. Wang is supported by NSF of China under Grant 12101431. D. Wu is supported by NSF of China under Grant 12101245.  Z. Zhang is partially supported by  NSF of China  under Grant 12171010 and  12288101.

\end{document}